\documentclass[reprint, amsmath, amssymb, showpacs, floatfix, longbibliography, aps, twocolumn, superscriptaddress, nofootinbib]{revtex4-1}
\usepackage[utf8]{inputenc}
\usepackage{amsmath,amssymb,amsfonts,amsthm}
\usepackage{physics}
\usepackage{graphicx}
\usepackage[normalem]{ulem}
\usepackage{xcolor}
\usepackage{cancel}
\usepackage{color}
\usepackage[linktocpage]{hyperref}
\usepackage{appendix}
\hypersetup{colorlinks=true,citecolor=blue,linkcolor=blue, urlcolor=blue, breaklinks=true}
\usepackage{mathtools}
\usepackage{scalerel}
\usepackage{tikz}
\usepackage[eulergreek]{sansmath}
\usepackage{bm}
\usepackage{subfigure} 
\usepackage{environ}
\usepackage{url}
\usepackage{hyperref}
\usepackage{CJKutf8}
\graphicspath{{./Figures/}}

\makeatletter
\DeclareSymbolFont{bbmind}{U}{bbm}{m}{n}
\SetSymbolFont{bbmind}{normal}{U}{bbm}{m}{n}
\SetSymbolFont{bbmind}{bold}{U}{bbm}{m}{n}
\@ifundefined{mv@sans}{}{%
  \SetSymbolFont{bbmind}{sans}{U}{bbm}{m}{n} 
}
\DeclareMathSymbol{\mathbbone}{\mathord}{bbmind}{"31}
\makeatother

\newcommand{\bbone}{\mathbbone}

\newcommand{\NN}{{\mathbb N}}
\newcommand{\ZZ}{{\mathbb Z}}

\newcommand{\RR}{{\mathbb R}}
\newcommand{\R}{{\mathbb R}}
\newcommand{\CC}{{\mathbb C}}

\newcommand{\ba}{\boldsymbol{a}}
\newcommand{\bA}{\textbf{\textsf{A}}}
\newcommand{\bH}{\textbf{\textsf{H}}}
\newcommand{\balpha}{{\sansmath{\bm \alpha}}}
\newcommand{\bd}{{\sansmath{\bm \delta}}}
\newcommand{\bv}{\boldsymbol{v}}
\newcommand{\be}{\boldsymbol e}
\newcommand{\bface}{\boldsymbol{f}}
\newcommand{\bt}{\boldsymbol{t}}
\newcommand{\bc}{\boldsymbol{c}}

\newcommand{\bsig}{\boldsymbol{\sigma}}
\newcommand{\bZ}{\textbf{\textsf{Z}}}
\newcommand{\bO}{\textbf{\textsf{O}}}
\newcommand{\bX}{\textbf{\textsf{X}}}
\newcommand{\bU}{\textbf{\textsf{U}}}
\newcommand{\bG}{\textbf{\textsf{G}}}
\newcommand{\bW}{\textbf{\textsf{W}}}
\newcommand{\bM}{\textbf{\textsf{M}}}
\newcommand{\bB}{\textbf{\textsf{B}}}

\newcommand{\oZ}{\overline{Z}}
\newcommand{\oX}{\overline{X}}

\newcommand{\oA}{\overline{A}}
\newcommand{\tU}{\tilde{U}}

\newcommand{\eps}{\epsilon}

\newcommand{\bx}{{\overline{x}}}
\newcommand{\by}{{\overline{y}}}
\newcommand{\mX}{{\mathcal{X}}}
\newcommand{\mZ}{{\mathcal{Z}}}
\newcommand{\mS}{{\mathcal{S}}}

\def\CZ{{\mathcal Z}}
\def\CX{{\mathcal X}}

\newcommand{\lr}[1]{ \langle {#1} \rangle}
\newcommand*\colvec[3][]{
    \begin{pmatrix}\ifx\relax#1\relax\else#1\\\fi#2\\#3\end{pmatrix}
}
\newcommand{\Ext}{\mathrm{Ext}}

\NewEnviron{eqs}{%
\begin{equation}\begin{aligned}
    \BODY
\end{aligned}\end{equation}
}

\newcommand{\msquare}{{\mathord{{\scalerel*{\Box}{t}}}}}

\newcommand{\ihat}{\hat{\imath}}

\newtheorem{defn}{Definition} 

\newtheorem{lemma}{Lemma}
\newtheorem{corollary}{Corollary}

\newtheorem{prop}{Proposition}

\newtheorem{theorem}{Theorem}
\newcommand{\change}{\color{black}}

\begin{document}

\title{Clifford quantum cellular automata from topological quantum field theories and invertible subalgebras}

\author{\mbox{Meng Sun (\begin{CJK}{UTF8}{gkai}{孙萌}\end{CJK})}}
\thanks{These authors contributed equally to this work.}
\affiliation{International Center for Quantum Materials, School of Physics, Peking University, Beijing 100871, China}

\author{\mbox{Bowen Yang (\begin{CJK}{UTF8}{gkai}{杨博闻}\end{CJK})}}
\thanks{These authors contributed equally to this work.}
\affiliation{Center of Mathematical Sciences and Applications, Harvard University, Cambridge, Massachusetts 02138, USA}

\author{\mbox{Zongyuan Wang (\begin{CJK}{UTF8}{gkai}{王宗远}\end{CJK})}}
\thanks{These authors contributed equally to this work.}
\affiliation{International Center for Quantum Materials, School of Physics, Peking University, Beijing 100871, China}
\affiliation{Department of Physics and Institute for Quantum Information and Matter, California Institute of Technology, Pasadena, CA 91125, USA}

\author{Nathanan Tantivasadakarn}
\email[E-mail: ]{nathanan.tantivasadakarn@stonybrook.edu}
\affiliation{C. N. Yang Institute for Theoretical Physics, Stony Brook University, Stony Brook, NY 11794, USA}
\affiliation{Department of Physics and Institute for Quantum Information and Matter, California Institute of Technology, Pasadena, CA 91125, USA}
\affiliation{Walter Burke Institute for Theoretical Physics, California Institute of Technology, Pasadena, CA 91125, USA}

\author{\mbox{Yu-An Chen (\begin{CJK*}{UTF8}{bkai}{陳昱安}\end{CJK*})}}
\email[E-mail: ]{yuanchen@pku.edu.cn}
\affiliation{International Center for Quantum Materials, School of Physics, Peking University, Beijing 100871, China}

\date{\today}





\begin{abstract}
We present a general framework for constructing quantum cellular automata (QCA) from topological quantum field theories (TQFT) and invertible subalgebras (ISA) using the cup-product formalism. This approach explicitly realizes all $\mathbb{Z}_2$ and $\mathbb{Z}_p$ Clifford QCAs (for prime $p$) in all admissible dimensions, in precise agreement with the classification predicted by algebraic $L$-theory. We determine the orders of these QCAs by explicitly showing that finite powers reduce to the identity up to finite-depth quantum circuits (FDQC) and lattice translations. In particular,  we demonstrate that the $\mathbb{Z}_2$ Clifford QCAs in $(4l{+}1)$ spatial dimensions can be disentangled by non-Clifford FDQCs.  
Our construction applies beyond cubic lattices, allowing $\mathbb{Z}_2$ QCAs to be defined on arbitrary cellulations.
Furthermore, we explicitly construct invertible subalgebras in higher dimensions, obtaining $\mathbb{Z}_2$ ISAs in $2l$ spatial dimensions and $\mathbb{Z}_p$ ISAs in $(4l{-}2)$ spatial dimensions.
These ISAs give rise to $\mathbb{Z}_2$ QCAs in $(2l{+}1)$ dimensions and $\mathbb{Z}_p$ QCAs in $(4l{-}1)$ dimensions. {\change
We further prove that the QCAs in $3$ spatial dimensions constructed via TQFTs and ISAs are equivalent by identifying their boundary algebras, and show that this approach extends to higher dimensions.
}
Together, these results establish a unified and dimension-periodic framework for Clifford QCAs, connecting their explicit lattice realizations to field theories.  

\end{abstract}

\maketitle
\tableofcontents

\section{Introduction}

\begin{table*}[t]
    \centering
    \renewcommand{\arraystretch}{1.8}
    \begin{tabular}{|l|c|c|c|}
    \hline
    ~\begin{tabular}{cc}
        Spacetime \\
        dimension~$D$ 
    \end{tabular} 
    &
    \begin{tabular}{cc}
        Topological action in \\
        $H^D(K(\ZZ_2 \times \ZZ_2, D/2), \RR/\ZZ)$
    \end{tabular} 
    &
    \begin{tabular}{c}
        Trivializable by \\
        FDQC + shift
    \end{tabular}
    &
    \begin{tabular}{c}
        Order \\
        of QCA
    \end{tabular}
    \\
    \hline
    ~$4$ (Sec.~\ref{Sec. 4d Z2 QCA})  & $\frac{1}{2}(A_2\cup A_2+A_2\cup B_2+B_2\cup B_2$) & No*&2 \\
    \hline
    ~$6$ (Sec.~\ref{sec: w1w2^2}) & $\frac{1}{2}( A_3\cup A_3+A_3\cup B_3+B_3\cup B_3$) & Yes & 2\\
    \hline
    ~$4l$ (Sec.~\ref{Sec. even d Z2 QCA})  & $\frac{1}{2}(A_{2l}\cup A_{2l}+A_{2l}\cup B_{2l}+B_{2l}\cup B_{2l}$) & No* &2\\
    \hline
    ~$4l{+}2$  (Sec.~\ref{Sec. even d Z2 QCA})~  & ~$\frac{1}{2}(A_{2l+1}\cup A_{2l+1}+A_{2l+1}\cup B_{2l+1}+B_{2l+1}\cup B_{2l+1})$~ & Yes & 2\\
    \hline
    \end{tabular}
    \caption{\change Summary of $\mathbb{Z}_2$ Clifford QCAs in even spacetime dimensions. 
    The $3{+}1$D $\mathbb{Z}_2$ QCA corresponds to the 3-fermion Walker–Wang model~\cite{Walker2012TQFT}, which can be understood as a gauged $\mathbb{Z}_2 \times \mathbb{Z}_2$ 1-form symmetry-protected topological (SPT) phase. 
    The corresponding topological actions admit a direct generalization to higher dimensions, from which we construct the associated $\mathbb{Z}_2$ Clifford QCAs. 
    In spacetime dimensions $D = 4l$, these QCAs are conjectured to be nontrivial under non-Clifford finite-depth quantum circuits (FDQCs), based on the expectation that the corresponding anomalous boundary theories do not admit commuting-projector Hamiltonians.
    On the other hand, in $D = 4l{+}2$ they become trivial upon allowing non-Clifford circuits, since the term $\frac{1}{2} A_{2l+1} \cup A_{2l+1}$ is a coboundary. 
    Consequently, the classification of $\mathbb{Z}_2$ Clifford QCAs exhibits a periodicity of 2 in dimension, while allowing non-Clifford circuits reveals that the 3-fermion–type QCAs occur only every 4 dimensions.}
    \label{table: Z2 QCA}
\end{table*}

\begin{table*}[t]
    \centering
    \renewcommand{\arraystretch}{1.8}
    \begin{tabular}{|l|c|c|c|}
    \hline
    ~\begin{tabular}{cc}
        Spacetime \\
        dimension~$D$
    \end{tabular}
    &
    \begin{tabular}{cc}
        Topological action in \\
        $H^D(K(\ZZ_p, D/2), \RR/\ZZ)$
    \end{tabular}
    &
    \begin{tabular}{c}
        Trivializable by \\
        FDQC + shift
    \end{tabular}
    &
    \begin{tabular}{c}
        Order \\
        of QCA
    \end{tabular}
    \\
    \hline
    ~$4$ (Sec.~\ref{Sec. 3d Zp QCA})
    & $\dfrac{k}{p}\, A_2 \cup A_2$
    & No*
    & \begin{tabular}{l}
        $2$ ~if~ $p \equiv 1 \pmod{4}$ \\
        $4$ ~if~ $p \equiv 3 \pmod{4}$
      \end{tabular}
    \\
    \hline
    ~$4l$ (Sec.~\ref{sec:Zp_QCA_(4l-1)+1D_TQFT})
    & $\dfrac{k}{p}\, A_{2l} \cup A_{2l}$
    & No*
    & \begin{tabular}{l}
        $2$ ~if~ $p \equiv 1 \pmod{4}$ \\
        $4$ ~if~ $p \equiv 3 \pmod{4}$
      \end{tabular}
    \\
    \hline
    \end{tabular}
    \caption{\change Summary of $\mathbb{Z}_p$ Clifford QCAs in even spacetime dimensions. 
    These QCAs exist only in spacetime dimensions $D = 4l$ and are conjectured remain nontrivial even upon allowing non-Clifford FDQCs. 
    Depending on the residue of $p \bmod 4$, the order of the $\mathbb{Z}_p$ QCA is either $2$ or $4$.}
    \label{table: Zp QCA}
\end{table*}

Quantum cellular automata (QCAs) are locality-preserving automorphisms of operator algebras defined on lattices. By definition, a QCA maps any strictly local operator to another whose support increases by at most a finite radius~\cite{haah_QCA_23}. On a finite lattice, such automorphisms can always be realized by a unitary operator; however, in the thermodynamic limit, QCAs need not decompose into finite-depth quantum circuits.  
This distinction makes QCAs both mathematically appealing and physically significant. Mathematically, they provide a complete characterization of locality-preserving unitaries. Physically, they capture nontrivial structures in symmetries and many-body dynamics. QCAs have become central to the modern understanding of lattice symmetries, offering a natural framework for extending symmetry-protected topological phases beyond cohomology~\cite{kapustin2014symmetry, Fidkowski2020beyondcohomology, chen2023exactly, fidkowski2024qca, jones2025holography}. They also serve as a foundational tool in the homotopical study of higher symmetries and lattice 't Hooft anomalies~\cite{kobayashi2024generalized, shirley2025anomaly, tu2025anomalies, seifnashri2025disentangling, kapustin2025higher, kapustin2025higher2, Kapustin2025Anomalous, ma2024QCA, kawagoe2025anomaly}.
Moreover, QCAs have found wide-ranging applications, including discretized quantum field theories~\cite{QFT2}, the classification of Floquet phases~\cite{Zhang2023Floquet, PoChiral, PoRadical, PotterMorimoto17, PotterVishwanathFidkowski18, Zhang2021classification, Glorioso21}, tensor-network unitary operators~\cite{IgnacioCirac2017, Sahinoglu2018, GongSunderhaufSchuchCirac20, Piroli21Fermionic, Piroli2020}, and entanglement growth in quantum dynamics~\cite{GongPiroliCirac21, Ranard20, GongNahumPiroli21}.

A complete classification of QCAs is known only in one and two spatial dimensions. In these cases, QCAs are labeled by the rational Gross-Nesme-Vogts-Werner (GNVW) index, which quantifies the net flow of quantum information along a given direction~\cite{Gross2012GNVWindex}. Every one- or two-dimensional QCA can be expressed as a composition of finite-depth quantum circuits and lattice translations~\cite{Freedman2020ClassificationQCA}. In higher dimensions, the structure is considerably richer. Prominent examples include the 3-fermion QCA~\cite{haah_QCA_23}, which disentangles the time-reversal-protected topological phase in three spatial dimensions, along with generalizations constructed from chiral topological orders~\cite{Haah2021CliffordQCA, Shirley2022QCA}. Many of these higher-dimensional QCAs are \emph{Clifford} QCAs, mapping Pauli operators to Pauli operators.
The study of Clifford QCAs lies at the intersection of operator algebras, quantum information theory, and condensed matter physics, connecting algebraic $L$-theory~\cite{haah2025topological,yang2025categorifying}, fault-tolerant protocols in measurement-based quantum computation \cite{Williamson2024threefermion, Bombin2024unifyingflavorsof, bombin2023fault, Stephen2019subsystem}, and the construction of topological phases~\cite{Fidkowski2020beyondcohomology, Chen2023HigherCup}.

Despite significant progress, several fundamental issues remain unresolved. First, although a classification of Clifford QCAs exists~\cite{haah2025topological,yang2025categorifying}, no general procedure is known for constructing all non-trivial Clifford QCAs in arbitrary dimensions. Second, the algebraic tools for determining their order and composition laws remain incomplete. Third, existing constructions are restricted to (hyper-)cubic lattices and do not extend naturally to arbitrary triangulations or cellulations.

In this work, we address these challenges by constructing Clifford QCAs through two complementary approaches: one based on {\bf topological quantum field theories (TQFTs)} and the other on {\bf invertible subalgebras (ISAs)}. In the first approach, we discretize TQFT actions in the cup product formalism on cellulations and realize the corresponding topological phases on the lattice through Pauli stabilizer models. In the second approach, we generalize the ISA construction of Ref.~\cite{Haah2023InvertibleSubalgebras} to higher dimensions --- also using the cup product --- extending its applicability to arbitrary cellulations.

Using these frameworks, we construct explicit $\mathbb{Z}_2$ and $\mathbb{Z}_p$ Clifford QCAs in all spatial dimensions, where $p$ is an odd prime. The two approaches yield consistent results, and the resulting families of QCAs align precisely with the predictions of algebraic $L$-groups reported in Ref.~\cite{haah2025topological}, providing strong evidence that our construction captures all non-trivial Clifford classes. For each QCA, we further show that some finite power reduces to a finite-depth quantum circuit up to a lattice translation, thereby determining its exact order in the quotient group of QCAs modulo circuits and shifts. Moreover, we find that $\mathbb{Z}_2$ Clifford QCAs in $(4l{+}1)$ spatial dimensions can be disentangled by non-Clifford finite-depth circuits, whereas their $(4l{-}1)$ counterparts remain intrinsically non-trivial.
{\change
Tables~\ref{table: Z2 QCA} and~\ref{table: Zp QCA} summarize all $\mathbb{Z}_2$ and $\mathbb{Z}_p$ QCAs constructed in this work. 
The tables list the associated topological actions, indicate whether each QCA is trivialized under non-Clifford finite-depth quantum circuits with shifts, and specify the order of each QCA.

}

In summary, this paper develops a unified and dimension-periodic method for generating Clifford QCAs, establishes their correspondence with TQFT data, and clarifies their algebraic properties. We expect these results to provide both a conceptual advance in the classification of locality-preserving unitaries and a practical toolkit for designing Floquet protocols and fault-tolerant logical gates in quantum devices.

\subsection*{Organization of the paper}

This paper is organized as follows.  
Sec.~\ref{sec:cohomology_tools} reviews the necessary background on simplicial (co)homology, higher cup products, and the polynomial formalism, and introduces the notations and conventions used throughout this work.  
Sec.~\ref{sec:3d_examples} constructs the $3{+}1$D $\mathbb{Z}_2$ and $\mathbb{Z}_p$ QCAs from the TQFT approach, and determines their orders.  
In Sec.~\ref{sec:Z2_Zp_higher_QCA_TQFT}, we generalize the construction to $\mathbb{Z}_2$ QCAs in $(2l{+}1){+}1$D and to $\mathbb{Z}_p$ QCAs in $(4l{-}1){+}1$D, establishing their periodicities in spatial dimensions.  
Sec.~\ref{sec:ISA} develops the ISA framework, first in $2{+}1$D and then in higher dimensions, including $\mathbb{Z}_2$ ISAs in $(2l{+}1)$D and $\mathbb{Z}_p$ ISAs in $(4l{-}2){+}1$D, which is consistent with the TQFT construction.
{\change
In Sec.~\ref{Sec: Equivalence of QCAs Constructed from TQFTs and ISAs}, we rigorously establish the equivalence between the $3{+}1$D QCAs obtained from the TQFT and ISA constructions by explicitly comparing their associated skew-Hermitian forms. We also present a cup-product analysis of the boundary algebra, which admits a natural extension to higher dimensions.
}
Sec.~\ref{sec: Algebraic formalism for Clifford QCA} reformulates Clifford QCAs in terms of band-diagonal symplectic matrices, connects them to bounded algebraic $K$-theory of Pedersen and Weibel, and demonstrates that two Clifford QCAs are equivalent when they disentangle the same stabilizer group. We anticipate that this formalism may be useful in future work.

Technical details and background material are collected in the appendices. Appendix~\ref{app: Review of the Laurent polynomial formalism} reviews the Laurent polynomial formalism, and Appendix~\ref{app:highercuphypercube} summarizes higher cup products on hypercubes.
{\change
Appendix~\ref{app: Explicit matrices for the 3-fermion-type QCA} lists the matrix representations of the $3$-fermion-type QCAs in $3{+}1$D and $5{+}1$D.
}

By providing explicit constructions that are general across dimensions, together with a comprehensive algebraic analysis, we aim for this work to serve both as a practical toolkit for quantum information applications, such as Floquet engineering in interacting many-body systems, and as a step toward a unified classification of locality-preserving unitaries.

\section{Notation}\label{sec:cohomology_tools}

Our constructions are formulated using topological quantum field theories (TQFTs) and invertible subalgebras (ISAs), both of which rely on higher cup products and the Laurent polynomial ring. This section reviews the notational conventions used throughout the main text.

We begin with a brief overview of simplicial cohomology, following the conventions of Refs.~\cite{Chen2021Disentangling, Chen2023HigherCup}. We define cochains, the coboundary operator, and various cup products, including their higher versions. For extensions to group and higher-group cohomology, we refer the reader to Refs.~\cite{chen2012symmetry, Kapustin2017Higher}.
We then introduce a new compact Einstein summation notation tailored for the Laurent polynomial formalism of Pauli stabilizer codes. The latter was developed in Refs.~\cite{Schlingemann2008structure, Gutschow2010Clifford, haah_module_13, haah2016algebraic}, and our summation notation can be considered an extension of that. A detailed review of the polynomial formalism is provided in Appendix~\ref{app: Review of the Laurent polynomial formalism}.

\subsection{Simplicial Cohomology} \label{app:terminology}

Consider a simplicial complex or cellulation of a manifold $M$. We may label an $n$-dimensional simplex/cell as $\sigma_n$. In low dimensions, we will also call $\sigma_0 \equiv v$ as vertices, $\sigma_1 =e$ as edges , $\sigma_2 =f$ as faces, and $\sigma_3 = t$ as tetrahedron or $\sigma_3=c$ as cubes. In general, a $p$-simplex is denoted by vertices as $\langle 0, 1, 2,\ldots, p \rangle$. A $p$-chain is a formal sum of $p$-simplices in $M$ with coefficients in the ring $R$, where $R$ is typically $\ZZ$ or $\ZZ_N$. The $p$-chains form the basis of the chain group $C_p(M,R)$. That is, denoting an arbitrary $p$-chain as $c_p$, we have $c_p = \sum_{\sigma_p} a_{\sigma_p} \sigma_p$ with $a_{\sigma_p} \in R$.

Consider the boundary map 
\begin{equation}
    \partial: C_p(M, R) \to C_{p-1}(M, R).
\end{equation}
which maps a $p$-simplex to the sum of its boundary simplices in one lower dimension while taking the orientation into account. Specifically, for the oriented $p$-simplex $\langle 0, 1, \dots, p \rangle$, its boundary is defined as 
\begin{equation}
    \partial \langle 0, 1, \dots, p \rangle = \sum_{i=0}^p ~(-1)^i ~\langle 0,\dots, \hat{i},\dots,p \rangle,
\end{equation}
where $\langle 0,\dots, \hat{i},\dots,p \rangle$ denotes the $i$th face of $\langle 0, 1, \dots, p \rangle$, obtained by deleting the vertex $i$.

A \(p\)-cochain on \(M\) is an $R$-valued linear function acting on \(p\)-chains. The set of all \(p\)-cochains on \(M\) is \(C^p(M, R)\). We use boldface symbols for cochains on \(M\), e.g., \(\boldsymbol{c} \in C^p(M, R)\).
Each simplex can be associated with a dual basis element. For instance, a 0-cochain \(\boldsymbol{v}\) is dual to a vertex \(v\) via
\begin{align}
    \boldsymbol{v}(v') = 
    \begin{cases}
      1 & \text{if } v' = v,\\
      0 & \text{otherwise}.
    \end{cases}
    \label{example: chain cochain correspondance}
\end{align}
Similarly, \(\boldsymbol{e}\) is a 1-cochain dual to an edge \(e\), and \(\boldsymbol{f}\) is a 2-cochain dual to a face \(f\) and so forth. The cochain dual to $c_p= \sum_{\sigma_p} a_{\sigma_p} \sigma_p$ is $\boldsymbol{c}_p = \sum_{\sigma_p} a_{\sigma_p} \boldsymbol{\sigma}_p$.

Dual to the boundary operator is the coboundary operator \(\delta\), which maps a \(p\)-cochain to a \((p+1)\)-cochain,
\begin{align}
    \delta: C^p(M, R) \to C^{p+1}(M,R).
\end{align}
For a \(p\)-cochain \(\boldsymbol{c}\) and a \((p+1)\)-simplex \(s\), we define $\delta$ via $\partial$ as
\begin{align} \label{coboundarydef1}
    \delta \boldsymbol{c}(s) = \boldsymbol{c}(\partial s).
\end{align}
The cup product \(\cup\) combines a \(p\)-cochain and a \(q\)-cochain to produce a \((p+q)\)-cochain:
\begin{align}
    \cup: C^p(M, R) \times C^q(M, R) \to C^{p+q}(M, R).
\end{align}
For instance, for \(\boldsymbol{c} \in C^p(M, \mathbb{Z}_2)\) and \(\boldsymbol{d} \in C^q(M, \mathbb{Z}_2)\), $\boldsymbol{c} \cup \boldsymbol{d}$ evaluates on a \((p+q)\)-simplex \(\langle 0, 1, 2,\ldots, p+q \rangle\) as
\begin{equation}
\begin{aligned}
    &\bigl(\boldsymbol{c} \cup \boldsymbol{d}\bigr)\bigl(\langle 0, 1,\ldots, p+q \rangle\bigr) \\
    & \quad\quad= \boldsymbol{c}\bigl(\langle 0, 1,\ldots, p \rangle\bigr)\,\boldsymbol{d}\bigl(\langle p, p+1, \ldots, p+q \rangle\bigr).
\end{aligned}
\end{equation}
The coboundary operator satisfies the derivation (Leibniz) rule:
\begin{align} \label{leibnizrule}
    \delta (\boldsymbol{c} \cup \boldsymbol{d})
    = \delta \boldsymbol{c} \cup \boldsymbol{d} + (-1)^p\boldsymbol{c} \cup \delta \boldsymbol{d},
\end{align}
for any $p$-cochain $\boldsymbol{c}$ and $q$-cochain $\boldsymbol{d}$.

We can similarly define higher cup products. A cup-$i$ product $\cup_i$ takes a $p$-cochain and a $q$-cochain to a \((p+q-i)\)-cochain:
\begin{align}
    \cup_i : C^p(M, R) \times C^q(M, R) \to C^{p+q-i}(M,R).
\end{align}
Closed-form expressions for higher cup products on simplicial complexes are given in Ref.~\cite{Steenrod1947Products, Gaiotto2015SpinTQFTs, Chen2020Exactbosonization, lokman2020lattice}.
They satisfy the following recursive relation
\begin{equation}\label{highercupleibniz}
\begin{aligned} 
    \delta (\boldsymbol{c} \cup_i \boldsymbol{d}) 
    =& \delta \boldsymbol{c} \cup_i \boldsymbol{d} 
    + (-1)^p\boldsymbol{c} \cup_i \delta \boldsymbol{d} 
    + \\
    &(-1)^{p+q-i}\boldsymbol{c} \cup_{i-1} \boldsymbol{d} 
    + (-1)^{pq+p+q}\boldsymbol{d} \cup_{i-1} \boldsymbol{c},
\end{aligned}
\end{equation}
where $\cup_0 \equiv \cup$. Next, we give a few examples.
\begin{widetext}
For the cup–1 product we have
\begin{eqs}
    \bigl(\boldsymbol{c}_p \cup_1 \boldsymbol{d}_q\bigr)\bigl(\langle 0, \ldots, p+q-1 \rangle\bigr) 
    &= \sum_{i=0}^{p-1} (-1)^{(p-i)(q+1)}\boldsymbol{c}_p\bigl(\langle 0, \ldots, i, q+i, \ldots, p+q-1 \rangle\bigr)~
    \boldsymbol{d}_q\bigl(\langle i, \ldots, q+i \rangle\bigr) \\
    &:= \sum_{i=0}^{p-1} (-1)^{(p-i)(q+1)} \boldsymbol{c}_p\bigl(\langle 0 \rightarrow i,~ q+i \rightarrow p+q-1 \rangle\bigr)~
    \boldsymbol{d}_q\bigl(\langle i \rightarrow q+i \rangle\bigr).
\end{eqs}

For the cup–2 product (non-zero only when $2\le p,q$), we obtain
\begin{eqs}
    \bigl(\boldsymbol{c}_p \cup_2 \boldsymbol{d}_q\bigr)\bigl(\langle 0, \ldots, p+q-2 \rangle\bigr)
    &=\sum_{0 \le i_1 < i_2 \le p+q-2}
    (-1)^{(p-i_1)(i_2-i_1-1)}~
    \boldsymbol{c}_p (\langle 0\rightarrow i_1, ~i_2 \rightarrow p+i_2-i_1-1 \rangle)\\
    & \qquad \qquad \boldsymbol{d}_q (\langle i_1\rightarrow i_2, ~p+i_2-i_1-1 \rightarrow p+q-2\rangle).
\end{eqs}

For a general integer $k$ with $0\le k\le\min\{p,q\}$, the cup-$k$ product is
\begin{eqs}
    &\bigl(\boldsymbol{c}_p \cup_k \boldsymbol{d}_q \bigr) \bigl(\langle0,1,\dots ,p+q-k\rangle\bigr) \\
    &= \sum_{0\le i_1<\cdots<i_k\le p+q-k}
    (-1)^{\mathsf{sgn}}  \boldsymbol{c}_p \bigl( \langle 0 \rightarrow i_{1}, i_{2} \rightarrow i_{3},~ \dots \rangle \bigr)
    \boldsymbol{d}_q \bigl( \langle i_{1} \rightarrow i_{2},\; i_{3} \rightarrow i_{4},~ \dots \rangle \bigr),
\end{eqs}
where $\mathsf{sgn}(i_{1},\dots ,i_{k})$ counts the permutations required
to rearrange the interleaved sequence
\[
  0\rightarrow i_{1},\; i_{2}\rightarrow i_{3},\;\dots;\;
  i_{1}+1\rightarrow i_{2}-1,\; i_{3}+1\rightarrow i_{4}-1,\;\dots
\]
into the natural order $0\rightarrow p+q-k$.
\end{widetext}

Finally, as a convenient shorthand, we write
\begin{align}
    \int_{N_p} \boldsymbol{c}_p := \sum_{s_p \in N_p} \boldsymbol{c}_p (s_p),
\end{align}
where \(N_p\) is a \(p\)-dimensional manifold.
When it is clear from the context, we will suppress the subscripts and write $\int_{N_p} \boldsymbol{c}_p$ as $\int \boldsymbol{c}$.

\subsection{Einstein notation for linear maps}

We introduce a new notation reminiscent of the Einstein notation that will be useful in bookkeeping many linear maps in cohomology. Readers who prefer to first see the concrete setup may proceed directly to later sections and refer back to this subsection as needed.

First, in this notation, because chains and cochains are dual, we will not use the bold symbol to distinguish cochains from chains to avoid cluttering expressions. Second, the linear maps in this notation will be denoted in bold sans serif font to distinguish it from the bold operators denoting cochains.

In the basis of simplices, the coboundary operator $\delta: C^p \rightarrow C^{p+1}$ may be represented as a linear map, with matrix elements $\delta_{\bsig_{p+1},\bsig_{p}}$. Our new notation will be
\begin{align}
    \bd_{\sigma_{p+1}, \delta \sigma_{p}} := \delta_{\bsig_{p+1},\bsig_{p}}.
\end{align}
The notational convenience of introducing $\bd$ is that we will no longer treat it as the coboundary operator, but as if it is a Kronecker delta. Thus, the above notation means that it is a matrix setting $\sigma_{p+1}$ equal to $\delta \sigma_{p}$.

The canonical pairing of chains and cochains allows us to lift the action of the boundary operator on chains, into an action on cochains $\delta^\dagger$, where $\dagger$ denotes the matrix transpose (Later the when using the polynomial formalism, e.g. in Sec.~\ref{sec:Symplectic_Rep_Pauli}, $\dagger$ will denote the transpose followed by the antipode.). Thus, as a slight abuse of notation, we will take $\partial = \delta^\dagger$ when acting on cochains. With this, we may also define
\begin{align}
  \bd_{ \sigma_{p}, \partial \sigma_{p+1}} =   \bd_{\delta \sigma_{p}, \sigma_{p+1}}  :=   (\bd^\dagger)_{\sigma_{p+1}, \delta \sigma_{p}},
\end{align}
which intuitively means that we may move the $\delta$ from one side to $\partial$ on the other side and vice versa. In this notation, the (co)chain complex condition reads:
\begin{align}
\label{eq:complexcondition}
\bd_{\sigma_{p+1},\delta \sigma_{p}} \bd_{ \sigma_{p},\delta  \sigma_{p-1}} &=0, & \bd_{\partial  \sigma_{p-1}, \sigma_{p}} \bd_{\partial  \sigma_{p}, \sigma_{p+1}} &=0.
\end{align}
where in the above, we implicitly sum over the repeated ``index" $\sigma_{p}$. 

Next, for any map $Q: C^q \rightarrow C^p$ with matrix elements $\textbf{\textsf{Q}}_{\sigma_p,\sigma_q}$ we may define the matrix with ``contracted" indices as follows:
\begin{align}
 \textbf{\textsf{Q}}_{\sigma_p,\delta \sigma_{q-1}} := \textbf{\textsf{Q}}_{\sigma_p,\sigma_{q}} \bd_{\sigma_{q},\delta \sigma_{q-1}},\\
 \textbf{\textsf{Q}}_{\sigma_p,\partial \sigma_{q+1}} := \textbf{\textsf{Q}}_{\sigma_p,\sigma_{q}} \bd_{\sigma_{q},\partial \sigma_{q+1}},\\
 \textbf{\textsf{Q}}_{\delta \sigma_{p-1},\sigma_q} := \bd_{\delta \sigma_{p-1},\sigma_p} \textbf{\textsf{Q}}_{\sigma_p,\sigma_q},\\
 \textbf{\textsf{Q}}_{\partial \sigma_{p+1},\sigma_q} := \bd_{\partial \sigma_{p+1},\sigma_p} \textbf{\textsf{Q}}_{\sigma_p,\sigma_q}.
\end{align}
The above notation allows us to simplify expressions by performing a sum over the repeated index. In particular, this also applies to the matrices $\bd$ themselves. For example, by summing over $\sigma_p$ using the Kronecker delta, in the complex condition Eq.~\eqref{eq:complexcondition}, we may write
\begin{align}
\bd_{\sigma_{p+1},\delta \sigma_p} \bd_{\sigma_p,\delta \sigma_{p-1}} =\bd_{\sigma_{p+1},\delta^2 \sigma_{p-1}} =0,
\end{align}
where we ``substituted" $\sigma_p$ for $\delta\sigma_{p-1}$ and used the fact that $\delta^2=0$.

We now turn to the linear map corresponding to cup products. Because the cup-$i$ product is a pairing between $p$-cochains and $d-p+i$-cochains, for each $p=1, \ldots , d$, we can use it to define multiple linear maps. For each $p$, we define $M^{(p,i)}: C^p \rightarrow C^{d-p+i}$, which acts as
\begin{align}
    M^{(p,i)} (\boldsymbol{c}_p) =\sum_{\boldsymbol{\sigma}_{d+i-p}}  \left( \int  \boldsymbol{\sigma}_{d+i-p} \cup_i \boldsymbol{c}_{p}\right) \boldsymbol{\sigma}_{d+i-p}.
\end{align}
Said differently, in the basis of $\boldsymbol{\sigma}_p \in C^p$ and $\boldsymbol{\sigma}_{d+p-i} \in C^{d+p-i}$, the matrix element is exactly
\begin{align}
    M^{(p,i)} _{\boldsymbol{\sigma}_{d+p-i},\boldsymbol{\sigma}_p} = \int  \boldsymbol{\sigma}_{d+p-i} \cup_i \boldsymbol{\sigma}_{p}.
\end{align}
We note that an analogous map was defined in Ref.~\cite{fidkowski2024qca} in order to obtained higher dimensional analogues of the 3F QCA. Here, we slightly improve on the notation of the matrix as follows. We will denote in Einstein notation the matrix elements as
\begin{align}
  \bM_{\sigma_{d+p-i} \cup_i \sigma_p} := M^{(p,i)}_{\boldsymbol{\sigma}_{d+p-i},\boldsymbol{\sigma}_p},
\end{align}
where we opt to use the cup symbol in lieu of the comma for the following aesthetic reason: any relation between cup products automatically gives rise to similar looking relations between these linear maps. For example, starting with the Leibniz rule for higher cup products Eq.~\eqref{highercupleibniz}, performing the integral and matching the terms with the matrix elements gives rise to the following identity
\begin{eqs}
\label{eq:highercuprelationmatrix}
  &\bM_{\delta \sigma_{q} \cup_i \sigma_p} + (-1)^p\bM_{\sigma_{q} \cup_i \delta \sigma_{p}}  \\
  &+ (-1)^{p+q-i}\bM_{\sigma_{q} \cup_{i-1} \sigma_{p}} +  (-1)^{pq+p+q}\bM_{\sigma_{p}\cup_{i-1}  \sigma_{q} }^\dagger=0 ,
\end{eqs}
where $q = d-p+i-1$. In particular, for $d=3$ we have
\begin{eqs}
     - \bM_{\delta e\cup_1 f} + \bM_{e\cup_1 \delta f} - \bM_{e\cup f} + \bM_{f\cup e}^\dagger   &=0,\\
  \bM_{\delta f'\cup_2 f} + \bM_{f'\cup_2 \delta f}+ \bM_{f\cup_1 f'} +\bM_{f'\cup_1 f}^\dagger  &=0.
\end{eqs}
Lastly, one can show that these matrices satisfy
\begin{align}
   \bM_{\partial \sigma_{d-p+2} \cup \partial \sigma_p}  =0.
\end{align}

\subsection{Symplectic representation of Pauli matrices}
\label{sec:Symplectic_Rep_Pauli}

We will now apply this formalism to the symplectic representation of Pauli matrices. See Appendix~\ref{app: Review of the Laurent polynomial formalism} for a review of the symplectic representation. We place $d$-dimensional qudits on each $p$-dimensional simplex $\sigma_p$. The total Hilbert space is then $\bigotimes_{|C^p|} \CC^{d}$. The generalized Pauli matrices can be represented in symplectic form as 
 \begin{align}
    \bZ_{\sigma_p} &= \begin{pmatrix}
        0\\
      \bbone
          \end{pmatrix},
          &  \bX_{\sigma_p} &= \begin{pmatrix}
       \bbone\\
       0
    \end{pmatrix},
\end{align} 
where $\bbone$ is the $|C^p|\times |C^p|$ identity matrix. Let us restrict our attention for now to $d=3$ and consider qudits placed on faces $f= \sigma_2$. Using the index notation, we may write this as
\begin{align}
    \bZ_{f} &= \begin{pmatrix}
        0\\
      \bd_{f_0,f}
          \end{pmatrix},
          &  \bX_f &= \begin{pmatrix}
       \bd_{f_0,f}\\
      0
    \end{pmatrix},
\end{align}
where $f_0$ denotes a dummy index for the rows of $\bZ_{f}$ and $\bX_f$, and $\bd_{f_0,f}$ is, as expected, the Kronecker delta representing the identity matrix. The symplectic inner product in this notation corresponds to summing over the index $f_0$
\begin{align}
   \langle \bX_f,\bZ_{f'}\rangle = \sum_{f_0}\bX_f^\dagger \bZ_{f'} = \bd_{f,f_0}\bd_{f_0,f'} = \bd_{f,f'}.
\end{align}

As a further warmup of this notation, we may define the stabilizers of the 3d toric code as
\begin{equation}
\begin{aligned}
    &\bW_c := \bZ_{\partial c} = \begin{pmatrix}
        0\\
     \bd_{f_0,\partial c}
     \end{pmatrix}, \\
      &\bB_e  := \bX_{\delta e} = \begin{pmatrix}
        \bd_{f_0,\delta e}\\
     0
     \end{pmatrix} .
\end{aligned}
\end{equation}
Indeed, we can verify that they commute
\begin{align}
    \langle \bB_e,\bW_c \rangle = \bd_{\partial e,f} \langle \bX_f,\bZ_{f'} \rangle \bd_{f',\partial c} = \bd_{e,\partial f}\bd_{f,\partial c}=0.
\end{align}

In specific constructions, we'll often meet terms like $\prod_{\sigma_p}P_p^{\int \boldsymbol{\sigma}_p\cup_i \boldsymbol{\sigma}_q}$, where $P_p$ is a Pauli operator $Z$ or $X$ living on a p-simplex $\sigma_p$. The $\boldsymbol\sigma_p$ and $\boldsymbol\sigma_q$ in the integral denote the cochains corresponding to the simplexes $\sigma_p$ and $\sigma_q$ according to Eq.~\eqref{example: chain cochain correspondance}. In our Einstein notation, we can express it as $P_p \times \bM_{\sigma_{p}\cup_i \sigma_{q}}$. Likewise, the term $P_q^{\int \boldsymbol\sigma_p\cup_i \boldsymbol\sigma_q}$ can be represented by $P_q \times \bM^\dagger_{\sigma_p\cup_i \sigma_q}$. Note that the dagger is required for the index to contract correctly in the latter case.

\subsection{Einstein notation in polynomial formalism}

In the case of a hypercubic lattice with translation invariance, the notation can be further simplified using a polynomial representation. More specifically, the matrices become circulant matrices, which can be represented by their associated polynomials~\cite{haah_module_13, panteleev2021quantum}. We denote $x_i$ as the translation in the $i^\text{th}$ direction, with periodic boundary conditions $x_i^{L_i}=1$ for integers $L_i$. The polynomials are elements of the ring $R[x_1,\ldots,x_d]/(x_1^{L_1}-1, \ldots, x_d^{L_d}{-}1)$. In lower dimensions, we use $x_1=x$, $x_2=y$, and $x_3=z$. A detailed review is given in Appendix~\ref{app: Review of the Laurent polynomial formalism}. The polynomial formalism has been used to construct QCAs in three and higher dimensions in Ref.~\cite{fidkowski2024qca}. In this work, we develop the framework systematically and introduce new notations that allow us to express the resulting QCAs more compactly.

Fig.~\ref{fig:example_poly} illustrates how Pauli operators on a two-dimensional square lattice can be represented by a $4 \times 1$ polynomial matrix. The first and second entries correspond to Pauli-$X$ operators pointing in the $x$- and $y$-directions, while the third and fourth entries correspond to Pauli-$Z$ operators in each direction. The entries of the matrix depend on the choice of base point, which in Fig.~\ref{fig:example_poly} is taken to be the origin $(0,0)$. Shifting the base point corresponds to multiplying the matrix representation by an overall monomial factor.

\begin{figure}
    \centering
    \hspace{2cm}
    \includegraphics[width=\linewidth]{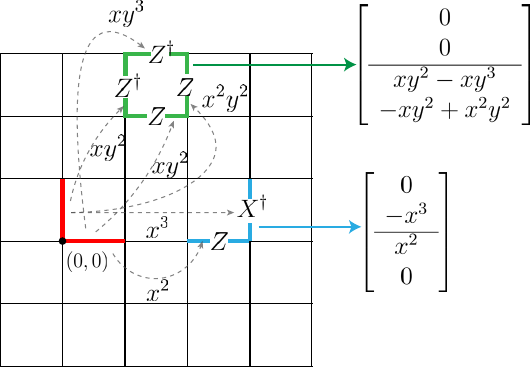}
    \caption{Examples of polynomial expressions for Pauli operators adapted from Ref.~\cite{liang2023extracting}. The flux operator on a plaquette and the $XZ$ operator on edges are shown. Red edges mark the chosen base edge, while all other edges follow from it by lattice translations. Monomials such as $x^2 y^2$ and $x^2$ specify operator locations relative to the origin. A detailed review is provided in Appendix~\ref{app: Review of the Laurent polynomial formalism}.
    }
    \label{fig:example_poly}
\end{figure}

\begin{figure*}
    \centering
    \includegraphics[width=0.9\textwidth]{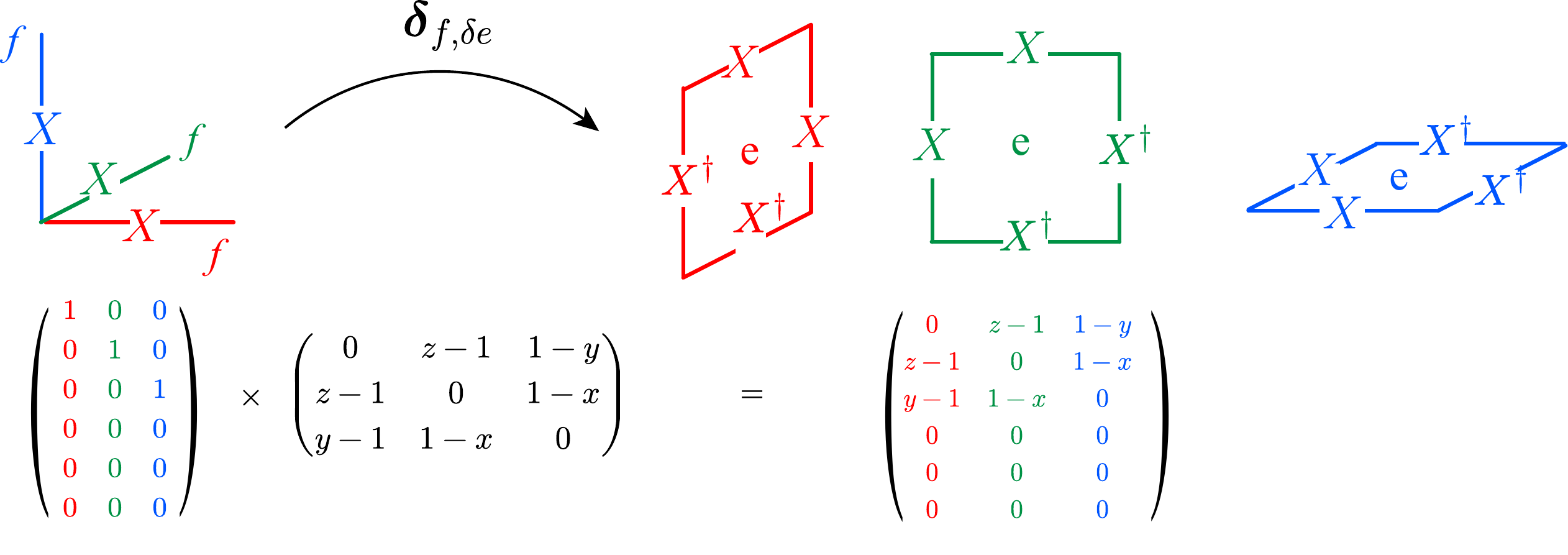}
    \caption{
    An illustration on the dual lattice showing the construction of the flux terms $B_e$ via the action of the coboundary matrix $\bd_{f,\delta e}$ on the vector of operators $X_f$. A Pauli operator $X_f$ is supported on each face $f$, and the product $X_f \bd_{f,\delta e}$ yields the flux term $X_{\delta e} = B_e$. The three colored plaquettes correspond to the $yz$, $xy$, and $xy$ orientations, with the associated polynomial expressions shown below. }
    \label{fig:example how delta matrix act on Pauli}
\end{figure*} 

Let us consider the three-dimensional cubic lattice in particular. The cochain complex takes the form of a tensor product of one-dimensional complexes and expands as
\begin{equation}
\begin{aligned}
C^0  &\xrightarrow{ \delta_{e,v}} C^1_x \oplus  C^1_y \oplus  C^1_z \\ 
&\xrightarrow{\delta_{f,e}} C^2_{yz}  \oplus C^2_{xz} \oplus C^2_{xy} \\
&\xrightarrow{\delta_{c,f}} C^3 ,
\end{aligned}
\end{equation}
where $C^1_x$ denotes edges oriented in the $x$ direction, $C^2_{yz}$ denotes faces oriented along the $yz$ plane, and similarly for the other orientations.

Likewise, the corresponding chain complex expands as
\begin{equation}
\begin{aligned}
C^0  &\xleftarrow{ \partial_{v,e}} C^1_x \oplus  C^1_y \oplus  C^1_z \\
&\xleftarrow{\partial_{e,f}} C^2_{yz}  \oplus C^2_{xz} \oplus C^2_{xy}  \\ 
&\xleftarrow{\partial_{f,c}} C^3 .
\end{aligned}
\end{equation}
The (co)boundary operators are defined as the maps between the respective (co)chains on the direct lattice, using the polynomial representation.
However, throughout this work, we will adopt a polynomial representation defined on the dual lattice. This is because our figures are drawn on the dual lattice where faces are depicted as edges for visual clarity.
In our Einstein notation, we have
\begin{align}
    \bd_{e,\delta v} = \bd_{\partial e,v} &= \begin{pmatrix}
        1-x \\
        1-y \\
        1-z
    \end{pmatrix}, \\
    \bd_{f,\delta e}= \bd_{\partial f,e}&= \begin{pmatrix}
 0 & z-1 & 1-y \\
 z-1 & 0 & 1-x \\
 y-1 & 1-x & 0 
    \end{pmatrix},\\
    \bd_{c,\delta f} =\bd_{\partial c,f} &= \begin{pmatrix}
        1- x &
         y-1 &
       1- z
    \end{pmatrix}.
\end{align}
If we want to construct the Pauli-X or Z operators on coboundaries like $\delta e$,  we just multiply the original $X_f$ or $Z_f$ matrix by the coboundary operator $\bd_{f,\delta e}$. Similar rules hold in higher dimensions. In Fig.~\ref{fig:example how delta matrix act on Pauli}, we give an example of constructing stabilizers using our Einstein notation on the dual cubic lattice. Specifically, in polynomial representation, the stabilizers of the $3d$ toric code are constructed as
\begin{align}
    & \bW_c = \bZ_{\partial c} = \bZ_f \bd_{f,\partial c} =
    \begin{pmatrix}
        0\\0\\0\\
        1-\bar x\\
        \bar y-1\\1-\bar z    
    \end{pmatrix}, \\
    & \bB_e = \bX_{\delta e} = \bX_f \bd_{f,\delta e}= \begin{pmatrix}
     0 & z-1 & 1-y \\
     z-1 & 0 & 1-x \\
     y-1 & 1-x & 0 \\
     0 & 0 & 0 \\
     0 & 0 & 0 \\
     0 & 0 & 0 
    \end{pmatrix}.
\end{align}

\section{Clifford QCAs in $3{+}1$D from TQFTs}\label{sec:3d_examples}

In this section, we use cup products to construct all $\mathbb{Z}_2$ and $\mathbb{Z}_p$ Clifford QCAs in $3{+}1$D associated with the (classical) Witt group. These QCAs generate bulk states whose truncations yield commuting-projector models on the $2{+}1$D boundary, supporting chiral Abelian anyon theories. The $\mathbb{Z}_2$ Clifford QCA was first obtained in Ref.~\cite{haah_QCA_23}, while the $\mathbb{Z}_p$ cases were later developed in Ref.~\cite{Haah2021CliffordQCA}. Here, we unify these constructions within a TQFT framework, showing that they arise from the same method under different choices of topological actions.  
Previous approaches were restricted to cubic lattices, whereas our formulation of the $\mathbb{Z}_2$ Clifford QCA applies to arbitrary cellulations. In addition, the construction extends naturally to higher dimensions, as we demonstrate in Sec.~\ref{sec:Z2_Zp_higher_QCA_TQFT}.

{\change
First, we briefly review the Walker--Wang framework, which provides a systematic construction of $3{+}1$D lattice models with nontrivial boundary anyon theories~\cite{Walker2012TQFT}. When the input data consist of a \emph{premodular} anyon theory (braided fusion category containing nontrivial transparent anyons) the resulting Walker-Wang model exhibits a nontrivial bulk Hamiltonian whose $(2{+}1)$D boundary realizes precisely the input anyon theory~\cite{Keyserlingk2013Threedimensional}.  
In contrast, when the input theory is \emph{modular}, the bulk Hamiltonian is trivial, in the sense that it supports no deconfined excitations, and the unitary mapping from a product state to the ground state defines a QCA.

An Abelian anyon theory is specified by a finite Abelian group $A$ together with a quadratic function
$q: A \rightarrow \mathbb{Q}/\mathbb{Z}$~\cite{Wang2022in},
where the topological spin of an anyon $a \in A$ is given by $q(a)$.
Such quadratic functions are in one-to-one correspondence with cohomology classes in
$H^4(K(A,2), \mathbb{R}/\mathbb{Z})$~\cite{Eilenberg1953OnTG, Eilenberg1954OnTG, Kapustin2014theta, Kapustin2017Higher, Delcamp2019On},
which classify 2-form Dijkgraaf-Witten TQFT~\cite{dijkgraaf1990topological}.
Consequently, the Abelian Walker--Wang model can be naturally interpreted as a gauged 1-form symmetry-protected topological (SPT) phase~\cite{chen2012symmetry}.



}

We are interested in constructing QCAs corresponding to the generators of the Witt group of Abelian anyon theories. The Witt group gives an equivalence class of anyon theories up to stacking with anyon theories that admit a gapped boundary (Drinfeld center of some fusion category)~\cite{Davydov2013Witt}. Formally, two modular tensory categories $\mathcal{A}$ and $\mathcal{B}$ lie in the same Witt class if their product $\mathcal{A} \otimes \overline{\mathcal{B}}$ admits a gapped boundary, or equivalently, if a set of bosons can be condensed to trivialize the combined theory. Intuitively, the Witt group captures the “indecomposable” topological orders that cannot be removed by boson condensation. Its group operation is stacking of anyon theories, and the inverse of a theory is given by its time-reversal. In this work, we restrict to Abelian anyon theories, and the Witt group reduces to the classical Witt group of quadratic forms over finite Abelian groups.

{\change
The relation to the $L$-theoretic classification of Clifford QCAs can be summarized
as follows.
Any finite abelian group admits a primary decomposition
\[
\bigoplus_i \mathbb{Z}_{p_i^{r_i}},
\]
where the $p_i^{r_i}$ are powers of (not necessarily distinct) primes.
Quadratic forms on finite abelian groups decompose compatibly with this primary
decomposition (Proposition~76 of Ref.~\cite{ruba2025witt}).
Consequently, the associated Witt group decomposes as a direct sum of its
$p_i$-primary components (Proposition~5.16 of Ref.~\cite{davydov2013structure}).

Restricting to a fixed prime-$p$ power cyclic group, one may realize its quadratic forms by reduction from quadratic forms on free $\mathbb{Z}_{p^r}$-modules, for $r$ sufficiently large. The classification of Clifford QCAs for qudit dimension $p^r$ developed in Refs.~\cite{haah2025topological, yang2025categorifying} identifies the relevant invariant as $L_{3-D}(\mathbb{Z}_{p^r})$. In spatial dimension $D=3$, this reduces to $L_0(\mathbb{Z}_{p^r})$. This group is precisely the Witt group of free $\mathbb{Z}_{p^r}$-modules equipped with a quadratic form, and hence reproduces the corresponding $p$-primary direct summand of the Witt group upon allowing all possible values of $r$.
}

For Abelian anyon theories with a single generator (i.e., $A$ cyclic), we follow the notation of Ref.~\cite{Bonderson2007} and denote the anyon theory described by the group $A = \mathbb{Z}_n$ and the quadratic form $q(a) = \tfrac{m}{n}a^2 : \mathbb{Z}_n \to \mathbb{Q}/\mathbb{Z}$ as $\mathbb{Z}_n^{(m)}$.
The generators of the Witt group decompose according to prime numbers as follows:~\cite{Davydov2013Witt}
\begin{enumerate}
    \item For $p=2$ (anyon groups of order $2^k$), the Witt group contains the subgroup $\mathbb{Z}_8 \times \mathbb{Z}_2$. The relevant elements include:
    \begin{itemize}
        \item $(1,0) \in \mathbb{Z}_8 \times \mathbb{Z}_2$: the chiral semion theory ($\mathbb{Z}_2^{(1/2)}$),
        \item $(1,1) \in \mathbb{Z}_8 \times \mathbb{Z}_2$: the $U(1)_4$ theory ($\mathbb{Z}_4^{(1/2)}$),
        \item $(4,0) \in \mathbb{Z}_8 \times \mathbb{Z}_2$: the 3-fermion theory, with $A=\mathbb{Z}_2 \times \mathbb{Z}_2$ and $q((a,b)) = \tfrac{1}{2}(a^2 + ab + b^2)$.
    \end{itemize}
    Among these, only the 3-fermion theory gives rise to a Clifford QCA.
    \item For each prime $p \equiv 3 \pmod{4}$, the Witt group contributes a $\mathbb{Z}_4$ subgroup generated by $\mathbb{Z}_p^{(1)}$, all of which correspond to Clifford QCAs.  
    \item For each prime $p \equiv 1 \pmod{4}$, the Witt group contributes a $\mathbb{Z}_2 \times \mathbb{Z}_2$ subgroup generated by $\mathbb{Z}_p^{(1)}$ and $\mathbb{Z}_p^{(r)}$, where $r$ is a quadratic nonresidue modulo $p$. All of these also correspond to Clifford QCAs.
\end{enumerate}

{\change

To construct these QCAs explicitly, we first review the definition of a QCA in terms of \emph{separators} and \emph{flippers}.

\begin{defn}[\textbf{Separators and flippers}~\cite{haah_nontrivial_2023}]
A locally flippable $\mathbb{Z}_p$ separator consists of a set of operators $\{\overline{Z}_a\}$ together with another set of operators $\{\tilde{X}_a\}$, called flippers. Each operator is supported within a disk of finite radius, and the following conditions are satisfied:
\begin{enumerate}
    \item ${(\overline{Z}_a)}^p = 1$.
    \item $[\overline{Z}_a, \overline{Z}_b] = 0$ for all $a,b$.
    \item $\overline{Z}_a \tilde{X}_b = e^{2\pi i/p}\, \tilde{X}_b \overline{Z}_a$ if $a=b$; otherwise, $\overline{Z}_a$ and $\tilde{X}_b$ commute.
    \item For an arbitrary assignment $a \mapsto \omega(a)$, where $\omega(a)$ is a $p$-th root of unity, the space of states $\ket{\psi}$ satisfying
    \[
        \overline{Z}_a \ket{\psi} = \omega(a)\ket{\psi}
    \]
    for all $a$ is one-dimensional.
\end{enumerate}
\label{definition: separators}
\end{defn}

In the above definition, the flippers $\tilde{X}_a$ are not required to commute with one another. However, this issue can be resolved by appropriately decorating them with the operators $\overline{Z}_a$, yielding modified flippers $\overline{X}_a$ that commute mutually.
With these ingredients, a QCA $\alpha$ is then defined by its action on the Pauli generators,
\begin{equation}
    \alpha(Z_a) = \overline{Z}_a, 
    \qquad 
    \alpha(X_a) = \overline{X}_a .
\end{equation}

Our construction of QCAs proceeds as follows. We begin with TQFT data, namely a cohomology class, and follow the standard procedure to obtain the ground-state wavefunction of a symmetry-protected topological (SPT) phase~\cite{chen2012symmetry}. Upon gauging the symmetry, this yields the ground-state wavefunction of a Dijkgraaf-Witten TQFT~\cite{dijkgraaf1990topological}, whose special cases include the Walker--Wang models considered in this work.

Given a ground-state wavefunction, however, there are infinitely many possible choices of parent Hamiltonians. Our goal is to identify a particular parent Hamiltonian that can be written as a sum of locally flippable separators,
\begin{equation}
    H_{\mathrm{parent}} = - \sum_i \overline{Z}_i ,
\end{equation}
such that the operators $\overline{Z}_i$ admit flippers in the sense of Definition~\ref{definition: separators}. Organizing a parent Hamiltonian into this form is generally a highly nontrivial task. Fortunately, the cup-product formalism provides a systematic and efficient way to construct such Hamiltonians in our setting.
}

\subsection{3-fermion QCA from topological action $S=\tfrac{1}{2}(A_2 \cup A_2 + B_2 \cup B_2 + A_2 \cup B_2)$}
\label{Sec. 4d Z2 QCA}

We begin with the QCA constructed from the 3-fermion Walker-Wang model on arbitrary cellulations.
When restricted to the cubic lattice, our construction here reproduces the QCA obtained in Ref.~\cite{Shirley2022QCA}, which is equivalent to the original construction in Ref.~\cite{haah_QCA_23} up to a finite-depth quantum circuit.

The 3-fermion anyon theory corresponds to the anyon group $A=\ZZ_2 \times \ZZ_2$ and $q(a,b) = \frac{1}{2}(a^2+ab+b^2)$, where $a,b$ are the generators of the group. The terms $a^2$ and $b^2$ endow the generating anyons with fermionic self-statistics, while the term $ab$ gives the mutual statistics, resulting in their bound state also having fermionic self-statistics. To write down the corresponding Walker-Wang model on the lattice, we place two qubits (denoted by its sublattice $A$ and $B$) on each face $f$. We then write down the gauge theory with the following topological action on a (3+1)-dimensional manifold~\cite{Chen2023HigherCup}:
\begin{equation}
    S = \frac{1}{2} \int A_2 \cup A_2 + B_2 \cup B_2 + A_2 \cup B_2,
\label{eq: 3f WW S}
\end{equation}
where the subscript denotes the $A_2$ and $B_2$ are 2-cocycles. Here, and throughout the paper, we will implicitly assume that all gauge fields are dynamical, rather than background gauge fields. This means that we implicitly sum over all configurations which are flat $\delta A_2 = \delta B_2 =0$. 

Following Refs.~\cite{Chen2012cohomology, Chen2023HigherCup}, we write the Hamiltonian for the gauge theory described above on a three-dimensional lattice. On each face $f$, we introduce Pauli operators $Z_f^A, X_f^A$ and $Z_f^B, X_f^B$. The Hamiltonian takes the form
\begin{eqs}
    H^{\mathrm{3f}} =& - \sum_t Z^A_{\partial t} - \sum_t Z^B_{\partial t} \\
    &- \sum_e { G^B_e }  \prod_f {Z^A_f}^{\int \bface \cup \be} 
    - \sum_e {G^A_e} \prod_f {Z^B_f}^{\int \be \cup \bface},
    \label{eq: 3fWW_Hamiltonian 1}
\end{eqs}
where $t$ denotes a 3-cell, $f$ a 2-cell, and $G^\alpha_e$ is the gauge constraint of the $3{+}1$D bosonization~\cite{Chen2019Bosonization}, defined as
\begin{equation}
    G^\alpha_{e} :=  X^\alpha_{\delta \be} \prod_{f'} {Z^\alpha_{f'}}^{\int  \delta \be \cup_1 \bface'},
    \quad  \alpha \in \{A, B\}.
\label{eq: definition of G}
\end{equation}
Here, $Z^\alpha_{\partial t}$ denotes the flux term, defined as the product of $Z^\alpha_f$ over the surface of a 3-dimensional cell.
\begin{equation}
\begin{aligned}
    Z_{\partial t} &:=  \prod_f Z_f^{\bface(\partial t)} = \prod_f Z_f^{\int \delta \bface \cup_3 \bt}, \\
    X_{\delta \be} &:= \prod_f X_f^{ \be (\partial f)}~.
\end{aligned}
    \label{eq:Z_partialt_defintion}
\end{equation}

In the special case of a cubic lattice, where edges $e$ and faces $f$ are in one-to-one correspondence through $\be \cup \bface$ or $\bface \cup \be$, the Hamiltonian simplifies to
\begin{eqs}
    H^{\mathrm{3f}} =& - \sum_c Z^A_{\partial c} - \sum_c Z^B_{\partial c} \\
    &- \sum_f Z^A_f \prod_e { G^B_e }^{\int \bface \cup \be} - \sum_f Z^B_f \prod_e {G^A_e}^{\int \be \cup \bface}~.
    \label{eq: 3fWW_Hamiltonian 2}
\end{eqs}
{\change
The 3-cell is now labeled by a cube $c$, and the operator $G^\alpha_e$ in Eq.~\eqref{eq: definition of G} can be represented diagrammatically:
\begin{equation*}
    \raisebox{-0.5\height}{\includegraphics[scale=.33]{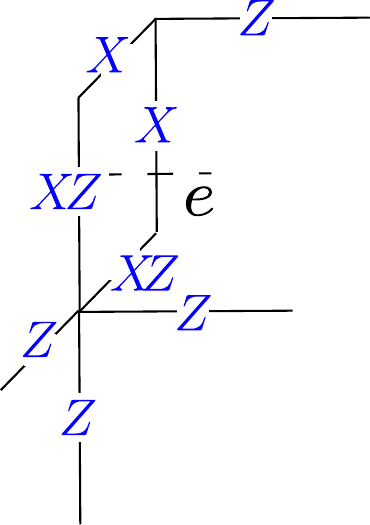}},
    ~\raisebox{-0.5\height}{\includegraphics[scale=.33]{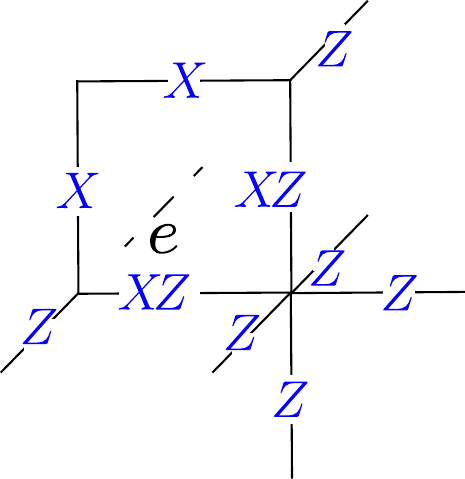}},
    ~\raisebox{-0.5\height}{\includegraphics[scale=.33]{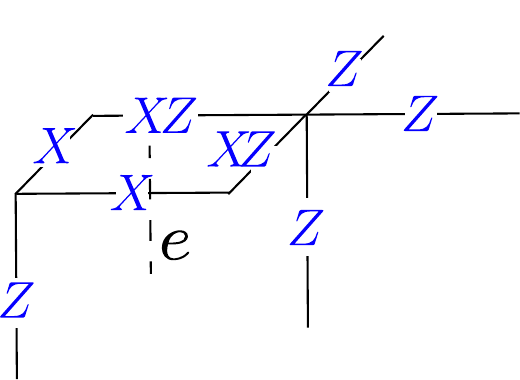}}~.
\end{equation*}
where, for clarity, all operators are illustrated on the dual lattice (solid lines), while the dashed line marks the edge $e$ on the original lattice.
}
Consequently, the third and fourth terms of Eq.~\eqref{eq: 3fWW_Hamiltonian 2} can be depicted as
\begin{widetext}
\begin{align}
Z^A_f \prod_e { G^B_e }^{\int \bface \cup \be} &=\raisebox{-0.5\height}{\includegraphics[scale=.25]{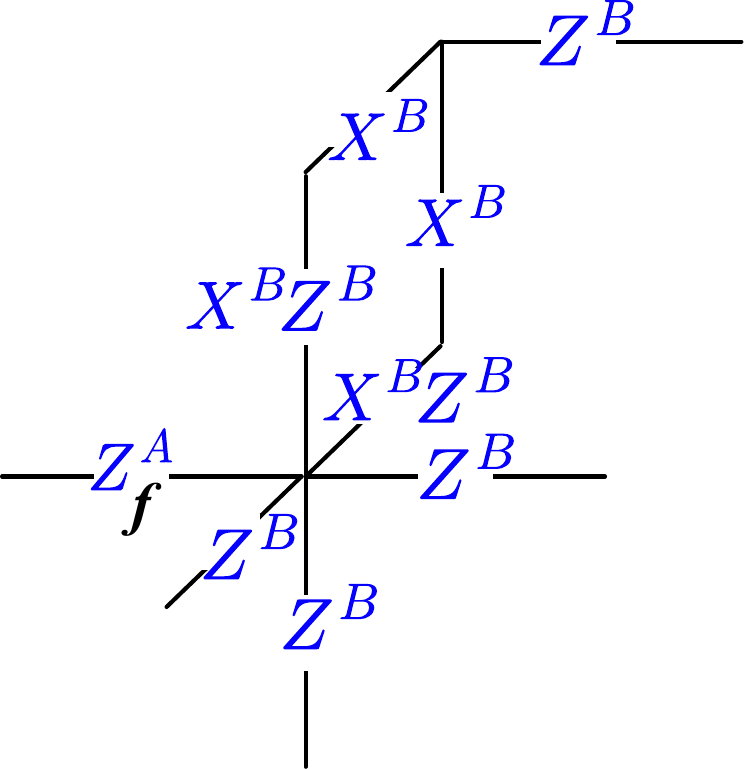}}~,
\quad \raisebox{-0.5\height}{\includegraphics[scale=.25]{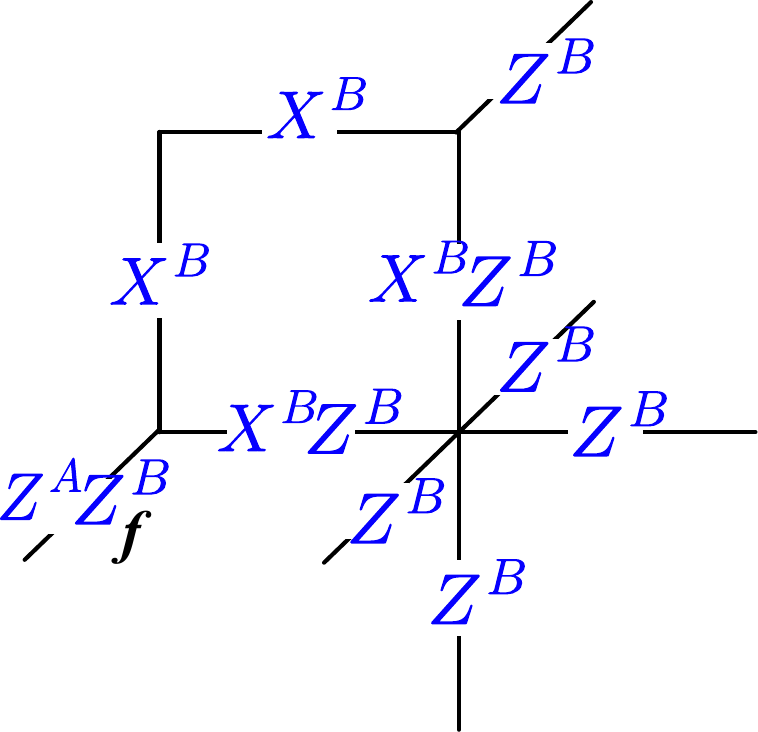}}~,
\quad\raisebox{-0.5\height}{\includegraphics[scale=.25]{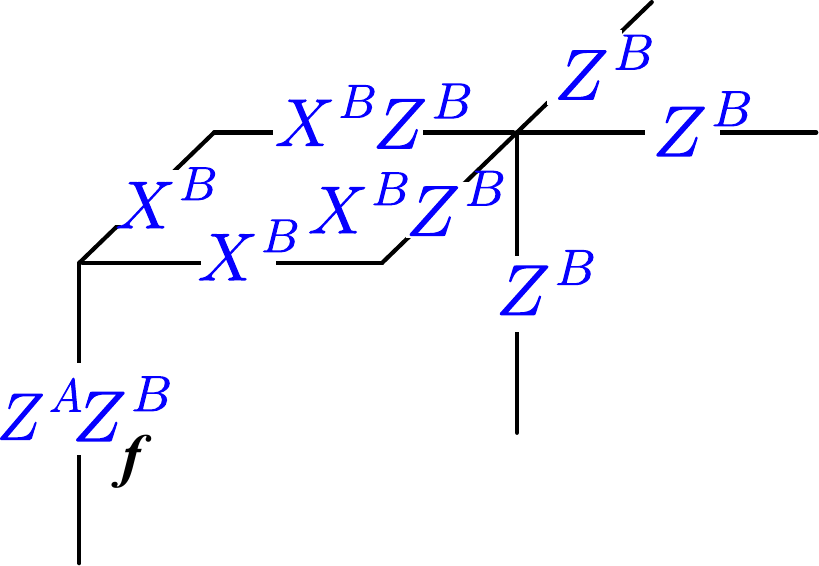}}~,\\
Z^B_f \prod_e {G^A_e}^{\int \be \cup \bface} &= \quad\raisebox{-0.5\height}{\includegraphics[scale=.25]{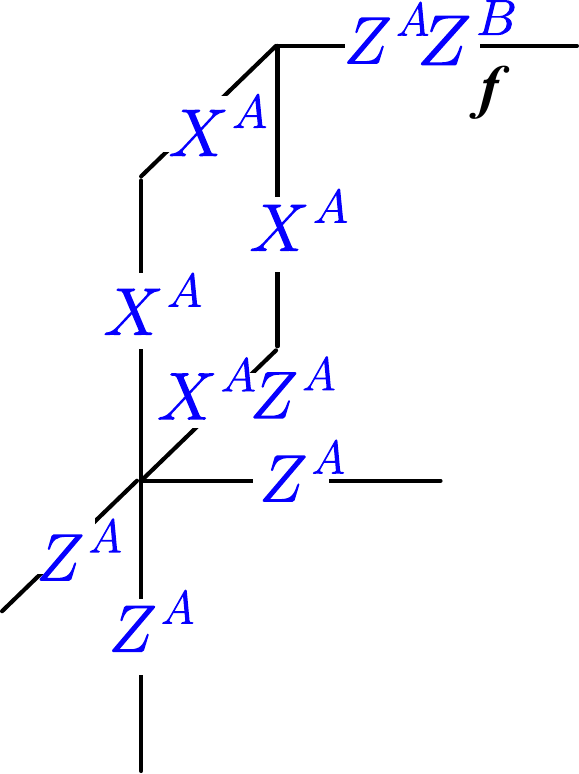}},
\quad\raisebox{-0.5\height}{\includegraphics[scale=.25]{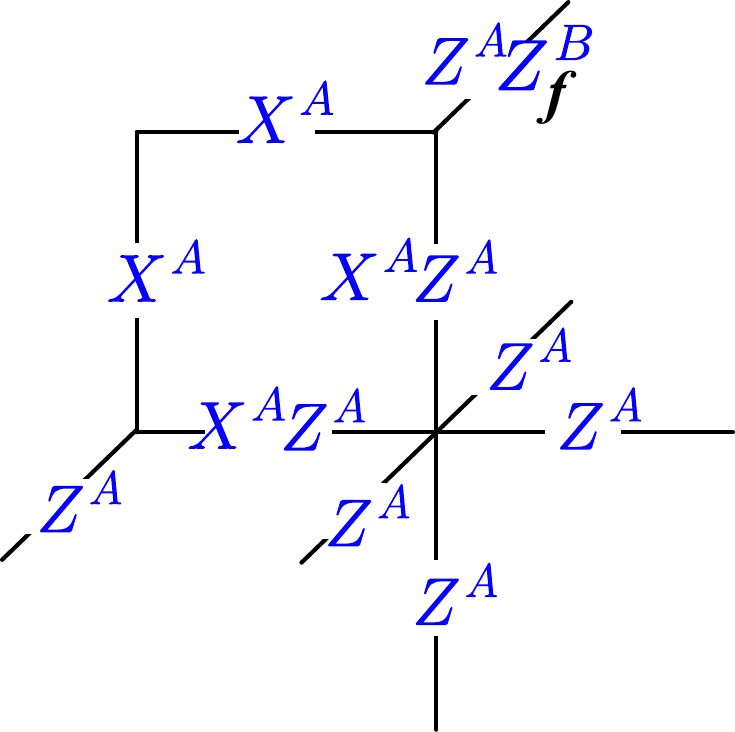}}~,
\quad\raisebox{-0.5\height}{\includegraphics[scale=.25]{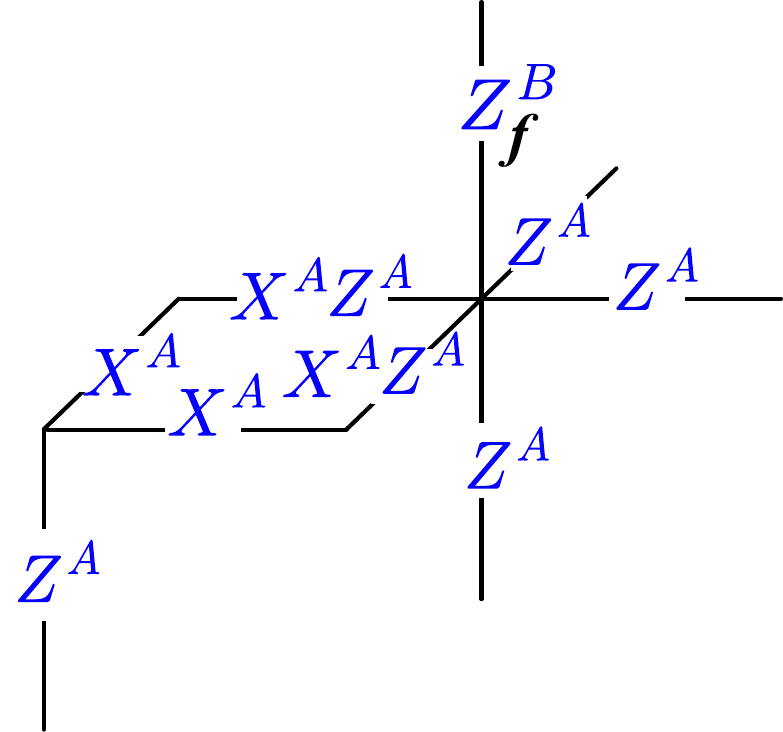}}~.
\end{align}
\end{widetext}
Note that if we multiply $G_e^\alpha$ around faces of a closed membrane on the dual lattice (equivalently, around a vertex on the direct lattice), the product becomes the identity
\begin{equation}
    \prod_{e \supset v} G^\alpha_e = 1, \quad \forall~ \alpha \in \{A, B\}.
\label{eq: closed condition of G_e}
\end{equation}
Consequently, we observe that the third and fourth terms of Eq.~\eqref{eq: 3fWW_Hamiltonian 2} generate the first and second terms. Thus, we may remove these two terms from the Hamiltonian without affecting its ground state.

Since the ground state is uniquely determined by the last two terms, we may define the separators of the QCA as
\begin{equation}
\begin{aligned}
    &\overline{Z}^A_f = Z^A_f \prod_e {G^B_e}^{\int \bface \cup \be}, \\ &\overline{Z}^B_f = Z^B_f \prod_e {G^A_e }^{\int \be \cup \bface},
\end{aligned}
\label{eq:3fWW_separator}
\end{equation}
on each face. They are the original $Z$ operators decorated with the $G$ operator for the other species of qubits.
In the $3{+}1$D bosonization~\cite{Chen2019Bosonization}, the fermion hopping operators are given by
\begin{equation}
\begin{aligned}
    &U^A_f := X^A_f \prod_{f'} {Z^A_{f'}}^{\int \bface' \cup_1 \bface}, \\ &U^B_f := X^B_f \prod_{f'} {Z^B_{f'}}^{\int \bface' \cup_1 \bface},
\end{aligned}
\label{eq: fermionic hopping operators}
\end{equation}
where the cup-1 product encodes the interaction between Pauli operators on adjacent faces.
They satisfy the commutation relations
\begin{eqs}
\label{eq:Ucommutation}
    U^{\alpha}_f U^{\alpha}_{f'} = (-1)^{\int \bface \cup_1 \bface' + \bface' \cup_1 \bface} U^{\alpha}_{f'} U^{\alpha}_f~,
\end{eqs}
$\forall~ \alpha \in \{A, B\}$. The explicit realizations on the cubic lattice are
\begin{equation}
    U_f =\vcenter{\hbox{\includegraphics[scale=.22]{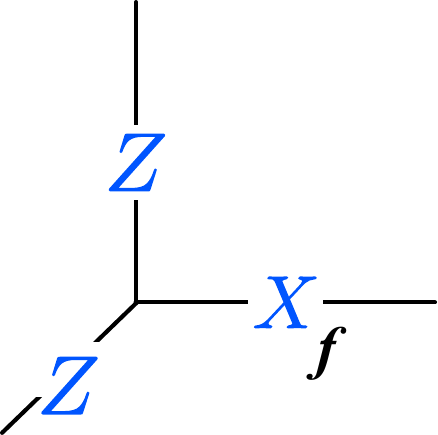}}}~,
    \quad\vcenter{\hbox{\includegraphics[scale=.22]{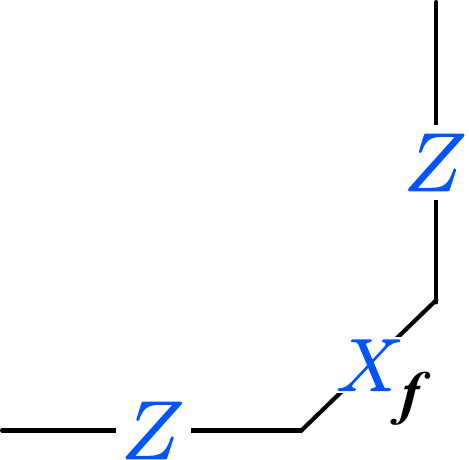}}}~,
    \quad\vcenter{\hbox{\includegraphics[scale=.22]{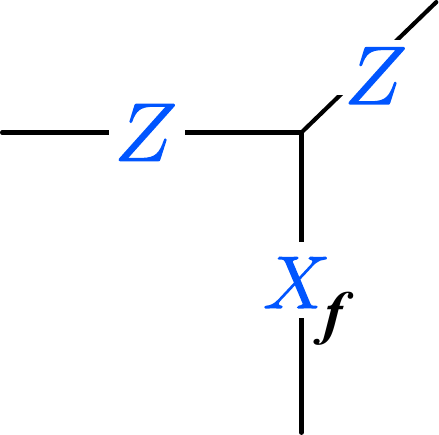}}}~.
\end{equation}
Since $[U^\alpha_f, G^\beta_e]=0$ for all $\alpha, \beta \in \{A, B\}$, we also find $[\overline{Z}^A_f , U^B_{f'}] = [\overline{Z}^B_f , U^A_{f'}] = 0$. Therefore, we have
\begin{eqs}
    \overline{Z}^A_f U^A_{f'} &= (-1)^{\delta_{f,f'}} U^A_{f'} \overline{Z}^A_f~, \\
    \overline{Z}^B_f U^B_{f'} &= (-1)^{\delta_{f,f'}} U^B_{f'} \overline{Z}^B_f~.
\end{eqs}
Hence, the hopping operators $U^A_f$ and $U^B_f$ act as (non-commuting) flippers for the separators $\overline{Z}^A_f$ and $\overline{Z}^B_f$. We may fix this non-commutation by attaching separators to cancel the sign coming from the commutation relation in Eq.\eqref{eq:Ucommutation}. Define
\begin{eqs} 
    \overline{X}^A_f &:= U^A_f \prod_{f'} {\overline{Z}^A_{f'}}^{\int \bface' \cup_1 \bface } 
    = X^A_f \prod_{f'} \left( Z^A_{f'} \overline{Z}^A_{f'} \right)^{\int \bface' \cup_1 \bface } \\
    &= X^A_f \prod_{e,f'} \left({G^B_e}^{\int \bface'\cup \be} \right)^{\int \bface' \cup_1 \bface}, \\
    \overline{X}^B_f &:= U^B_f \prod_{f'} {\overline{Z}^B_{f'}}^{\int \bface' \cup_1 \bface}
    = X^B_f \prod_{f'} \left( Z^B_{f'} \overline{Z}^B_{f'} \right)^{\int \bface' \cup_1 \bface} \\
    &= X^B_f \prod_{e,f'} \left({G^A_e}^{\int \be \cup \bface'} \right)^{\int \bface' \cup_1 \bface}.
\label{eq: definition of tilde X^A and X^B}
\end{eqs}
These flippers on the cubic lattice can be visualized in Fig.~\ref{flippers of 3-fermion QCA}. 
\begin{figure*}[htbp]
    \centering
    \includegraphics[scale=0.25]{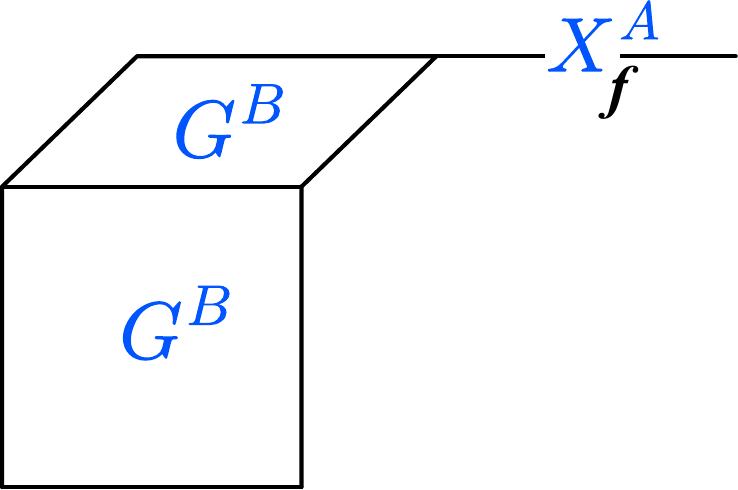}
    \hspace{1em}
    \includegraphics[scale=0.25]{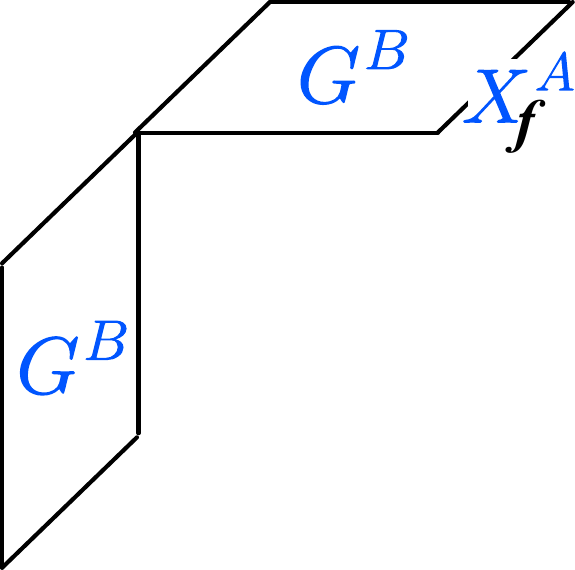}
    \hspace{1em}
    \includegraphics[scale=0.25]{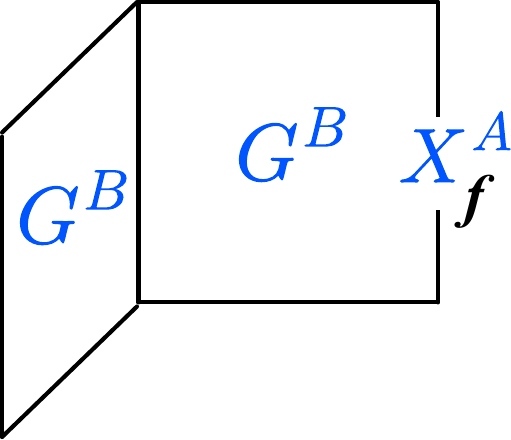}
    \\
    \vspace{1em}
    \includegraphics[scale=0.25]{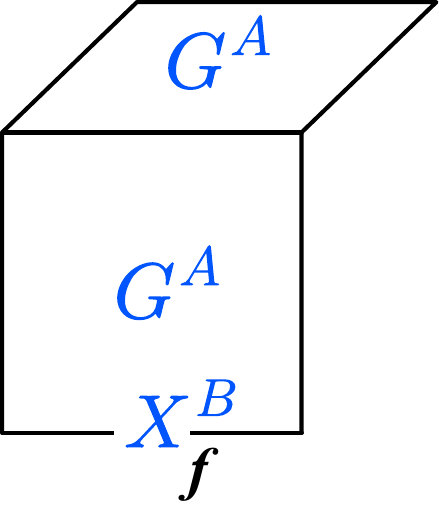}
    \hspace{2em}
    \includegraphics[scale=0.25]{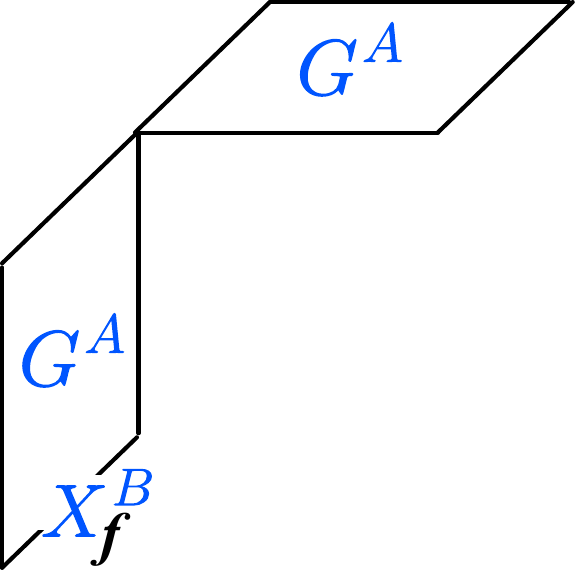}
    \hspace{2em}
    \raisebox{0em}{\includegraphics[scale=0.25]{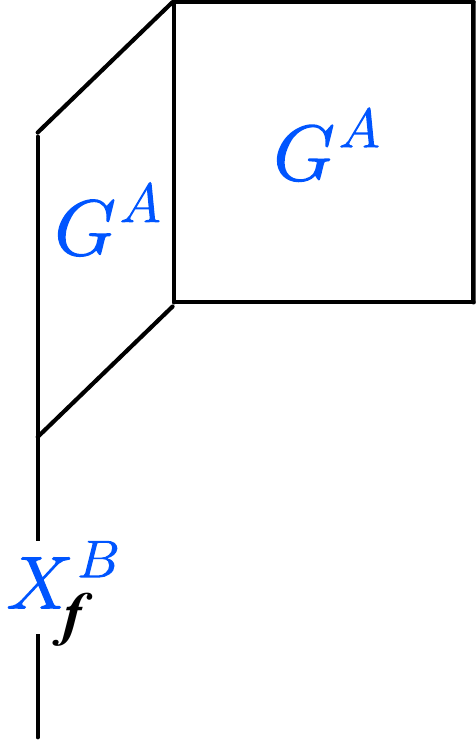}}
    \caption{
    The flippers $\overline{X}_f^A$ and $\overline{X}_f^B$ of the 3-fermion QCA on the dual lattice, with $G^A$ and $G^B$ defined in Eq.~\eqref{eq: definition of G}.
    }
    \label{flippers of 3-fermion QCA}
\end{figure*}
We may also seperately verify that $[\overline{X}^A_f, \overline{X}^A_{f'}] = [\overline{X}^B_f, \overline{X}^B_{f'}] = 0$ using the final expression by noting that $G^\alpha_e$ commutes for distinct edges $e$ (this property holds only for $\mathbb{Z}_2$ qubits).

The 3-fermion quantum cellular automaton (QCA) is defined as
\begin{eqs}
    X^A &\rightarrow \overline{X}^A, \quad  Z^A \rightarrow \overline{Z}^A, \\
    X^B &\rightarrow \overline{X}^B, \quad  Z^B \rightarrow \overline{Z}^B.
\end{eqs}

\subsubsection{Order of 3-fermion QCA}

We will show that this QCA squares to the identity. To prove this, it suffices to show that $G^\alpha_e$ is invariant under the QCA, since the action of the QCA attaches an appropriate set of $G_e$ to both $X$ and $Z$. Thus, if this is the case, then we have
\begin{equation}
    X \rightarrow X \prod G_e \rightarrow \left( X \prod G_e^2 \right) = X.
\end{equation}
and similarly for Pauli $Z$.
To verify this, observe that
\begin{eqs}
    U^A_{f} \rightarrow &~ \overline{X}^A_{f} \prod_{f'} {\overline{Z}^A_{f'}}^{\int \bface' \cup_1 \bface} \\
    &= \left[ U^A_f \prod_{f'} {\overline{Z}^A_{f'}}^{\int \bface' \cup_1 \bface } \right]
    \prod_{f'} {\overline{Z}^A_{f'}}^{\int \bface' \cup_1 \bface} 
    = U^A_{f}~,
\label{eq:U^A_invariant_under_QCA}
\end{eqs}
\begin{eqs}
    U^B_{f} \rightarrow &~ \overline{X}^B_{f} \prod_{f'} {\overline{Z}^B_{f'}}^{\int \bface' \cup_1 \bface} \\
    &= \left[ U^B_f \prod_{f'} {\overline{Z}^B_{f'}}^{\int \bface' \cup_1 \bface } \right]
    \prod_{f'} {\overline{Z}^B_{f'}}^{\int \bface' \cup_1 \bface} 
    = U^B_{f}~.
    \label{eq:U^B_invariant_under_QCA}
\end{eqs}
Thus, the hopping operator remains invariant under the QCA. By construction, the flux term $Z^\alpha_{\partial c}$ is also invariant under the QCA, due to Eq.~\eqref{eq: closed condition of G_e}. Consequently, the gauge constraint $G^\alpha_e$ is invariant under the QCA. Explicitly, we have:
\begin{equation}
\begin{aligned}
    G_e &= \left( X_{\delta \be} \prod_{f'} Z_{f'}^{\int \bface' \cup_1 \delta \be} \right)
    \prod_{f'}Z^{\int \delta \be \cup_1 \bface' + \bface' \cup_1 \delta \be} \\
    &= U_{\delta \be} \prod_{f'}Z_{f'}^{\int \delta \be \cup_2 \delta \bface'} 
    = U_{\delta \be} \prod_c Z_{\partial c}^{\delta \be \cup_2 \bc},
\end{aligned}
\end{equation}
where $U_{\delta \be} := X_{\delta \be} \prod_{f'} Z_{f'}^{\int \bface' \cup_1 \delta \be}$ represents a product of $U_f$ around an edge $e$, up to an overall $\pm$ sign. It is important to emphasize that the property of $G_e$ being a product of $U_f$ and $Z_{\partial c}$ is specific to $\mathbb{Z}_2$ qubits. In the QCA with $\mathbb{Z}_p$ qudits discussed in the next section, $G_e$ is no longer expressible as a simple product of $U_f$ and $Z_{\partial c}$. Therefore, while it is generally true in our construction that $U_f$ and $Z_{\partial c}$ (as well as their generalization in higher dimensions) remain invariant under the QCA, the invariance of $G_e$ under the QCA is unique to $\mathbb{Z}_2$ qubit systems.

\subsubsection{Polynomial expression for 3-fermion QCA}

Since we have two qubits species $A$ and $B$, we define
\begin{align}
    &\bZ^A =
    \begin{pmatrix}
        0\\0\\
      \bd_{f_0,f}\\
      0
    \end{pmatrix}, \quad
    \bZ^B =
    \begin{pmatrix}
        0\\0\\0\\
      \bd_{f_0,f}
    \end{pmatrix}, \\
    &\bX^A =
    \begin{pmatrix}
       \bd_{f_0,f}\\0\\0\\0
    \end{pmatrix}, \quad
    \bX^B =
    \begin{pmatrix}
      0 \\\bd_{f_0,f}\\0\\0
    \end{pmatrix}~.
\end{align}
The hopping operators and the gauge constraint can be written in polynomial form as
\begin{align}
    \bU_f^{\alpha} &= \bX_f^\alpha + k \bZ_{f'}^\alpha \bM_{f' \cup_1 f},   \\
    \tilde \bU_f^\alpha &=\bX_f + k \bZ_{f'}^\alpha \bM_{f \cup_1 f'}^\dagger ,\\
    \bG_e^{(k)}&= \tilde \bU_{\delta e}^\alpha . 
\end{align}
Finally, we can write the separators and flippers as
\begin{align}
    \overline{\bZ}_f^A &=\bZ_f^A +\bG_e^B \bM_{f\cup e}^\dagger, \nonumber\\
    \overline{\bZ}_f^B &=\bZ_f^B +\bG_e^A \bM_{e\cup f}, \nonumber\\
    \overline{\bX}_f^A &= \bU_f^{A} +\overline{\bZ}_{f'}^A \bM_{f'\cup_1 f}  = \bX_f^A +\bG^{B}_e\bM_{f'\cup e}^\dagger\bM_{f'\cup_1f}, \nonumber\\
    \overline{\bX}_f^B &= \bU_f^{B} +\overline{\bZ}_{f'}^B \bM_{f'\cup_1 f}  = \bX_f^B +\bG^{A}_e\bM_{e\cup f'}\bM_{f'\cup_1f}.\nonumber
\end{align}
\begin{widetext}
In matrix form,
\begin{eqs}
    \overline{\bZ}^A &=
    \begin{pmatrix}
        0\\\bM_{f\cup\partial f_0}^\dagger\\
        \bd_{f_0,f}\\
        \bM_{ f' \cup_1 f_0}^\dagger  \bM_{f\cup \partial f'}^\dagger
    \end{pmatrix}~, \qquad\qquad~
    \overline{\bZ}^B =
    \begin{pmatrix}
        \bM_{\partial f_0 \cup f}\\0\\ \bM_{f' \cup_1 f_0}^\dagger  \bM_{\partial f'\cup f}\\
        \bd_{f_0,f}
    \end{pmatrix}~,\\
    \overline{\bX}^A &=
    \begin{pmatrix}
       \bd_{f_0,f}\\\bM_{f' \cup \partial f_0}^\dagger\bM_{f'\cup_1 f}\\0\\\bM_{f'' \cup_1 f_0}^\dagger \bM^\dagger_{f'\cup \partial f''} \bM_{f\cup_1 f},
    \end{pmatrix}, \quad
    \overline{\bX}^B =
    \begin{pmatrix}
      \bM_{\partial f_0 \cup f'}\bM_{f' \cup_1 f} \\\bd_{f_0,f}\\\bM_{f''\cup_1 f_0}^\dagger\bM_{\partial f'' \cup f'}\bM_{f' \cup_1 f} \\0
    \end{pmatrix}.
\end{eqs}
This allows us to define the 3F QCA as
\begin{align}
    \sansmath{\bm \alpha} =  \begin{pmatrix}\overline{\bX}^A  & \overline{\bX}^B &  \overline{\bZ}^A& \overline{\bZ}^B\end{pmatrix}.
    \label{eq:3FQCA_TQFT}
\end{align}
The above QCA has been presented in polynomial form in Ref.~\cite{Shirley2022QCA}, albeit in a different basis. {\change The matrix elements of our 3F QCA can be found in Appendix~\ref{app:3FQCA_TQFT_3D}.} We will later show in Sec.~\ref{sec:Z2_Zp_higher_QCA_TQFT} that higher dimensional generalizations of the 3F QCA can be straightforwardly written by replacing the simplices to higher dimensional ones. In fact, in 5+1D, these matrices coincide with those constructed in Ref.~\cite{fidkowski2024qca}.

We may show explicitly that the QCA squares exactly to the identity. Before that, we first prove some simple identities. Let us show that
\begin{eqs}
    &\begin{pmatrix}
      \bO^{(11)}_{f_0,f} & \cdots & \bO^{(14)}_{f_0,f}\\
      \vdots & \ddots & \vdots\\
      \bO^{(41)}_{f_0,f} & \cdots& \bO^{(44)}_{f_0,f}
    \end{pmatrix} 
    \coloneqq  \begin{pmatrix} \bM_{\partial f_0 \cup f'} \\ \bM^\dagger_{f' \cup \partial f_0} \\ \bM_{  f_0 \cup \partial f'}\\
    \bM^\dagger_{\partial f' \cup f_0}
    \end{pmatrix} (\bM_{f'\cup_1 f''}+\bM^\dagger_{f''\cup_1 f'}) 
    \times \begin{pmatrix}
     \bM_{\partial f'' \cup f}^\dagger & \bM_{f'' \cup \partial f}  
 & \bM_{  f \cup \partial f''}^\dagger & \bM_{\partial f'' \cup f}
    \end{pmatrix} 
    =0~.
\end{eqs}
For example,
\begin{align*}
\bO^{(13)}_{f_0,f} &=\bM_{\partial f_0 \cup f'}(\bM_{f'\cup_1 f''}+\bM^\dagger_{f''\cup_1 f'})\bM^\dagger_{f\cup \partial f''}
= \bM_{\partial f_0 \cup f'}(\bM_{\delta f'\cup_2 f''}+\bM^\dagger_{f'\cup_2 \delta f''})\bM^\dagger_{f\cup \partial f''}\\
&={{\bM_{\partial f_0 \cup \partial c}}} ~\bM_{c\cup_2 f''}~ \bM^\dagger_{f\cup \partial f''} +\bM_{\partial f_0 \cup f'}~\bM^\dagger_{f'\cup_2 c}~{\bM^\dagger_{f\cup \partial^2 c}} ~=0~,
\end{align*}
and
\begin{align*}
\bO^{(24)}_{f_0,f} &=   \bM^\dagger_{f' \cup \partial f_0}(\bM_{f'\cup_1 f''}+\bM^\dagger_{f''\cup_1 f'})\bM_{\partial f'' \cup f} 
= \bM^\dagger_{f' \cup \partial f_0}(\bM_{\delta f'\cup_2 f''}+\bM^\dagger_{f'\cup_2 \delta f''})\bM_{\partial f'' \cup f}\\
&={\bM^\dagger_{f' \cup \partial^2 c}}~\bM_{c\cup_2 f''}~ \bM_{\partial f'' \cup f}  +\bM^\dagger_{f' \cup \partial f_0}~\bM^\dagger_{f'\cup_2 c}~{\bM_{\partial^2 c \cup f}} =0~,
\end{align*}
evaluate to zero using the relation of the $\cup_1$ and $\cup_2$ products in Eq.~\eqref{eq:highercuprelationmatrix}, similarly for all other entries.
Returning to the squaring of the QCA, a direct computation shows that
\begin{align}
    \boldsymbol{\alpha}^2 =  \begin{pmatrix}
    \bd_{f_0,f}&0&\bO^{(13)}_{f_0,f}&0\\
   0& \bd_{f_0,f}+ \bO^{(31)}_{f_0,f'}\bM_{f'\cup_1 f}&0&\bO^{(31)}_{f_0,f}\\
   \bM^\dagger_{f' \cup_1 f_0} \bO^{(13)}_{f',f''}\bM_{f''\cup_1 f}&0& \bd_{f_0,f}+\bM^\dagger_{f' \cup_1 f_0} \bO^{(13)}_{f',f}& 0\\
    0&\bM^\dagger_{f'\cup_1 f_0}\bO^{(24)}_{f',f''}\bM_{f''\cup_1 f}&0& \bd_{f_0,f}\\
    \end{pmatrix} = \bbone_{12\times 12} .
\end{align}
\end{widetext}

\subsubsection{3-fermion QCA on arbitrary cellulations}

We emphasize that our QCA construction applies to arbitrary cellulations. We place qubits of type $A$ on the faces of the cellulation, and qubits of type $B$ on the faces of its dual cellulation. The separators in Eq.~\eqref{eq:3fWW_separator} are then
\begin{eqs}
    \overline{Z}^A_f = Z^A_f G^B_{\mathrm{PD}(f)}, \quad 
    \overline{Z}^B_f = Z^B_f G^A_{\mathrm{PD}(f)},
\label{eq:3fWW_separator triangulation}
\end{eqs}
while the flippers in Eq.~\eqref{eq: definition of tilde X^A and X^B} become
\begin{eqs} 
    &\overline{X}^A_f 
    = X^A_f \prod_{f'} \left(G^B_{\mathrm{PD}(f')}\right)^{\int \bface' \cup_1 \bface}, \\
    &\overline{X}^B_f 
    = X^B_f \prod_{f'} \left(G^A_{\mathrm{PD}(f')}\right)^{\int \bface' \cup_1 \bface}~.
\label{eq: definition of tilde X^A and X^B triangulation}
\end{eqs}
Here $\mathrm{PD}$ denotes the Poincaré dual: $\mathrm{PD}(f)$ is the edge in the dual (direct) lattice intersecting the face $f$ in the direct (dual) lattice. We use the dual cellulation because, in passing from Eq.~\eqref{eq: 3fWW_Hamiltonian 1} to Eq.~\eqref{eq: 3fWW_Hamiltonian 2}, we exploited the correspondence between edges $e$ and faces $f$ on the cubic lattice, which is naturally interpreted as the Poincaré duality between two lattices. The definition of cup products on arbitrary cellulations can be found in Ref.~\cite{Chen2023HigherCup, Bauer2025lowoverheadnon}.

\subsection{$\ZZ_p^{(k)}$ QCA from topological action $S = \frac{k}{p}B_2 \cup B_2$}
\label{Sec. 3d Zp QCA}
In this section, we construct the QCA associated with the anyon theory $\mathbb{Z}_p^{(k)}$. This theory consists of anyons $\mathcal{A} = \{1, a, a^2, \ldots, a^{p-1}\}$ with fusion rules and topological spins given by
\begin{equation}
    a^p = 1, \qquad 
    \theta_{a} = \exp\!\left(\tfrac{2\pi i k}{p}\right).
\end{equation}
Although a representative of this QCA  has been constructed before in Ref.~\cite{Haah2021CliffordQCA} via the shifting of invertible subalgebras~\cite{Haah2023InvertibleSubalgebras}, we provide here a novel construction through TQFTs and the cup product formalism. This approach places the $\ZZ_p$ QCA in $3{+}1$D on the same footing as other $\ZZ_2$ QCAs in $3{+}1$D and higher dimensions. Moreover, we explicitly show that the order of these QCAs, matches that of the Witt group.

The Hamiltonian of the gauged $\mathbb{Z}_p$ 1-form SPT phase with action $S = \tfrac{k}{p} B_2 \cup B_2 \in H^4(B^2\mathbb{Z}_p, \mathbb{R}/\mathbb{Z})$ is
\begin{eqs}
    H^{\ZZ_p^{(k)}} =& -\sum_t Z_{\partial t} - \sum_e X_{\delta \be} \prod_{f'} Z_{f'}^{k \int \bface' \cup \be + \be \cup \bface'} \\
    &+ h.c. \, ,
\label{eq: 1/p BUB QCA initial Hamiltonian}
\end{eqs}
where $Z_{\partial t}$ is the flux operator, defined as the product of $Z_f$ over the 2-cell faces of a 3-cell $t$, and $X_{\delta e}$ is the product of $X_f$ around an edge $e$, with orientations specified as in Eq.~\eqref{eq:Z_partialt_defintion}. On the cubic lattice, the corresponding Hamiltonian terms are shown in Fig.~\ref{fig:Zp^k_original_Hamiltonian}.

\begin{figure*}[t]
    \centering
    \includegraphics[scale=0.5]{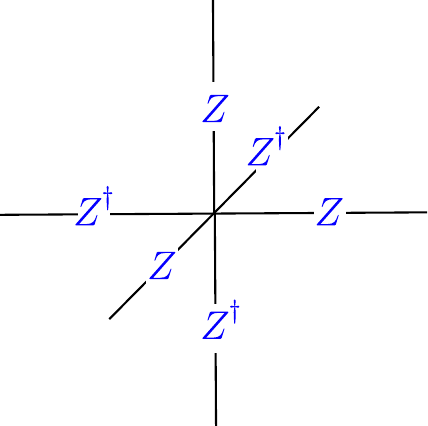}
    \hfill
    \includegraphics[scale=0.5]{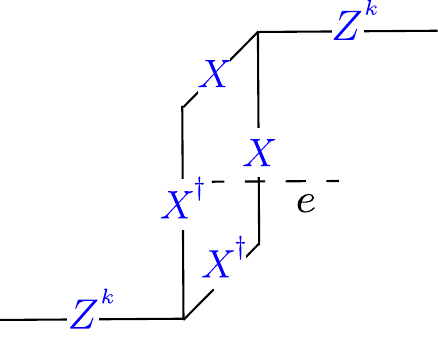}
    \hfill
    \includegraphics[scale=0.5]{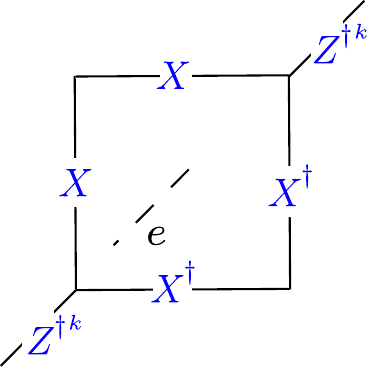}
    \hfill
    \includegraphics[scale=0.5]{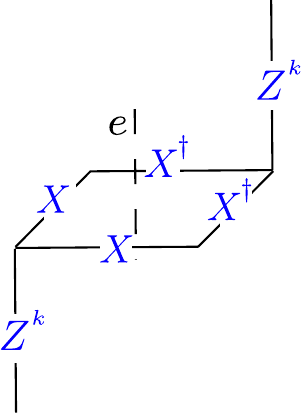}
\caption{
A valid choice of parent Hamiltonian for the gauged 1-form SPT phase associated with the cocycle $S = \tfrac{k}{p} B_2 \cup B_2 \in H^4(B^2\mathbb{Z}_p, \mathbb{R}/\mathbb{Z})$, illustrated on the dual lattice. However, the Hamiltonian terms are not independent, and therefore cannot be used directly as separators in the QCA construction.
}
\label{fig:Zp^k_original_Hamiltonian}
\end{figure*}

Unlike the 3-fermion case in Eq.~\eqref{eq: 3fWW_Hamiltonian 2}, we cannot simply remove the $Z_{\partial c}$ term on the cubic lattice (here we use $c$ instead of $t$ to indicate that the construction is defined on the cubic lattice).
The second term in Eq.~\eqref{eq: 1/p BUB QCA initial Hamiltonian} is an $X$ plaquette operator multiplied by two $Z$ operators. Consequently, taking the product around a cube leaves behind two $Z_{\partial c}$ terms instead of one. To address this, we look for a better set of stabilizer generators that eliminates $Z_{\partial t}$. These new generators will serve as separators for the QCA construction.
Our approach is to multiply the second term by $Z_{\partial c}$, which preserves the stabilizer group. This allows us to rewrite it as a $Z_f$ operator times an additional operator $G$ defined below in Eq.~\eqref{eq: G_e^{(k)} definition}. First, using the identity
\begin{eqs}
    &\bface' \cup \be - \be \cup \bface' \\ 
    & = \delta \be \cup_1 \bface' - \be \cup_1 \delta \bface' 
    -\delta (\be \cup_1 \bface'),
\label{eq: rearrange cup to cup1}
\end{eqs}
we can transform the second term of Eq.~\eqref{eq: 1/p BUB QCA initial Hamiltonian} to:
\begin{eqs}
    - \sum_e X_{\delta \be} \prod_{f'} Z_{f'}^{k\int 2 \be \cup \bface' + \delta \be \cup_1 \bface' - \be \cup_1 \delta \bface'}.
\label{eq: e U f and f U e to de U1 f}
\end{eqs}
Multiplying by $Z_{\partial c}^k$ does not change the stabilizer group generated by these operators, which leads to
\begin{eqs}
    & -\sum_e  \Biggl[\prod_c Z_{\partial c}^{k\int \be \cup_1 \bc} \, X_{\delta \be} \prod_{f'} Z_{f'}^{k\int 2 \be \cup \bface' + \delta \be \cup_1 \bface' - \be \cup_1 \delta \bface'} \Biggr]\\
    &= -\sum_e X_{\delta \be} Z_{f'}^{k\int 2 \be \cup \bface' + \delta \be \cup_1 \bface'} \\
    &= -\sum_e G_e^{(k)} \prod_{f'} Z_{f'}^{2k\int \be \cup \bface'} \, ,
\label{eq: stabilizers of G times Z^2}
\end{eqs}
where by analogy with the $3{+}1$D bosonization \cite{Chen2019Bosonization}, we define
\begin{eqs}
    G_e^{(k)} := X_{\delta \be} \prod_{f'} Z_{f'}^{k\int \delta \be \cup_1 \bface'}.
\label{eq: G_e^{(k)} definition}
\end{eqs}
In the $\mathbb{Z}_p$ case, the operators $G_e^{(k)}$ on different edges do not commute. However, as shown in Eq.~\eqref{eq: stabilizers of G times Z^2}, the combinations $G_e^{(k)} \prod_{f'} Z_{f'}^{2k\int \be \cup \bface'}$ for different edges do commute.

Next, since $\be \cup \bface'$ defines a one-to-one correspondence between edges and faces on the cubic lattice, we can rewrite this term as
\begin{eqs}
    -\sum_f Z_f^{2k} \prod_e {G_e^{(k)}}^{\int \be \cup \bface},
\label{eq: 1/p BUB QCA modified plaquette terms}
\end{eqs}
and we will show that it generates the entire stabilizer group. In fact, the product of these terms around a cube reproduces $Z_{\partial c}^{2k}$, with all $G_e^{(k)}$ factors canceling. An alternative way to see this is by substituting $\be = \delta \bv$ into $G_e^{(k)}$, which gives
\begin{equation}
    \prod_{e \supset v} G_e^{(k)} = X_{\delta (\delta \bv)} Z_{f'}^{k\int \delta (\delta \bv) \cup_1 \bface'} = 1.
\end{equation}
For odd $p$, the operator $Z_{\partial c}^{2k}$ already generates $Z_{\partial c}$, ensuring that the entire stabilizer group is captured.
Consequently, we obtain an equivalent Hamiltonian with the same ground state,
\begin{eqs}
    H^{\ZZ_p^{(k)}}_{\text{separator}} =& -\sum_f Z_f^{2k} \prod_e {G_e^{(k)}}^{\int \be \cup \bface} + h.c. \\
    =& -\sum_f \left({\oZ_f^{(k)}}\right)^{2k}+ h.c.~,  
\label{eq: 1/p BUB QCA separator Hamiltonian}
\end{eqs}
where we define
\begin{equation}
    \oZ_f^{(k)} := Z_f \prod_e {G_e^{(k)}}^{\frac{1}{2k}\int \be \cup \bface}~,
\end{equation}
with $\tfrac{1}{2k}$ denoting the multiplicative inverse of $2k$ in $\mathbb{Z}_p$. The construction is illustrated in Fig.~\ref{fig:Separators_Zp^k}.

\begin{figure*}[t]
    \centering
    \includegraphics[scale=.3]{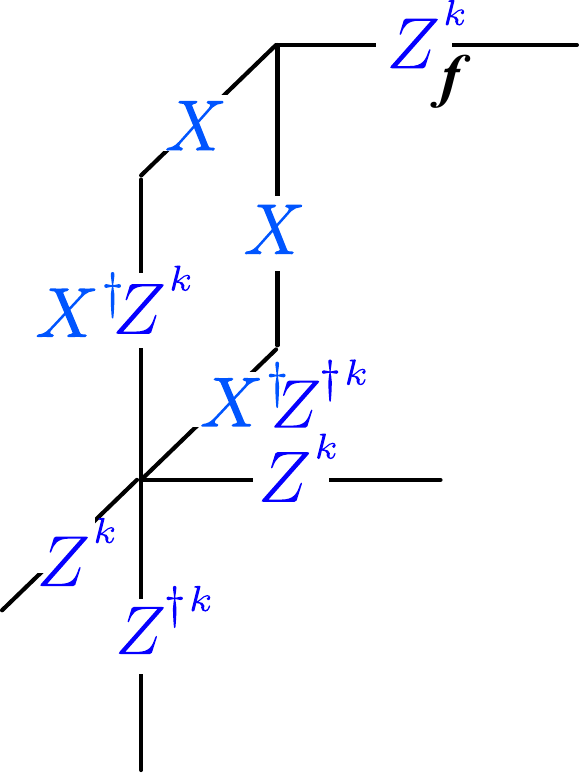} \quad\quad\quad
    \includegraphics[scale=.3]{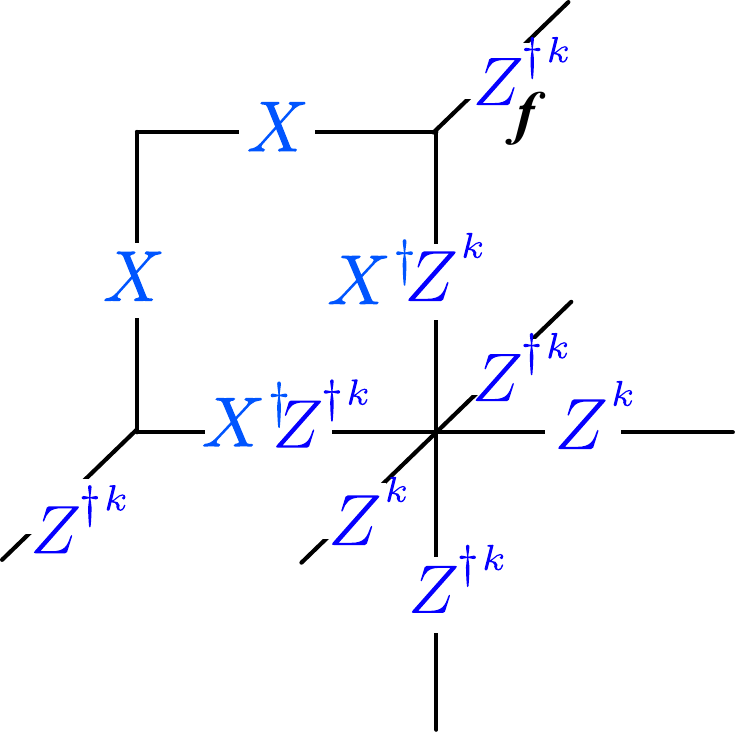} \quad\quad\quad
    \includegraphics[scale=.3]{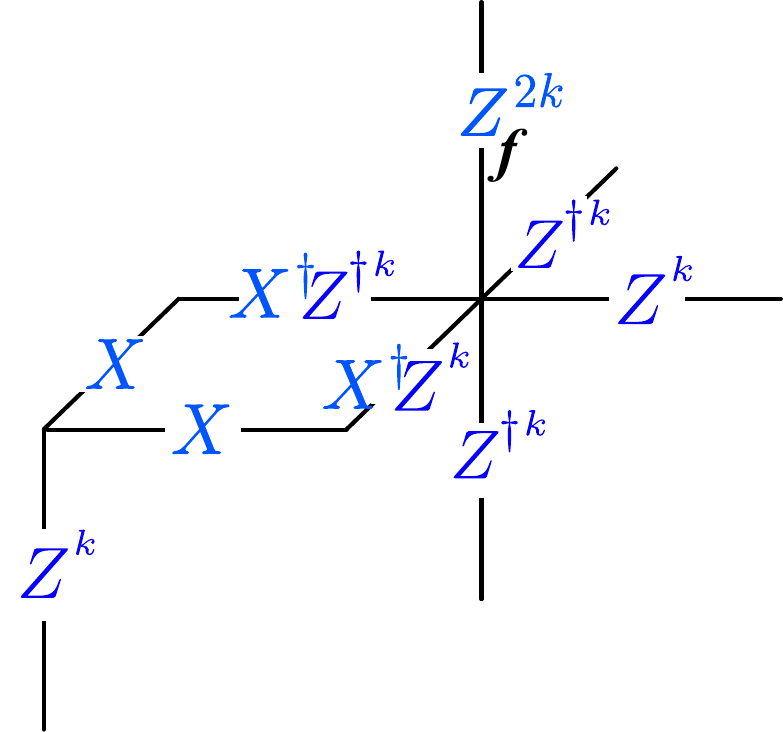}
    \caption{Separators $\overline{Z}_f^{2k}$ shown on the dual lattice.
    }
\label{fig:Separators_Zp^k}
\end{figure*}

We also define the ``hopping term'' as
\begin{eqs}
    U_f^{(k)} := X_f \prod_{f'} Z_{f'}^{k\int \bface' \cup_1 \bface}.
\end{eqs}
On a cubic lattice, $U_f^{(k)}$ can be depicted as 
\begin{equation}
    U_f^{(k)} = \vcenter{\hbox{\includegraphics[scale=0.23]{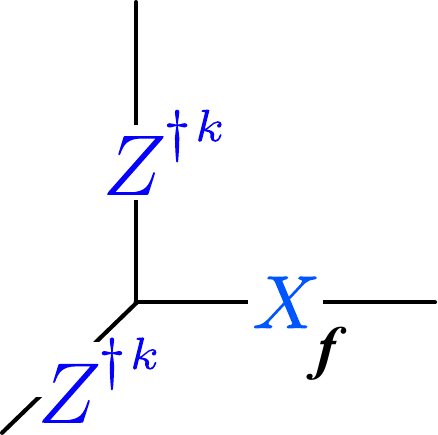}}},
    \hspace{1em}
    ~\vcenter{\hbox{\includegraphics[scale=0.23]{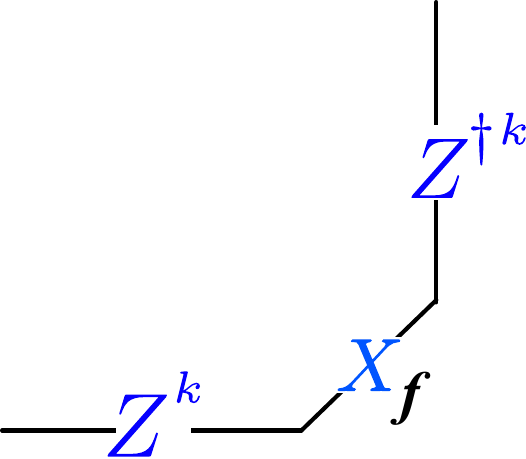}}},
    \hspace{1em}
    ~\vcenter{\hbox{\includegraphics[scale=0.23]{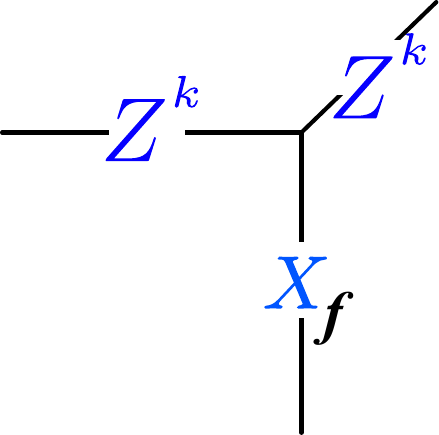}}},
\end{equation}
where faces on the original lattice are represented as edges on the dual lattice.
One checks that $U_f^{(k)}$ commutes with $G_e^{(k)}$
\begin{eqs}
    U_f^{(k)} G_e^{(k)} 
    &=
    \Bigl( X_f \prod_{f'} Z_{f'}^{k\int \bface' \cup_1 \bface} \Bigr)
    \Bigl( X_{\delta \be} \prod_{f''} Z_{f''}^{k\int \delta \be \cup_1 \bface''}  \Bigr) \\
    &= \exp\Bigl(\frac{2 \pi i k}{p}
    \bigl[\int \delta \be \cup_1 f - \delta \be \cup_1 f\bigr] \Bigr) \\
    & \quad~~ \Bigl( X_{\delta \be} 
    \prod_{f''} Z_{f''}^{k\int \delta \be \cup_1 \bface''}  \Bigr)
    \Bigl( X_f \prod_{f'} Z_{f'}^{k\int \bface' \cup_1 \bface} \Bigr) \\
    &= G_e^{(k)} U_f^{(k)}~.
\end{eqs}
As a result, each $U_f^{(k)}$ fails to commute with exactly one $Z_f^2$ term in Eq.~\eqref{eq: 1/p BUB QCA modified plaquette terms}. Hence, $\{U_f^{(k)}\}$ are the non-commuting ``flippers'' used to build the QCA, just as in Eq.~\eqref{eq: fermionic hopping operators} for the 3-fermion example.

To define the commuting flippers, we decorate $U_f^{(k)}$ with a product of $\overline{Z}_f$ 
\begin{eqs}
    \oX_f^{(k)} :=& U_f^{(k)} \prod_{f'} {\oZ_{f'}^{\dagger}}^{k\int \bface' \cup_1 \bface} \\
    =& X_f \prod_{f'} \bigl(Z_{f'} \oZ_{f'}^{\dagger}\bigr)^{k\int \bface' \cup_1 \bface} \\
    =& X_f \prod_{f'}
    \Bigl[ \prod_e {G_e^{(k)}}^{-\frac{1}{2k}\int \be \cup \bface} \Bigr]^{k\int \bface' \cup_1 \bface} \\
    =& X_f \prod_{f'}
    \Bigl[ \prod_e {G_e^{(k)}}^{-\frac{1}{2}\int \be \cup \bface} \Bigr]^{\int \bface' \cup_1 \bface},
\label{eq: complete Zp QCA}
\end{eqs}
in analogy to Eq.~\eqref{eq: definition of tilde X^A and X^B}. Hence, the complete $\ZZ_p^{(k)}$ QCA, denoted $\alpha_p^{(k)}$, is defined by
\begin{eqs}
    \alpha_p^{(k)}: Z_f \rightarrow&~\alpha_p^{(k)}(Z_f) 
    = \oZ_f^{(k)} \\
    &= Z_f \prod_e {G_e^{(k)}}^{\frac{1}{2k}\int \be \cup \bface}, \\
    \alpha_p^{(k)}: X_f \rightarrow& ~\alpha_p^{(k)}(X_f) 
    = \oX_f^{(k)} \\ 
    &= X_f \prod_{f'} \Bigl[ \prod_e {G_e^{(k)}}^{-\frac{1}{2}\int \be \cup \bface} \Bigr]^{\int \bface' \cup_1 \bface}.
    \label{eq: definition of alpha_p}
\end{eqs}
As in Eq.~\eqref{eq:U^A_invariant_under_QCA} and Eq.~\eqref{eq:U^B_invariant_under_QCA}, we find that $U_f^{(k)}$ remains invariant under this QCA
\begin{equation}
\begin{aligned}
    U_f^{(k)} \rightarrow & ~\oX_f \prod_{f'} \oZ_{f'}^{k\int \bface' \cup_1 \bface} \\ 
    &= U_f^{(k)} \prod_{f'} {\oZ_{f'}^{\dagger}}^{k\int \bface' \cup_1 \bface} \prod_{f'} \oZ_{f'}^{k\int \bface' \cup_1 \bface} 
    = U_f^{(k)}.
\end{aligned}
\end{equation}
It also follows from Eq.~\eqref{eq: complete Zp QCA} that $Z_{\partial c}\oZ_f^{(k)}$ is invariant under this QCA
\begin{equation}
    \overline{Z}_{\partial c}^{(k)} = \prod_f \left(\oZ_f^{(k)}\right)^{[\bface](\partial c)} = \prod_f Z_f^{[\bface](\partial c)} = Z_{\partial c}.
\end{equation}
We stress, however, that $G_e^{(k)}$ is \emph{not} invariant under the QCA, because in the $\ZZ_p$ case, $G_e^{(k)}$ cannot be expressed as a product of $U_f^{(k)}$ and $Z_{\partial c}$. Concretely, there is a leftover term of $Z^2$, which vanishes only when $p=2$
\begin{eqs}
    G_e^{(k)} =& X_{\delta \be} \prod_{f'} Z_{f'}^{k\int \delta \be \cup_1 \bface'} \\
    =& \Bigl( X_{\delta \be} \prod_{f'} Z_{f'}^{k\int  \bface' \cup_1  \delta \be} \Bigr)
    \prod_{f'} Z_{f'}^{k\int  \delta \be \cup_1 \bface' - \bface' \cup_1  \delta \be} \\
    =& U_{\delta \be}^{(k)} \prod_{f'} Z_{f'}^{k\int  \delta \be \cup_2 \delta \bface' + 2 \delta \be \cup_1 \bface'}\\
    =&U_{\delta \be}^{(k)} \prod_{f'} Z_{f'}^{k\int  \delta \be \cup_2 \delta \bface' + 2 (\bface' \cup \be - \be \cup \bface' + \be \cup_1 \delta \bface')}\\
    =&  U_{\delta \be}^{(k)} \prod_{c} Z_{\partial c}^{k\int  \delta \be \cup_2 \bc + 2 \be \cup_1 \bc}
    \prod_{f'} Z_{f'}^{2k \int \bface' \cup \be - \be \cup \bface'},
\end{eqs}
where we have employed Eq.~\eqref{eq: rearrange cup to cup1} and the recursive formula for higher cup products
\begin{eqs}
    \delta (A_2 \cup_2 B_2) &= \delta A_2 \cup_2 B_2 + A_2 \cup_2 \delta B_2 \\
    & \hspace{2em} + A_2 \cup_1 B_2 + B_2 \cup_1 A_2.
\label{eq: 2 2 2 recursive formual}
\end{eqs}
Notably, on the cubic lattice, there is a one-to-one correspondence between edges and faces, characterized by $\int \be \cup \bface \neq 0$ (and similarly a dual correspondence given by $\int \bface \cup \be \neq 0$). More precisely, the following identity holds on the cubic lattice, but not on a generic triangulation:
\begin{equation}
    \prod_{e} \Bigl( \prod_{f'} Z_{f'}^{\int \be \cup \bface'}\Bigr)^{\int \be \cup \bface} = Z_f.
\end{equation}
We can now decompose $\alpha_p^{(k)}(Z_f)$ as
\begin{eqs}
    &\alpha_p^{(k)}(Z_f) \\
    &= Z_f \prod_e {G_e^{(k)}}^{\frac{1}{2k}\int \be \cup \bface} \\
    &= \prod_e \Bigl( U_{\delta \be}^{(k)} \prod_{c} Z_{\partial c}^{k\int  \delta \be \cup_2 \bc + 2 \be \cup_1 \bc} \Bigr)^{\frac{1}{2k}\int \be \cup \bface} \\
    & \hspace{3em} \times \prod_e
    \Bigl( \prod_{f'} Z_{f'}^{\int \bface' \cup \be } \Bigr)^{\int \be \cup \bface} \\
    &= Z_{T(f)} \prod_e \Bigl( U_{\delta \be}^{(k)} \prod_{c} Z_{\partial c}^{k\int  \delta \be \cup_2 \bc + 2 \be \cup_1 \bc} \Bigr)^{\frac{1}{2k}\int \be \cup \bface}~,
\label{eq: alpha_p^k translate Zf along (-1, -1, -1)}
\end{eqs}
in which we have used
\begin{equation}
    \prod_e \Bigl( \prod_{f'} Z_{f'}^{\int \bface' \cup \be } \Bigr)^{\int \be \cup \bface}
    =Z_{T(f)},
    \label{eq: two e f duality are shift}
\end{equation}
where $T(f)$ denotes the face $f$ shifted by $(-1, -1, {-}1)$ on the cubic lattice. Thus, $\oZ_f$ is simply $Z_{T(f)}$ multiplied by $U_{f'}^{(k)}$ and $Z_{\partial c}$. When we apply $\alpha^{(k)}_p$ once again, the $U_{f'}^{(k)}$ and $Z_{\partial c}$ factors remain unchanged, and only $Z_{T(f)}$ is translated to $Z_{T(T(f))}$ (up to some additional $U_{f'}^{(k)}$ and $Z_{\partial c}$ factors), as illustrated in Fig.~\ref{fig:twice action of Zp QCA}.
\begin{figure*}
    \centering
    \includegraphics[width=0.9\linewidth]{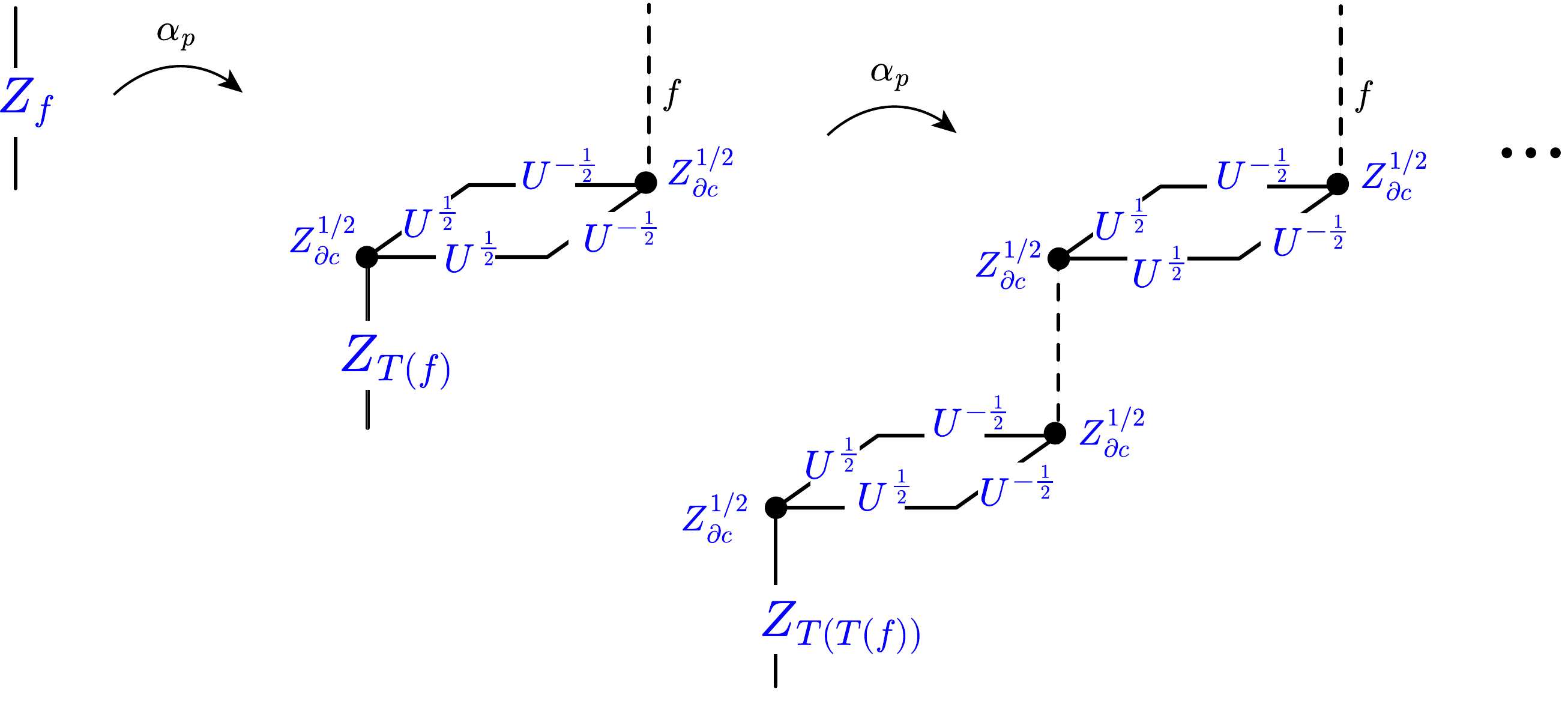}
    \caption{An illustration of the action of $\alpha^{(k)}_p$ QCA on a single Pauli-$Z$ operator on an $xy$ face (shown here as a $z$ edge on the dual lattice). Repeated applications forms a stair-like structure, where face $f$ is represented by an edge and cube $c$ is represented by a dot. The Pauli-$Z$ operators in other directions follow a similar inductive law.}
    \label{fig:twice action of Zp QCA}
\end{figure*}
We emphasize that the QCA constructed in Eq.~\eqref{eq: definition of alpha_p} is not the unique way to disentangle the ground state of the $\mathbb{Z}_p^{(k)}$ Walker–Wang Hamiltonian in Eq.~\eqref{eq: 1/p BUB QCA initial Hamiltonian}.
For example, Eq.~\eqref{eq: e U f and f U e to de U1 f} can be instead written as
\begin{equation}
    - \sum_e X_{\delta \be} \prod_{f'} Z_{f'}^{k\int 2 \bface' \cup \be  - \delta \be \cup_1 \bface' + \be \cup_1 \delta \bface'}.
\end{equation}
Following the same procedure, we obtain separators and flippers
\begin{eqs}
    \beta_p^{(k)}: Z_f \rightarrow& ~\beta_p^{(k)}(Z_f) \\
    &= Z_f \prod_e {G_e^{(-k)}}^{\frac{1}{2k}\int \bface \cup \be}, \\
    \beta_p^{(k)}: X_f \rightarrow& ~\beta_p^{(k)}(X_f) \\
    &= U_f^{(-k)} \prod_{f'} \Bigl[\beta_p^{(k)}(Z_f)^\dagger\Bigr]^{-k \int \bface' \cup_1 \bface} \\
    &= X_f \prod_{f'} \Bigl[Z_{f'} \beta_p^{(k)}(Z_{f'})^\dagger\Bigr]^{-k \int \bface' \cup_1 \bface} \\
    &= X_f \prod_{f'} \Bigl[ \prod_e {G_e^{(k)}}^{\frac{1}{2}\int \bface \cup \be} \Bigr]^{\int \bface' \cup_1 \bface}.
    \label{eq: definition of beta_p}
\end{eqs}
The QCA $\beta_p^{(k)}$ preserves both $U_f^{(-k)}$ and $Z_{\partial c}$. Combined with Eqs.~\eqref{eq: alpha_p^k translate Zf along (-1, -1, -1)} and~\eqref{eq: definition of beta_p}, it is straightforward to verify that
\begin{equation}
    \beta_p^{(-k)}(\alpha_p^{(k)}(Z_f)) = Z_f,
\end{equation}  
on the cubic lattice. Therefore, $\beta_p^{(-k)}$ is the inverse of $\alpha_p^{(k)}$ on the cubic lattice.  

Both $\alpha_p^{(-k)}$ and $\beta_p^{(-k)}$ disentangle the same Walker-Wang ground state, suggesting that these two QCAs belong to the same class. In Sec.~\ref{sec: Algebraic formalism for Clifford QCA}, we will establish the operator algebra formalism to show that if two QCAs disentangle the same stabilizer group, they are equivalent up to finite-depth quantum circuits (FDQC) and translations. Consequently, we conclude that $\alpha_p^{(-k)} \circ \alpha_p^{(k)} \sim \mathbf{1}$, where $\mathbf{1}$ is the identity operator and $\sim$ denotes equivalence up to FDQC and translations.

At the end of this section, we will prove that the order of $\ZZ_p^{(k)}$ QCA is 2 or 4, depending on p mod 4 be 1 or 3. The proof will be divided into two parts.

\subsubsection{Order of $\ZZ_p$ QCA for $p\equiv 1 \pmod{4}$}

We will prove that when \(p \equiv 1 \ (\text{mod}\ 4)\), the \(\ZZ_p^{(k)}\) QCA can be transformed into the \(\ZZ_p^{(-k)}\) QCA \(\alpha_p^{(-k)}\)---with action \(S=-\tfrac{k}{p}B_2\cup B_2\)---by an on-site unitary transformation. Furthermore, in the Appendix~\ref{app: Polynomial formalism for Clifford QCA}, we use a polynomial representation to show that \(\ZZ_p^{(-k)}\) QCA \(\alpha_p^{(-k)}\) is the inverse of \(\ZZ_p^{(k)}\) QCA \(\alpha_p^{(k)}\) up to a finite-depth circuit.

When \(p \equiv 1 \ (\text{mod}\ 4)\), $-1$ is a quadratic residue mod $p$, i.e., there exists $a$ such that $a^2 \equiv -1 \pmod{p}$.

We now define the following transformation \(c_p\) on each face
\begin{equation}
    c_p (Z_f) := Z_f^a, 
    \quad 
    c_p (X_f) := X_f^{-a}.
\end{equation}
This transformation is unitary because \(-a\) is the multiplicative inverse of \(a\) mod $p$. Moreover, it can be viewed as a squareroot of charge conjugation. In the cochain language, this corresponds to the map \(B \rightarrow aB\). Substituting \(B \to aB\) into the original \(\ZZ_p^{(k)}\) action \(S=\tfrac{k}{p}B_2\cup B_2\) yields the action \(\,S'=-\tfrac{k}{p}B_2\cup B_2\), with the corresponding \(\ZZ_p^{(-k)}\) QCA defined as $c_p \circ \alpha_p^{(k)} \circ c_p^{-1}$.
Indeed, we can verify that this matches the expression for $\alpha_p^{(-k)}(Z_f)$ defined in Eq.~\eqref{eq: definition of alpha_p}
\begin{eqs}
    c_p(\alpha_p^{(k)} &(c_p^{-1}(Z_f))) \\
    =& 
    c_p^{-1}\left(
    Z_f \prod_e {\Bigl[ X_{\delta \be} \prod_{f'} Z_{f'}^{k\int \delta \be \cup_1 \bface'}\Bigr]}^{\frac{1}{2k}\int \be \cup \bface}
    \right)^{-a} \\
    =& Z_f \prod_e {\Bigl[ X^\dagger_{\delta \be} \prod_{f'} Z_{f'}^{k\int \delta \be \cup_1 \bface'}\Bigr]}^{\frac{1}{2k}\int \be \cup \bface} \\
    =& Z_f \prod_e {G^{(-k)}_e}^{-\frac{1}{2k}\int \be \cup \bface}
    = \alpha_p^{(-k)}(Z_f).
\end{eqs}
 There is another way to express the above quantity as
\begin{eqs}
    & Z_f \prod_e {\Bigl[ X^\dagger_{\delta \be} \prod_{f'} {Z^\dagger_{f'}}^{k\int \delta \be \cup_1 \bface'}\Bigr]}^{\frac{1}{2k}\int \be \cup \bface} \\
    & \hspace{8em} \times \prod_{e, f'} \Bigl[ Z_{f'}^{\int \delta \be \cup_1 \bface'}
    \Bigr]^{\int \be \cup \bface} \\
    &= Z_f \prod_e {G_e^{(k)}}^{-\frac{1}{2k}\int \be \cup \bface}
    \prod_{e, f'} \Bigl[ Z_{f'}^{\int \delta \be \cup_1 \bface'}
    \Bigr]^{\int \be \cup \bface}.
\end{eqs}
Now, we focus on the cubic lattice to proceed.
Note that \(\int \be \cup \bface \neq 0\) induces a one-to-one correspondence between edges \(e\) and faces \(f\) on the cubic lattice. Hence, we can rewrite the above expression as
\begin{eqs}
    &{G_e^{(k)}}^{-\frac{1}{2k}}
    \prod_{f} Z_f^{\int \be \cup \bface}
    \prod_{f'}  Z_{f'}^{\int \delta \be \cup_1 \bface'} \\
    &= {G_e^{(k)}}^{-\frac{1}{2k}} \prod_{f} Z_f^{\int \bface \cup \be+ \be \cup_1 \delta \bface},
\end{eqs}
where we have used Eq.~\eqref{eq: rearrange cup to cup1}. By applying the other duality between \(e\) and \(f\) for \(\int \bface \cup \be \neq 1\), we can rearrange it into
\begin{equation}
\begin{aligned}
    &\alpha_p^{(-k)}(Z_{T^{-1}(f)}) \\
    &= Z_f \prod_e {G_e^{(k)}}^{-\frac{1}{2k} \int \bface \cup \be} 
    \prod_{e, f'} \Bigl[ Z_{f'}^{\int \be \cup_1 \delta \bface'}\Bigr]^{\int \bface \cup \be} \\
    &= Z_f \prod_e {G_e^{(k)}}^{-\frac{1}{2k} \int \bface \cup \be} 
    \prod_{e, c} \Bigl[ Z_{\partial c}^{\int \be \cup_1 \bc}\Bigr]^{\int \bface \cup \be},
\end{aligned}
    \label{eq: alpha_p^-k on Z_f}
\end{equation}
where the overall shift along \(({-1}, {-1}, {-1})\) comes from Eq.~\eqref{eq: two e f duality are shift}.

Comparing \(\beta_p^{(-k)}(Z_f)\) from Eq.~\eqref{eq: definition of beta_p} with \(\alpha_p^{(-k)}(Z_{T^{-1}(f)})\) in Eq.~\eqref{eq: alpha_p^-k on Z_f}, we find that they differ only by a \(Z_{\partial c}\) term. Since such \(Z_{\partial c}\) factors lie in the stabilizer group, we conclude that both \(\alpha_p^{(-k)}\) and \(\beta_p^{(-k)}\) disentangle the same \(\ZZ_p^{(-k)}\) Walker-Wang ground state, which corresponds to the action \(S' = -\frac{k}{p} B_2 \cup B_2\). 

{
\change 
Before proceeding, we emphasize a key result stating that any Clifford QCAs that disentangle the same stabilizer group are equivalent. 
Corollary~\ref{corollary: separated QCAs trivial}, proved in Sec.~\ref{sec: QCAs with the same stabilizer group}, is essential for our argument. Even though two QCAs disentangle the same Walker-Wang ground state, it is not immediately obvious that these two QCAs should be equivalent since they map the stabilizers differently. The corollary guarantees that it suffices to consider a single representative QCA that disentangles the given ground state, without concerning that different choices of Hamiltonian lead to inequivalent QCAs.

Therefore, by Corollary~\ref{corollary: separated QCAs trivial}}, $ \alpha_p^{(-k)} $ and $ \beta_p^{(-k)} $ are related by a finite-depth quantum circuit (FDQC) and translations.
In Sec.~\ref{order of 1 mod 4 Zp QCA in matrix form}, we write $ \alpha_p^{(-k)} $ explicitly as matrices over polynomial rings and exhibit the FDQC and translations that relate these two matrices.

\subsubsection{Order of $\ZZ_p$ QCA for $p\equiv 3 \pmod{4}$}

In the case that $p \equiv 3 \ (\text{mod}\ 4)$, we'll show that two copies of $\ZZ_p^{(1)}$ QCA are transformed into two copies of $\ZZ_p^{(-1)}$ QCA by a local unitary communicating the two copies.
In the case where $p \equiv 3 \ (\text{mod}\ 4)$, the term $e^{\frac{p-1}{4}}$ is no longer valid, so we should seek help from another copy of QCA. To begin with, we present a number theory proposition, that there exists two integers $a$ and $b$, satisfying $a^2+b^2 \equiv -1 \ (\text{mod}\ p)$.
To prove this, we first modify the equation to be $(1+x^2)a^2 \equiv -1 \ (\text{mod}\ p)$. There must exist an $x$ for $1+x^2$ to be quadratic non-residue. Otherwise quadratic residue can be extended to the whole residue class, which is obviously contradictory. Since  quadratic non-residues correspond to the odd power of any primitive roots, the quotient of two quadratic non-residues (in number theory that means the multiplication of one and the other's inverse) is exactly a quadratic residue. Namely, there must exists an $a$ for $(1+x^2)a^2 \equiv -1 \ (\text{mod}\ p)$ when $1+x^2$ is a quadratic non-residue. Denoting $x\times a$ as $b$, we obtain the solution for $a^2+b^2 \equiv -1 \ (\text{mod} p)$ 

We label the qudits of the two QCA copies as $Z_f^1$ and $Z_f^2$. The transformation is then given by 
\begin{align}
    Z_f^1&\rightarrow (Z_f^1)^a (Z_f^2)^b~,\\
    Z_f^2&\rightarrow (Z_f^1)^b(Z_f^2)^{-a}~,\\
    X_f^1&\rightarrow (X_f^1)^{-a}(X_f^2)^{-b}~,\\
    X_f^2&\rightarrow (X_f^1)^{-b}(X_f^2)^a~.
\end{align}
While in the cochain representation, that can be shown as $B^1\rightarrow aB^1+bB^2, B^2\rightarrow bB^1-aB^2$. Substituting the transformation into the two-copy $\ZZ_p^{(1)}$ action, we obtain $S'=-\frac{1}{p}(B^{1}\cup B^{1}+B^{2}\cup B^{2})$, the action for two-copy $\ZZ_p^{(-1)}$ QCA.

As we have proved that $\ZZ_p^{(1)}$ QCA and $\ZZ_p^{(-1)}$ QCA are inverse in the appendix, we finally reached the conclusion that $\ZZ_p^{(1)}$ QCA has order 2 when $p \equiv 1 \ (\text{mod}\ 4)$,and that it has order 4 when $p \equiv 3 \ (\text{mod}\ 4)$.

\subsubsection{Polynomial expression of odd prime QCA}

First, we have
\begin{align}
    \bU_f^{(k)} &= \bX_f^{(k)} + k \bZ_{f'} \bM_{f' \cup_1 f}, \\
    \tilde \bU_f^{(k)} &=\bX_f^{(k)} + k \bZ_{f'} \bM_{f \cup_1 f'}^\dagger \\
    \bG_e^{(k)}&= \tilde \bU_{\delta e}^{(k)} .
\end{align}
In term of the polynomial, they are
\begin{align}
    \bU^{(k)} &= \begin{pmatrix}
        \bd_{f_0,f}\\k \bM_{f_0\cup_1f}\end{pmatrix}, \\
    \tilde \bU^{(k)}&= \begin{pmatrix}
       \bd_{f_0,f}\\k \bM_{f'\cup_1f}^\dagger
    \end{pmatrix}, \\
    \bG^{(k)}&= \begin{pmatrix}
        \bd_{f_0,\delta e}\\k \bM_{\delta e\cup_1f_0}^\dagger \end{pmatrix}.     
\end{align}

Then separators and flippers for the odd prime QCA are
\begin{align}
       \overline{\bZ}_f^{(k)} &= \bZ_f + \frac{1}{2k} \bG_e^{(k)} \bM_{e \cup f} \label{eq:bZ_odd_prime}\\
       \overline{\bX}_f^{(k)} &= \bU_f^{(k)} - k\overline{\bZ}_{f'} \bM_{f'\cup_1 f} \notag \\
    &= \bX_f -\frac{1}{2}\bG^{(k)}_e\bM_{e\cup f'}\bM_{f'\cup_1f} \label{eq:bX_odd_prime}
\end{align}
In polynomial form, we have
\begin{align}
    &\overline{\bZ}^{(k)}=  \begin{pmatrix} \frac{1}{2k}\bM_{\partial f_0\cup f}\\
    \bd_{f_0,f} + \frac{1}{2} \bM^\dagger_{f'\cup_1 f_0} \bM_{\partial f'\cup f}\end{pmatrix}~,
\end{align}
whose matrix elements are
\begin{equation*}
    \begin{pmatrix}
 0 & \frac{1}{2 k x z}-\frac{1}{2 k x} & \frac{1}{2 k x y}-\frac{1}{2 k x} \\
 \frac{1}{2 k y}-\frac{1}{2 k y z} & 0 & \frac{1}{2 k x y}-\frac{1}{2 k y} \\
 \frac{1}{2 k z}-\frac{1}{2 k y z} & \frac{1}{2 k z}-\frac{1}{2 k x z} & 0 \\
 \frac{1}{2 y z}+\frac{1}{2} & \frac{1}{2 x z}-\frac{1}{2 z} & \frac{1}{2}-\frac{1}{2 x} \\
 \frac{1}{2 y^2 z}-\frac{1}{2 y z} & \frac{1}{2 x y z}-\frac{1}{2 y z}+\frac{1}{2 z}+\frac{1}{2} & \frac{1}{2 y}-\frac{1}{2} \\
 \frac{1}{2 y z}-\frac{1}{2 y z^2} & \frac{1}{2 z^2}-\frac{1}{2 z} & \frac{1}{2 x y z}-\frac{1}{2 z}+1 \\
\end{pmatrix}
\end{equation*}
and
\begin{align}
    \overline{\bX}^{(k)}&=\begin{pmatrix} \bd_{f_0,f}  - \frac{1}{2}\bM_{\partial f_0 \cup f'}\bM_{f'\cup_1f}\\
   -\frac{k}{2}\bM_{\delta e\cup_1f_0}^\dagger \bM_{ e\cup f'}\bM_{f'\cup_1f} \end{pmatrix}~,
\end{align}
whose matrix elements are
\begin{widetext}
\begin{equation}
    \begin{pmatrix}
 \frac{1}{2 x y z}-\frac{1}{2 x}+1 & \frac{1}{2 x}-\frac{y}{2 x} & \frac{z}{2 x}-\frac{1}{2 x} \\
 \frac{1}{2 x y}-\frac{1}{2 y} & \frac{1}{2 x y z}-\frac{1}{2 x y}+\frac{1}{2 x}+\frac{1}{2} & \frac{1}{2 x y}-\frac{z}{2 x y} \\
 \frac{1}{2 y z}-\frac{1}{2 x y z} & \frac{1}{2 x y z}-\frac{1}{2 x z} & \frac{1}{2 x y}+\frac{1}{2} \\
 \frac{k}{2 x y z}-\frac{k}{2 x}-\frac{k}{2 y z}+\frac{k}{2} & -\frac{k}{2 x y z}-\frac{k y}{2 x}+\frac{k}{2 x}+\frac{k y}{2} & -\frac{k}{2 x y}+\frac{k z}{2 x}-\frac{k}{2 x}+\frac{k}{2} \\
 \frac{k}{2 x y^2 z}-\frac{k}{2 y^2 z}+\frac{k}{2 y z}-\frac{k}{2} & -\frac{k}{2 x y^2 z}+\frac{k}{2 x y z}-\frac{k y}{2}+\frac{k}{2} & -\frac{k}{2 x y^2}+\frac{k}{2 y}+\frac{k z}{2}-\frac{k}{2} \\
 \frac{k}{2 x y z}+\frac{k}{2 y z^2}-\frac{k}{2 y z}-\frac{k}{2 z} & \frac{k}{2 x y z^2}-\frac{k}{2 x y z}+\frac{k}{2 x z}-\frac{k y}{2 z} & \frac{k}{2 x y z}-\frac{k}{2 x y}-\frac{k}{2 z}+\frac{k}{2} 
\end{pmatrix} ~.
\end{equation}
We can confirm that $\overline{\bZ}_f^{(k)}$ generates $\bZ_{\partial c}$ since
\begin{align*}
    \overline{\bZ}_{\partial c} &= \bZ_{\partial c}  + \frac{1}{2k} \tilde \bG_e^{(k)} \bM_{e\cup \partial c} 
    = \bZ_{\partial c}  + \frac{1}{2k}  \bG_e^{(k)} \bM_{e\cup \partial c} 
    = \bZ_{\partial c}  + \frac{1}{2k} \tilde \bU_{\delta e}^{(k)} \cancelto{0}{\bM_{\partial e\cup \partial c}}
\end{align*}
With this, we can now define the QCA as
\begin{align}\label{eq:ZkpQCA_from_TQFT}
    \balpha^{(k)} &= \begin{pmatrix}\overline{\bX}^{(k)} &  \overline{\bZ}^{(k)}\end{pmatrix}
    =\begin{pmatrix}
         \bd_{f_0,f} - \frac{1}{2}\bM_{\partial f_0 \cup f'}\bM_{f'\cup_1f} &  \frac{1}{2k}\bM_{\partial f_0\cup f}\\
   -\frac{k}{2}\bM_{f''\cup_1f_0}^\dagger \bM_{\partial f''\cup f'}\bM_{f'\cup_1f} & \bd_{f_0,f}  + \frac{1}{2} \bM^\dagger_{f'\cup_1 f_0} \bM_{\partial f'\cup f} \end{pmatrix}~.
\end{align}
From the expression above, it is clear that $ \balpha^{(k)}\bU^{(k)} = \bU^{(k)}$.

\subsubsection{$p \equiv 1 \ (\text{mod} \ 4)$}\label{order of 1 mod 4 Zp QCA in matrix form}
In this case, there exists $a$ such that $a^2=-1$ (mod $p$). The onsite basis transformation is given by
$\textbf{\textsf{c}} = \begin{pmatrix}
    -a \bbone & 0\\ 0 & a\bbone
\end{pmatrix}$, which is a symplectic matrix. Conjugating a matrix results in
\begin{align}
    \begin{pmatrix}
    -a \bbone & 0\\ 0 & a\bbone
\end{pmatrix} \begin{pmatrix}
   \alpha_{11}& \alpha_{12}\\ \alpha_{21} & \alpha_{22}
\end{pmatrix}
 \begin{pmatrix}
    a \bbone & 0\\ 0 & -a\bbone
\end{pmatrix} = \begin{pmatrix}
   \alpha_{11}& -\alpha_{12}\\ -\alpha_{21} & \alpha_{22}
\end{pmatrix}
\end{align}
Since the dependence of $k$ in $\alpha$ is exactly off-diagonal, we can immediately see that $\textbf{\textsf{c}}   \balpha^{(k)} \textbf{\textsf{c}}^{-1} = \balpha^{(-k)}$~.
Let us now compute the square of the QCA (up to circuits):
\begin{align}
\balpha^{(-k)} \balpha^{(k)} &= 
\begin{pmatrix}
         \bd_{f_0,f} + (\frac{1}{4}\bO^{(14)}_{f_0,f'}-\bM_{\partial f_0 \cup f'})\bM_{f'\cup_1 f} & 0 \\
          0 & \bd_{f_0,f} + \bM^\dagger_{f'\cup_1 f_0}(\frac{1}{4}\bO^{(14)}_{f',f} +\bM_{\partial f'\cup f})
    \end{pmatrix}\\
    &=\begin{pmatrix}
         \bd_{f_0,f} -\bM_{\partial f_0 \cup f'}\bM_{f'\cup_1 f} & 0 \\
          0 & \bd_{f_0,f} + \bM^\dagger_{\delta e\cup_1 f_0}\bM_{e\cup f}
    \end{pmatrix}\\
    &=\begin{pmatrix}
         \bd_{f_0,f} -\bM_{\partial f_0 \cup f'}\bM_{f'\cup_1 f} & 0 \\
          0 & \bd_{f_0,f} + (\bM_{f_0\cup e} - \bM^\dagger_{e\cup f_0}+\bM^\dagger_{e\cup_1 \delta f_0})\bM_{e\cup f}
    \end{pmatrix}~.
\end{align}
Thus, this QCA does not mix $Z$ and $X$ stabilizers. The result of Sec.~\ref{sec: Algebraic formalism for Clifford QCA} therefore implies that this QCA can be decomposed into circuits and shifts.

We may further explicitly find such a decomposition on the cubic lattice. Specifically, due to the unique pairing of edges and faces via the cup product, we have a special property that $\bM^\dagger_{e\cup f_0} \bM_{e\cup f} = \bd_{f_0,f}$~.
Thus,
\begin{align}
\balpha^{(-k)} \balpha^{(k)} &=\begin{pmatrix}
         \bd_{f_0,f} -\bM_{\partial f_0 \cup f'}\bM_{f'\cup_1 f} & 0 \\
          0 & (\bM_{f_0\cup e}+\bM^\dagger_{e\cup_1 \delta f_0})\bM_{e\cup f}
    \end{pmatrix}~.
\end{align}
Next, we show that $\balpha^{(-k)} \balpha^{(k)}$ is composed of circuits together with shifts. Consider the following transformation
\begin{align}
    \bU&=\begin{pmatrix}
        \bM_{f_0\cup e}^\dagger+\bM_{e\cup_1 \delta f_0} & 0 \\
          0 & (\bM_{f_0\cup e}+\bM^\dagger_{e\cup_1 \delta f_0})^{-1}
    \end{pmatrix}\\ &= \begin{pmatrix}
 1 & y-1 & 1-z & 0 & 0 & 0 \\
 \frac{1}{x}-1 & \frac{y}{x}-\frac{1}{x}-y & \frac{1}{x}-\frac{z}{x} & 0 & 0 & 0 \\
 \frac{1}{x y}-\frac{1}{y} & \frac{1}{x}-\frac{1}{x y} & -\frac{z}{x y}+\frac{1}{x y}+z & 0 & 0 & 0 \\
 0 & 0 & 0 & -\frac{x}{y z}+x+\frac{1}{y z} & \frac{x}{z}-\frac{1}{z} & 1-x \\
 0 & 0 & 0 & \frac{1}{y z}-\frac{1}{z} & \frac{y}{z}-y-\frac{1}{z} & 1-y \\
 0 & 0 & 0 & \frac{1}{y z}-\frac{1}{y} & 1-\frac{1}{z} & 1 \\
\end{pmatrix} .
\end{align}
This can be decomposed in terms of CNOT gates, and so it is a circuit. In particular,
\begin{align}
    \bU = \text{diag}(1,-1,1,1,-1,1)\bU_1 \bU_2 \bU_3 \bU_4
\end{align}
where 
\begin{align}
    \bU_1 &= \begin{pmatrix}
 1 & 0 & 0 & 0 & 0 & 0 \\
 0 & 1 & 0 & 0 & 0 & 0 \\
 \bar x-1 & 1-\bar y & 1 & 0 & 0 & 0 \\
 0 & 0 & 0 & 1 & 0 & 1-x \\
 0 & 0 & 0 & 0 & 1 & y-1 \\
    \end{pmatrix}, &
        \bU_2 &= \begin{pmatrix}
 1 & 0 & 0 & 0 & 0 & 0 \\
 1-\bar x & 1 & 0 & 0 & 0 & 0 \\
 0 & 0 & 1 & 0 & 0 & 0 \\
 0 & 0 & 0 & 1 & x-1 & 0 \\
 0 & 0 & 0 & 0 & 1 & 0 \\
 0 & 0 & 0 & 0 & 0 & 1 \\
    \end{pmatrix}, \\
    \bU_3 &= \begin{pmatrix}
 1 & 0 & 1-z & 0 & 0 & 0 \\
 0 & 1 & z-1 & 0 & 0 & 0 \\
 0 & 0 & 1 & 0 & 0 & 0 \\
 0 & 0 & 0 & 1 & 0 & 0 \\
 0 & 0 & 0 & 0 & 1 & 0 \\
 0 & 0 & 0 & \bar z-1 & 1-\bar z & 1 \\
    \end{pmatrix}, &
    \bU_4 &= \begin{pmatrix}
 1 & y-1 & 0 & 0 & 0 & 0 \\
 0 & 1 & 0 & 0 & 0 & 0 \\
 0 & 0 & 1 & 0 & 0 & 0 \\
 0 & 0 & 0 & 1 & 0 & 0 \\
 0 & 0 & 0 & 1-\bar y & 1 & 0 \\
 0 & 0 & 0 & 0 & 0 & 1 \\
    \end{pmatrix}.
\end{align}

Acting $\bU$ on $\balpha^{(-k)} \balpha^{(k)}$, we have
\begin{align}
   \bU  \balpha^{(-k)} \balpha^{(k)} = 
   \begin{pmatrix}
        \bM_{e\cup f} & 0 \\
          0 & \bM_{e\cup f}
    \end{pmatrix} = \text{diag}(\bar y \bar z, -\bar x \bar z, \bar x \bar y, \bar y \bar z, -\bar x \bar z, \bar x \bar y)
\end{align}
Thus, we have shown that the QCA is order two up to circuits and shifts.

\subsubsection{$p\equiv3\ (\text{mod} \ 4)$ }

In this case, we consider two copies of QCA, which is just the direct sum $\balpha^{(k)}\oplus\balpha^{(k)}$. In the matrix notation, it can be expressed as a block diagonal matrix
$\begin{pmatrix}
    \balpha^{(k)}&0\\
    0&\balpha^{(k)}
\end{pmatrix}.$
When $p\equiv3\ (\text{mod} \ 4)$, there exists $a$ and $b$ satisfying $a^2+b^2=-1$ (mod $p$). We can consider the following transformation matrix:
\begin{align}
\bc=\begin{pmatrix}
    a\bbone&0&b\bbone&0\\
    0&-a\bbone&0&-b\bbone\\
    b\bbone&0&-a\bbone&0\\
    0&-b\bbone&0&a\bbone
\end{pmatrix}~.
\end{align}
This corresponds to a transformation connecting the same sites in each QCA copy. Applying this transformation to the original QCA, we obtain
\begin{eqs}
\bc \begin{pmatrix}
    \balpha^{(k)}&0\\
    0&\balpha^{(k)}
\end{pmatrix} \bc^{-1}
&=\begin{pmatrix}
    a\bbone&0&b\bbone&0\\
    0&-a\bbone&0&-b\bbone\\
    b\bbone&0&-a\bbone&0\\
    0&-b\bbone&0&a\bbone
\end{pmatrix}
\begin{pmatrix}
    a_{11}&a_{12}&0&0\\
    a_{21}&a_{22}&0&0\\
    0&0&a_{11}&a_{12}\\
    0&0&a_{21}&a_{22}
\end{pmatrix}
\begin{pmatrix}
    -a\bbone&0&-b\bbone&0\\
    0&a\bbone&0&b\bbone\\
    -b\bbone&0&a\bbone&0\\
    0&b\bbone&0&-a\bbone
\end{pmatrix}\\
&=\begin{pmatrix}
        a_{11}&-a_{12}&0&0\\
    -a_{21}&a_{22}&0&0\\
    0&0&a_{11}&-a_{12}\\
    0&0&-a_{21}&a_{22}
\end{pmatrix}
=\begin{pmatrix}
    \balpha^{(-k)}&0\\
    0&\balpha^{(-k)}
\end{pmatrix}~.
\end{eqs}
This implies that after transformation, the original QCA is again decomposed into two copies of $\balpha^{(-k)}$ QCA.

According to the previous subsection, the composition of these two QCAs can be decomposed into two separate set of finite depth circuits $\bU\oplus\bU$
\begin{eqs}
&\begin{pmatrix}
    \bU&0\\
    0&\bU
\end{pmatrix}\begin{pmatrix}
    \balpha^{(-k)}&0\\
    0&\balpha^{(-k)}
\end{pmatrix}\begin{pmatrix}
    \balpha^{(k)}&0\\
    0&\balpha^{(k)}
\end{pmatrix}
=
\text{diag}\Bigl(\bar y \bar z, -\bar x \bar z, \bar x \bar y, \bar y \bar z, -\bar x \bar z, \bar x \bar y,\bar y \bar z,  -\bar x \bar z, \bar x \bar y, \bar y \bar z, -\bar x \bar z, \bar x \bar y\Bigl)~.
\end{eqs}
Since two copies of a QCA are equivalent to a single QCA applied twice, the order of this QCA is 4.

\end{widetext}

\section{$\ZZ_2$ and $\ZZ_p$ Clifford QCAs in higher dimensions from TQFTs}\label{sec:Z2_Zp_higher_QCA_TQFT}

In this section, we generalize the 3-fermion QCA from the previous section to arbitrary $2k{+}1$ spatial dimensions. The construction is derived directly from the corresponding topological action, allowing the QCA to be defined on arbitrary cellulations and thereby providing greater flexibility in lattice realizations. When restricted to the hypercubic lattices, our formulation reproduces the construction of Ref.~\cite{fidkowski2024qca}.

\subsection{$\ZZ_2$ QCA in $5{+}1$D from topological action $S=\tfrac{1}{2}\int A_3\cup A_3+A_3\cup B_3+B_3\cup B_3$}
\label{sec: w1w2^2}
The action for this QCA is given by
\begin{align}
    S=\frac{1}{2}\int A_3\cup A_3+A_3\cup B_3+B_3\cup B_3~.
\end{align}
This expression closely resembles the 3-fermion Walker-Wang QCA in $3{+}1$D, with the primary difference being the change in the dimensionality of the cocycles. Consequently, the lattice construction for this QCA follows the same general framework as in the 3-fermion Walker-Wang case.
The corresponding separators and (non-commuting) flippers on hypercubic lattices are
\begin{eqs}
    \overline{Z}^A_c=Z^A_c\prod_f G^{B \int \boldsymbol{c}\cup\bface}_f~,
    &\quad \overline{Z}^B_c=Z^B_c\prod_f G^{A \int \bface\cup\boldsymbol{c}}_f~,\\
    U_c^A = X^A_c\prod_{c'}Z^{A \int \boldsymbol{c}'\cup_1\boldsymbol{t}}_{c'}~,
    &\quad U_c^B = X^B_c\prod_{c'}Z^{B \int\boldsymbol{c}'\cup_1\boldsymbol{t}}_{c'}~,
\end{eqs}
where qubits are placed on cubes $c$, and $G_f^\alpha$ is defined in analogy with Eq.~\eqref{eq: definition of G}:
\begin{equation}
    G^\alpha_{f} :=  X^\alpha_{\delta \bface} \prod_{c'} {Z^\alpha_{c'}}^{\int  \delta \bface \cup_1 \bc'},
    \quad  \alpha \in \{A, B\}.
\label{definition:G $5{+}1$D}
\end{equation}
See Appendix~\ref{app: Polynomial formalism for Clifford QCA} for the polynomial formalism of this QCA on the hypercubic lattice.
On a general cellulation, we place qubit $A$ on each 3-cell of the direct lattice and qubit $B$ on each 3-cell of the dual lattice. The separators and (non-commuting) flippers then take the form
\begin{eqs}
    \overline{Z}^A_t = Z^A_t G^B_{\mathrm{PD}(t)}~,\quad\quad~~~
    &\quad \overline{Z}^B_t = Z^B_t G^A_{\mathrm{PD}(t)}~,\\
    U_t^A = X^A_t\prod_{t'}Z^{A \int \boldsymbol{t}'\cup_1\boldsymbol{t}}_{t'}~,
    &\quad U_t^B = X^B_t\prod_{t'}Z^{B \int\boldsymbol{t}'\cup_1\boldsymbol{t}}_{t'}~,
\end{eqs}
where $\mathrm{PD}$ denotes the Poincaré dual: $\mathrm{PD}(t)$ is the face in the dual lattice intersecting the tetrahedron $t$ in the direct lattice.
For simplicity, we focus on hypercubic lattices in the following construction, while noting that the argument extends straightforwardly to arbitrary cellulations.

As noted in Ref.~\cite{fidkowski2024qca}, this QCA should be disentangled by a non-Clifford finite-depth quantum circuit, and we argue that our construction indeed possesses this property.
First, observe that
\begin{equation}
    \frac{1}{2} A_3 \cup A_3
\end{equation}
is a trivial cocycle in $H^6(B^3 \mathbb{Z}_2, \mathbb{R}/\mathbb{Z})$ when $A$ is a closed 3-form. Indeed, consider the coboundary
\begin{equation}
\begin{aligned}
&\frac{1}{4}\,\delta\!\left(A_3 \cup_1 A_3 + A_3 \cup_2 \delta A_3\right) \\
=& - \frac{1}{2} A_3 \cup A_3 \;-\; \frac{1}{2} A_3 \cup_1 \delta A_3 \;+\; \frac{1}{4}\,\delta A_3 \cup_2 \delta A_3
\end{aligned}
\label{trivial cocycle in 6D}
\end{equation}
where we have used the standard coboundary identity for higher cup products.  
To be consistent with the denominator $4$, we lift the cochain $A$ from $\ZZ_2$ to $\ZZ$. Since $A$ is closed, we have $\delta A \equiv 0 \pmod{2}$. Consequently,
\begin{equation}
    \frac{1}{2} A_3 \cup_1 \delta A_3 \;-\; \frac{1}{4} \delta A_3 \cup_2 \delta A_3 \equiv 0 \pmod{1}~.
\end{equation}
Thus, $\tfrac{1}{2} A_3 \cup A_3$ is a coboundary.
Adding this coboundary to the topological action modifies the ground state only by a finite-depth quantum circuit (FDQC) parameterized by  
$\tfrac{1}{4}(A_3 \cup_1 A_3 + A_3 \cup_2 \delta A_3)$. Because the denominator is $4$, the resulting circuit is non-Clifford, acting on the $\mathbb{Z}_2$ qubits located on tetrahedra (in particular, it includes the controlled-$S$ gate). 

{\change
We note that although two topological actions that differ by a coboundary give rise to the same ground state, the corresponding Hamiltonian-to-QCA constructions may \textit{a priori} yield different QCAs. 
Corollary~\ref{corollary: separated QCAs trivial} guarantees that it suffices to consider a representative QCA that disentangles the given ground state, without concerning that different choices of Hamiltonian may lead to inequivalent QCAs.
}

The coboundary term $\frac{1}{2}A_3\cup A_3$ transforms the gauged SPT ground state via the circuit  
\begin{equation}
\begin{aligned}
    &V^a_{\mathrm{FDQC}} \ket{\ba_t} \\
    & :=\exp\!\left( -\frac{2\pi i}{4} \int_{M_5} \ba_t \cup_1 \ba_t + \ba_t \cup_2 \delta \ba_t \right) \ket{\ba_t},
\end{aligned}
\end{equation}
where $\ba_t$ is the 3-cochain labeling the $A$-type qubit configuration on tetrahedra, and $M_5$ is the spatial manifold. Consequently, the two Clifford QCAs defined with or without the $\tfrac{1}{2} A_3 \cup A_3$ term differ only by $V^a_{\mathrm{FDQC}}$ and are therefore equivalent. The same reasoning applies to the term $\tfrac{1}{2} B_3 \cup B_3$, for which we define the associated $V^b_{\mathrm{FDQC}}$.
Hence, the QCA generated by  
\begin{equation}
    \tfrac{1}{2}\,(A_3 \cup A_3 + A_3 \cup B_3 + B_3 \cup B_3)~,
\end{equation}
conjugated by the non-Clifford circuit $V^a_{\mathrm{FDQC}} V^b_{\mathrm{FDQC}}$, is identical to the QCA generated by  
\begin{equation}
    \tfrac{1}{2}\,A_3 \cup B_3~.
\end{equation}

Finally, we show that the QCA from $\tfrac{1}{2} A_3 \cup B_3$ is in fact a Clifford circuit. The Hamiltonian associated with this action is
\begin{eqs}
H=-\sum_f X_{\delta f}^A\prod_c Z_c^{B \int \bface\cup\boldsymbol{c}}-\sum_f X_{\delta f}^B\prod_c Z_c^{A \int \boldsymbol{c}\cup\bface}~.
\end{eqs}
On a hypercubic lattice, separators can be written as 
\begin{eqs}
\overline{Z}^A_c=Z^A_c\prod_f X^{B \int \boldsymbol{c}\cup\bface}_{\delta f}~,\\
\overline{Z}^B_c=Z^B_c\prod_f X^{A \int \bface\cup\boldsymbol{c}}_{\delta f}~,
\end{eqs}
while the flippers remain $X^A_c$ and $X^B_c$.  
We define the Clifford circuit
\begin{align}
U=\prod_{c,c'}CZ(X^A_c,X^B_{c'})^{\int \boldsymbol{c}\cup \partial c'}
\end{align}
Here $\partial c'$ means a 2-cochain taking value 1 on $\partial c$ and 0 on other faces. It can be checked geometrically that $\prod_f Z_{\delta f}^{\int \bface\cup\boldsymbol{c}}=\prod_{c'} Z_{c'}^{\int \partial{c'}\cup\boldsymbol{c}}=\prod_{c'}Z_{c'}^{\int \boldsymbol{c}'\cup\partial c}$. Then it is direct that
\begin{eqs}
    \overline{Z}^A_c=UZ^A_cU^\dagger, \quad
    \overline{Z}^B_c=UZ^B_cU^\dagger~.
\end{eqs}
Thus, the QCA derived from $S=\tfrac{1}{2} A_3 \cup B_3$ is realized by a Clifford circuit.  



\subsection{$\ZZ_2$ QCA in $7{+}1$D from topological action $S=\tfrac{1}{2}\int A_4\cup A_4 + A_4\cup B_4+ B_4\cup B_4$}
\label{sec: w8}

The topological action in $7{+}1$D is 
\begin{align}
    S=\frac{1}{2}\int A_4\cup A_4+A_4\cup B_4+B_4\cup B_4~.
\end{align}
The resulting QCA is structurally similar to the $3{+}1$D case, and the separators and flippers can be written directly as
\begin{eqs}
    \overline{Z}^A_{s_4} &= Z^A_t \prod_{s_4} G^{B \int \boldsymbol{s}_4\cup\boldsymbol{t}}_t~, \\
    \overline{Z}^B_{s_4} &= Z^B_t \prod_{s_4} G^{A \int \boldsymbol{t}\cup\boldsymbol{s}_4}_t~,\\
    U_{s_4}^A &= X^A_{s_4} \prod_{s'_4} Z^{A \int \boldsymbol{s}'_4\cup_1\boldsymbol{s}_4}_{s'_4}~, \\
    U_{s_4}^B &= X^B_{s_4} \prod_{{s_4}'} Z^{B \int\boldsymbol{s}'_4\cup_1\boldsymbol{s}_4}_{s'_4}~,
\end{eqs}
where $s_4$ denotes a 4-simplex.

Unlike the $5{+}1$D case, however, the term $\tfrac{1}{2} A_4 \cup A_4$ is no longer a trivial cocycle. In $7{+}1$D there is no analogue of Eq.~\eqref{trivial cocycle in 6D}, and therefore this cocycle cannot be removed by adding a coboundary. As a result, we expect that the corresponding QCA cannot be simplified to a non-Clifford circuit. This conclusion is consistent with the fact that $w_8$ is a nontrivial Stiefel–Whitney class in $7{+}1$D, which indeed classifies a nontrivial QCA.  
A polynomial representation of this QCA can be obtained by translating the cup products of cochains into matrices, as in the previous cases. The resulting matrix coincides with that obtained in Ref.~\cite{fidkowski2024qca}.

\subsection{$\ZZ_2$ QCA in $(2l{-}1){+}1$D from topological action $S=\tfrac{1}{2}\int A_l\cup A_l + A_l\cup B_l + B_l \cup B_l$}
\label{Sec. even d Z2 QCA}
In a general even dimension, the action can be written as
\begin{eqs}
S=\frac{1}{2}\int A_l \cup A_l + A_l \cup B_l + B_l \cup B_l~.
\end{eqs}
We denote a $(l{-}1)$-simplex by $s_{l-1}$ and a $l$-simplex by $s_l$. The corresponding QCA is then
\begin{eqs}
    \overline{Z}^A_{s_l} &= Z^A_{s_l} \prod_{s_{l-1}} G^{B \int \boldsymbol{s}_l \cup \boldsymbol{s}_{l-1}}_{s_{l-1}}~,\\
    \overline{Z}^B_{s_l} &= Z^B_{s_l} \prod_{s_{l-1}} G^{A \int \boldsymbol{s}_{l-1} \cup \boldsymbol{s}_l}_{s_{l-1}}~,\\
    U_{s_l}^A &= X^A_{s_l} \prod_{s_l'} Z^{A \int \boldsymbol{s}'_l \cup_1 \boldsymbol{s}_l}_{s_l'}~,\\
    U_{s_l}^B &= X^B_{s_l} \prod_{s_l'} Z^{B \int \boldsymbol{s}'_l \cup_1 \boldsymbol{s}_l}_{s_l'}~.
\end{eqs}

In an arbitrary dimension, Eq.~\eqref{trivial cocycle in 6D} generalizes to 
\begin{equation}
\begin{aligned}
&\frac{1}{4}\,\delta\!\left(A_l \cup_1 A_l + A_l \cup_2 \delta A_l\right) \\
=& ~\frac{1}{4}\,\delta A_l \cup_2 \delta A_l 
+ \frac{(-1)^l-1}{4}\,\bigl(A_l \cup A_l + A_l \cup_1 \delta A_l\bigr).
\end{aligned}
\label{AcupA in arbitrary dimension}
\end{equation}
When $l$ is odd, using the fact that $\delta A_l \equiv 0 \pmod{2}$, we obtain 
\begin{eqs}
\frac{1}{2}A_l \cup A_l 
= -\frac{1}{4}\,\delta\!\left(A_l \cup_1 A_l + A_l \cup_2 \delta A_l\right),
\end{eqs}
so that $\tfrac{1}{2}A_l \cup A_l$ is a coboundary in $H^{2l}(B^l \mathbb{Z}_2,\mathbb{R}/\mathbb{Z})$. In this case, as in $5{+}1$D, the corresponding QCA is trivial up to a non-Clifford circuit.

When $l$ is even, the $A_l \cup A_l$ term in Eq.~\eqref{AcupA in arbitrary dimension} cancels, leaving a nontrivial cocycle. Consequently, the associated QCA is also nontrivial. We therefore establish a 4-periodicity in the spatial dimensions for the generalized “3-fermion” Walker–Wang QCAs. Moreover, if non-Clifford circuits are not allowed, the QCAs exhibit only a 2-periodicity in the spatial dimensions.

\subsection{$\ZZ_p$ Clifford QCAs in $(4l{-}1){+}1$D from TQFTs}\label{sec:Zp_QCA_(4l-1)+1D_TQFT}

The generalization of the $\mathbb{Z}_2$ 3-fermion QCA to higher dimensions can be extended to $\mathbb{Z}_p^{(k)}$ QCAs as well. However, in this case, the construction exhibits a 4-periodicity in the spatial dimension, so that nontrivial QCAs appear only in $(4l{-}1)$ spatial dimensions.

Consider the action
\begin{eqs}
S = \frac{k}{p}\int A_{2l} \cup A_{2l} ~\in~ H^{4l}(B^{2l} \mathbb{Z}_p, \RR/\ZZ)~.
\end{eqs}
Following the framework of the $3{+}1$D case, the corresponding separators and flippers are
\begin{eqs}
\overline{Z}_{s_l} &= Z_{s_l} \prod_{s_{l-1}} \left(G^{(k)}_{s_{l-1}}\right)^{\tfrac{1}{2k}\int \boldsymbol{s}_{l-1}\cup \boldsymbol{s}_l}, \\
\overline{X}_{s_l} &= X_{s_l} \prod_{s_l'} \left[\prod_{s_{l-1}} \left(G^{(k)}_{s_{l-1}}\right)^{-\tfrac{1}{2}\int \boldsymbol{s}_{l-1}\cup \boldsymbol{s}_l}\right]^{\int \boldsymbol{s}'_l \cup_1 \boldsymbol{s}_l},
\end{eqs}
where $s_{l-1}$ and $s_l$ denote $(l{-}1)$- and $l$-simplices, and
\begin{eqs}
G^{(k)}_{s_{l-1}} := X_{\delta \boldsymbol{s}_{l-1}} 
\prod_{s_l'} Z_{s_l'}^{\,k\int \delta \boldsymbol{s}_{l-1} \cup_1 \boldsymbol{s}'_l}~.
\end{eqs}
This construction is valid because the coboundary identities for higher cup products repeat with a 4-periodicity. Specifically, the formulas
\begin{eqs}
    \delta (\boldsymbol{c}_p \cup_i \boldsymbol{d}_q) 
    =&~ \delta \boldsymbol{c}_p \cup_i \boldsymbol{d}_q \\ 
    &+ (-1)^p\boldsymbol{c}_p \cup_i \delta \boldsymbol{d}_q \\
    &+ (-1)^{p+q-i}\boldsymbol{c}_p \cup_{i-1} \boldsymbol{d}_q \\
    &+ (-1)^{pq+p+q}\boldsymbol{d}_q \cup_{i-1} \boldsymbol{c}_p~, \\
    \delta (\boldsymbol{c}_{p+2} \cup_i \boldsymbol{d}_{q+2}) 
    =&~ \delta \boldsymbol{c}_{p+2} \cup_i \boldsymbol{d}_{q+2} \\
    &+ (-1)^p\boldsymbol{c}_{p+2} \cup_i \delta \boldsymbol{d}_{q+2} \\
    &+ (-1)^{p+q-i}\boldsymbol{c}_{p+2} \cup_{i-1} \boldsymbol{d}_{q+2} \\
    &+ (-1)^{pq+p+q}\boldsymbol{d}_{q+2} \cup_{i-1} \boldsymbol{c}_{p+2}~,
\end{eqs}
have identical sign structures. Promoting each cochain by two degrees therefore preserves the relations, implying a 4-periodicity. Thus, our $\mathbb{Z}_p$ QCA constructions are valid in every fourth spatial dimension, consistent with the classification of Ref.~\cite{haah2025topological}.

We now argue that QCAs in $(4l{+}1)$ spatial dimensions are always trivial, even when restricted to Clifford circuits.  
Consider the identity
\begin{eqs}
& ~\frac{k(p-1)}{2p}\,\delta(A_{2l+1} \cup_1 A_{2l+1})  \\
&= \frac{k(p-1)}{2p}\Bigl(\delta A_{2l+1} \cup_1 A_{2l+1} 
- A_{2l+1} \cup_1 \delta A_{2l+1} \\
&\qquad\qquad\qquad - A_{2l+1} \cup A_{2l+1} 
- A_{2l+1} \cup A_{2l+1}\Bigr) \\
&= \frac{k}{p}\,A_{2l+1} \cup A_{2l+1}~.
\label{Zp AcupA in arbitrary dimension}
\end{eqs}
Hence, $\tfrac{k}{p} A_{2l+1} \cup A_{2l+1}$ is a coboundary.
More importantly,
\begin{equation}
\begin{aligned}
    &\frac{k(p-1)}{2p}\,\delta(A_{2l+1} \cup_1 A_{2l+1})   \\
    =& ~\delta\left( \frac{k(p-1)/2}{p}\,A_{2l+1} \cup_1 A_{2l+1}\right),
\end{aligned}
\end{equation}
is implemented by a Clifford circuit, since the denominator involves only $p$. (By contrast, in the $\mathbb{Z}_2$ case the denominator is $4$, which goes beyond Clifford.) Therefore, the QCA derived from
\begin{equation}
S=\tfrac{k}{p} A_{2l+1} \cup A_{2l+1}
\end{equation}
is a Clifford circuit, and thus no nontrivial QCA arises in $(4l{+}1)$ spatial dimensions.

We conclude that $\mathbb{Z}_p$ QCAs exhibit a 4-periodicity in spatial dimensions, nontrivial only in $(4l{-}1)$ spatial dimensions.

\section{Invertible subalgebras in $2{+}1$D and higher dimensions}\label{sec:ISA}

A closely related notion to QCAs is that of \emph{invertible subalgebras} (ISAs), introduced by Haah in Ref.~\cite{Haah2023InvertibleSubalgebras}. An ISA $A$ is defined as a subalgebra of operators on a Hilbert space such that every operator can be locally expressed using elements of $A$ and its commutant $\bar A$. This induces a decomposition of the Hilbert space as 
\begin{equation}
    \mathcal H = \mathcal H_A \otimes \mathcal H_{\bar A},
\end{equation}
where $\mathcal H_A$ is the Hilbert space generated by operators in $A$. A non-trivial ISA arises when a tensor product Hilbert space $\mathcal H$ splits into factors $\mathcal H_A$ and $\mathcal H_{\bar A}$ that do not themselves form a tensor product structure. 

Haah showed that an ISA in $d$ spatial dimensions can be used to construct a QCA in $(d\!+\!1)$ dimensions. Concretely, consider a stack of $d$-dimensional ISAs layered along the extra dimension. The total Hilbert space takes the form
\begin{equation}
    \bigotimes_i \mathcal H_{A}^{(i)} \otimes \mathcal H_{\bar A}^{(i)},
\end{equation}
where $i$ labels the layers. The QCA acts by shifting operators in $\mathcal H_A^{(i)}$ forward to $\mathcal H_A^{(i+1)}$, while leaving $\mathcal H_{\bar A}^{(i)}$ invariant. This establishes a one-to-one correspondence between ISAs in $d$ dimensions and QCAs in $d\!+\!1$ dimensions. 

To date, the only known example of a non-trivial ISA appears in two spatial dimensions on the square lattice. Ref.~\cite{Haah2023InvertibleSubalgebras} constructed $\mathbb{Z}_p^{(k)}$ ISAs in this setting, showing that each such ISA uniquely generates a QCA in one higher dimension.  

In this section, we introduce a new construction of $\mathbb{Z}_2$ ISAs on arbitrary two-dimensional cellulations, where the subalgebra $A$ realizes the 3-fermion chiral topological order, and extend this construction to all even spatial dimensions.
In addition, we reformulate the $\mathbb{Z}_p^{(k)}$ ISAs in two spatial dimensions and generalize them to all $(4l{-}2)$ dimensions. These constructions are consistent with the classification of algebraic $L$-groups over Laurent polynomial rings given in Ref.~\cite{haah2025topological}.

\subsection{$2{+}1$D $\mathbb{Z}_2$ invertible subalgebras}

We present the invertible subalgebra on the $2{+}1$D square lattice with two $\ZZ_2$ qubits per edge, which corresponds to the $3{+}1$D 3-fermion QCA.
{\change 
First, we can rearrange two copies of $\ZZ_2$ toric codes $\{1, e_A, m_A, f_A\} \times \{1, e_B, m_B, f_B\}$ into the 3-fermion topological order $\mathcal{A}$ and its anti-chiral copy $\bar{\mathcal{A}}$ as follows
\begin{eqs}
    \mathcal{A} \times  \bar{\mathcal{A}}:=&~ \{1, e_A f_B, m_A f_B, f_A\} \\
    & \times \{1, e_B f_A, m_B f_A, f_B\}~.\\
\label{eq: 2 TC = 2 3-fermion}
\end{eqs}
}
We define two shorthands that will be useful later on:
\begin{equation}
\begin{aligned}
    U_e^\alpha &:= X_e^\alpha \prod_{e'} {Z_{e'}^\alpha}^{\int \be' \cup \be}~, \\
    \tU_e^\alpha &:= X_e^\alpha \prod_{e'} {Z_{e'}^\alpha}^{\int \be \cup \be'}~,
\end{aligned}
    \label{eq: definition U-tilde Z2 ISA}
\end{equation}
$\alpha \in \{A, B \}$. Diagrammatically, they are
\begin{eqs}
    &U_e^\alpha =\vcenter{\hbox{\includegraphics[scale=.25]{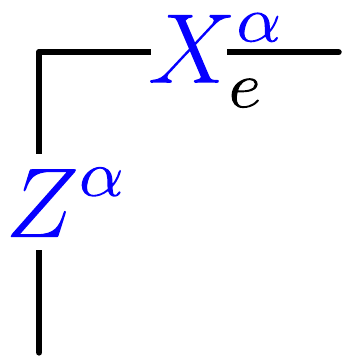}}}~, \quad \vcenter{\hbox{\includegraphics[scale=.25]{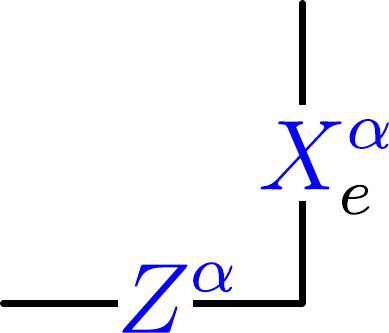}}}~,
    \\
    &\tU_e^\alpha =\vcenter{\hbox{\includegraphics[scale=.25]{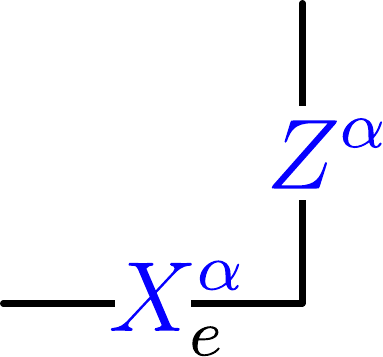}}}~, \quad \vcenter{\hbox{\includegraphics[scale=.25]{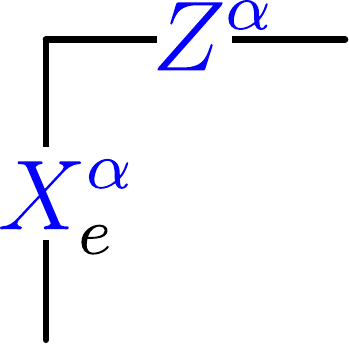}}}~.
\end{eqs}
For simplicity, we define the flux operator $W^\alpha_f$ as:
\begin{equation}
W^\alpha_f=Z^\alpha_{\partial f}=\raisebox{-2.5em}{\includegraphics[scale=.25]{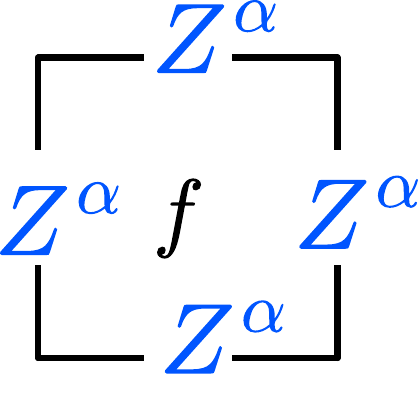}}~.
\end{equation}
The following are the generators of the ISA for the 3-fermion theory
\begin{eqs}
    A_e^A &= \tU_e^A \prod_f W_f^{A\int \bface \cup_1 \be}~, \\
    A_e^B &= \tU_e^B \prod_f W_f^{B\int \be \cup_1 \bface} \prod_{e'} {Z_{e'}^A}^{\int \be\cup \be'}~, \\
    \overline{A}_e^A &= \tU_e^A \prod_f W_f^{A\int \be  \cup_1 \bface} \prod_{e'} {Z^B_{e'}}^{\int \be'\cup \be}~, \\
    \overline{A}_e^B &= \tU_e^B \prod_f W_f^{B\int \bface \cup_1 \be}~.
\label{eq:Z2_ISA_2+1D}
\end{eqs}
These generators are constructed by attaching the flux terms $W_f^{A/B}$ and single-qubit Pauli operators $Z_f^{A/B}$ to the hopping terms $\tU_e^{A/B}$, in accordance with the pattern specified by the cup and cup-1 products.
This structure ensures that the resulting operators satisfy the definition of an ISA.

\begin{widetext}
On the square lattice, these terms are
\begin{eqs}
    A_e^A=&~\vcenter{\hbox{\includegraphics[scale=.25]{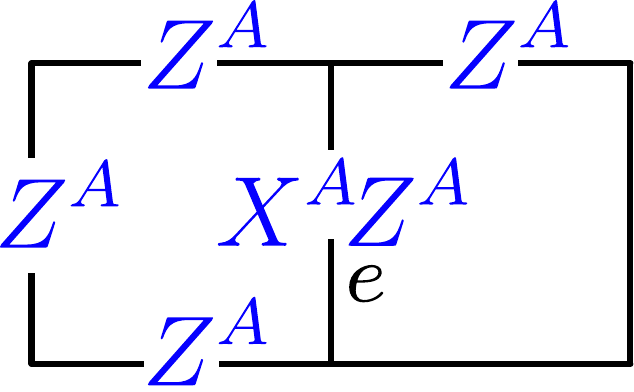}}}~,
    \quad
    \vcenter{\hbox{\includegraphics[scale=.25]{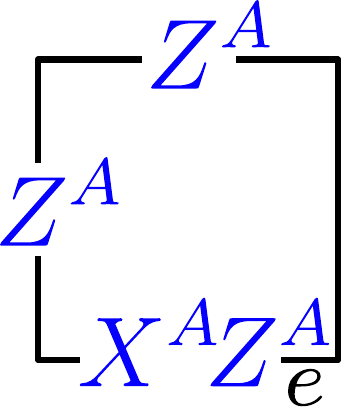}}}~, \qquad\qquad~
    A_e^B=~\vcenter{\hbox{\includegraphics[scale=.25]{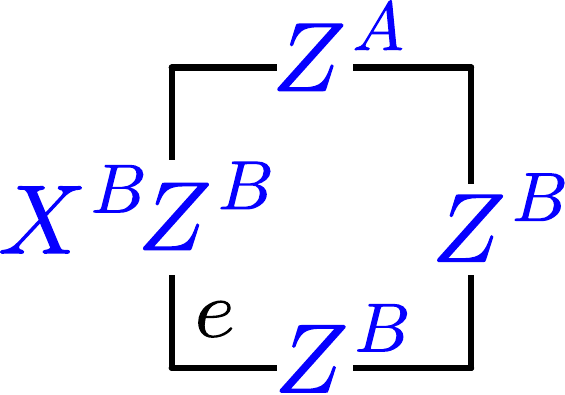}}}~,
    \quad
    \vcenter{\hbox{\includegraphics[scale=.25]{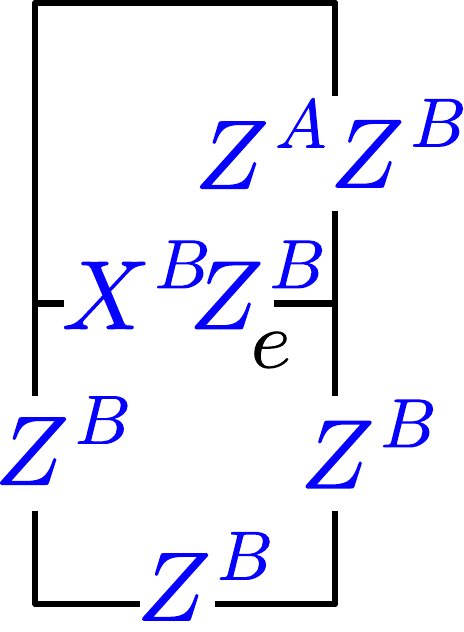}}}~, \\
    \overline{A}_e^A=&~\vcenter{\hbox{\includegraphics[scale=.25]{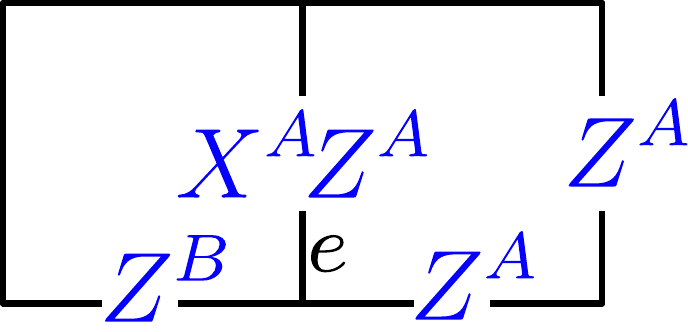}}}~,
    \quad
    \vcenter{\hbox{\includegraphics[scale=.25]{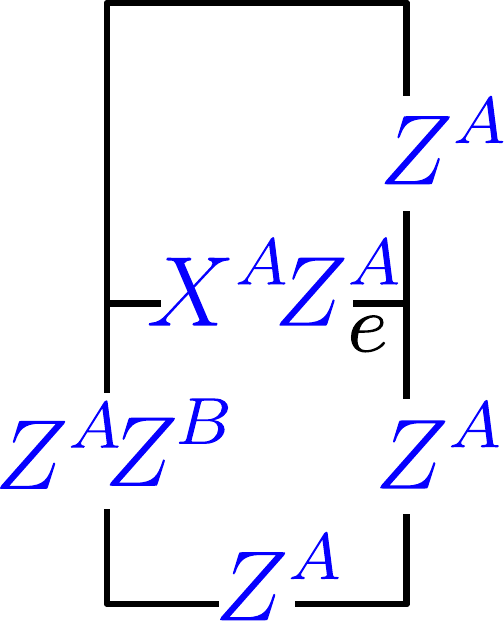}}}~, \qquad
    \overline{A}_e^B=~\vcenter{\hbox{\includegraphics[scale=.25]{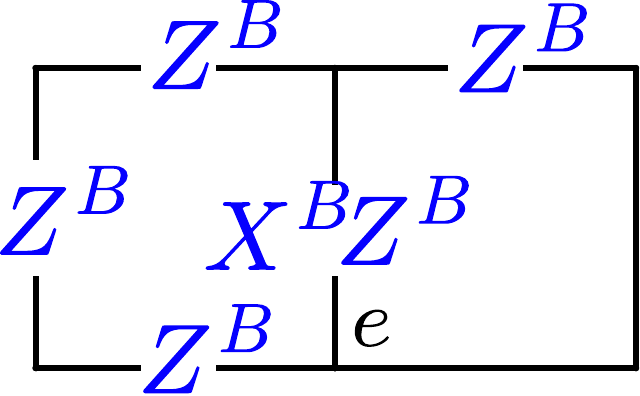}}}~,
    \quad
    \vcenter{\hbox{\includegraphics[scale=.25]{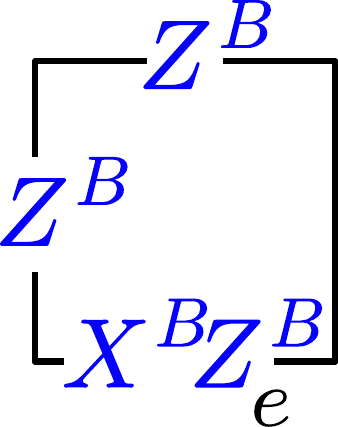}}}~.
\end{eqs}
Next, we prove that it is invertible, which means that the single $X$ and $Z$ terms can be generated from $A$ and $\overline{A}$. We first construct the flux operator of $B$ qubits from $A_e^A$ and $\overline{A}_e^A$ as $W_f^B=\prod_\nu (A^A_{\delta\nu}\overline{A}^A_{\delta\nu})^{\int \bface\cup\boldsymbol{\nu}}$:
\begin{eqs}
\raisebox{-3.5em}{\includegraphics[scale=.25]{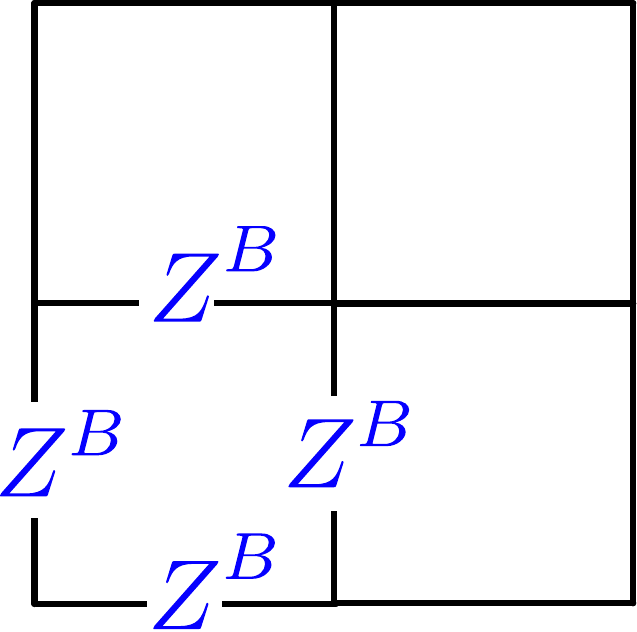}}=&\raisebox{-3.5em}{\includegraphics[scale=.25]{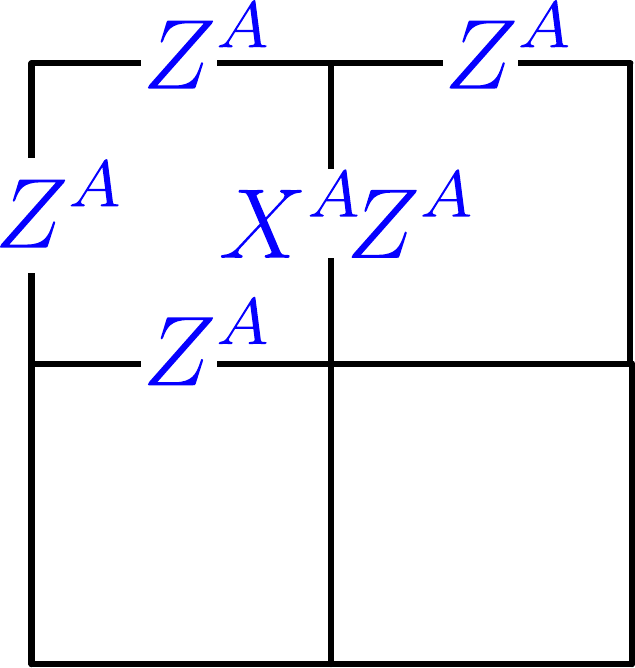}}\times\raisebox{-3.5em}{\includegraphics[scale=.25]{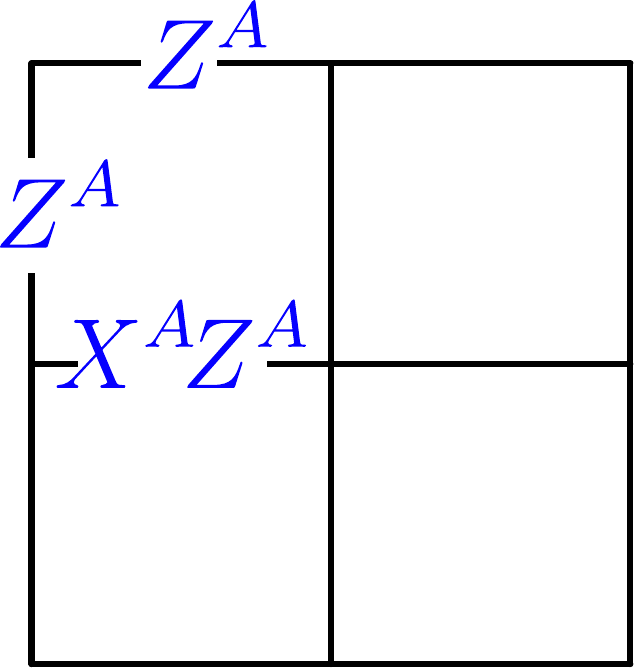}}\times\raisebox{-3.5em}{\includegraphics[scale=.25]{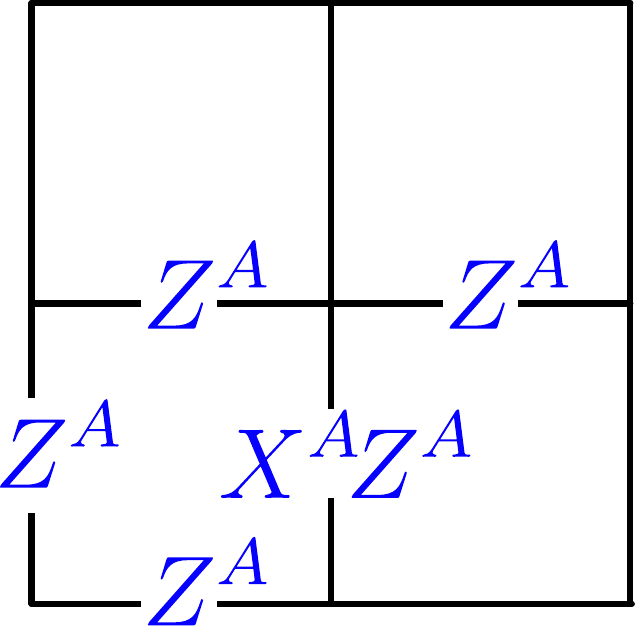}}\times\raisebox{-3.5em}{\includegraphics[scale=.25]{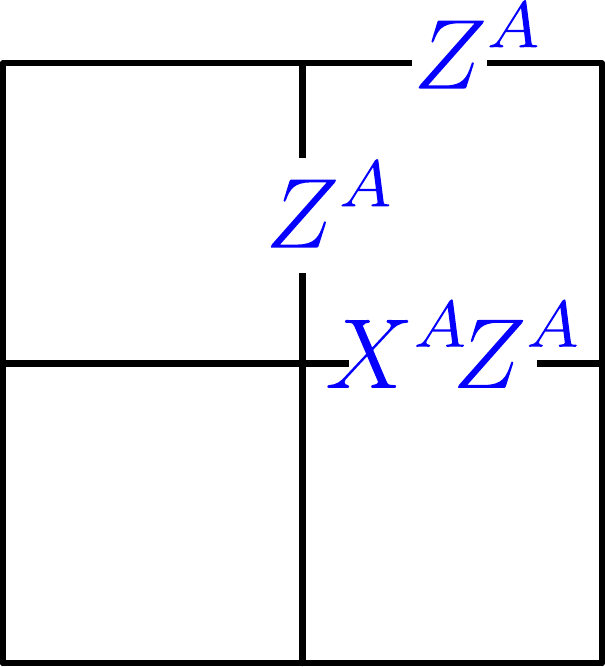}}\\
\times&\raisebox{-3.5em}{\includegraphics[scale=.25]{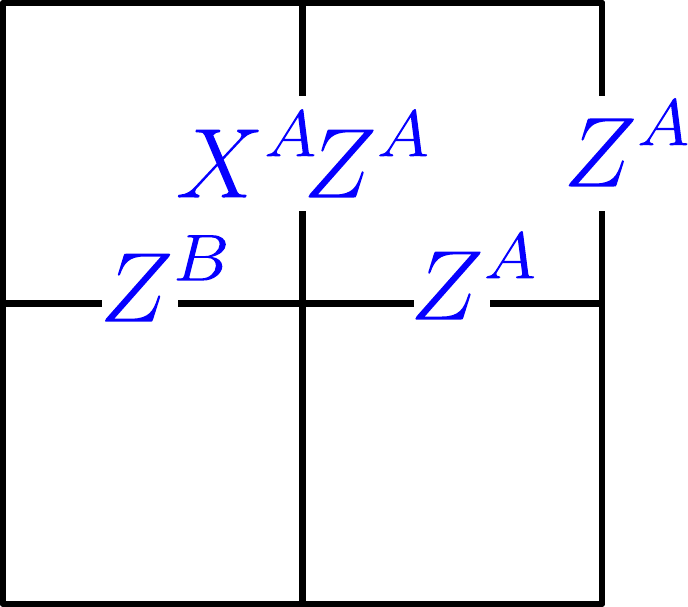}}\times\raisebox{-3.5em}{\includegraphics[scale=.25]{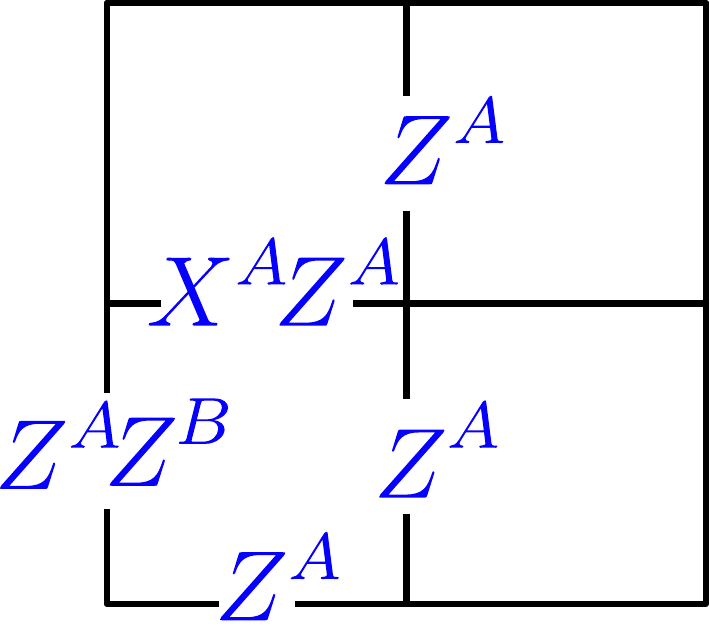}}\times\raisebox{-3.5em}{\includegraphics[scale=.25]{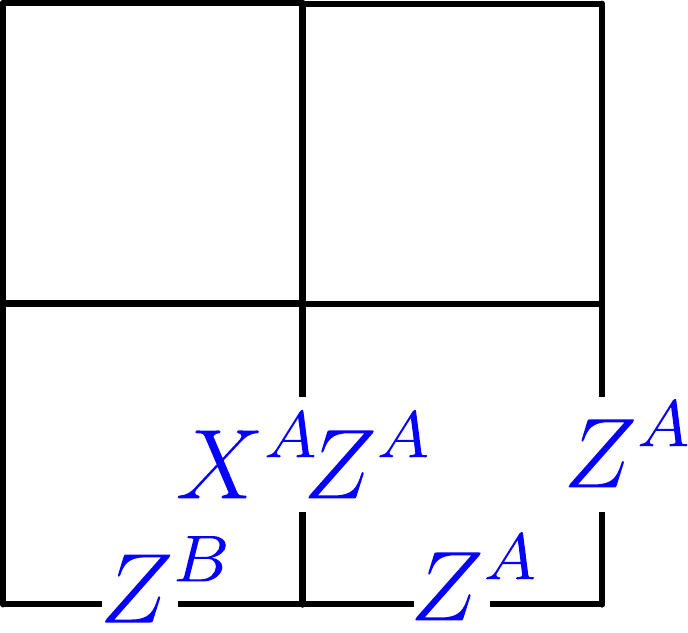}}\times\raisebox{-3.5em}{\includegraphics[scale=.25]{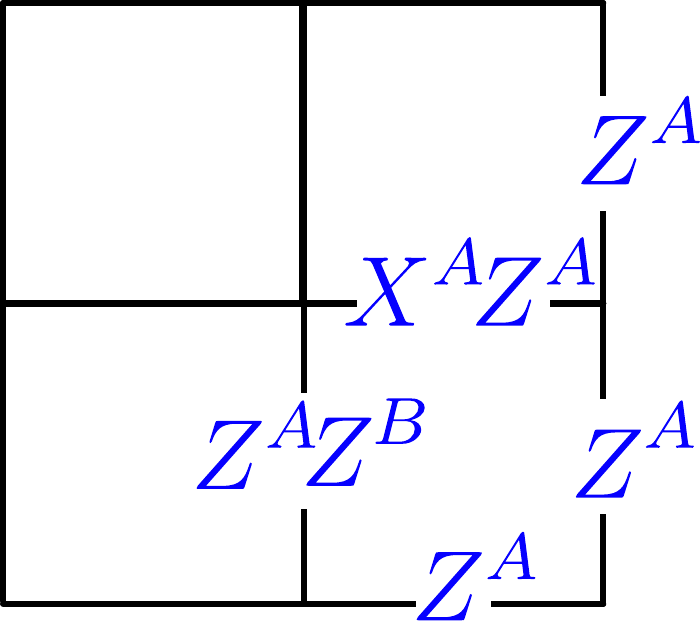}}~.
\nonumber
\end{eqs}
With this $W_f^B$, we construct single $Z_e^A$ from $A_e^B$ and $\overline{A}_e^B$
\begin{eqs}
Z_e^A=\prod_{e'}(A^B_{e'}\overline{A}^B_{e'}\prod_f W_f^{B \int \bface\cup_1\be'+\be'\cup_1\bface})^{\int \be'\cup\be}~,
\end{eqs}
which can be shown diagrammatically as
\begin{eqs}
\raisebox{-1.5em}{\includegraphics[scale=.25]{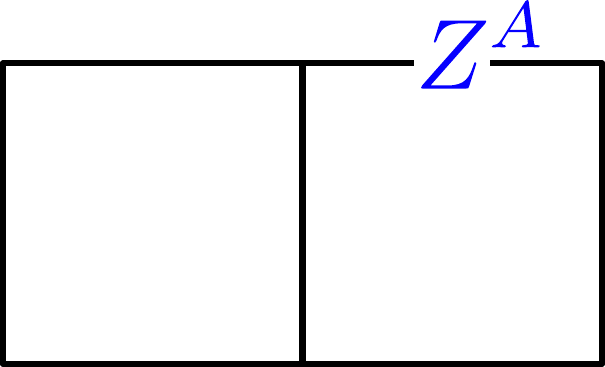}}&=\raisebox{-1.5em}{\includegraphics[scale=.25]{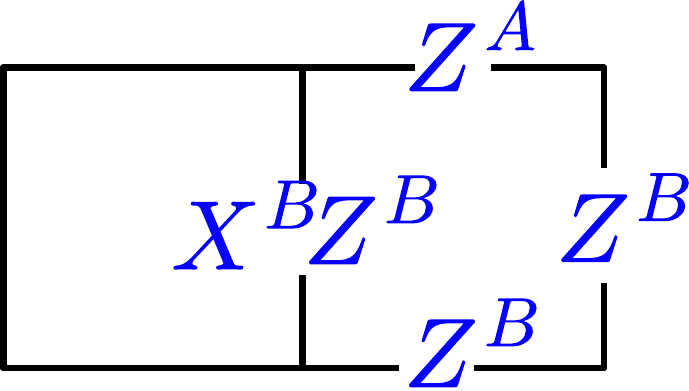}}\times\raisebox{-1.5em}{\includegraphics[scale=.25]{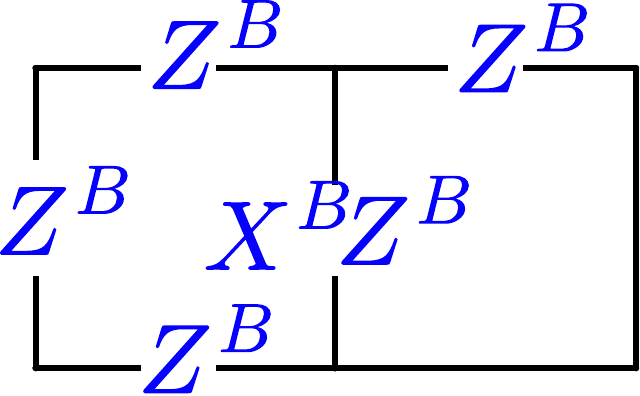}}\times\raisebox{-1.5em}{\includegraphics[scale=.25]{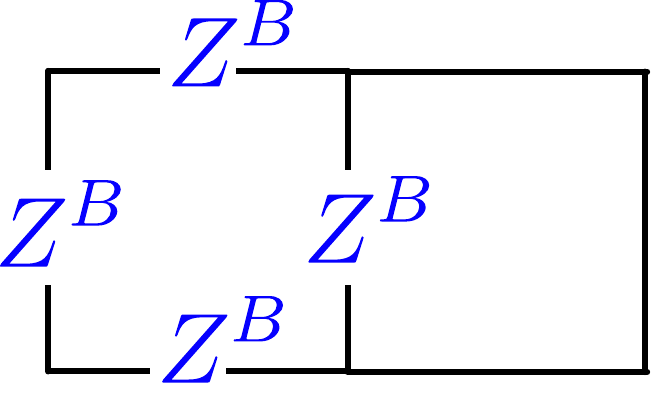}}\times\raisebox{-1.5em}{\includegraphics[scale=.25]{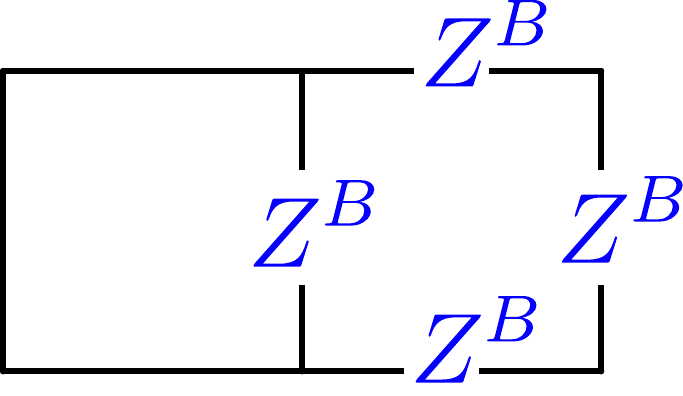}} , \\
\raisebox{-3.5em}{\includegraphics[scale=.25]{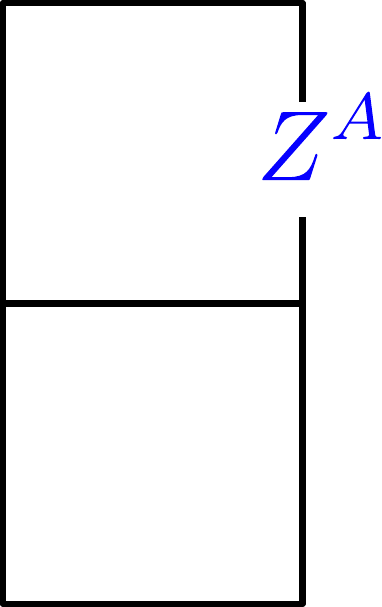}}&=\raisebox{-3.5em}{\includegraphics[scale=.25]{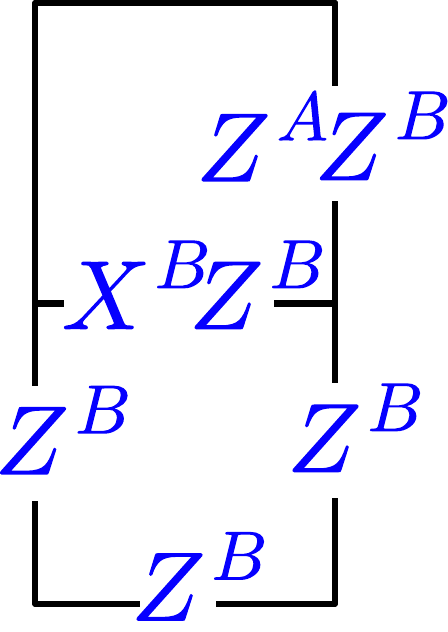}}\times\raisebox{-3.5em}{\includegraphics[scale=.25]{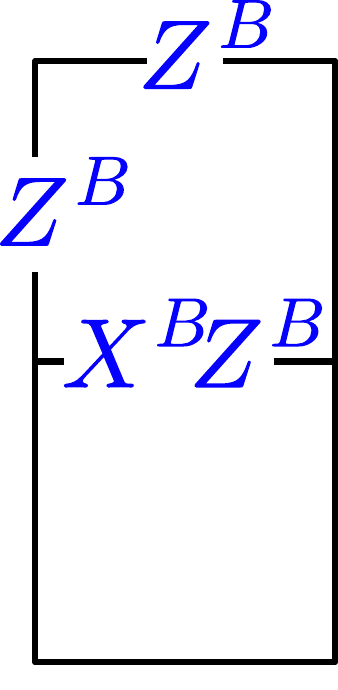}}\times\raisebox{-3.5em}{\includegraphics[scale=.25]{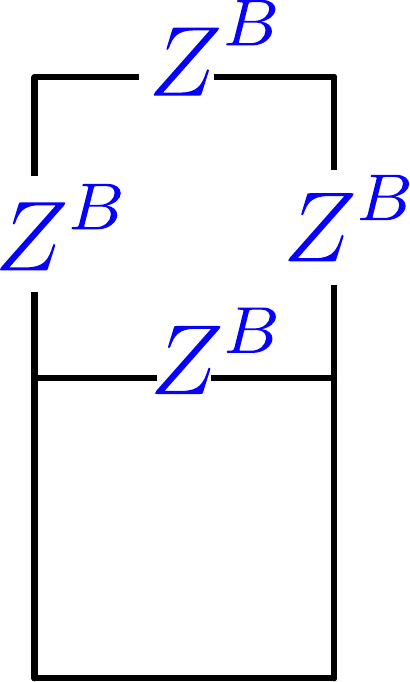}}\times\raisebox{-3.5em}{\includegraphics[scale=.25]{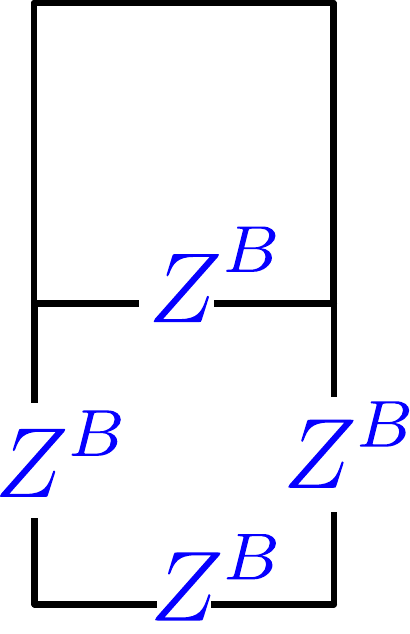}}~.
\nonumber
\end{eqs}
With this single $Z^A$ obtained, the construction of other single Pauli operators are obvious. In cup product formalism, they can be shown as
\begin{eqs}
    X^A_e&=A^A_e \prod_{e'}Z_{e'}^{A \int \be\cup\be'}\prod_f W_f^{A \int \bface\cup_1\be}~,\\
    Z^B_e&=\prod_{e'}\left(X^A_{e'}\overline{A}^A_{e'}\prod_{e''}Z_{e''}^{A \int \be'\cup\be''}\prod_fW_f^{A \int \be'\cup\bface}\right)^{\int \be\cup\be'}~,\\
    X^B_e&=\overline{A}_e^B \prod_{e'}Z_{e'}^{B \int \be\cup\be'}\prod_fW_f^{B \int \bface\cup_1\be}~.
    \label{eq: generation single X B cup product}
\end{eqs}
{\change
Operators $A_e^\alpha$ and $\bar A_e^\alpha$ are related by the unitary
\begin{equation}
    U = \prod_{e,f'}CZ(Z_e^A,W_f^A)^{\int{\bface \cup_1 \be +\be \cup_1 \bface}}  CZ(Z_e^B,W_f^B)^{\int{\bface \cup_1 \be +\be \cup_1 \bface}} \times\prod_{e,e'}CZ(Z_e^A,Z^B_{e'})^{\int{\be' \cup \be}}~,
\end{equation}
where $CZ(P_1,P_2)= (-1)^{\frac{1-P_1}{2}\frac{1-P_2}{2}}$ is the generalized controlled gate of two commuting Pauli operators $P_1$ and $P_2$.
Note that $W_f^{\int (\bface \cup_1 \be +\be \cup_1 \bface)}$ on the square lattice doesn't contain $Z_e$ so the $CZ$ operator is well defined.
}
\end{widetext}

\subsubsection{Polynomial form}
The hopping operators in polynomial form are
\begin{align}
    \bU^\alpha_e &= \bX^\alpha_e + \bZ^\alpha_{e'} \bM_{e' \cup e}~, \\
    \tilde \bU^\alpha_e &=\bX^\alpha_e +\bZ^\alpha_{e'} \bM_{e \cup e'}^\dagger~, \\
    \bW^\alpha_f &= \bZ^\alpha _{\delta e}~,
\end{align}
where $\alpha \in \{A, B \}$. The generators of the ISA are
\begin{align}
    \bA^A_e &= \tilde\bU^A_e+ \bW^A_f \bM_{f\cup_1e}~, \\
    \overline{\bA}^A_e &= \tilde\bU^A_e+ \bW^A_f \bM_{e\cup_1f}^\dagger+\bZ^B_{e'}\bM_{e'\cup e}~, \\ \bA^B_e &= \tilde\bU^B_e+ \bW^B_f\bM_{e\cup_1f} ^\dagger+\bZ^A_{e'}\bM_{e\cup e'}^\dagger~,  \\  \overline{\bA}^B_e &= \tilde\bU^B_e+ \bW^B\bM_{f\cup_1e}~.
\end{align}
In other words, we have
\begin{align}
    \bA &= \begin{pmatrix} \bA^A & \bA^B \end{pmatrix} =  \begin{pmatrix}
        \bbone\\
        \boldsymbol{\mathcal{M}} \end{pmatrix}~, \\
    \overline{\bA} &= \begin{pmatrix} \overline{\bA}^A & \overline{\bA}^B\end{pmatrix} =  \begin{pmatrix} 
        \bbone\\
        \overline{\boldsymbol{\mathcal{M}}}
    \end{pmatrix}~,
\end{align}
where
\begin{align}
    \boldsymbol{\mathcal{M}} &= \begin{pmatrix} \bM^\dagger_{e\cup e_0} +\bM_{\delta e_0 \cup_1 e} & \bM^\dagger_{e\cup e_0}\\
    0 &  \bM^\dagger_{e\cup e_0}+\bM^\dagger_{e\cup_1\delta e_0}\end{pmatrix}~, \\\overline{\boldsymbol{\mathcal{M}}}&= \begin{pmatrix}  \bM^\dagger_{e\cup e_0}+ \bM^\dagger_{ e \cup_1 \delta e_0} & 0\\
    \bM_{e_0\cup e} & \bM^\dagger_{e\cup e_0}+\bM_{\delta e_0 \cup_1 e}
    \end{pmatrix}~.
\end{align}
In fact, $\boldsymbol{\mathcal{M}}^\dagger =  \overline{\boldsymbol{\mathcal{M}}}$, where the equality for the diagonal terms follows from the Leibniz rule for the $\cup_1$ product.

\begin{figure}[t]
    \begin{subfigure}[]
        \centering
        \includegraphics[width=0.6\linewidth]{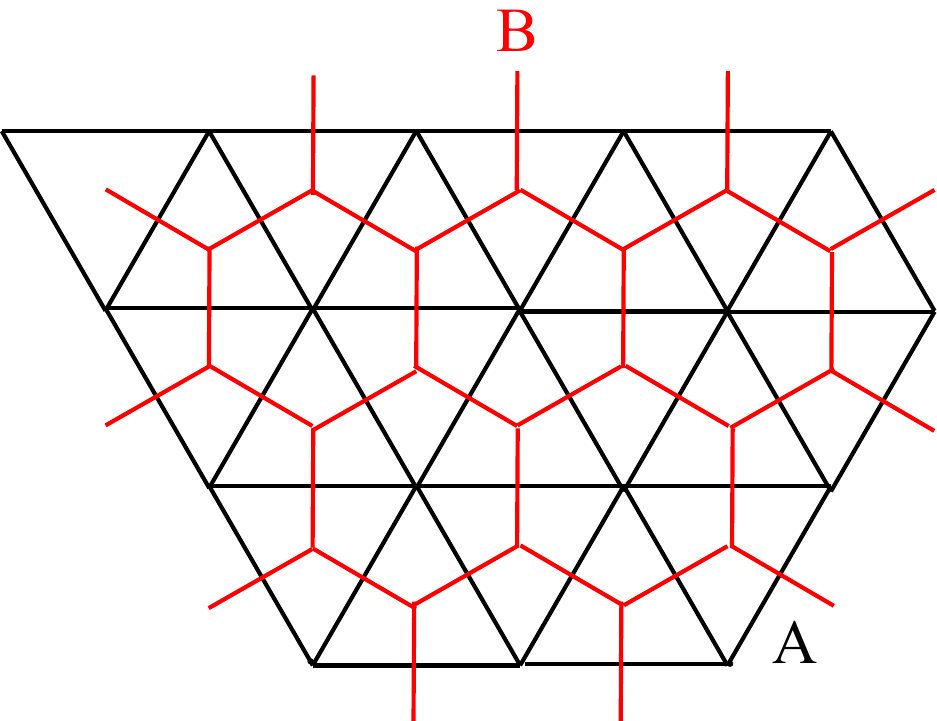}
    \end{subfigure}
    \vspace{1em}
    \begin{subfigure}[]
        \centering
        \includegraphics[width=0.4\linewidth]{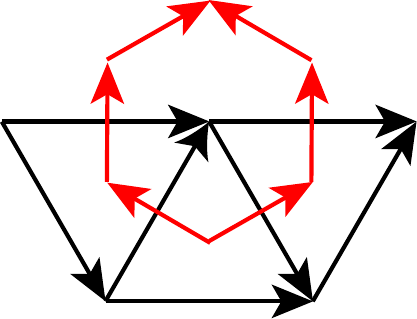}
    \end{subfigure}
    \caption{(a) The lattice is constructed by stacking a triangular lattice together with its dual, the honeycomb lattice. (b) Branching structure of the mixed lattice. }
    \label{fig:triangular-honeycomb mixed lattice}
\end{figure}

{\change
A QCA can be constructed from the ISA by pushing $\bA$ in an arbitrary direction of $m = x^i y^jz^k$ as
\begin{equation}
     \bH \begin{pmatrix} m & 0 \\ 0 & 1\end{pmatrix}\bH^{-1}~, \text{ where }
     \bH = \begin{pmatrix}
        \bA & \bar \bA
    \end{pmatrix}~,
\end{equation}
of which the simplest is to take $m = z$. We may write
\begin{equation}
    \balpha^\text{3F}_\text{ISA} = \bH \begin{pmatrix} z & 0 \\ 0 & 1\end{pmatrix}\bH^{-1}~.
    \label{eq:3FQCA_ISA}
\end{equation}
}

\subsubsection{$\mathbb{Z}_2$ invertible subalgebra on dual triangular-honeycomb lattices}

\begin{figure}[t]
    \begin{subfigure}[]
        \centering
        \includegraphics[width=0.5\linewidth]{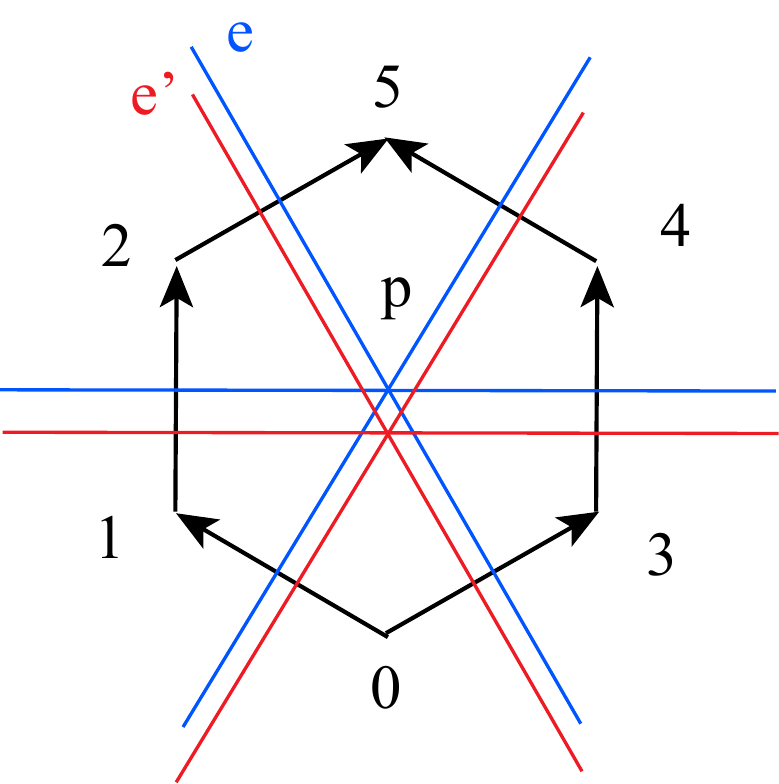}
    \end{subfigure}
    \vspace{1em}
    \begin{subfigure}[]
        \centering
        \includegraphics[width=0.8\linewidth]{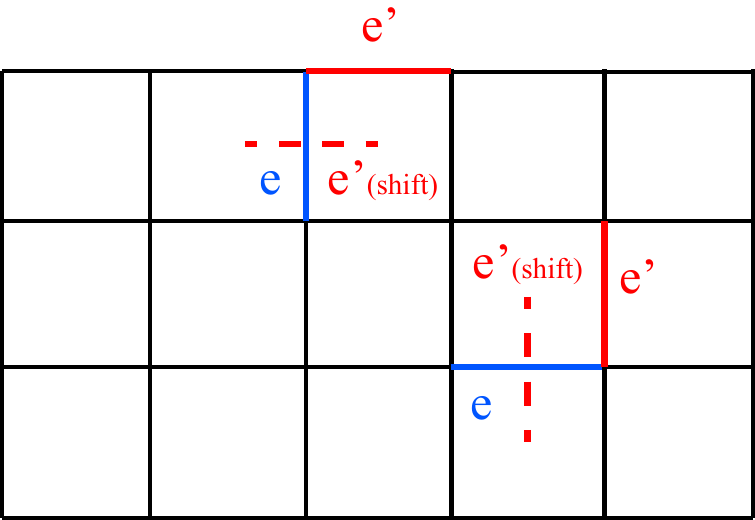}
    \end{subfigure}
    \caption{(a) The intersections of shifted dual lines indicate the cup product definition. (b) An illustration of the cup product on a square lattice, realized as the intersection of one edge with the other after a left-downward shift. }
    \label{fig:cup product and Poincare dual}
\end{figure}

In this subsection, we generalize our cubic lattice construction for $\mathbb{Z}_2$ invertible subalgebra to other lattices. Specifically, we introduce the dual triangular-honeycomb lattices illustrated in Fig.~\ref{fig:triangular-honeycomb mixed lattice}(a).

We place $A$ qubits on the edges of the triangular lattice and $B$ qubits on the edges of the honeycomb lattice. To incorporate this model into our cup product framework, we introduce the rules for cup products on the mixed lattice.

We begin by establishing the branching structure of the lattice, illustrated in Fig.~\ref{fig:triangular-honeycomb mixed lattice}(b). The cup product and coboundary operators on the triangular lattice follow directly from simplicial cohomology, as reviewed in Sec.~\ref{app:terminology}. Our main focus, therefore, is on defining the cup products on the honeycomb lattice and the mixed cup products such as $\be^A \cup \be^B$ across the two lattices.

Following the geometric definition of the cup product on the cubic lattice, we define the cup and cup-1 products on the honeycomb lattice as illustrated in Fig.~\ref{fig:cup product and Poincare dual}(a):
\begin{eqs}
    \be\cup_1\bface(p)&=\be(01)+\be(12)+\be(25)]\bface(p)~,\\
    \bface\cup_1\be(p)&=[\be(03)+\be(34)+\be(45)]\bface(p)~,\\
    \be\cup\be'(p)&=\be(03)\be'(34)+\be(03)\be'(45)+\be(34)\be'(45)\\
    &+\be(01)\be'(12)+\be(01)\be'(25)+\be(12)\be'(25)~.
\end{eqs}
Next, we derive the formula for $\be^A \cup \be^B$. On the square lattice, the cup product of two edges corresponds geometrically to the intersection of one edge with the dual edge shifted toward the down-left direction, as illustrated in Fig.~\ref{fig:cup product and Poincare dual}(b). Motivated by this, we define $\int \be^A \cup \be^B \neq 0$ whenever $e^A$ and $e^B$ are Poincaré dual to each other.
As an example, the operator $A^A_e$ can be written as
\begin{equation}
\begin{aligned}
    A_e^A &= \tU_e^A \prod_f W_f^{A\int \be \cup_1 \bface} \prod_{e'} {Z_{e'}^B}^{\int \be\cup \be'} \\
    &= \tU_e^A \left( \prod_f W_f^{A\int \be \cup_1 \bface} \right){Z_{\mathrm{PD}(e)}^B}~.
\end{aligned}
\end{equation}
where $\mathrm{PD}(e)$ denotes the edge Poincaré dual to $e$.

With all these cup products defined, we can now derive the invertible subalgebra on this triangular-honeycomb lattice from the original formalism. Starting with the $\tilde{U}$ operator defined in Eq.~\eqref{eq: definition U-tilde Z2 ISA}.
\begin{widetext}
\begin{eqs}
\tilde{U}^A_e&=X^A_e \prod_{e'} {Z^A_{e'}}^{\int \be \cup \be'}=
\raisebox{-2em}{\includegraphics[scale=.450]{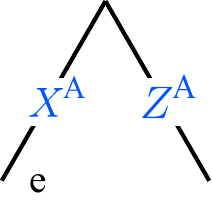}}~,
\quad
\raisebox{-2em}{\includegraphics[scale=.450]{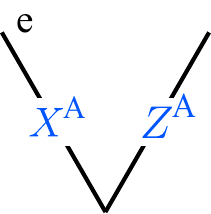}}~,
\quad
\raisebox{-2em}{\includegraphics[scale=.450]{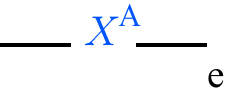}}~,\\
\tilde{U}^B_e&=X^B_e \prod_{e'} {Z^B_{e'}}^{\int \be \cup \be'}=
\raisebox{-2em}{\includegraphics[scale=.450]{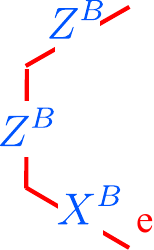}}~,
\quad
\raisebox{-2em}{\includegraphics[scale=.450]{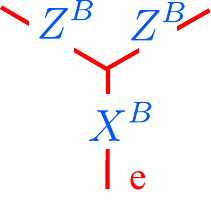}}~,
\quad
\raisebox{-2em}{\includegraphics[scale=.450]{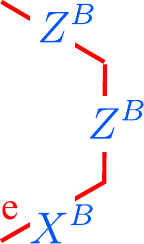}}~.
\end{eqs}
We can then derive the entire invertible subalgebra as follows
\begin{eqs}
    A_e^A &= \tU_e^A \prod_f W_f^{A\int \be \cup_1 \bface} \prod_{e'} {Z_{e'}^B}^{\int \be\cup \be'}=
    \raisebox{-2em}{\includegraphics[scale=.450]{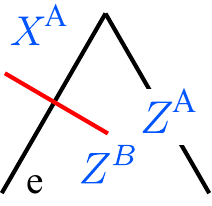}}~,
    \quad
    \raisebox{-2em}{\includegraphics[scale=.450]{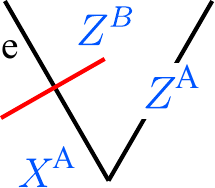}}~,
    \quad
    \raisebox{-2em}{\includegraphics[scale=.450]{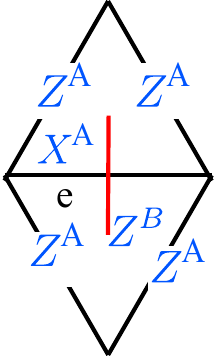}}~,
    \\
    A^B_e&= \tU_e^B \prod_f W_f^{B\int \bface \cup_1 \be}=
    \raisebox{-2em}{\includegraphics[scale=.450]{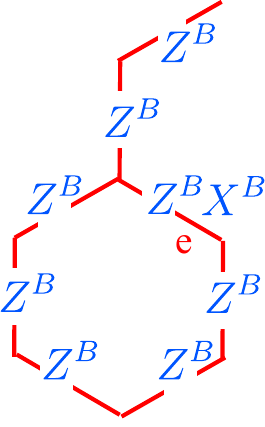}}~,
    \quad
    \raisebox{-2em}{\includegraphics[scale=.450]{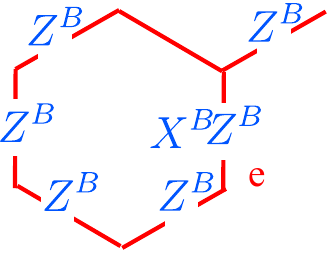}}~,
    \quad
    \raisebox{-2em}{\includegraphics[scale=.450]{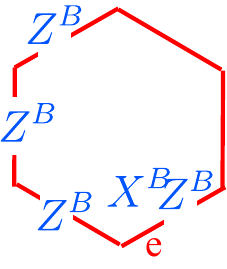}}~,
\end{eqs}
and
\begin{eqs}
    \overline{A}_e^A &= \tU_e^A \prod_f W_f^{A\int \bface \cup_1 \be}=
    \raisebox{-2em}{\includegraphics[scale=.450]{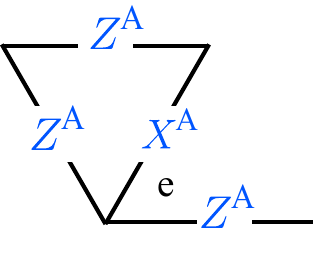}}~,
    \quad
    \raisebox{-2em}{\includegraphics[scale=.450]{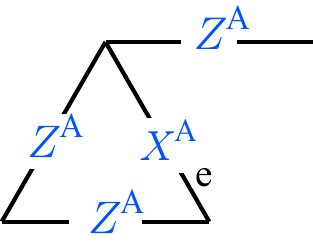}}~,
    \quad
    \raisebox{-2em}{\includegraphics[scale=.450]{Figures/U_tilde_A_3.pdf}}~,\\
    \bar A^B_e&=\tU_e^B \prod_f W_f^{B\int \be  \cup_1 \bface} \prod_{e'} {Z^A_{e'}}^{\int \be'\cup \be}=
    \raisebox{-2em}{\includegraphics[scale=.450]{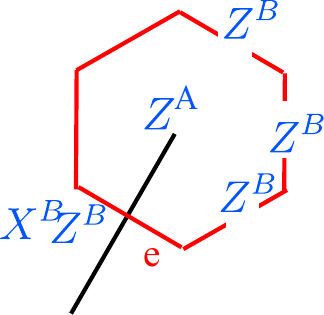}}~,
    \quad
    \raisebox{-2em}{\includegraphics[scale=.450]{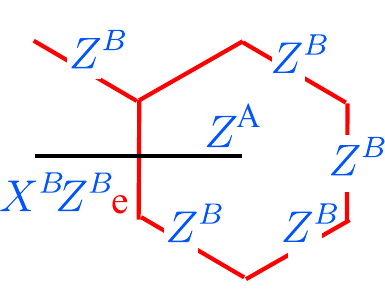}}~,
    \quad
    \raisebox{-2em}{\includegraphics[scale=.450]{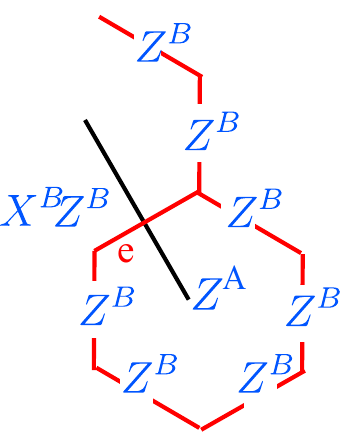}}~.
\end{eqs}
Following the similar calculation of Eq.~\eqref{eq: generation single X B cup product}, we conclude that these operators also generate the entire qubit algebra, meaning that they form an invertible subalgebra.

Thus, we generalize the cubic $\mathbb{Z}_2$ invertible subalgebra to a specific non-cubic lattice. Before moving on, we claim that this kind of generalization only works for $\mathbb{Z}_2$ cases because in other cases the additional $\pm$ sign would break the commutativity of $A$ and $\bar A$. From a mathematical point of view, the $\ZZ_2$ invertible subalgebra (and its corresponding QCA) is guided by a $\ZZ_2$ topological invariant $w_2^2$, while $\ZZ_k$ cases aren't. So the cup product formula for $\ZZ_2$ ISA is robust to lattice changes, but $\ZZ_k$ ISA only lives on a certain lattice (the cubic one).
\end{widetext}

\subsection{$2{+}1$D $\mathbb{Z}_p^{(k)}$ invertible subalgebras}
\label{sec: Zpk invertible subalgebras}
On the 2d square lattice, we put one $\mathbb{Z}_p$ qudit on each edge, spanning the whole Pauli algebra. We can define a subalgebra of the Pauli group generated by the operator
\begin{eqs}
    A_e^{(k)} :=& X_e \prod_{e'} Z_{e'}^{k \int \be' \cup \be}
    \prod_f Z_{\partial f}^{\frac{k}{2}\int(\be \cup_1 \bface - \bface \cup_1 \be)},
\end{eqs}
where we again emphasize that $\frac{1}{2}$ in the exponent stands for the integer $\frac{p+1}{2}$, which is the multiplicative inverse of $2$ mod $p$.
The orthogonal subalgebra is generated by
\begin{equation}
    \oA_e^{(k)} := X_e \prod_{e'} Z_{e'}^{-k \int \be' \cup \be}
    \prod_f Z_{\partial f}^{-\frac{k}{2}\int(\be \cup_1 \bface - \bface \cup_1 \be)},
\end{equation}
which is the charge conjugation, $Z \rightarrow Z^{-1}$, of the subalgebra generated by $A_e^{(k)}$.

These two terms can be shown in the following diagrams:
\begin{eqs}
    {A^{(k)}_e}&=\raisebox{-1.9em}{\includegraphics[scale=.25]{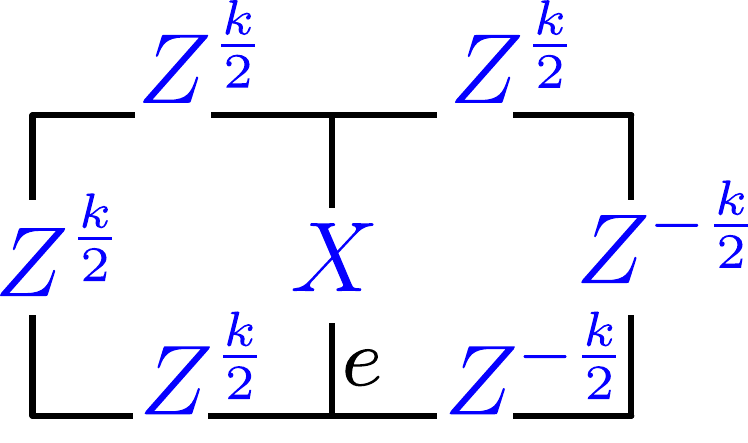}}~,
    \quad
    \raisebox{-4em}{\includegraphics[scale=.25]{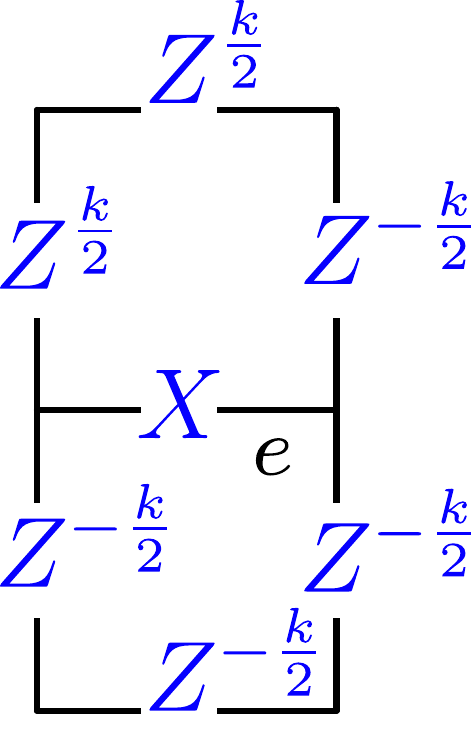}}~,
\end{eqs}
\begin{eqs}
    {\overline{A}^{(k)}_e}&=\raisebox{-2em}{\includegraphics[scale=.25]{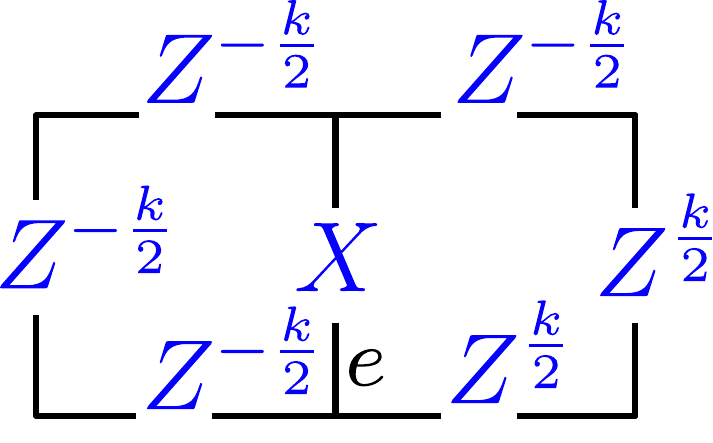}}~,
    \quad~
    \raisebox{-4em}{\includegraphics[scale=.25]{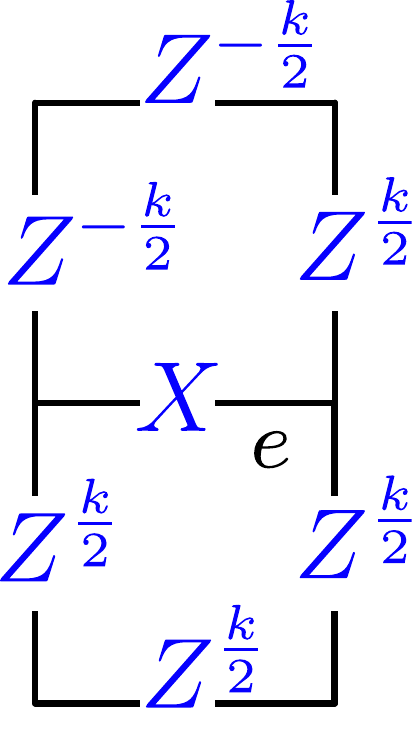}}~.
\end{eqs}
First, we verify elements in $\{ A_e^{(k)} \}$ and $\{ \oA_e^{(k)} \}$ mutually commute, i.e.,
\begin{equation}
    A_{e_1}^{(k)} \oA_{e_2}^{(k)} = \exp(\frac{2\pi i}{p} \chi) \oA_{e_2}^{(k)} A_{e_1}^{(k)},
\end{equation}
with
\begin{equation}
\begin{aligned}
    \chi &= \frac{k}{2} \int -2 \be_1 \cup \be_2 - \be_2 \cup_1 \delta \be_1 + \delta \be_1 \cup_1 \be_2 \\
    & \hspace{4em} - 2 \be_2 \cup \be_1  - \be_1 \cup_1 \delta \be_2 + \delta \be_2 \cup_1 \be_1 \\
    &= 0~,
\end{aligned}
\label{eq: 2d ISA cup product 1}
\end{equation}
where we have used
\begin{equation}
\begin{aligned}
    \delta(A_1 \cup B_1) =&~ \delta A_1 \cup_1 B_1 - A_1 \cup_1 \delta B_1 \\
    &- A_1 \cup B_1 - B_1 \cup A_1.
\end{aligned}
\end{equation}

Next, we need to show that the single Pauli $X_e$ or $Z_e$ can be generated from $\{ A_e^{(k)} \}$ and $\{ \oA_e^{(k)} \}$. We note that this scenario is specific to the square lattice. For a generic triangulation, $\{ A_e^{(k)} \}$ and $\{ \oA_e^{(k)} \}$ might not generate all operators in the Hilbert space, and $\{ A_e^{(k)} \}$ will not realize an invertible subalgebra.

First, it is useful to notice the following identity:
\begin{equation}
\begin{aligned}
    A_{\delta v}^{(k)} &= X_{\delta v} \prod_{e'}Z_{e'}^{k \int \be' \cup \delta \bv + \frac{k}{2}\int (\delta \bv \cup_1 \delta \be' - \delta \be' \cup_1 \delta \bv)} \\
    &= X_{\delta v} \prod_{f}Z_{\partial f}^{k \int \bv \cup \bface}~,
\end{aligned}
\label{eq: A_dv is Gv}
\end{equation}
where we have used
\begin{eqs}
    \delta (\be' \cup \bv) =& \delta \be' \cup \bv - \be' \cup \delta \bv~, \\
    \delta ( \delta \bv \cup_2 \delta \be') =& - \delta \bv \cup_1 \delta \be' - \delta \be' \cup_1 \delta \bv~, \\
    \delta (\bv \cup_1 \delta \be')=& \delta \bv \cup_1 \delta \be' - \bv \cup \delta \be' + \delta \be' \cup \bv~.
\label{eq: 2d ISA cup product 2}
\end{eqs}
We then multiply $A_e^{(k)}$ and $\oA_e^{(k)}$ to obtain $X_e^2$ and therefore single Pauli $X_e$:
\begin{equation}
\label{eq:XfromAAbar}
    X_e = \left(A_e^{(k)} \oA_e^{(k)} \right)^{\frac{1}{2}}.
\end{equation}
Diagrammatically, they are
\begin{eqs}
\raisebox{-0.5em}{\includegraphics[scale=.25]{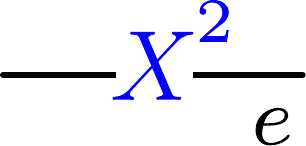}}&=\raisebox{-3.5em}{\includegraphics[scale=.22]{Figures/Zp_ISA_horizontal.pdf}}\times\raisebox{-3.5em}{\includegraphics[scale=.22]{Figures/Zp_ISA_rev_horizontal.pdf}}~, \\
\raisebox{-1.5em}{\includegraphics[scale=.25]{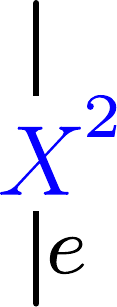}}\quad&=\raisebox{-1.5em}{\includegraphics[scale=.22]{Figures/Zp_ISA_vertical.pdf}}\times\raisebox{-1.5em}{\includegraphics[scale=.22]{Figures/Zp_ISA_rev_vertical.pdf}}~.
\end{eqs}
Subsequently, on the square lattice, we use Pauli $X_e$ and Eq.~\eqref{eq: A_dv is Gv} to obtain a single $Z_{\partial f}$ term:
\begin{equation}
    Z_{\partial f}=\prod_{v} (A_{\delta v} X_{\delta v}^{-1})^{\frac{1}{k}\int \bv \cup \bface}
    = \prod_{v} (A_{\delta v} \oA_{\delta v}^{-1} )^{\frac{1}{2k}\int \bv \cup \bface}.
\end{equation}
This product can be illustrated by the following diagram:
\begin{widetext}
\begin{equation}
\includegraphics[scale=.21]{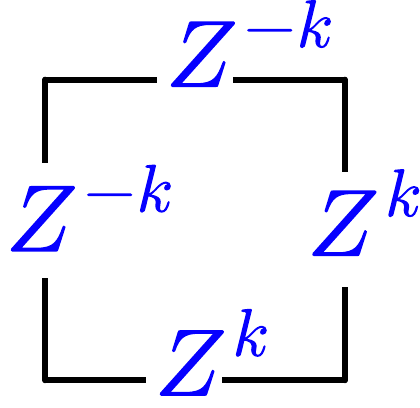}=\raisebox{-2.75em}{\includegraphics[scale=.21]{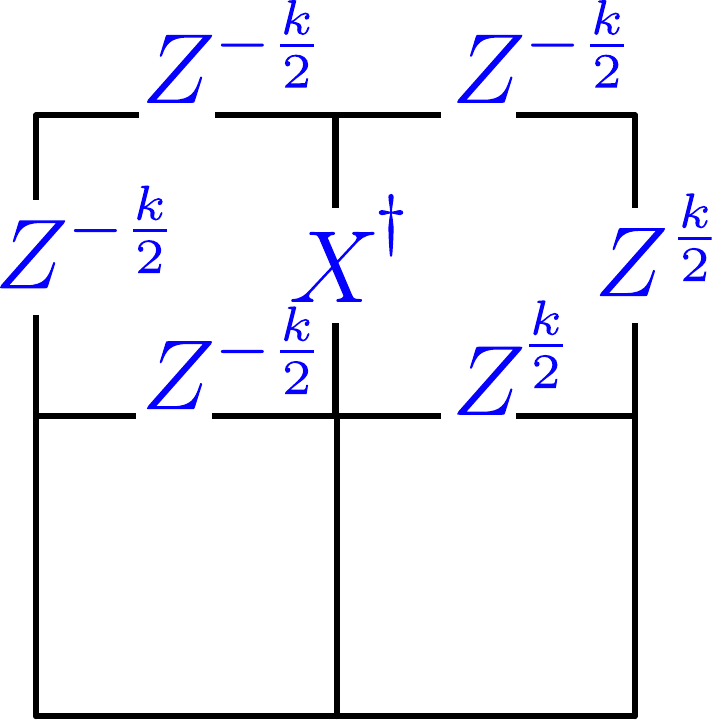}}\times\raisebox{-2.75em}{\includegraphics[scale=.21]{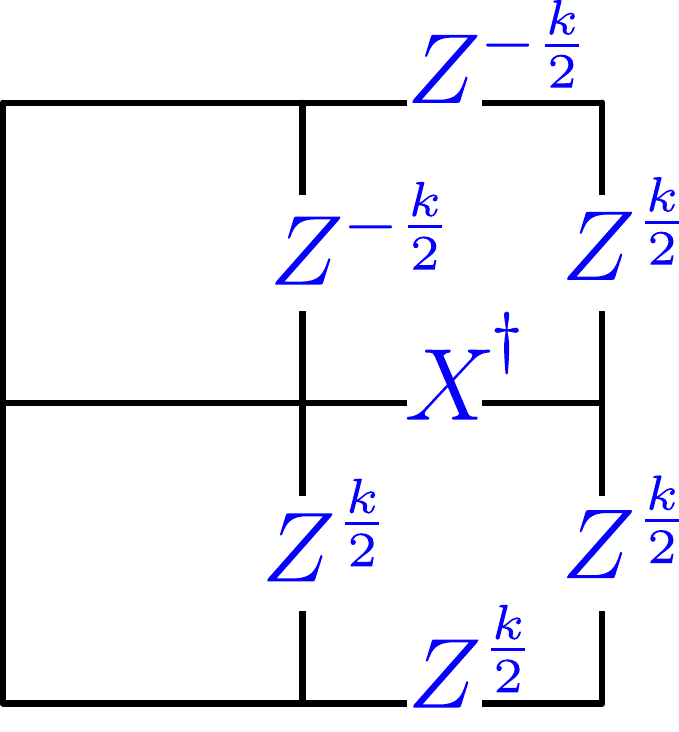}}\times\raisebox{-2.75em}{\includegraphics[scale=.21]{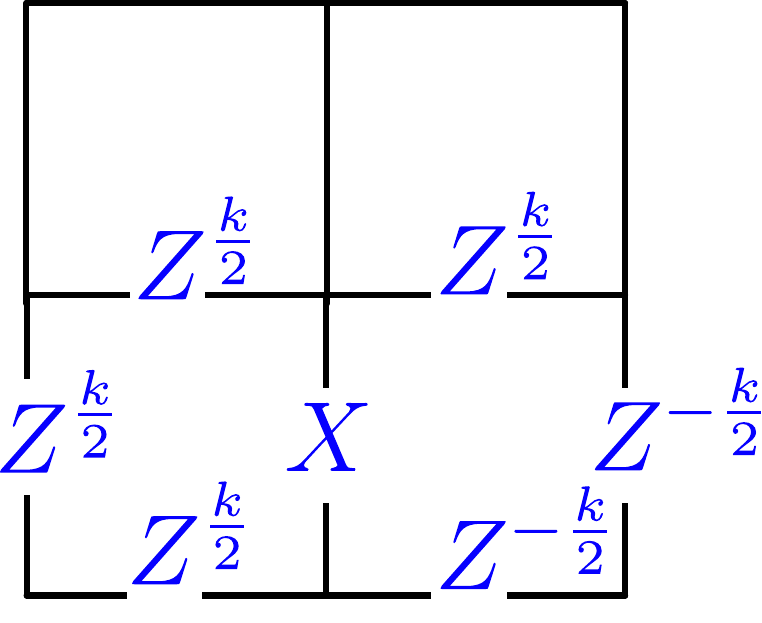}}\times
\raisebox{-2.75em}{\includegraphics[scale=.21]{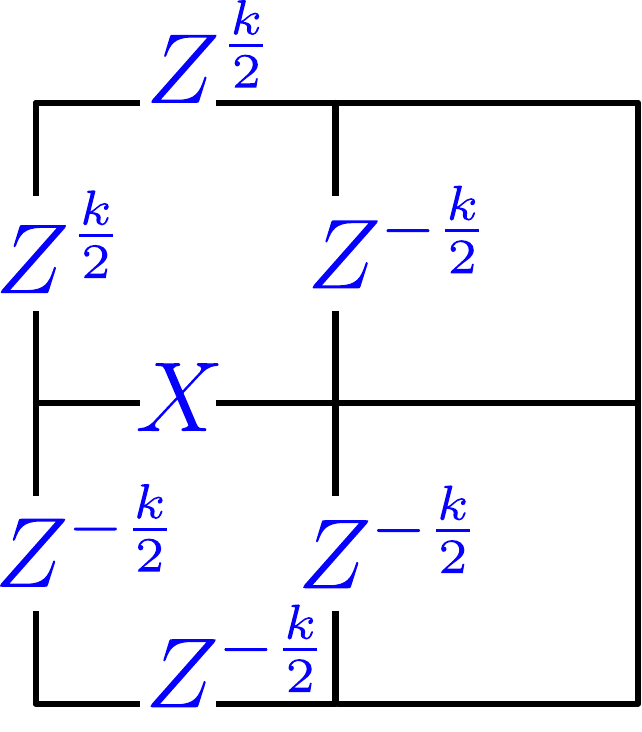}}\times\raisebox{-2.75em}{\includegraphics[scale=.21]{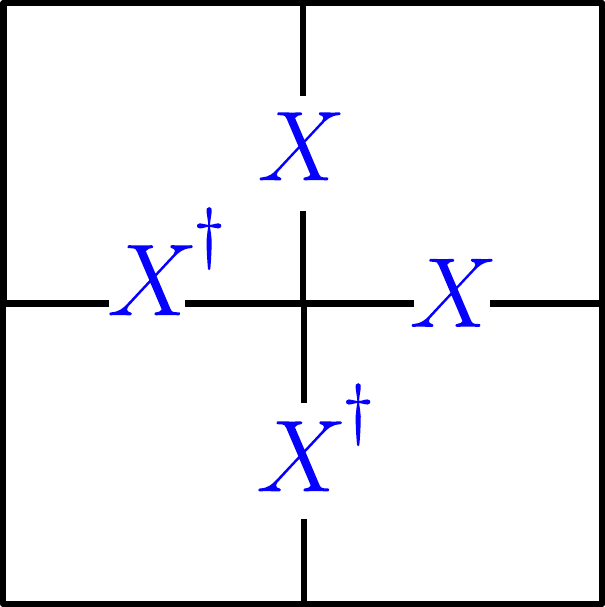}} ~.
\end{equation}
Finally, we multiply the Pauli $X_e$ and plaquette term $Z_{\partial f}$ to $A_e^{(k)}$, to derive single Pauli $Z_e$ as
\begin{equation}
    Z_{e'}= \prod_{e} \left( A_e^{(k)} X_e^{-1}
    \prod_{f} Z_{\partial f}^{-\frac{k}{2}\int(\be \cup_1 \bface - \bface \cup_1 \be)}
    \right)^{\frac{1}{k} \int \be' \cup \be}.
\end{equation}
Graphically, this is
\begin{eqs}
\raisebox{-1.5em}{\includegraphics[scale=.25]{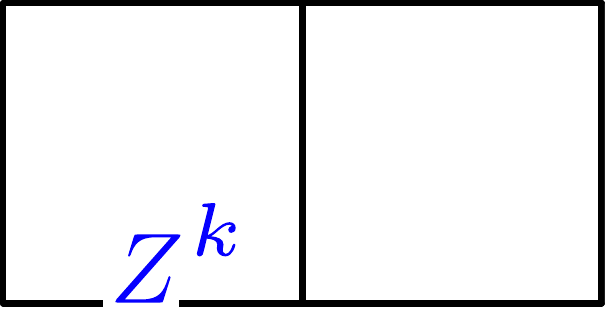}}&=\raisebox{-1.5em}{\includegraphics[scale=.25]{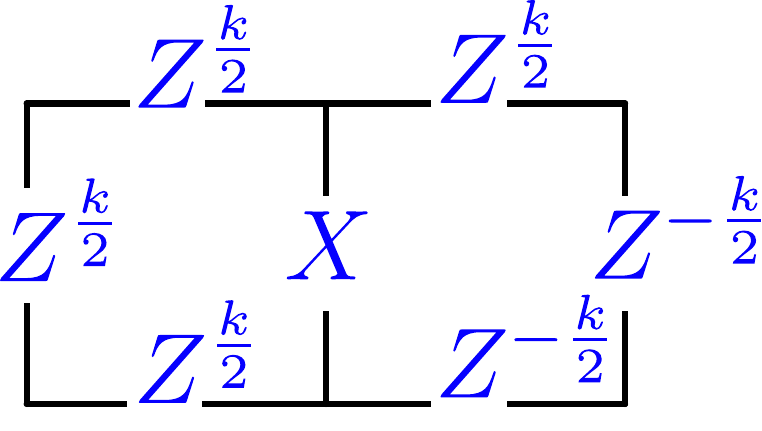}}\times\raisebox{-1.5em}{\includegraphics[scale=.25]{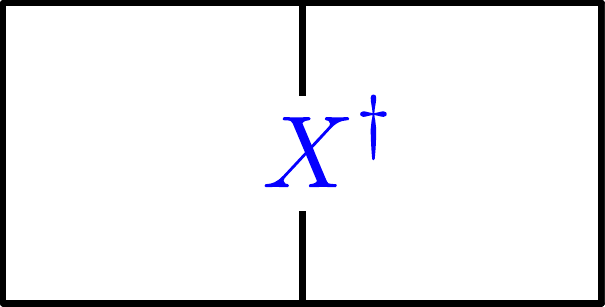}}\times\raisebox{-1.5em}{\includegraphics[scale=.25]{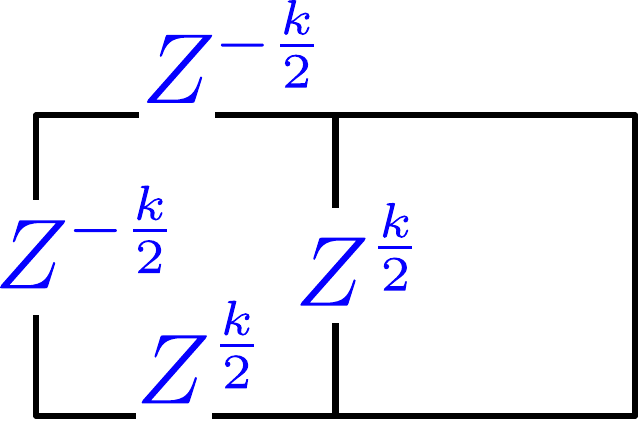}}\times\raisebox{-1.5em}{\includegraphics[scale=.25]{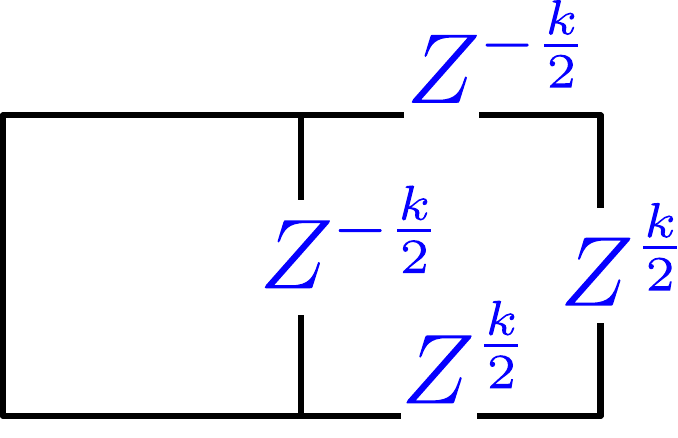}} , \\
\raisebox{-3.5em}{\includegraphics[scale=.25]{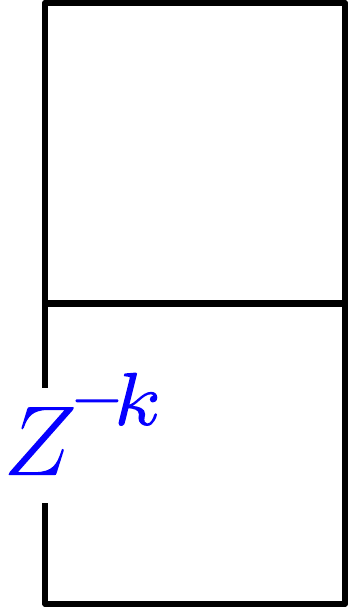}} &=\raisebox{-3.5em}{\includegraphics[scale=.25]{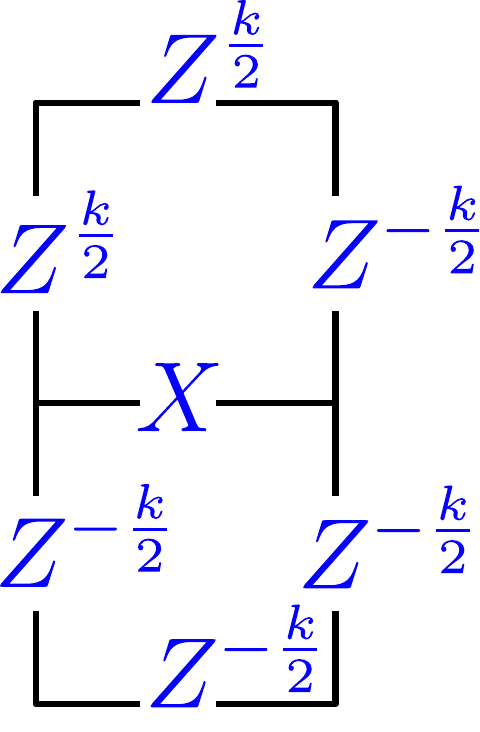}}\times\raisebox{-3.5em}{\includegraphics[scale=.25]{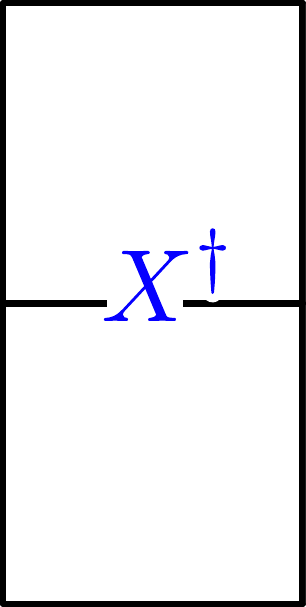}}\times\raisebox{-3.5em}{\includegraphics[scale=.25]{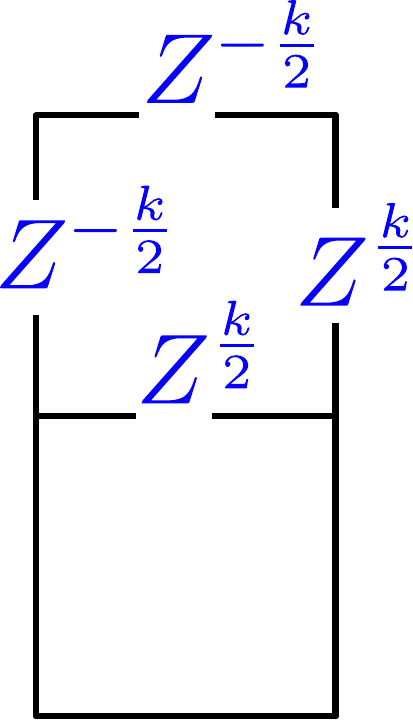}}\times\raisebox{-3.5em}{\includegraphics[scale=.25]{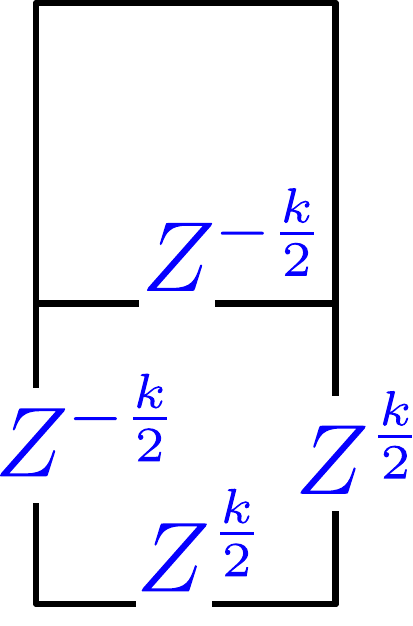}}.
\end{eqs}
Thus, this can be expressed entirely in terms of $A_e^{(k)}$ and $\bar A_e^{(k)}$ as 
\begin{equation}
\label{eq:ZfromAAbar}
\begin{aligned}
    Z_{e'} =& \prod_{e} \Biggl[\Biggl( {A_e^{(k)}} {\big(\bar A_e^{(k)}\big)}^{-1} \Biggl) \times  \prod_{v,f}\left( {A_{\delta v}^{(k)}} \big(\bar A_{\delta v}^{(k)}\big)^{-1} \right)^{-\frac{1}{2} \int \bv \cup \bface \int(\be \cup_1 \bface - \bface \cup_1 \be)}\Biggl]
    ^{\frac{1}{2k} \int \be' \cup \be}.
\end{aligned}
\end{equation}
\end{widetext}

When $p \equiv 1 \ (\mathrm{mod}\ 4)$, the element $-1$ is a quadratic residue in $\mathbb{Z}_p$. 
Let $a \in \mathbb{Z}_p$ satisfy $a^2 \equiv -1 \ (\mathrm{mod}\ p)$. 
Then the diagonal symplectic transformation
\begin{equation}
\begin{pmatrix}
a & 0 & 0 & 0 \\
0 & a & 0 & 0 \\
0 & 0 & -a & 0 \\
0 & 0 & 0 & -a
\end{pmatrix}
\end{equation}
maps $A_e^{(k)}$ to its conjugate $\bar{A}_e^{(k)}$. 
Consequently, two copies of the invertible subalgebra $A_e^{(k)}$ are equivalent to one copy of $A_e^{(k)}$ and one copy of $\bar{A}_e^{(k)}$. 
This tensor product Hilbert space is reducible to a trivial factorization, implying that the combined system of two copies of $A_e^{(k)}$ is trivial ISA. 
This conclusion is fully consistent with the Witt group classification, where for primes $p \equiv 1 \ (\mathrm{mod}\ 4)$ the corresponding quadratic forms yield only order-two elements, so stacking two copies indeed produces a trivial class.

\subsubsection{Polynomial form}

In the odd prime cases, the ISA are constructed as
\begin{align}
    \bA_e^{(k)} &= \bU_e^{(k)}+\frac{k}{2} \bZ_{e'} (\bM^\dagger_{e\cup_1 \delta e'}-\bM_{\delta e'\cup_1 e}) , \\
    \bA^{(k)} &= \begin{pmatrix}
        \bbone\\
        \boldsymbol{\mathcal{M}}
    \end{pmatrix},
\end{align}
where 
\begin{align}
        \boldsymbol{\mathcal{M}}_{e_0,e} = k\bM_{e_0\cup e} + \frac{k}{2}(\bM^\dagger_{e\cup_1 \delta e_0}-\bM_{\delta e_0\cup_1 e}).
\end{align}
Define $ \bar \bA^{(k)} \equiv \bA^{(-k)}$. One can check that
\begin{align}
    \langle \bA^{(k)}, \bar \bA^{(k)} \rangle =0.
\end{align}
$\bA^{(k)}$ defines an invertible subalgebra if $\bA^{(k)}$ and $\bar \bA^{(k)}$ generates all Pauli operators. That is, the matrix
\begin{align}
    \bH = \begin{pmatrix}
        \bA^{(k)} & \bar \bA^{(k)}
    \end{pmatrix}
\end{align}
is invertible. An explicit form of the inverse is
\begin{align}
    \bH^{-1} =\frac{1}{2} \begin{pmatrix} \bbone & \boldsymbol{\mathcal{M}}^{-1} \\
        \bbone &- \boldsymbol{\mathcal{M}}^{-1}
    \end{pmatrix}.
\end{align}
Thus, $\boldsymbol{\mathcal{M}}$ needs to be invertible. A sufficient condition for being invertible is if $\bM_{v\cup f}$ and $\bM_{e'\cup e}$ are invertible matrices, which is the case for the square lattice. Since the columns of $\boldsymbol{H}^{-1}$ describe the coefficients of $\bA^{(k)}$  and $\bar\bA^{(k)}$  that give $\bX$ and $\bZ$, we can directly obtain the expression for $\boldsymbol{\mathcal{M}}^{-1}$  in terms of the cup products from Eq.~\eqref{eq:ZfromAAbar} as
\begin{eqs}
    \boldsymbol{\mathcal{M}}^{-1}_{e_0,e} =& \frac{1}{k} \Bigl[\bd_{e_0,e} -\frac{1}{2} (\bM^\dagger_{\partial e_0,f})^{-1} \\
    & \hspace{2em} \times (\bM^\dagger_{e'\cup_1 f}-\bM_{f\cup_1 e'})\Bigl] (\bM_{e\cup e'})^{-1} .
\end{eqs}

{ \change
A QCA can then be constructed from the ISA by again pushing, for example $\bA^{(k)}$, in an arbitrary direction of $m = x^i y^jz^k$ as
\begin{equation}
     \bH \begin{pmatrix} m & 0 \\ 0 & 1\end{pmatrix}\bH^{-1}
\end{equation}
Take $m= x^{-1}y^{-1}z^{-1}$ as an example, we have 
\begin{equation}\label{eq:ZkpQCA_from_ISA}
    \balpha^{(k)}_\text{ISA} = A + z^{-1} B
\end{equation}
where
}
\begin{widetext}
{\change
\begin{equation*}
    A = \left(
\begin{array}{cccc}
 \frac{1}{2} & 0 & \frac{x}{4 k}-\frac{1}{4 k x} & \frac{y}{4 k x}+\frac{1}{4 k x}+\frac{y}{4 k}-\frac{1}{4 k} \\
 0 & \frac{1}{2} & -\frac{x}{4 k y}-\frac{x}{4 k}-\frac{1}{4 k y}+\frac{1}{4 k} & \frac{1}{4 k y}-\frac{y}{4 k} \\
 \frac{k}{4 y}-\frac{k y}{4} & -\frac{k y}{4 x}-\frac{k}{4 x}-\frac{k y}{4}+\frac{k}{4} & \frac{1}{2} & 0 \\
 \frac{k x}{4 y}+\frac{k x}{4}+\frac{k}{4 y}-\frac{k}{4} & \frac{k x}{4}-\frac{k}{4 x} & 0 & \frac{1}{2}
\end{array}
\right)~,
\end{equation*}
and
\begin{equation*}
    B = \left(
\begin{array}{cccc}
 \frac{1}{2 x y} & 0 & \frac{1}{4 k x^2 y}-\frac{1}{4 k y} & -\frac{1}{4 k x^2 y}-\frac{1}{4 k x^2}+\frac{1}{4 k x y}-\frac{1}{4 k x} \\
 0 & \frac{1}{2 x y} & \frac{1}{4 k x y^2}-\frac{1}{4 k x y}+\frac{1}{4 k y^2}+\frac{1}{4 k y} & \frac{1}{4 k x}-\frac{1}{4 k x y^2} \\
 \frac{k}{4 x}-\frac{k}{4 x y^2} & \frac{k}{4 x^2 y}+\frac{k}{4 x^2}-\frac{k}{4 x y}+\frac{k}{4 x} & \frac{1}{2 x y} & 0 \\
 -\frac{k}{4 x y^2}+\frac{k}{4 x y}-\frac{k}{4 y^2}-\frac{k}{4 y} & \frac{k}{4 x^2 y}-\frac{k}{4 y} & 0 & \frac{1}{2 x y} \\
\end{array}
\right)~.
\end{equation*}
}
Further, we may explicitly construct a QCA that squares to translation in an arbitrary direction. In particular, let the desired translation be given by a monomial $m = x^i y^jz^k$ for $p=1$ mod 4. Define
\begin{equation}
\begin{aligned}
    \balpha_\text{ISA} &= \bH \begin{pmatrix} m & 0 \\ 0 & 1\end{pmatrix}\bH^{-1}\begin{pmatrix} -a\bbone & 0 \\
        0 &a \bbone
    \end{pmatrix}
    = \frac{1}{2}\begin{pmatrix} -a(m+1) \bbone &a(m-1)\boldsymbol{\mathcal{M}}^{-1} \\
        -a(m-1)\boldsymbol{\mathcal{M}} &a(m+1) \bbone 
    \end{pmatrix} ,
\end{aligned}
\end{equation}
where we choose $a^2=-1$ (mod $p$) such that $\balpha_\text{ISA}$ is a symplectic transformation. Then, we see that
\begin{align}
    \balpha_\text{ISA}^2 = \begin{pmatrix}-m \bbone &0\\
       0 &-m \bbone
    \end{pmatrix}, 
\end{align}
which is a translation in the direction $m$ followed by charge conjugation.

When $p= -1 (\text{mod}~4)$, there exists no $a^2= -1(\text{mod}~p)$, but there always exist $a$ and $b$ that $a^2+b^2\equiv -1(\text{mod}~p)$. In this case, we must consider the stack of two invertible subalgebras $\bH\oplus\bH$ represented by a block-diagonal matrix $\text{diag}(\bH,\bH)$
We can then consider a QCA induced by this double subalgebra
\begin{eqs}
\balpha_\text{ISA} &=\begin{pmatrix}
    \bH&0\\0&\bH
\end{pmatrix}
\begin{pmatrix}
    m\bbone&0&0&0\\0&\bbone&0&0\\0&0&m\bbone&0\\0&0&0&\bbone
\end{pmatrix}
\begin{pmatrix}
    \bH^{-1}&0\\0&\bH^{-1}
\end{pmatrix}
\begin{pmatrix}
    a\bbone&0&b\bbone&0\\
    0&-a\bbone&0&-b\bbone\\
    b\bbone&0&-a\bbone&0\\
    0&-b\bbone&0&a\bbone
\end{pmatrix}~.
\end{eqs}
It is straightforward to check that this $\balpha_\text{ISA}$ also satisfies 
\begin{eqs}
\balpha^2_\text{ISA}=\begin{pmatrix}
    -m\bbone&0&0&0\\0&-m\bbone&0&0\\0&0&-m\bbone&0\\0&0&0&-m\bbone
\end{pmatrix},
\end{eqs}
which represents a shift along the direction $m$ along with a charge conjugation on both pieces of the invertible subalgebra.
\end{widetext}

\subsection{$4{+}1$D and $2l{+}1$D $\mathbb{Z}_2$ invertible subalgebras}

For $\mathbb{Z}_2$ ISAs, the overall signs in front of higher cup products can be ignored. Hence, each cochain can simply be promoted by one degree, and the same analysis applies. For instance, Eq.~\eqref{eq:Z2_ISA_2+1D} generalizes to $4{+}1$D as
\begin{eqs}
    A_f^A &= \tU_f^A \prod_t W_t^{A \int \bt \cup_1 \bface}~, \\
    A_f^B &= \tU_f^B \left( \prod_t W_t^{B \int \bface \cup_1 \bt} \right) Z_{\mathrm{PD}(f)}^A~, \\
    \bar A_f^A &= \tU_f^A \left( \prod_t W_t^{A \int \bface \cup_1 \bt} \right) Z_{\mathrm{PD}(f)}^B~, \\
    \bar A_f^B &= \tU_f^B \prod_t W_t^{B \int \bt \cup_1 \bface}~,
\label{eq:Z2_ISA_4+1D triangulation}
\end{eqs}
where qubits of type $A$ are placed on the faces of the direct lattice and qubits of type $B$ on the faces of the dual lattice. Here $\mathrm{PD}(f)$ denotes the face in the dual (direct) lattice that intersects the face $f$ in the direct (dual) lattice. The hopping and flux terms are defined as
\begin{eqs}
    \tU_f^\alpha := X_f^\alpha \prod_{f'} \left(Z_{f'}^\alpha\right)^{\int \bface \cup \bface'}~, \qquad
    W_t^\alpha := Z^\alpha_{\partial t},
\end{eqs}
$\alpha \in \{A,B\}$. It is straightforward to verify that these operators satisfy the definition of an ISA. ISAs in $6{+}1$D and higher dimensions can be constructed analogously by continuing to increase the cochain degree.

\subsection{$6{+}1$D and $(4l{-}2){+}1$D $\mathbb{Z}_p^{(k)}$ invertible subalgebras}

The recursive relation for cup products, introduced in the derivation of the $2{+}1$D invertible subalgebras in Eqs.~\eqref{eq: 2d ISA cup product 1} and~\eqref{eq: 2d ISA cup product 2}, generalizes naturally to higher dimensions, reappearing every fourth dimension.

As shown in Sec.~\ref{sec:Zp_QCA_(4l-1)+1D_TQFT}, promoting each cochain by two degrees yields the same structure. On a hypercubic lattice, each nonvanishing product $\boldsymbol{c}_i \cup \boldsymbol{c}_{d-i}$ establishes a one-to-one correspondence between $\boldsymbol{c}_i$ and $\boldsymbol{c}_{d-i}$, ensuring that the subalgebra remains invertible.  

As a concrete example, the $6{+}1$D $\mathbb{Z}_p^{(k)}$ invertible subalgebras are defined by
\begin{eqs}
    A_{c_3}^{(k)} &= X_{c_3} \prod_{{c_3}'} Z_{{c_3}'}^{k \int \bc_3' \cup \bc_3}
    \prod_{c_4} Z_{\partial c_3}^{\frac{k}{2}\int(\bc_3 \cup_1 \bc_4 - \bc_4 \cup_1 \bc_3)}~,\\
    \bar A_{c_3}^{(k)} &=  A_{c_3}^{(-k)},
\end{eqs}
where $c_3$ and $c_4$ denote 3- and 4-dimensional (hyper)cubes, respectively.
For higher dimensions, the construction proceeds analogously by increasing the cochain degree by two.

{\change

\section{Equivalence of QCAs Constructed from TQFTs and ISAs}
\label{Sec: Equivalence of QCAs Constructed from TQFTs and ISAs}

In this section, we first explicitly demonstrate that the $3{+}1$D Clifford QCAs constructed from TQFTs and from ISAs are equivalent by comparing their skew-Hermitian forms within the polynomial formalism. While this approach is, in principle, applicable to higher dimensions, the required explicit computations rapidly become prohibitively time-consuming.
To overcome this difficulty, we develop an algebraic approach based on the cup-product formalism to analyze the boundary excitations. We show that the excitations obtained from the TQFT construction and those from the ISA construction coincide, implying that the two approaches yield equivalent QCAs. This argument applies to arbitrary dimensions.

We begin by reviewing the definition of boundary algebra introduced in Ref.~\cite{liang2024operator} and further developed in Ref.~\cite{ruba2025witt}.
Related constructions and perspectives have appeared in
Refs.~\cite{schuster2023holographic, Kong2017Boundarybulk, kitaev2012models, Lan2014stringnet_boundary},
as well as in studies of phase transitions
Refs.~\cite{Wang2023RG, Chen2024sequential_circuit}.

Given a translation-invariant topological stabilizer code, we truncate the system by removing all qudits below the plane $z=0$ and discarding all stabilizers that involve any of the truncated qudits. The remaining stabilizers are referred to as \textbf{bulk stabilizers}.
    
    The \textbf{boundary algebra} is defined as the algebra generated by local Pauli operators that commute with all bulk stabilizers, modulo the bulk stabilizers themselves. This algebra captures the independent boundary degrees of freedom that are invisible to the bulk Hamiltonian. The structure of the boundary algebra can be encoded by a finitely generated module over a polynomial ring in one fewer variable, together with a skew-Hermitian form describing the commutation relations. In general, this module need not be free, nor is the form necessarily non-degenerate. This reflects the presence of nontrivial bulk excitation and braiding~\cite{liang2024operator, schuster2023holographic}. Such modules are referred to as \textbf{quasi-symplectic} in Ref.~\cite{ruba2025witt}. 

    Let \(\alpha\) be a translationally invariant Clifford QCA acting on \(\mathbb{Z}_p\) qudits, with \(p\) a prime. The separators of \(\alpha\) generate the stabilizer group of an associated stabilizer code. In this special case, there are no bulk excitation, and the boundary algebra constructed as above is a free module equipped with a non-degenerate skew-Hermitian form. Section~4.3 of Ref.~\cite{ruba2025witt} shows that this boundary algebra is Witt equivalent to the boundary algebra defined for a Clifford QCA in Ref.~\cite{Haah2021CliffordQCA}, thereby realizing the boundary algebra of a Clifford QCA as a special case of the boundary algebra of a topological stabilizer code.

\begin{theorem}[\textbf{Classification of Clifford QCAs by skew-Hermitian forms}~\cite{Haah2021CliffordQCA}]
Two Clifford QCAs are equivalent if and only if their boundary algebras are equivalent. More precisely, suppose their boundary algebras are free modules equipped with skew-Hermitian forms \(\Xi\) and \(\Xi'\), respectively. Then the QCAs are equivalent if and only if there exist integers \(q, q'\) and an invertible matrix \(E\) such that
\begin{equation}
    E^\dagger \, (\Xi \oplus \lambda_q) \, E
    = \Xi' \oplus \lambda_{q'} ,
\end{equation}
where \(\lambda_q\) and \(\lambda_{q'}\) denote the standard symplectic matrices of dimensions \(2q\) and \(2q'\), respectively.
\label{theom:QCA_equivalence}
\end{theorem}

\subsection{Polynomial formalism}
In this section, we represent the QCA $\alpha$ by a polynomial matrix $Q$. First, we briefly describe how to obtain a skew-Hermitian form $\Xi$ that appears in Theorem~\ref{theom:QCA_equivalence}. As a representation of $\alpha$, each column of $Q$ encodes the image of the onsite Pauli $X, Z$ operators. Assume we fixed some ordering such that the first $q$ columns are the flippers and the last $q$ columns are the respective separators/ stabilizers. It is always possible to coarse-grain the sites so that the separators and flippers have a width of at most two in the $x_D$-direction. In other words, $Q= A + x_D B$ where $A, B$ do not involve the variable $x_D$. The columns of matrix $B$ are the flippers and separators truncated above the hyperplane $x_D=0$. A priori, $Q$ is symplectic following from the fact that flippers and separators have a commutation relation governed by the standard symplectic matrix $\lambda_q$. After the truncation, the commutation relation is governed by a different skew-hermitian form. Focusing on the truncated separators, i.e., the last $q$ columns $B_0$ of $B$, we denote their commutation relation by the form $\Xi= B_0^\dagger \lambda_q B_0$. These truncated separators corresponding to the columns of $B_0$ are boundary gauge operators in Ref.~\cite{liang2024operator}.
The boundary algebra of a stabilizer code determines a Witt class that is independent of the choice of truncation procedure and of basis transformations of $B$~\cite{Haah2021CliffordQCA, ruba2025witt}.


The skew-Hermitian form $\Xi$ associated with the boundary algebra of a QCA can be computed using the following lemma.
\begin{lemma}[Lemma III.8 of Ref.~\cite{Haah2021CliffordQCA}]
    Let $Q = A + x_D B$ be a $2q \times 2q$ symplectic matrix where $A,B$ do not involve the variable $x_D$ (e.g. in $3{+}1$D, $x_D$ can be chosen to be $x^\pm, y^\pm, z^\pm$). Let $B_0$ be the last $q$ columns of $B$, and put $\Xi = B_0^\dagger \lambda_q B_0$. If the matrix $\Xi$ is invertible, then $\Xi$ is a skew-Hermitian form of Theorem~\ref{theom:QCA_equivalence}.
    \label{lemma:compute_anti_hermitian_form}
\end{lemma}
In addition, we will need the following Clifford gates expressed in terms of
polynomial matrices:
\begin{itemize}
    \item \textit{Hadamard:}
    \[
        \mathrm{H}_i \coloneqq
        E_{i,i+q}(-1)\, E_{i+q,i}(1)\, E_{i,i+q}(-1),
        \qquad 1 \leq i \leq q .
    \]
    \item \textit{Controlled-phase:}
    \[
        \mathrm{CPhase}^{\,i}_f \coloneqq E_{i+q,i}(f),
        \qquad f = \bar f,\quad 1 \leq i \leq q .
    \]
    \item \textit{Controlled-NOT:}
    \[
        \mathrm{CX}^{\,i,j}_a \coloneqq
        E_{i,j}(a)\, E_{j+q,i+q}(-\bar a),
        \qquad 1 \leq i \neq j \leq q .
    \]
\end{itemize}
Here, for $i \neq j$, $E_{i,j}(a)$ denotes the row-addition elementary
$2q \times 2q$ matrix with entries
\begin{equation}
    \bigl[E_{i,j}(a)\bigr]_{\mu\nu}
    = \delta_{\mu\nu} + \delta_{\mu i}\,\delta_{\nu j}\, a ,
\end{equation}
where $a$ is a polynomial.
We refer the reader to Sec.~2 of Ref.~\cite{haah_commuting_2013} for a detailed
discussion of these gates and illustrative examples.
}

\begin{widetext}
{\change
\subsubsection{$\mathbb{Z}_2$ case}

We are ready to show that the 3F QCA $\balpha^\text{3F}_\text{TQFT}$ obtained from the TQFT approach Eq.~\eqref{eq:3FQCA_TQFT}  is equivalent to $\balpha^\text{3F}_\text{ISA}$ obtained from the invertible subalgebra Eq.~\eqref{eq:3FQCA_ISA}. 
Observe that $\balpha^\text{3F}_\text{ISA}$ is already of the form $A+z \, B$. Using lemma~\ref{lemma:compute_anti_hermitian_form}, we find the skew-Hermitian form of

\begin{equation}
    \Xi^\text{3F}_\text{ISA} = \left(
        \begin{array}{cccc}
         x+\frac{1}{x} & \frac{y}{x}+\frac{1}{x}+y+1 & 0 & y \\
         \frac{x}{y}+x+\frac{1}{y}+1 & y+\frac{1}{y} & x & 0 \\
         0 & \frac{1}{x} & x+\frac{1}{x} & \frac{y}{x}+\frac{1}{x}+y+1 \\
         \frac{1}{y} & 0 & \frac{x}{y}+x+\frac{1}{y}+1 & y+\frac{1}{y} \\
        \end{array}
        \right)~.
    \label{eq:antihermitian_form_3F_from_ISA}
\end{equation}

Now consider $\balpha^\text{3F}_\text{TQFT}$ whose matrix elements are given in Appendix~\ref{app:3FQCA_TQFT_3D}. Recall that QCAs are defined up to FDQC and shifts. Applying a shift circuit of
\begin{equation}
    \bm{S} = \left(
    \begin{array}{cccc}
     x\,\mathbb{I}_{3\times 3} & 0 & 0 & 0 \\
     0 & \mathbb{I}_{3\times 3} & 0 & 0 \\
     0 & 0 & x\,\mathbb{I}_{3\times 3} & 0 \\
     0 & 0 & 0 & \mathbb{I}_{3\times 3} \\
    \end{array}
    \right)~,
\end{equation}
we find $\bm{S} \, \balpha^\text{3F}_\text{TQFT}$ is of the form $A + x\, B$, where $A$ and $B$ are matrices that do not involve the variable $x$. However, we cannot yet use lemma~\ref{lemma:compute_anti_hermitian_form}, as a skew-Hermitian form needs to be invertible. We act further by the basis transformation
\begin{equation}
    \bU ( \bm{\cdot} ) = U_\text{FDQC} \bigl(\bm{\cdot}\bigl) E~,
\end{equation}
where
\begin{equation}
\begin{aligned}
    U_\text{FDQC} =& ~\text{H}_6 \text{H}_4 \text{H}_3 \text{H}_2\text{CX}^{6,4}_{1+\frac{1}{z}} \text{CX}^{3,4}_{\frac{1}{y z^2}+\frac{1}{y z}} \text{CX}^{2,4}_{\frac{1}{y^2 z}+\frac{1}{yz}} \text{H}_6 \text{H}_3 \text{H}_2 
    \text{CX}^{5,4}_{1+\frac{1}{yz}}\text{CX}^{3,4}_{\frac{1}{z}+\frac{1}{yz}} \text{CX}^{2,4}_{\frac{1}{y}+\frac{1}{yz}} \text{H}_4 \text{H}_1 \text{CX}^{6,1}_{1+yz} \text{CX}^{5,1}_{1+\frac{1}{yz}} \\
    & \times\text{H}_6 \text{H}_5\text{CX}^{6,1}_{1+\frac{1}{z}} \text{CX}^{5,1}_{1+z}\text{H}_1 \text{H}_2 \text{H}_3 \text{CX}^{3,1}_{1} \text{CX}^{2,1}_{\frac{1}{y}}\text{H}_2\text{H}_3
    \text{CX}^{6,4}_{1} \text{H}_6 \text{CX}^{6,5}_{y}~,
\end{aligned}
\end{equation}
is a finite-depth quantum circuit and
\begin{equation}
\begin{aligned}
    E =& ~\text{T}^{4,12}_1 \text{T}^{4,11}_{\frac{1}{y}} \text{T}^{1,4}_{1+yz} \text{T}^{2,4}_{yz+y^2z} \text{T}^{3,4}_{yz + y z^2} \text{T}^{8,4}_{y + yz} \text{T}^{9,4}_{z+yz} \text{T}^{1,9}_1 \text{T}^{1,8}_{\frac{1}{y}} 
    \text{T}^{2,7}_{y} \text{T}^{3,7}_1 \text{T}^{7,1}_{\frac{1}{y z^2} + \frac{1}{z} + \frac{1}{y z} + z + yz + y z^2} \\
    & \times \text{T}^{10,11}_{1+\frac{1}{y}} \text{T}^{10,12}_{1+\frac{1}{z}} \text{T}^{2,4}_{yz + y^2 z} \text{T}^{3,4}_{yz + y z^2}
    \text{T}^{5,4}_{1+y} \text{T}^{6,4}_{1+z} \text{T}^{8,4}_{y+yz} \text{T}^{9,4}_{z+yz} \text{T}^{5,10}_{y} \text{T}^{6,10}_{1} \text{T}^{5,12}_{1+\frac{y}{z}} \text{T}^{6,12}_{1+\frac{1}{z}} \text{T}^{5,11}_{1+\frac{1}{y}}~,
\end{aligned}
\end{equation}
with $\text{T}^{i,j}_{f}$ is the matrix of elementary column operation that adds to the $i$th column by the $j$th column multiplied by $f$. The column operation is just multiplying our stabilizers generators, and it does not change the stabilizer group. Note that we have designed $E$ such that the transformed QCA is symplectic. Acting these transformations on our QCA, we find
\begin{equation}
     \bU \bigl(\bm{S} \balpha^\text{3F}_\text{TQFT}\bigl) = \left(
\begin{array}{cccccccccccc}
 x & 0 & 0 & 0 & 0 & 0 & 0 & 0 & 0 & 0 & 0 & 0 \\
 0 & * & 0 & 0 & * & * & 0 & 0 & 0 & 0 & 0 & * \\
 0 & 0 & * & 0 & * & * & 0 & 0 & 0 & 0 & * & 0 \\
 0 & 0 & 0 & 1 & 0 & 0 & 0 & 0 & 0 & 0 & 0 & 0 \\
 0 & * & * & 0 & * & * & 0 & * & * & 0 & * & * \\
 0 & * & * & 0 & * & * & 0 & * & * & 0 & 0 & 0 \\
 0 & 0 & 0 & 0 & 0 & 0 & x & 0 & 0 & 0 & 0 & 0 \\
 0 & 0 & 0 & 0 & * & * & 0 & * & 0 & 0 & * & * \\
 0 & 0 & 0 & 0 & * & * & 0 & 0 & * & 0 & * & * \\
 0 & 0 & 0 & 0 & 0 & 0 & 0 & 0 & 0 & 1 & 0 & 0 \\
 0 & * & * & 0 & * & 0 & 0 & 0 & * & 0 & 0 & 0 \\
 0 & * & * & 0 & * & 0 & 0 & * & * & 0 & 0 & * \\
\end{array}
\right)~,
\end{equation}
where the $*$ terms are non-zero polynomials which are given below. Observe that the first and fourth qubits are disentangled from the rest of the system. They can be removed. We now have the transformed QCA
\begin{equation}
    \balpha^\text{3F}_\text{TQFT, transformed} = A+x \, B
\end{equation}
where
\begin{equation*}
\begin{aligned}
     &A = \left(
\begin{array}{cccccccc}
 0 & 0 & \frac{1}{y z} & \frac{z+1}{y} & 0 & 0 & 0 & \frac{1}{y} \\
 0 & 0 & \frac{y+1}{z} & \frac{y+z+1}{y z} & 0 & 0 & \frac{1}{z} & 0 \\
 \frac{1}{y z} & \frac{1}{y} & y+1 & \frac{1}{y}+1 & 1 & \frac{1}{y} & 1 & \frac{1}{y} \\
 y (y+1) z+1 & z (y z+y+1) & y & 1 & y & y z & 0 & 0 \\
 0 & 0 & \frac{(y+1) (y z+1)}{y^2 z} & \frac{z}{y} & 0 & 0 & \frac{1}{y} & \frac{y+1}{y^2} \\
 0 & 0 & \frac{y z (y z+y+z)+z+1}{y z^2} & z+\frac{1}{z} & 0 & 0 & \frac{1}{z}+1 & \frac{y z+z+1}{y z} \\
 y z+z+1 & z (z+1) & 1 & 0 & 0 & z & 0 & 0 \\
 (y+1) y z+y+\frac{1}{z}+1 & (z+1) (y z+1) & y+1 & 0 & y (z+1) & y z+z+1 & 0 & 1 \\
\end{array}
\right)~,
\end{aligned}
\end{equation*}
and
\begin{equation*}
    B = \left(
\begin{array}{cccccccc}
 1 & 0 & \frac{1}{y}+\frac{1}{z}+1 & \frac{z+1}{y z} & 0 & 0 & 0 & \frac{1}{y} \\
 0 & 1 & \frac{y+1}{y z} & 1 & 0 & 0 & \frac{1}{z} & 0 \\
 y (y+1) z+\frac{1}{z} & y z (z+1)+\frac{1}{y z} & 0 & 0 & y z+y+\frac{1}{z} & y z+z+1 & 0 & 0 \\
 y & 1 & 0 & 0 & y & y z & 0 & 0 \\
 0 & 0 & \frac{y^3+1}{y^2 z}+y+1 & \frac{\frac{1}{y}+z^2+1}{z} & 1 & 0 & \frac{\frac{1}{y}+1}{z}+1 & \frac{1}{y}+1 \\
 0 & 0 & \frac{y^2 (z+1)+y+z+1}{y z^2} & \frac{1}{z^2}+1 & 0 & 1 & \frac{z+1}{z^2} & \frac{1}{z} \\
 y & z+1 & 0 & 0 & 0 & z & 0 & 0 \\
 y \left(y+\frac{1}{z}+1\right) & y z+y+\frac{1}{z}+1 & 0 & 0 & y \left(\frac{1}{z}+1\right) & y & 0 & 0 \\
\end{array}
\right)~.
\end{equation*}
Applying lemma~\ref{lemma:compute_anti_hermitian_form}, we have
\begin{equation}
    \Xi^\text{3F}_\text{TQFT} = \left(
\begin{array}{cccc}
 z+\frac{1}{z} & \frac{z}{y}+\frac{1}{y}+z+1 & 0 & \frac{1}{y} \\
 \frac{y}{z}+y+\frac{1}{z}+1 & y+\frac{1}{y} & \frac{1}{z} & 0 \\
 0 & z & z+\frac{1}{z} & \frac{z}{y}+\frac{1}{y}+z+1 \\
 y & 0 & \frac{y}{z}+y+\frac{1}{z}+1 & y+\frac{1}{y} \\
\end{array}
\right)~,
\end{equation}
which is exactly equal to $\Xi^\text{3F}_\text{ISA}$ of Eq.~\eqref{eq:antihermitian_form_3F_from_ISA} by re-defining our coordinate system by sending $z \to x^{-1},y\to y^{-1}$. Thus, Theorem~\ref{theom:QCA_equivalence} tells us $\balpha^\text{3F}_\text{TQFT}$ and $\balpha^\text{3F}_\text{ISA}$ are equivalent.

\subsubsection{$\mathbb{Z}_p$ case}

Let us now show that the QCAs, $\balpha^{(k)}_\text{TQFT}$ obtained from the TQFT Eq.~\eqref{eq:ZkpQCA_from_TQFT}  and $\balpha^{(k)}_\textbf{ISA}$ obtained from the invertible subalgebra Eq.~\eqref{eq:ZkpQCA_from_ISA}, are equivalent. First, we bring Eq.~\eqref{eq:ZkpQCA_from_TQFT} into the form of $A+z^{-1}B$ through an FDQC, which we pick to be
\begin{equation}
\begin{aligned}
    \bU =&~ \text{CPhase}^{3}_{\frac{1}{4} k x y z+\frac{k}{4 x y z}-\frac{k x y}{4}-\frac{k}{4 x y}-\frac{k z}{4}-\frac{k}{4 z}+\frac{k}{2}} \text{CX}^{1,3}_{-\frac{1}{2x}-\frac{yz}{2}} \text{CX}^{2,3}_{\frac{1}{2xy}+\frac{z}{2}} \text{H}_1 \text{H}_2 
    \text{CX}^{2,3}_{-\frac{k}{2}-\frac{k}{2x y^2} + \frac{k}{2y}+ \frac{kz}{2}} \text{CX}^{1,3}_{\frac{k}{2}-\frac{k}{2x}-\frac{k}{2xy}+\frac{kz}{2x}} \\
    & \times \text{H}_3 \text{H}_1 \text{H}_2 \text{CX}^{2,3}_{\frac{1}{2ky}- \frac{1}{2kxy}} 
    \text{CX}^{1,3}_{\frac{1}{2kx}-\frac{1}{2kxy}} \text{H}_2 \text{H}_1 \text{CX}^{2,3}_{\frac{1}{2}-\frac{1}{2y}} \text{CX}^{1,3}_{-\frac{1}{2}+\frac{1}{2x}} \text{CX}^{3,2}_{-\frac{1}{z}} \text{CX}^{3,1}_{\frac{1}{yz}} \text{H}_3 \text{H}_2 \text{H}_1
\end{aligned}
\end{equation}
And, we have
\begin{equation}
    \bU \balpha^{(k)}_\text{TQFT} = \left(
\begin{array}{cccccc}
 * & 0 & 0 & * & * & 0 \\
 0 & * & 0 & * & * & 0 \\
 0 & 0 & -1 & 0 & 0 & 0 \\
 * & * & 0 & * & 0 & 0 \\
 * & * & 0 & 0 & * & 0 \\
 0 & 0 & 0 & 0 & 0 & -1 \\
\end{array}
\right)~,
\end{equation}
where the $*$ terms are given below. Observe that the third qudit is now disentangled from the rest of the system and can be dropped. We then obtain the transformed QCA
\begin{equation*}
\begin{aligned}
    &\balpha^{(k)}_\text{TQFT,transformed}= \\
    &\hspace{0.5em} \left(
\begin{array}{cccc}
 \frac{1}{2 x y z}+\frac{1}{2} & 0 & -\frac{1}{2 k x y^2 z}+\frac{1}{2 k x y z}-\frac{y}{2 k}+\frac{1}{2 k} & \frac{1}{2 k x y z}+\frac{y}{2 k x}-\frac{1}{2 k x}-\frac{y}{2 k} \\
 0 & \frac{1}{2 x y z}+\frac{1}{2} & -\frac{1}{2 k x y^2 z}+\frac{1}{2 k y^2 z}-\frac{1}{2 k y z}+\frac{1}{2 k} & \frac{1}{2 k x y z}-\frac{1}{2 k x}-\frac{1}{2 k y z}+\frac{1}{2 k} \\
 \frac{k}{2 x y z}-\frac{k}{2 x}-\frac{k}{2 y z}+\frac{k}{2} & -\frac{k}{2 x y z}-\frac{k y}{2 x}+\frac{k}{2 x}+\frac{k y}{2} & \frac{1}{2 x y z}+\frac{1}{2} & 0 \\
 \frac{k}{2 x y^2 z}-\frac{k}{2 y^2 z}+\frac{k}{2 y z}-\frac{k}{2} & -\frac{k}{2 x y^2 z}+\frac{k}{2 x y z}-\frac{k y}{2}+\frac{k}{2} & 0 & \frac{1}{2 x y z}+\frac{1}{2} \\
\end{array}
\right)~.
\end{aligned}
\end{equation*}
Applying lemma~\ref{lemma:compute_anti_hermitian_form}, we find the skew-Hermitian form to be
\begin{equation}
\begin{aligned}
    &\Xi_\text{TQFT} = 
    \hspace{0.2em} \left(
\begin{array}{cc}
 \frac{1}{4 k y}-\frac{y}{4 k} & \frac{y}{4 k x}-\frac{1}{4 k x}-\frac{y}{4 k}-\frac{1}{4 k} \\
\frac{x}{4 k}+\frac{1}{4 k y}+\frac{1}{4 k} -\frac{x}{4 k y} & \frac{x}{4 k}-\frac{1}{4 k x} \\
\end{array}
\right)~.
\end{aligned}
\end{equation}
}
\end{widetext}
{\change
It is then straightforward to show that $\Xi_\text{TQFT}= E^\dagger\Xi_\text{ISA}E$ for the $\Xi_\text{ISA}$ computed from Eq.~\eqref{eq:ZkpQCA_from_ISA}, where
\begin{equation}
    E = \begin{pmatrix}
 0 & y^{-1}x \\
 1 & 0 \\
\end{pmatrix}~.
\end{equation}
By Theorem~\ref{theom:QCA_equivalence}, the QCAs obtained from the TQFT approach and the invertible subalgebra approach are equivalent.

We have computed the matrix representation of the skew-Hermitian form for the TQFT-based $3{+}1$D QCA and compared it with that obtained from the ISA construction, thereby explicitly verifying this equivalence. However, this comparison involves numerous intermediate basis transformations. Furthermore, in higher-dimensional QCAs, the matrix size grows rapidly, rendering it impractical to explicitly determine the required basis transformations or the corresponding unitary circuits that relate the QCAs constructed from TQFTs and ISAs.

\subsection{Cup product formalism}

Motivated by the limitations of the polynomial approach, we develop an alternative method to extract the boundary algebra that does not rely on an explicit matrix representation. In this subsection, we reformulate the boundary algebra using the cup-product formalism, which is directly applicable in arbitrary spatial dimensions.

\subsubsection{$\ZZ_2$ case}

We begin by studying the three-fermion Walker–Wang model on a spatial manifold $M_3$ with an open boundary $N_2 = \partial M_3$. We review the auxiliary-vertex approach (also known as the boundary cone construction) introduced in Refs.~\cite{Chen2021Disentangling, chen2023exactly}, which is consistent with the boundary Hamiltonian presented in Ref.~\cite{Chen2023HigherCup}.

To extend our bulk construction to manifolds with boundary, we embed the system into a closed manifold by adjoining a set of auxiliary simplices. Concretely, we introduce an auxiliary vertex $v_0$ and connect it to every vertex on the boundary manifold $N_2$. This generates auxiliary edges $\lr{0i}$ for each $i \in N_2$, auxiliary faces $\lr{0ij}$ for each boundary edge $\lr{ij} \subset N_2$, and higher-dimensional auxiliary simplices defined analogously. The union of these auxiliary simplices with $N_2$ defines the cone over $N_2$, which we denote by $CN_2$. By gluing $CN_2$ to the original manifold $M_3$ along $N_2$, we obtain a closed spatial three-manifold,
\begin{eqs}
    \widetilde M_3 \equiv M_3 \sqcup CN_2, \qquad \partial \widetilde M_3 = 0~.
\end{eqs}
With this closed manifold in hand, we may directly apply the bulk formalism. Given cochains $a,b \in C^1(M_4,\ZZ_2)$, we extend them to cochains in $C^1(\widetilde M_4,\ZZ_2)$ by setting $a(e)=b(e)=0$ on all auxiliary edges $e$. Upon gauging the $1$-form $\ZZ_2 \times \ZZ_2$ symmetry, the state $|a|_{\widetilde M_4}, b|_{\widetilde M_4}\rangle$ is mapped to
\begin{eqs}
    |\delta a|_{\widetilde M_3}, \delta b|_{\widetilde M_3}\rangle
    =
    |\delta a|_{M_3}, \delta b|_{M_3}, a|_{N_2}, b|_{N_2}\rangle,
\label{eq: gauged state with boundary}
\end{eqs}
where $(\cdot)|_S$ denotes restriction of a cochain to the subspace $S$.

To interpret the boundary degrees of freedom, we identify the value of $\delta a$ on an auxiliary face $\lr{0ij}$ with the edge variable $a(\lr{ij})$ on $N_2$, using the condition $a(\lr{0i})=0$. The same identification applies to $\delta b$ and $b|_{N_2}$. We denote the Pauli operators acting on the four components appearing in Eq.~\eqref{eq: gauged state with boundary} by $Z^A_f$, $Z^B_f$, $\CZ^A_e$, and $\CZ^B_e$, respectively.


We now analyze the Hamiltonian terms in Eq.~\eqref{eq: 3fWW_Hamiltonian 1} defined on $\widetilde{M}_3$. We decorate the terms involving $G_e$ with appropriate flux factors $Z_{\partial t}$, which preserves the stabilizer group:
\begin{eqs}
    & ~G^A_e \prod_f {\tilde{Z}^B_f}{}^{\int_{\widetilde{M}_3} \be \cup \bface} \\
    &= \tilde{X}^A_{\delta \be}
    \prod_{f} 
    {\tilde{Z}^A_{f}}{}^{\int_{\widetilde{M}_3}  \delta \be \cup_1 \bface}
    \prod_f {\tilde{Z}^B_f}{}^{\int_{\widetilde{M}_3} \be \cup \bface} \\
    & \sim 
    \tilde{X}^A_{\delta \be}
    \prod_{f} 
    {\tilde{Z}^A_{f}}{}^{\int_{\widetilde{M}_3}  \be \cup \bface + \bface \cup \be}
    \prod_f {\tilde{Z}^B_f}{}^{\int_{\widetilde{M}_3} \be \cup \bface}~,
\end{eqs}
and
\begin{eqs}
    & ~G^B_e \prod_f {\tilde{Z}^A_f}{}^{\int_{\widetilde{M}_3} \bface \cup \be} \\
    &= \tilde{X}^B_{\delta \be}
    \prod_{f} 
    {\tilde{Z}^B_{f}}{}^{\int_{\widetilde{M}_3}  \delta \be \cup_1 \bface}
    \prod_f {\tilde{Z}^A_f}{}^{\int_{\widetilde{M}_3} \bface \cup \be} \\
    &\sim \tilde{X}^B_{\delta \be}
    \prod_{f} 
    {\tilde{Z}^B_{f}}{}^{\int_{\widetilde{M}_3}  \be \cup \bface + \bface \cup \be}
    \prod_f {\tilde{Z}^A_f}{}^{\int_{\widetilde{M}_3} \bface \cup \be}~.
\end{eqs}
We now focus on the behavior of these stabilizers in the vicinity of the boundary $N_2$.
\begin{enumerate}
    \item $t=\lr{0ijk}$ denotes an auxiliary tetrahedron, and let $f=\lr{ijk}$ be the corresponding boundary face. The associated flux term can be written as
    \begin{equation}
        Z^\alpha_{\partial t} =  Z_f^\alpha \prod_{e \in N_2 | e\subset f}\CZ_{e}^\alpha := Z_f^\alpha \CZ_{\partial f}^\alpha,
    \end{equation}
    for $\alpha = A,B$. Here, $Z_f^\alpha$ acts on the bulk degrees of freedom in $M_3$, while $\CZ_{\partial f}^\alpha$ acts on the boundary degrees of freedom associated with $CN_2$. In other words, the Pauli $Z_f$ on the boundary face $f$ of $M_3$ is equal to the product of $\CZ_e$ around edge $e$ such that $e \in\partial f$.
    
    \item $e=\lr{0i}$ denotes an auxiliary edge, and let $v=\lr{i}$ be the corresponding boundary vertex. In this case, we obtain the stabilizer as
    \begin{eqs}
        &
        \tilde{X}^A_{\delta \be}
        \prod_{f} 
        {\tilde{Z}^A_{f}}{}^{\int_{\widetilde{M}_3}  \be \cup \bface + \bface \cup \be}
        \prod_f {\tilde{Z}^B_f}{}^{\int_{\widetilde{M}_3} \be \cup \bface} \\
        &= \CX_{\delta \bv}^A
        \prod_{f \in N_2} {(Z^A_{f} Z^B_{f})}^{\int_{N_2}  \bv \cup \bface} \\
        &= \CX_{\delta \bv}^A
        \prod_{f \in N_2} {(\CZ^A_{\partial f} \CZ^B_{\partial f})}^{\int_{N_2}  \bv \cup \bface}~,
    \end{eqs}
    where $\sim$ indicates equality up to multiplication by flux terms.
    Similarly,
    \begin{eqs}
        &\tilde{X}^B_{\delta \be}
        \prod_{f} 
        {\tilde{Z}^B_{f}}{}^{\int_{\widetilde{M}_3}  \be \cup \bface + \bface \cup \be}
        \prod_f {\tilde{Z}^A_f}{}^{\int_{\widetilde{M}_3} \bface \cup \be} \\
        &= \CX_{\delta \bv}^B
        \prod_{f \in N_2} {\CZ^B_{\partial f}}^{\int_{N_2}  \bv \cup \bface}~.
    \end{eqs}

    \item $e=\lr{ij}$ denotes a boundary edge. In this case, we obtain
    \begin{eqs}
        &
        \tilde{X}^A_{\delta \be}
        \prod_{f} 
        {\tilde{Z}^A_{f}}{}^{\int_{\widetilde{M}_3}  \be \cup \bface + \bface \cup \be}
        \prod_f {\tilde{Z}^B_f}{}^{\int_{\widetilde{M}_3} \be \cup \bface} \\
        &= X^A_{\delta \be} \CX^A_e
        \prod_{f \in M_3} 
        {Z^A_{f}}^{\int_{{M}_3}  \be \cup \bface + \bface \cup \be} \prod_{e' \in N_2} {\CZ^A_{e'}}^{\int_{N_2} \be' \cup \be} \\
        & \qquad \qquad \times  \prod_f {Z^B_f}^{\int_{M_3} \be \cup \bface} \\
        &= \CX^A_e \prod_{e' \in N_2} {\CZ^A_{e'}}^{\int_{N_2} \be' \cup \be}
        \times \text{ bulk terms}~, \nonumber
    \end{eqs}
    and
    \begin{eqs}
        &\tilde{X}^B_{\delta \be}
        \prod_{f} 
        {\tilde{Z}^B_{f}}{}^{\int_{\widetilde{M}_3}  \be \cup \bface + \bface \cup \be}
        \prod_f {\tilde{Z}^A_f}{}^{\int_{\widetilde{M}_3} \bface \cup \be} \\
        &= X^B_{\delta \be} \CX^B_e
        \prod_{f \in M_3} 
        {Z^B_{f}}^{\int_{{M}_3} \be \cup \bface + \bface \cup \be} \prod_{e' \in N_2} {\CZ^B_{e'}}^{\int_{N_2} \be' \cup \be} \\
        & \qquad \qquad \times \prod_{f \in M_3} 
        {Z^A_{f}}^{\int_{{M}_3}  \bface \cup \be} \prod_{e' \in N_2} {\CZ^A_{e'}}^{\int_{N_2} \be' \cup \be} \\
        &= \CX^B_e \prod_{e' \in N_2} {(\CZ^A_{e'}\CZ^B_{e'})}^{\int_{N_2} \be' \cup \be}
        \times \text{ bulk terms}~. \nonumber
    \end{eqs}
\end{enumerate}
To summarize, the boundary Hamiltonian takes the form
\begin{eqs}
    &\!H_\text{boundary} \\
    =&-\sum_{v \in N_2}
        \CX_{\delta \bv}^A
        \prod_{f \in N_2} {(\CZ^A_{\partial f} \CZ^B_{\partial f})}^{\int_{N_2}  \bv \cup \bface}\\
    &-\sum_{v \in N_2}\CX_{\delta \bv}^B
        \prod_{f \in N_2} {\CZ^B_{\partial f}}^{\int_{N_2}  \bv \cup \bface}\\
    &-\sum_{f \in N_2} \left( Z_f^A \CZ^A_{\partial f} 
    + Z_f^B \CZ^B_{\partial f} \right)
    \\
    &-\sum_{e \in N_2} \CX^A_e \prod_{e' \in N_2} {\CZ^A_{e'}}^{\int_{N_2} \be' \cup \be}
        \times \text{ bulk terms} \\
    &-\sum_{e \in N_2} \CX^B_e \prod_{e' \in N_2} {(\CZ^A_{e'}\CZ^B_{e'})}^{\int_{N_2} \be' \cup \be}
        \times \text{ bulk terms}~.
\label{eq: boundary Hamiltonian of 3F WW}
\end{eqs}
We now swap the roles of the bulk and the boundary, treating the cone $CN_2$ as the ``bulk,'' with Hilbert space generated by $\CX_e^{A,B}$ and $\CZ_e^{A,B}$. From this perspective, the first two lines can be interpreted as bulk stabilizers on $CN_2$, while the $\CX$ and $\CZ$ operators appearing in the remaining terms generate the boundary algebra.

In particular, we focus on the last two lines, which correspond to the anyons $m_A e_A$ and $m_B e_A e_B$ in the $\ZZ_2^A \times \ZZ_2^B$ toric code. These are precisely the same anyon generators that appear in the ISA construction.

Explicitly, the bulk stabilizer on $CN_2$ is given by
\begin{equation}
    S_{\mathrm{bulk}} = \left\langle 
    \raisebox{-3em}{\includegraphics[scale=0.25]{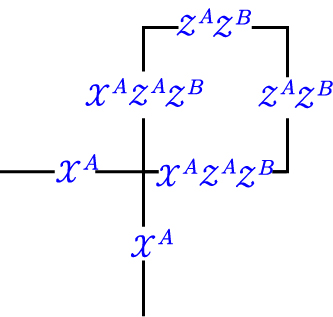}},~
    \raisebox{-3em}{\includegraphics[scale=0.25]{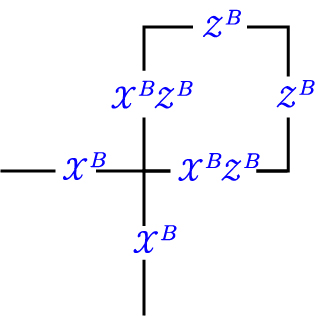}}
    \right\rangle~.
\end{equation}
The boundary algebra is defined as
\begin{equation}
    \mathcal{A}_{\mathrm{boundary}} = \frac{S_{\mathrm{bulk}}^\perp}{S_{\mathrm{bulk}}}~,
\end{equation}
where $S_{\mathrm{bulk}}^\perp$ consists of all finite Pauli operators generated by $\CX_e^{A,B}$ and $\CZ_e^{A,B}$ that commute with $S_{\mathrm{bulk}}$.
From the last three lines of Eq.~\eqref{eq: boundary Hamiltonian of 3F WW}, the corresponding boundary algebra is generated by the following operators:
\begin{eqs}
    \forall \alpha \in \{A, B\}~, \quad
    \CZ^\alpha_{\partial f}
    &= \raisebox{-2.5em}{\includegraphics[scale=.25]{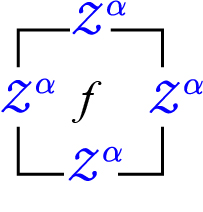}}~,  \\
    \CX^A_e \prod_{e' \in N_2} {\CZ^A_{e'}}^{\int_{N_2} \be' \cup \be}
    &= \raisebox{-2em}{\includegraphics[scale=.25]{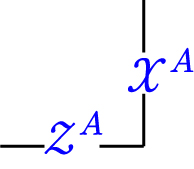}}~,
    \raisebox{-2em}{\includegraphics[scale=.25]{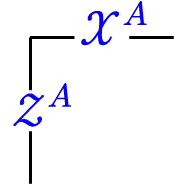}}~, \\
    \CX^B_e \prod_{e' \in N_2} {(\CZ^A_{e'}\CZ^B_{e'})}^{\int_{N_2} \be' \cup \be}
    &= \raisebox{-2em}{\includegraphics[scale=.25]{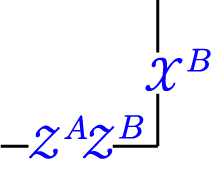}}~,
    \raisebox{-2em}{\includegraphics[scale=.25]{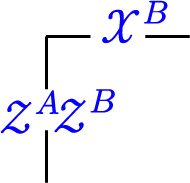}}~.
\end{eqs}
However, these generators are not independent. A straightforward analysis shows that the following two operators are sufficient to generate the entire boundary algebra (up to multiplication by bulk stabilizers):
\begin{eqs}
    &\CX^A_e \prod_{e' \in N_2} {\CZ^A_{e'}}^{\int_{N_2} \be' \cup \be + \be \cup_1 \delta \be'} \\
    & =~ \raisebox{-2.3em}{\includegraphics[scale=0.25]{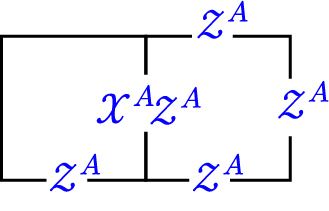}}~, \quad
    \raisebox{-3.8em}{\includegraphics[scale=.25]{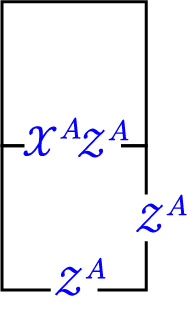}}~, \\
    &\CX^B_e \prod_{e' \in N_2} {\CZ^A_{e'}}^{\int_{N_2} \be' \cup \be} {\CZ^B_{e'}}^{\int_{N_2} \be' \cup \be + \delta \be' \cup_1 \be} \\
    & =~ \raisebox{-2.3em}{\includegraphics[scale=0.25]{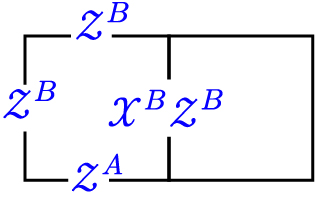}}~, \quad
    \raisebox{-3.8em}{\includegraphics[scale=0.25]{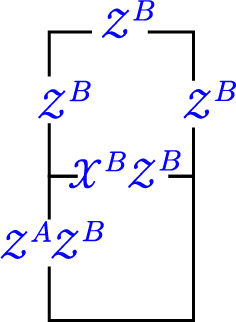}}~.
\end{eqs}
We note that this boundary algebra is exactly the ISA for the 3-fermion theory, up to a reordering of the cup products.

The same analysis extends to higher dimensions with a periodicity of 2: one simply shifts the degree of each cochain by 1, and the recursive relations for the cup products continue to hold modulo~$2$.


\subsubsection{$\ZZ_p$ case}

In this section, we study the $\ZZ_p^{(k)}$ Walker--Wang Hamiltonian defined in Eq.~\eqref{eq: 1/p BUB QCA initial Hamiltonian}, which consists of the following two types of terms:
\begin{eqs}
     Z_{\partial t}, \qquad
     X_{\delta \be} \prod_{f} Z_{f}^{k \int \bface \cup \be + \be \cup \bface}~.
\end{eqs}
As in the previous section, we construct the Hamiltonian on the closed manifold
${\widetilde M}_3 \equiv M_3 \sqcup CN_2$.
We place $\ZZ_p$ qudits on each face $f \in M_3$ and each edge $e \in N_2$, equipped with Pauli operators $X_f, Z_f$ and $\CX_e, \CZ_e$, respectively.

We now analyze the Hamiltonian terms supported near $N_2$:
\begin{enumerate}
    \item
    Let $t=\lr{0ijk}$ denote an auxiliary tetrahedron, and let $f=\lr{ijk}$ be the corresponding boundary face. The associated flux term can be written as
    \begin{equation}
        Z_{\partial t}
        = Z_{\lr{ijk}} \CZ_{\lr{jk}}^{-1} \CZ_{\lr{ik}} \CZ_{\lr{ij}}^{-1}
        := Z_f \CZ_{\partial f}^{-1}~.
    \end{equation}

    \item
    Let $e=\lr{0i}$ denote an auxiliary edge, and let $v=\lr{i}$ be the corresponding boundary vertex. In this case, the stabilizer takes the form
    \begin{eqs}
        & \tilde{X}_{\delta \be}
        \prod_{f} \tilde{Z}_{f}^{k \int_{\widetilde{M}_3} (\bface \cup \be + \be \cup \bface)} \\
        &= \CX_{\delta \bv}^{-1}
        \prod_{f \in N_2} Z_{f}^{k\int_{N_2} \bv \cup \bface} \\
        &= \CX_{\delta \bv}^{-1}
        \prod_{f \in N_2} \CZ_{\partial f}^{k\int_{N_2} \bv \cup \bface}~,
    \end{eqs}
    where the minus sign in the exponent of $\CX_{\delta \bv}$ arises because when $e=\lr{0i}$ appears with positive orientation in $\partial\lr{0ij}$, the corresponding boundary vertex $v=\lr{i}$ appears with negative orientation in $\partial\lr{ij}$.

    \item
    Let $e=\lr{ij}$ denote a boundary edge. In this case, we obtain
    \begin{eqs}
        & \tilde{X}_{\delta \be}
        \prod_{f} \tilde{Z}_{f}^{k \int_{\widetilde{M}_3} \bface \cup \be + \be \cup \bface} \\
        &= X_{\delta \be} \CX_e
        \prod_{f \in M_3}
        Z_{f}^{k\int_{M_3} \bface \cup \be + \be \cup \bface}
        \prod_{e' \in N_2} \CZ_{e'}^{k\int_{N_2} \be' \cup \be} \\
        &= \CX_e
        \prod_{e' \in N_2} \CZ_{e'}^{k\int_{N_2} \be' \cup \be}
        \times \text{bulk terms}~.
    \end{eqs}
\end{enumerate}

Collecting the above results, the boundary Hamiltonian takes the form
\begin{eqs}
    &H_{\text{boundary}}\\
    &= -\sum_{v \in N_2}
        \CX_{\delta \bv}
        \prod_{f \in N_2} \CZ_{\partial f}^{-k\int_{N_2} \bv \cup \bface}
    -\sum_{f \in N_2} Z_f \CZ_{\partial f}^{-1} \\
    &\quad -\sum_{e \in N_2}
        \CX_e
        \prod_{e' \in N_2} \CZ_{e'}^{k\int_{N_2} \be' \cup \be}
        \times \text{bulk terms} ~ +~ h.c.
\label{eq: boundary Hamiltonian of Zp QCA}
\end{eqs}
As in the previous case, we treat the operators $\CX_e$ and $\CZ_e$ supported on $CN_2$ as bulk degrees of freedom. With this convention, the first term in Eq.~\eqref{eq: boundary Hamiltonian of Zp QCA} defines a bulk stabilizer,
\begin{eqs}
    \CX_{\delta \bv}
        \prod_{f \in N_2} \CZ_{\partial f}^{-k\int_{N_2} \bv \cup \bface}
    = \raisebox{-4em}{\includegraphics[scale=0.25]{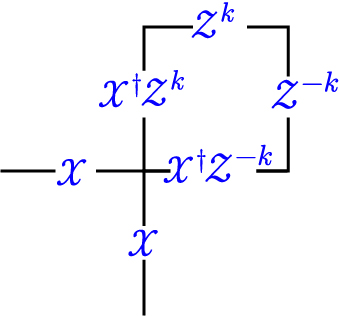}}~.
\end{eqs}

The remaining two terms in Eq.~\eqref{eq: boundary Hamiltonian of Zp QCA} generate the boundary algebra. In particular, the boundary operators can be taken as
\begin{eqs}
    \CZ_{\partial f}
    = \raisebox{-3em}{\includegraphics[scale=0.25]{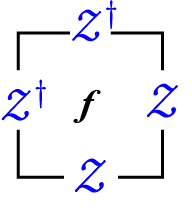}}~,
\end{eqs}
and
\begin{eqs}
    \CX_e \prod_{e'} \CZ_{e'}^{k\int_{N_2} \be' \cup \be}
    = \raisebox{-2em}{\includegraphics[scale=0.25]{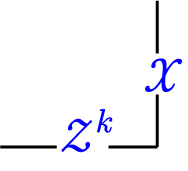}}~, \quad
      \raisebox{-1.5em}{\includegraphics[scale=0.25]{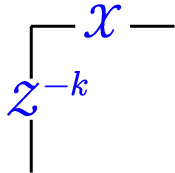}}~.
\end{eqs}
These operators are not independent: they generate redundant elements of the boundary algebra. By appropriately combining them, we obtain a single free generator,
\begin{eqs}
    A_e^{(k)} :=\;& \CX_e \prod_{e'} \CZ_{e'}^{k \int \be' \cup \be}
    \prod_f \CZ_{\partial f}^{\frac{k}{2}\int(\be \cup_1 \bface - \bface \cup_1 \be)} \\
    =\;&
    \raisebox{-2.5em}{\includegraphics[scale=0.25]{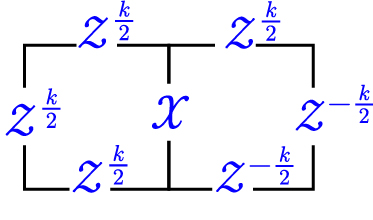}}~, \quad
    \raisebox{-4em}{\includegraphics[scale=0.25]{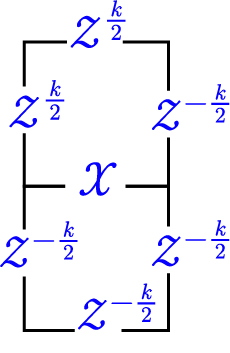}}~,
\end{eqs}
which coincides exactly with the generator of the $\mathbb{Z}_p$ ISA construction in Sec.~\ref{sec: Zpk invertible subalgebras}. 
Therefore, the QCAs constructed via the TQFT and ISA approaches are equivalent. 

The same analysis applies in higher dimensions, repeating every 4 dimensions: one simply increases the degree of each cochain by 2, and the recursive relations for the cup products continue to hold.
}
\section{Algebraic formalism for Clifford QCA}\label{sec: Algebraic formalism for Clifford QCA}

In this section, we develop an algebraic framework for analyzing Clifford QCAs. We model them as automorphisms of the generalized Pauli group, modulo phase, subject to a locality-preserving condition. This reformulation reduces the study of Clifford QCAs to symplectic automorphisms of an abelian group with a band-diagonality constraint. It also enables a natural connection to the questions in algebraic $K$-theory. We begin with the definition of generalized Pauli operators and show how Clifford QCAs correspond to the locality-preserving automorphisms of their group.

\subsection{Preliminaries}

Recall from Sec.~\ref{sec:Symplectic_Rep_Pauli}, we place $n$-dimensioal qudits on each $p$-dimensional simplex $\sigma_p$. Throughout Sec.~\ref{sec: Algebraic formalism for Clifford QCA}, we denote the collection of qudits by $\Lambda$, which is equipped with a distance function $d: \Lambda \times \Lambda \rightarrow \RR_{\geq 0}$. 

\begin{defn}
    Denote by $X_i, Z_i$ the generalized Pauli matrices supported on $i\in \Lambda$. Let \textit{generalized Pauli operators} $\mathbf P$ be the group generated by $\{X_i, Z_i\}_{i\in \Lambda}$.
\end{defn}

\begin{prop}
    A Clifford QCA $\alpha$ is completely determined by the images $\alpha(X_i), \alpha(Z_i)\in \mathbf P$ for every $i\in \Lambda$. In other words, a Clifford QCA is reduced to an automorphism $\alpha\in \mathrm{Aut}(\mathbf P)$ that meets the locality-preserving condition. 
\end{prop}
\begin{proof}
Generalized Pauli $X_i, Z_i$ generate the entire algebra of local operators. Therefore, the image under $\alpha$ of another operator is determined by automorphism properties. 
\end{proof}

It is possible to express \(\mathbf{P}\) modulo its center \(\ZZ_d\) using an abelian group~\cite{calderbank1997quantum}. The clock and shift operators $X_i$ and $Z_i$ defined are represented respectively by $\colvec{1}{0}$ and $\colvec{0}{1}$. Each generalized Pauli operator has a corresponding sum of $\colvec{1}{0}$ and $\colvec{0}{1}$, which is uniquely defined up to a phase. For example, both $X_i^2Z_i$ and $X_iZ_iX_i$ are represented by $\colvec{2}{1}$. As $X_i^d= Z_i^d= I$, we have $\colvec{d}{0}=\colvec{0}{d}=\colvec{0}{0}$. All calculations shall therefore be done modulo $d$, where $d$ is the dimension of the qudit. We denote the set of all length-$2$ column vectors modulo $d$ by $\ZZ_d^2$. 

Moreover, the commutation relation between two operators is recovered by the matrix $\begin{pmatrix}
    0 & 1\\
    -1 & 0
\end{pmatrix}$: two Pauli operators represented by $\colvec{a}{b}$ and $\colvec{c}{d}$ commute up to some $\exp(2\pi i \frac{m}{d})\in \mu_d$ with
$$
m =\omega_i\left(\colvec{a}{b}, \colvec{c}{d}\right)=
\begin{pmatrix}
    a & b
\end{pmatrix} 
\begin{pmatrix}
    0 & 1\\
    -1 & 0
\end{pmatrix}  \colvec{c}{d}.
$$
This is also called the standard symplectic form on  $\ZZ_d^2$
\begin{equation}
\omega_i: \ZZ_d^2\times \ZZ_d^2 \rightarrow \ZZ_d.
\end{equation}
\begin{defn}
    Motivated by this, we define $P:=\bigoplus_{\Lambda}\ZZ_d^2$ and $\omega:=\sum_{i\in \Lambda}\omega_i.$ Let $\mathrm{Aut}^{\omega}(P)=\{\theta\in \mathrm{Aut}(P): \omega(\theta(p_1), \theta(p_2))=\omega(p_1, p_2) \ \mathrm{for }\ p_1, p_2 \in P\}$ be the collection of symplectic automorphisms. 
\end{defn}

\begin{lemma}
    There is a surjective homomorphism $$\kappa: \mathrm{Aut}(\mathbf P)\rightarrow \mathrm{Aut}^{\omega}(P).$$ Furthermore, $\ker \kappa$ consists of only Clifford circuits. 
\end{lemma}

\begin{proof}
An automorphism $\alpha\in \mathrm{Aut}(\mathbf P)$ has to map the center to itself.
Therefore, it induces an automorphism $\kappa(\alpha)$ of $P$. Since $\alpha$ respects the commutation relation in $\mathbf P$, $\kappa(\alpha)\in \mathrm{Aut}^{\omega}(P)$. If $\alpha\in \ker \kappa$, then $\alpha(X_i)=\exp(2\pi i \frac{m_i}{d})X_i$ and $\alpha(Z_i)=\exp(2\pi i \frac{n_i}{d})Z_i$. Therefore, $\alpha$ is the single-layer quantum circuit $\prod_{i \in \Lambda}X_i^{n_i}Z_i^{-m_i}$.
\end{proof}
As Clifford circuits are deemed trivial, it suffices to study $\mathrm{Aut}^{\omega}(P)$. As for the locality-preserving condition, the following representation is convenient.
An automorphism $\theta\in \mathrm{Aut}^{\omega}(P)$ has an explicit representation $$\theta=\begin{pmatrix}
    \Theta^{XX} & \Theta^{XZ}\\
    \Theta^{ZX} & \Theta^{ZZ}
\end{pmatrix},$$
where $\Theta^{\bullet \bullet}:\Lambda\times \Lambda \rightarrow \ZZ_d$ are 2-tensors indexed by $\Lambda$. Here $\Theta^{\bullet \bullet}\in \{\Theta^{XX}, \Theta^{XZ}, \Theta^{ZX}, \Theta^{ZZ}\}$. Viewed as $\Lambda$-indexed square matrices, multiplication between 2-tensors is defined as matrix multiplication. We will refer to these 2-tensors simply as matrices. The automorphism being $\theta$ symplectic is equivalent to the identities
$$\theta^T J_\Lambda \theta = J_\Lambda\textbf{ and } \theta J_\Lambda \theta^T = J_\Lambda,$$ where 
$$J_\Lambda = \begin{pmatrix}
    0 & I_\Lambda\\
    -I_\Lambda & 0
\end{pmatrix}.$$
Observe that $\theta$ is invertible with inverse $-J_\Lambda \theta^T J_\Lambda$.
The locality-preserving condition translates into a ``band-diagonality" condition: there exists   $r>0$ \ such that $\Theta^{\bullet \bullet}(i,j)=0$ for all $i, j\in \Lambda$ satisfying $d(i,j)>r.$ Let $\mathrm{Sp}_{\Lambda}^{bd}(\ZZ_d)$ denote the subgroup of band-diagonal elements in $\mathrm{Aut}^{\omega}(P)$. Similar objects have been considered in topology under the name controlled linear algerba~\cite{freedman1995controlled}.

A single-layer Clifford circuit corresponds to  $\theta$ where $\Theta^{\bullet \bullet}$ are identically block-diagonal. In other words, there exists a partition of $\Lambda$ into disjoint blocks $\Lambda=\coprod_\beta \Lambda_\beta$ of finite block sizes $\mathrm{diam}(\Lambda_\beta)<\infty$. And $\Theta^{\bullet \bullet}=\sum_\beta \Theta_\beta^{\bullet \bullet}$ with $\Theta_\beta^{\bullet \bullet}(i,j)=0$ unless $i, j \in \Lambda_\beta$. Furthermore, the band-diagonal condition forces the block sizes to be uniformly bounded. Let $ESp_{\Lambda}^{bd}(\ZZ_d)$ denote the group generated by them. It corresponds to finite depth Clifford circuits.  

A QCA which only permutes qudits corresponds to $\theta$ with $\Theta^{XZ}=\Theta^{ZX}=0$ and $\Theta^{XX}=\Theta^{ZZ}$ a band-diagonal permutation matrix.\footnote{A permutation matrix is a square binary matrix that has exactly one entry of 1 in each row and each column with all other entries 0.} Let $PSp_{\Lambda}^{bd}(\ZZ_d)$ denote the group generated by them.

In conclusion, the classification of Clifford QCA is reduced to 
$$\frac{Sp_{\Lambda}^{bd}(\ZZ_d)}{\langle ESp_{\Lambda}^{bd}(\ZZ_d), PSp_{\Lambda}^{bd}(\ZZ_d)\rangle},$$ up to a procedure called stabilization (see below). Elements in $\langle ESp_{\Lambda}^{bd}(\ZZ_d), PSp_{\Lambda}^{bd}(\ZZ_d)\rangle$ are referred to as trivial Clifford QCAs.

Above, we implicitly assume that there is a single qudit at each site $i\in \Lambda$. However, it is customary to allow multiple qudits on each site. Stabilization allows one to compare QCAs defined on the same $\Lambda$ with different on-site qudits. A QCA may be extended to additional on-site qudits by the identity action on operators supported on them. These extended QCA are regarded equivalent to the original. The procedure is called stabilization. In view of this, we should allow $$P=\bigoplus_{i\in \Lambda} \ZZ_d^{2q_i},$$ with different numbers $q_i$ of qudits on site $i$ and its symplectic form $\omega$ accordingly. This generalizes $$\Theta^{\bullet \bullet}(j, k)\in \mathrm{Mat}_{q_k\times q_j}(\ZZ_d),$$ for $j, k\in \Lambda$.

\subsection{QCAs with the same stabilizer group}\label{sec: QCAs with the same stabilizer group}

This section addresses a fundamental question: \textit{Are two Clifford QCAs with the same stabilizer group (or flipper group) always equivalent?}
{\change
We begin by introducing a special class of QCAs for which this question reduces to a particularly simple form.

\begin{defn}
A Clifford QCA $\alpha$ is called \textbf{separated} if
\[
\alpha(\langle X_i\rangle_{i\in \Lambda})=\langle X_i\rangle_{i\in \Lambda},
\qquad
\alpha(\langle Z_i\rangle_{i\in \Lambda})=\langle Z_i\rangle_{i\in \Lambda}.
\]
Here, $\langle X_i\rangle_{i\in \Lambda}$ (resp.\ $\langle Z_i\rangle_{i\in \Lambda}$)
denotes the group of Pauli operators generated purely by $X$ (resp.\ purely by $Z$)
operators and phase.  
For a separated QCA $\alpha$, the associated symplectic transformation
$\kappa(\alpha)$ necessarily takes the block-diagonal form
\[
\kappa(\alpha)
=
\begin{pmatrix}
    \Theta & 0 \\
    0 & (\Theta^T)^{-1}
\end{pmatrix}
\in Sp_{\Lambda}^{\mathrm{bd}}(\mathbb{Z}_d).
\]
We refer to such elements of $Sp_{\Lambda}^{\mathrm{bd}}(\mathbb{Z}_d)$ as
\textbf{separated} symplectic transformations.
\end{defn}
Note that a separated Clifford QCA maps CSS codes to CSS codes, and hence preserves the CSS structure.

\begin{theorem}
\label{thm:same-stabilizer}
Two Clifford QCAs with the same stabilizer group (or the same flipper
group) are equivalent up to finite-depth quantum circuits if and only if every separated Clifford QCA is trivial.
\end{theorem}
}

\begin{lemma}
 Any $\theta$ with $\Theta^{XZ}=0$ is equivalent to a separated $\theta'=\begin{pmatrix}
    \Theta & 0\\
    0 & (\Theta^T)^{-1}
\end{pmatrix}\in Sp_{\Lambda}^{bd}(\ZZ_d).$
\end{lemma}
\begin{proof}
    Write $\theta=\begin{pmatrix}
    \Theta^{XX} & 0\\
    \Theta^{ZX} & \Theta^{ZZ}
\end{pmatrix}.$ Then $\theta^T J_\Lambda \theta = J_\Lambda\text{ and } \theta J_\Lambda \theta^T = J_\Lambda$ is equivalent to $$\Theta^{ZZ}=((\Theta^{XX})^T)^{-1}$$ and 
$(\Theta^{XX})^T\Theta^{ZX}, \Theta^{ZZ}(\Theta^{ZX})^T$ are symmetric. By the first equality, $\theta'=\begin{pmatrix}
    \Theta^{XX} & 0\\
    0 & \Theta^{ZZ}
\end{pmatrix}\in Sp_{\Lambda}^{bd}(\ZZ_d).$ Moreover, $\theta'^{-1}\theta=\begin{pmatrix}
    I_\Lambda & 0\\
    (\Theta^{XX})^T\Theta^{ZX} & I_\Lambda
\end{pmatrix}\in ESp_{\Lambda}^{bd}(\ZZ_d).$ In other words, $\theta'^{-1}\theta$ lifts to a QCA with $X_j\mapsto X_j$ and $Z_j\mapsto \mathcal X_jZ_j$ where 
\begin{equation*}
\begin{aligned}
    \mathcal X_j &=\exp\{\pi i \left((\Theta^{XX})^T\Theta^{ZX}\right)(j, j)/d\} \\
    & \hspace{2em} \times \prod_{k\in \Lambda}X_k^{\left((\Theta^{XX})^T\Theta^{ZX}\right)(j, k)}
\end{aligned}
\end{equation*} consists only of Pauli $X$ matrices and a phase. Such a QCA is always a circuit. This proves the lemma. 
\end{proof}
\begin{proof}[Proof of Theorem]
     Let $\alpha$ and $\beta$ be two QCAs that both bijectively map $\langle X_i\rangle_{i\in \Lambda}$ to the same  image. Their difference $\beta^{-1}\circ\alpha$ is a Clifford QCA. Since $\beta^{-1}\circ\alpha$ maps $\langle X_i\rangle_{i\in \Lambda}$ bijectively to itself, $\kappa(\beta^{-1}\circ\alpha)\in A:= \{\theta\in \mathrm{Sp}_\Lambda^{bd}(\ZZ_d): \Theta^{XZ}=0\}.$ 
     By the lemma above, $\kappa(\beta^{-1}\circ\alpha)$ is equivalent to some $\theta'=\begin{pmatrix}
    \Theta & 0\\
    0 & (\Theta^T)^{-1}
\end{pmatrix}$. Hence $\alpha$ and $\beta$ are equivalent if $\theta'$ is trivial. 

Conversely, a nontrivial $\theta'$ of this form has the same stabilizer group as the identity QCA. Therefore, it provides a counterexample. 
\end{proof}

Define $\mathrm{GL}_\Lambda^{bd}(\ZZ_d)$ to be the collection of invertible band-diagonal matrices indexed by $\Lambda$.  There is an injective homomorphism $$D: \Theta\mapsto\begin{pmatrix}
    \Theta & 0\\
    0 & (\Theta^T)^{-1}
\end{pmatrix}$$ embedding $\mathrm{GL}_\Lambda^{bd}(\ZZ_d)$ in $\mathrm{Sp}_\Lambda^{bd}(\ZZ_d)$. 
The image of $D$ is precisely the subgroup of separated elements $\Sigma=\{\theta\in \mathrm{Sp}_\Lambda^{bd}(\ZZ_d): \Theta^{12}=\Theta^{21}=0\}.$ Define $E_{\Lambda}^{bd}(\ZZ_d)\subset \mathrm{GL}_\Lambda^{bd}(\ZZ_d)$ to be the subgroup generated by block diagonal matrices and $P_{\Lambda}^{bd}(\ZZ_d)\subset \mathrm{GL}_\Lambda^{bd}(\ZZ_d)$ the permutation matrices. The classification group of separeted QCAs is isomorphic to $\mathrm{GL}_\Lambda^{bd}(\ZZ_d)/\langle E_{\Lambda}^{bd}(\ZZ_d), P_{\Lambda}^{bd}(\ZZ_d)\rangle$ up to stabilization.
\subsection{Separated QCAs and K-theory}

The following construction is taken from \cite{pedersen1989k, pedersen2006nonconnective}. It provides a natural categorical setting for Clifford QCA \cite{yang2025categorifying}.

\begin{defn}
    
Let $X$ be a metric space and $\mathcal{A}$ a filtered additive category. We then define the filtered category $\mathcal{C}_X(\mathcal{A})$ as follows:
\begin{enumerate}
    \item An object $A$ of $\mathcal{C}_X(\mathcal{A})$ is a collection of objects $A(x)$ of $\mathcal{A}$, one for each $x \in X$, satisfying the condition that for each ball $B \subset X$, $A(x) \ne 0$ for only finitely many $x \in B$.
    \item A morphism $\phi : A \rightarrow B$ is a collection of morphisms $\phi^x_y : A(x) \rightarrow B(y)$ in $\mathcal{A}$ such that there exists $r$ depending only on $\phi$ so that
    \begin{enumerate}
        \item $\phi^x_y = 0$ for $d(x, y) > r$
        \item all $\phi^x_y$ are in $F_r \mathrm{Hom}(A(x), B(y))$
    \end{enumerate}
    (We then say that $\phi$ has filtration degree $\le r$.)
\end{enumerate}

Composition of $\phi : A \rightarrow B$ with $\psi : B \rightarrow C$ is given by $(\psi \phi)^x_z = \sum_{y \in X} \psi^y_z \phi^x_y$. Notice that the sum makes sense because the category is additive and because the sum will always be finite. 
\end{defn}

The categories $\mathcal{C}_X(\mathcal{A})$ and $\mathcal{C}_X(\mathcal{A})$ are equivalent for coarsely equivalent metric spaces $X, Y$. See \cite{pedersen1989k} for a proof. 
\begin{defn}
Let $\mathcal{A}$ be an additive category. The \emph{algebraic \(K_1\) group} of $\mathcal{A}$, denoted \(K_1(\mathcal{A})\), is defined as follows:

\medskip

We first consider the category $\operatorname{Aut}(\mathcal{A})$, whose objects are pairs \((A, \phi)\) where \(A\) is an object of \(\mathcal{A}\) and \(\phi : A \to A\) is an automorphism in \(\mathcal{A}\). A morphism \(f : (A, \alpha) \to (B, \beta)\) in \(\operatorname{Aut}(\mathcal{A})\) is a morphism \(f : A \to B\) in \(\mathcal{A}\) such that
\[
f \circ \alpha = \beta \circ f.
\]

Then \(K_1(\mathcal{A})\) is defined as the abelian group obtained from the group completion of the monoid of isomorphism classes in \(\operatorname{Aut}(\mathcal{A})\), modulo the relations:
\begin{enumerate}
    \item $[A, \alpha]+ [A, \alpha']= [A, \alpha\circ \alpha']$
    \item $[A, \alpha]+[C, \gamma]= [B, \beta]$ whenever there is a diagram with exact rows
    \[
\begin{array}{ccccccccc}
0 & \longrightarrow & A & \longrightarrow & B & \longrightarrow & C & \longrightarrow & 0 \\
  &                 & ~~\downarrow{\scriptstyle \alpha} 
  &                 & ~~\downarrow{\scriptstyle \beta} 
  &                 & ~~\downarrow{\scriptstyle \gamma} 
  &                 &   \\
0 & \longrightarrow & A & \longrightarrow & B & \longrightarrow & C & \longrightarrow & 0
\end{array}
\]

\end{enumerate}

\end{defn}

In particular, $K_1(\mathcal{C}_X(\mathcal A))$ gives a description of the homology theory associated with any algebraic K-theory. The following theorem \cite{pedersen1984k_i} will be useful.  
\begin{theorem}
    Let $\mathcal{A}$ be the category of finitely generated free $R$-modules for a ring $R$. $$K_1(\mathcal{C}_{\ZZ^{D+1}}(\mathcal A))\cong K_{-D}(R)$$ the negative algebraic $K$-theory of $R$.
\end{theorem}

Let $\mathcal{A}$ be the category of finitely generated free $\ZZ_d$-modules and $X=\Lambda$. The definition of $\mathrm{GL}_\Lambda^{bd}(\ZZ_d)$ coincides with \(\operatorname{Aut}(\mathcal{A})\). Moreover, $\mathrm{GL}_\Lambda^{bd}(\ZZ_d)/E_{\Lambda}^{bd}(\ZZ_d)$ with stabilization is isomorphic to $K_1(\mathcal{C}_\Lambda(\mathcal A))$. Therefore, the classification of separated QCA is a quotient of $K_1(\mathcal{C}_\Lambda(\mathcal A))$. For QCAs on Euclidean spaces, the theorem above implies 
$$K_1(\mathcal{C}_{\ZZ^{D}}(\mathcal A))\cong K_{-D+1}(\ZZ_d)=\begin{cases}
\ZZ^r, \text{ for } D=1\\
0, \text{ for } D>1
\end{cases}$$ where $r$ is the number of distinct prime factors of $d$. In dimension 1, nontrivial $K_1$-classes correspond to translations.
{\change
Since separated QCAs are classified by the corresponding negative $K$-groups, the above result immediately implies the following corollary.
\begin{corollary}
    On any lattice coarsely equivalent to a Euclidean space, all separated QCAs are trivial. Consequently, Clifford QCAs with the same stabilizer group are always equivalent on such lattices according to Theorem~\ref{thm:same-stabilizer}.
\label{corollary: separated QCAs trivial}
\end{corollary}
}

In the translation-invariant case, the same conclusion follows from Suslin’s stability theorem (see Lemma IV.10 in Ref.~\cite{haah_QCA_23}).
On these lattices, Clifford QCAs with the same stabilizer group are always equivalent. On the other hand, there exist metric spaces $X$ on which the $K_1$ group of $\mathcal C_X(\ZZ_d)$  is nontrivial. Separated QCAs on such $X$ could be nontrivial. Examples include open cones
$$X:=O(Y)=\{tx\in \R^{n+1}: t\geq 0 \text{ and } x\in Y\subset S^n\}$$
over certain subspaces $Y$ of the unit $n$-sphere $S^n\subset \R^{n+1}$ with nontrivial $K$-theory homology groups~\cite{pedersen1989k}. This should motivate further study of QCAs on lattices of the form 
$$\{nx\in \R^{n+1}: n\in \NN \text{ and } x\in \Gamma\subset S^n\},$$ where $\Gamma$ is some discrete lattice in the unit $n$-sphere.

\subsection{Local flippability}
Lastly, we record a useful result in the translation invariant setup. 
Consider a Pauli translation-invariant stabilizer code over $\ZZ_{p^r}$ defined by a module homomorphism (see \cite{haah_module_13, ruba2024homological}) $R^t\xrightarrow[]{\sigma} P=R^{2q}$, where $L:=$ image of $\sigma$ satisfies the topological order (aka Lagrangian) condition: $$ R^t\xrightarrow[]{\sigma} P\xrightarrow[]{\eps} R^t$$ is exact. For the purpose of Clifford QCA construction, we impose the additional conditions that $t=q$ and $\sigma$ injective. In particular, $L$ is free of rank $q$. 
\begin{theorem}
    The following are equivalent:
    \begin{enumerate}
        \item Separators defined by $\sigma$ are locally flippable.
        \item The short exact sequence $0\rightarrow L \rightarrow P \rightarrow P/L\rightarrow 0$ splits.
        \item $P/L$ is projective
        \item $P/L$ is free
        \item The module of topological point excitation\footnote{This algebraic invariant was first defined this way in~\cite{ruba2024homological} but appeared already in Ref.~\cite{haah_module_13}.} $\Ext^1(P/L, R)$ vanishes.
    \end{enumerate}
\end{theorem}
    \begin{proof}
        $1 \implies 2:$ the local separators provides a basis for $P/L$. \\
        $2 \implies 3:$ as the short exact sequence splits, $P/L$ is a direct summand of free module $P$. \\
        $3 \implies 4:$ by Quillen-Suslin theorem, every projective module over $R/pR=\ZZ_{p}[x_1^{\pm}, \dots, x_D^{\pm}]$ is free. Assume $M$ is a projective module over $R=\ZZ_{p^r}[x_1^{\pm}, \dots, x_D^{\pm}]$, then $M/pM$ is a free module over $R/pR$. Lift a basis of $M/pM$ to $\{m_1, \dots, m_k\}\subset M$ and denote the submodule it generates by $M'\subset M$. Notice $M/M'\subset p(M/M')$, by Nakayama's lemma $M/ M'=0$. In other words, we have a short exact sequence $0\rightarrow K\rightarrow R^k\xrightarrow{\sigma} M\rightarrow 0$, where $\sigma$ is generated by $\{m_1, \dots, m_k\}$ and $K=\ker \sigma$. By the projectivity of $M$, $R^k\cong K \oplus M$. Moreover, $K\cong R^k/M\subset pK$. Therefore, by Nakayama's lemma again, $K=0$.\\
        $4 \implies 5:$ follows from the definition of $\Ext$. \\
        $5 \implies 1:$ As $L$ is free, $\Ext^1(P/L, R)=0$ implies $\Ext^1(P/L, L)=0$ as well. In other words, all extensions of $P/L$ by $L$ is trivial and $P\cong L\oplus P/L$. Since $L$ is Lagrangian, there exists a standard symplectic basis for this direct sum decomposition of $P$.
    \end{proof}

\section{Discussions}

Our work establishes a unified framework for constructing $\mathbb{Z}_2$ and $\mathbb{Z}_p$ Clifford QCAs from both TQFTs and ISAs, clarifying their orders, periodicities in dimensions, and realizations on arbitrary cellulations (for $\mathbb{Z}_2$ QCAs).
{\change Several natural directions for future research emerge from this framework.

A particularly important direction is to move beyond the Clifford regime. Notable examples, such as the chiral semion Walker--Wang model, already indicate the existence of intrinsically non-Clifford QCAs for qubits~\cite{Shirley2022QCA}. Developing systematic methods to construct and classify such models would extend our framework to a much broader class of topological phases.

As a concrete example, the semion QCA is associated with the topological action
\begin{equation}
    \frac{1}{4}\,\bigl(B_2 \cup B_2 + B_2 \cup_1 \delta B_2\bigr)
    \in H^4\bigl(K(\mathbb{Z}_2,2), \mathbb{R}/\mathbb{Z}\bigr)~,
\end{equation}
whose surface anyon theory is $\{1,s\}$, where $s$ is a semion obeying $\mathbb{Z}_2$ fusion rules and carrying topological spin $\theta(s)=i$. This theory has chiral central charge $c_- = 1$. The fourth power of the semion QCA is equivalent to the $3$-fermion QCA (with $c_- = 4$), since their surface anyon theories lie in the same Witt class. Explicitly, starting from
\begin{equation}
    \mathcal{A}
    = \{1,s_1\} \times \{1,s_2\} \times \{1,s_3\} \times \{1,s_4\},
\end{equation}
condensing the boson $b = s_1 s_2 s_3 s_4$ yields
\begin{equation}
    \mathcal{A}' = \{1,\, s_1 s_2,\, s_1 s_3,\, s_1 s_4\},
\end{equation}
which is precisely the $3$-fermion anyon theory.

Generalizing this construction to higher dimensions may lead to new families of non-Clifford QCAs. For instance, we expect that the topological action
\begin{equation}
    \frac{1}{4}\,\bigl(B_{2l} \cup B_{2l} + B_{2l} \cup_1 \delta B_{2l}\bigr)
    \in H^{4l}\bigl(K(\mathbb{Z}_2,2l), \mathbb{R}/\mathbb{Z}\bigr)~,
\end{equation}
gives rise to semion-type QCAs in $(4l{-}1)$ spatial dimensions. It would also be interesting to construct a non-Clifford ISA corresponding to the chiral semion theory and to generalize such constructions to higher dimensions.

Another promising perspective is to derive QCAs from cobordism classifications. The $3$-fermion Walker--Wang model corresponds to the $3{+}1$D $w_2^2$ TQFT, and its $4{+}1$D generalization, associated with the $w_2 w_3$ TQFT, was constructed in Ref.~\cite{chen2023exactly} via the topological action
\begin{equation}
    \frac{1}{2} \int A_3 \cup_1 A_3
    + A_3 \cup B_2
    + B_2 \cup (B_2 \cup_1 B_2)~,
\end{equation}
which closely parallels the $3$-fermion Walker--Wang action in Eq.~\eqref{eq: 3f WW S}. It is natural to conjecture that all nontrivial cobordism classes give rise to nontrivial QCAs~\cite{chen2023exactly, fidkowski2024qca}.

Finally, we note that whenever a topological action involves coefficients beyond $1/2$ or contains terms higher than quadratic, the resulting QCA is necessarily non-Clifford in prime-dimensional qudits, since the corresponding Hamiltonian can no longer be written as a Pauli stabilizer and must involve at least $S$ or $CZ$ gates. Therefore, exploring cocycles for higher-form symmetries provides a systematic route to constructing new non-Clifford QCAs. However, extracting explicit separators and flippers from such Hamiltonians is highly nontrivial and remains an important open problem.
}

Next, a more explicit connection between the TQFT and ISA constructions remains to be established. Although we have argued that both descriptions capture the same phases, a direct finite-depth circuit relating them has not yet been constructed in general. Constructing such a circuit would provide an operator-algebraic proof of their equivalence and elucidate the role of locality-preserving unitaries in pumping chirality~\cite{Fidkowski2024Pumping}.

Another natural direction is to explore QCAs that act on Hilbert spaces beyond simple tensor products of local qudits. Examples include the Kramers–Wannier duality, symmetry-protected QCAs, and more general categorical dualities in spin systems~\cite{ma2024QCA,jones2024QCA}. Extending our framework to these cases would connect QCA classification to recent work on lattice anomalies and higher-form symmetries, thereby enriching the interplay between algebraic topology, category theory, and condensed matter physics.  

Finally, several structural generalizations remain open. While we demonstrated that $\mathbb{Z}_2$ QCAs can be realized on arbitrary cellulations, extending $\mathbb{Z}_p$ QCAs systematically to triangulations and general cell complexes would further clarify their robustness and their relation to generalized cohomology.
One key property of hypercubic lattices is that {\change $\bc_{i} \cup \bc_{d-i}$} induces a one-to-one correspondence between cells. It would be interesting to identify other lattices with similar properties, or to exploit Poincaré duality, as in our $\mathbb{Z}_2$ construction. For example, vertex-face pairing seems to be important in the construction of Chern-Simons theory on the lattice~\cite{ELIEZER92anyon,ELIEZER92,Sun15,dlF24}.

More broadly, our results suggest a deep correspondence between TQFT classifications and QCA classifications. Making this link precise, both in the Clifford and non-Clifford regimes, would provide a unified understanding of the algebraic invariants underlying QCAs and their role in higher-dimensional phases of matter. A systematic study along these lines would not only extend the mathematical classification of QCAs in higher dimensions, but also clarify their relevance to fault-tolerant quantum computation, Floquet dynamics, generalized symmetries, and lattice anomalies.

\section*{Acknowledgment}

We thank Lukasz Fidkowski, Jeongwan Haah, Po-Shen Hsin, Wenjie Ji, Anton Kapustin, Ryohei Kobayashi, Ruochen Ma, Shmuel Weinberger, and Beni Yoshida for valuable and insightful discussions.

Y.-A.C. is supported by the National Natural Science Foundation of China (Grant No.~12474491), and the Fundamental Research Funds for the Central Universities, Peking University. N.T. was supported by the Walter Burke Institute for Theoretical Physics at Caltech. Z.W. is supported by the Simons Foundation, award number 651438. B.Y. gratefully acknowledges support from Harvard CMSA and the Simons Foundation through Simons
Collaboration on Global Categorical Symmetries.

\onecolumngrid
\appendix

\section{Review of the Laurent polynomial formalism}\label{app: Review of the Laurent polynomial formalism}

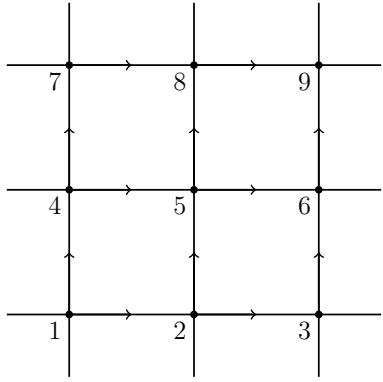
\begin{figure}[htb]
\centering
\resizebox{5cm}{!}{%
\begin{tikzpicture}
\draw[thick] (-3,0) -- (3,0);\draw[thick] (-3,-2) -- (3,-2);\draw[thick] (-3,2) -- (3,2);
\draw[thick] (0,-3) -- (0,3);\draw[thick] (-2,-3) -- (-2,3);\draw[thick] (2,-3) -- (2,3);
\draw[->] [thick](0,0) -- (1,0);\draw[->][thick] (0,2) -- (1,2);\draw[->][thick] (0,-2) -- (1,-2);
\draw[->][thick] (0,0) -- (0,1);\draw[->][thick] (2,0) -- (2,1);\draw[->][thick](-2,0) -- (-2,1);
\draw[->][thick] (-2,0) -- (-1,0);\draw[->][thick] (-2,2) -- (-1,2);\draw[->][thick](-2,-2) -- (-1,-2);
\draw[->][thick] (-2,-2) -- (-2,-1);\draw[->][thick] (0,-2) -- (0,-1);\draw[->] [thick](2,-2) -- (2,-1);
\filldraw [black] (-2,-2) circle (1.5pt) node[anchor=north east] {\large 1};
\filldraw [black] (0,-2) circle (1.5pt) node[anchor=north east] {\large 2};
\filldraw [black] (2,-2) circle (1.5pt) node[anchor=north east] {\large 3};
\filldraw [black] (-2,-0) circle (1.5pt) node[anchor=north east] {\large 4};
\filldraw [black] (0,0) circle (1.5pt) node[anchor=north east] {\large 5};
\filldraw [black] (2,0) circle (1.5pt) node[anchor=north east] {\large 6};
\filldraw [black] (-2,2) circle (1.5pt) node[anchor=north east] {\large 7};
\filldraw [black] (0,2) circle (1.5pt) node[anchor=north east] {\large 8};
\filldraw [black] (2,2) circle (1.5pt) node[anchor=north east] {\large 9};
\end{tikzpicture}}
\caption{A qudit is placed on each edge, with generalized Pauli operators $X_e$ and $Z_e$ acting on it.}
\label{fig:square}
\end{figure}

This appendix reviews the Laurent polynomial representation and its application to translation-invariant stabilizer codes. The formalism was first introduced in Refs.~\cite{Schlingemann2008structure, Gutschow2010Clifford, haah_module_13, haah2016algebraic}. In this work, we adopt the conventions and notations of Refs.~\cite{liang2023extracting, liang2024operator, liang2025generalized, liang2025planar}.  
We begin with a general $\ZZ_d$ qudit system, where the $d \times d$ generalized Pauli matrices are defined as
\begin{eqs}
    X = 
    \begin{bmatrix}
    0 & 0 & \cdots & 0 & 1 \\
    1 & 0 & \cdots & 0 & 0 \\
    0 & 1 & \cdots & 0 & 0 \\
    \vdots & \vdots & \ddots & \vdots & \vdots \\
    0 & 0 & \cdots & 1 & 0
    \end{bmatrix}, \quad
    Z = 
    \begin{bmatrix}
    1 & 0 & 0 & \cdots & 0 \\
    0 & \omega & 0 & \cdots & 0 \\
    0 & 0 & \omega^2 & \cdots & 0 \\
    \vdots & \vdots & \vdots & \ddots & \vdots \\
    0 & 0 & 0 & \cdots & \omega^{d-1}
    \end{bmatrix},
\end{eqs}
with $\omega := \exp(2 \pi i / d)$. These matrices satisfy the commutation relation
\begin{eqs}
    Z X = \omega X Z.
\end{eqs}
We develop the formalism in the setting of a two-dimensional square lattice, from which the extension to the three-dimensional cubic lattice and higher dimensions follows naturally.

We consider the case of two $\mathbb{Z}_d$ qudits per unit cell (e.g., one qudit on each edge of a square lattice), generalizable to $k$ qudits per cell. Any Pauli operator---a finite tensor product of Pauli matrices---can be represented, up to an overall phase, as a column vector over the polynomial ring
\begin{equation}
    R = \ZZ_d [x, y, x^{-1}, y^{-1}],
\end{equation}
which includes all polynomials in $x^{\pm1}$, $y^{\pm1}$ with coefficients in $\mathbb{Z}_d$.
We assign column vectors over $\ZZ_d$ to the (generalized) Pauli matrices $X_{12}$, $Z_{12}$, $X_{14}$, and $Z_{14}$, depicted in Fig.~\ref{fig:square}:
\begin{eqs}
    \mX_{12} =
    \left[\begin{array}{c}
        1 \\
        0 \\
        \hline
        0 \\
        0
    \end{array}\right],~
    \mZ_{12} =
    \left[\begin{array}{c}
        0 \\
        0 \\
        \hline
        1 \\
        0
    \end{array}\right], 
    \mX_{14} =
    \left[\begin{array}{c}
        0 \\
        1 \\
        \hline
        0 \\
        0
    \end{array}\right],~
    \mZ_{14} =
    \left[\begin{array}{c}
        0 \\
        0 \\
        \hline
        0 \\
        1
    \end{array}\right],
\end{eqs}
where column vector representations of operators are indicated using curly letters.
The coefficients in these vectors correspond to their powers:
\begin{eqs}
    \mathcal{P} =
    \left[\begin{array}{c}
        i \\
        j \\
        \hline
        k \\
        l
    \end{array}\right]
    ~\Rightarrow~
    \mathcal{P}^m =
    \left[\begin{array}{c}
        m i \\
        m j \\
        \hline
        m k \\
        m l
    \end{array}\right]
    \forall~m \in \ZZ_d.
\end{eqs}
The translation of operators is achieved using polynomials of $x$ and $y$ to denote translations in the $x$ and $y$ directions, respectively. To illustrate, translating the operator on edge $e_{12}$ to edge $e_{78}$ or to edge $e_{58}$ involves multiplying the column vector of the operator by $y^2$ or $xy$, respectively:
\begin{eqs}
    \mZ_{78}= y^2
    \mZ_{12}
    =
    \left[\begin{array}{c}
        0 \\
        0 \\
        \hline
        y^2 \\
        0
    \end{array}\right],~
    \mX_{58}= x y
    \mX_{14}
    =
    \left[\begin{array}{c}
        0 \\
        xy \\
        \hline
        0 \\
        0
    \end{array}\right].
\end{eqs}
A general Pauli operator can be expressed as
\begin{equation}
    P = \eta X^{a_1}_{e_1} X^{a_2}_{e_2} \cdots X^{a_n}_{e_n} Z^{b_1}_{e'_1} Z^{b_2}_{e'_2} \cdots Z^{b_m}_{e'_m},
\end{equation}
where $\eta$ represents a root of unity of order $2d$. After dropping the overall phase $\eta$, the corresponding column vector for this operator is a linear combination of individual Pauli matrices, expressed as
\begin{equation}
\begin{aligned}
    \mathcal{P} =&  a_1 \mX_{e_1} + a_2 \mX_{e_2} + \cdots + a_n \mX_{e_n}  + b_1 \mZ_{e'_1} + b_2 \mZ_{e'_2} + \cdots + b_m \mZ_{e'_m}~.
\end{aligned}
\end{equation}
More examples are included in Fig.~\ref{fig:example_poly}.

Next, we introduce the {\bf antipode map} that is a $\ZZ_d$-linear map from $R$ to $R$ defined by
\begin{eqs}
    x^a y^b \rightarrow \overline{x^a y^b}:=x^{-a} y^{-b}.
\end{eqs}
To determine whether two Pauli operators represented by vectors $v_1$ and $v_2$ commute or anti-commute, we define the dot product as
\begin{eqs}
    v_1 \cdot v_2 = \overline{v}_1^{T} \Lambda v_2,
\end{eqs}
where $T$ is the transpose operation on a matrix and
\begin{eqs}
    \Lambda=
    \left[\begin{array}{cc | cc}
        0 & 0 & 1 & 0 \\
        0 & 0 & 0 & 1 \\
        \hline
        -1 & 0 & 0 & 0 \\
        0 & -1 & 0 & 0 \\
    \end{array}\right]
\end{eqs}
is the matrix representation of the standard {\bf symplectic bilinear form}. For simplicity, we denote $\overline{(\cdots)}^T$ as $(\cdots)^\dagger$.

A translation-invariant stabilizer code corresponds to an $R$-submodule $\sigma$ such that
\begin{eqs}
    v_1 \cdot v_2 = v_1^\dagger \Lambda v_2 = 0, \quad \forall\, v_1, v_2 \in \sigma,
\end{eqs}
called the \textbf{stabilizer module}. 
The Hamiltonian could have two (or more) terms per square to have a unique ground state on a simply-connected manifold, denoted as $H = -\sum_{\text{cells}} (S_1 + S_2)$.
For example, the trivial phase $H_0 = - \sum_e X_e$ is
\begin{eqs}
    \mS_1 = \left[\begin{array}{c}
        1 \\
        0 \\
        \hline
        0 \\
        0
    \end{array}\right], \quad
    \mS_2 = \left[\begin{array}{c}
        0 \\
        1 \\
        \hline
        0 \\
        0
    \end{array}\right],
\label{eq: trivial H0 SA SB}
\end{eqs}
and the standard $\ZZ_d$ toric code Hamiltonian 
\begin{eqs}
    H_{\text{TC}} = - \sum_v \vcenter{\hbox{\includegraphics[scale=.22]{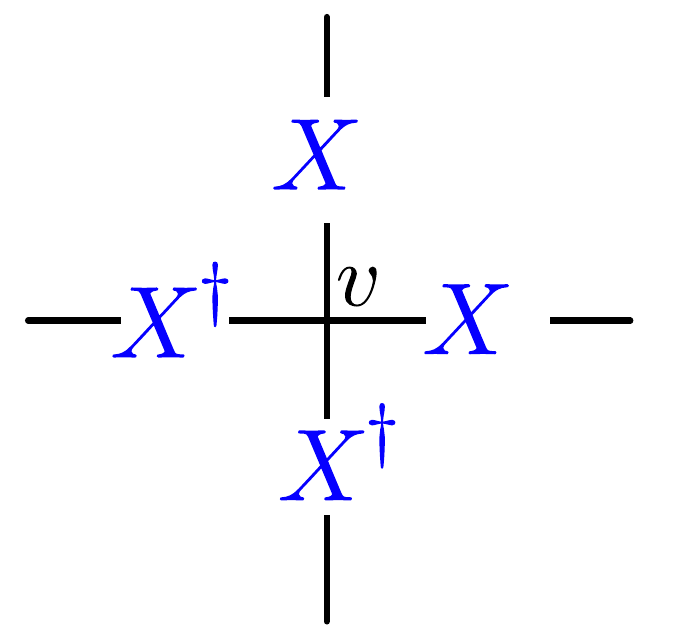}}} \, - \sum_p \vcenter{\hbox{\includegraphics[scale=.23]{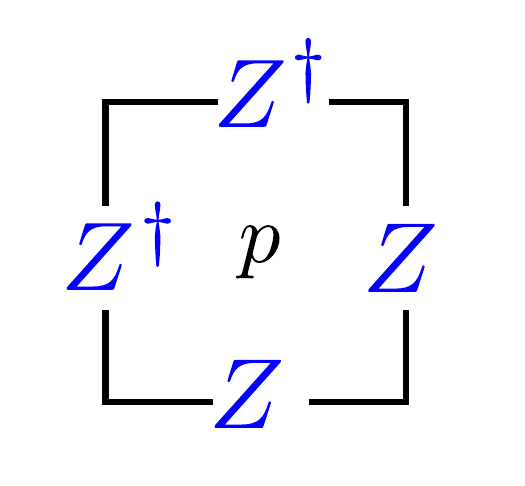}}},
\label{eq: toric code Hamiltonian}
\end{eqs}
corresponds to
\begin{eqs}
    \mS_1 = \left[\begin{array}{c}
        1-\bx \\
        1- \by \\
        \hline
        0 \\
        0
    \end{array}\right], \quad
    \mS_2 = \left[\begin{array}{c}
        0 \\
        0 \\
        \hline
        1-y \\
        -1+x
    \end{array}\right].
\label{eq: standard toric code SA SB}
\end{eqs}

\section{Review of higher cup products on hypercubes}\label{app:highercuphypercube}

\begin{figure}[b]
    \centering
    \subfigure[2d square]{\includegraphics[scale=0.3]{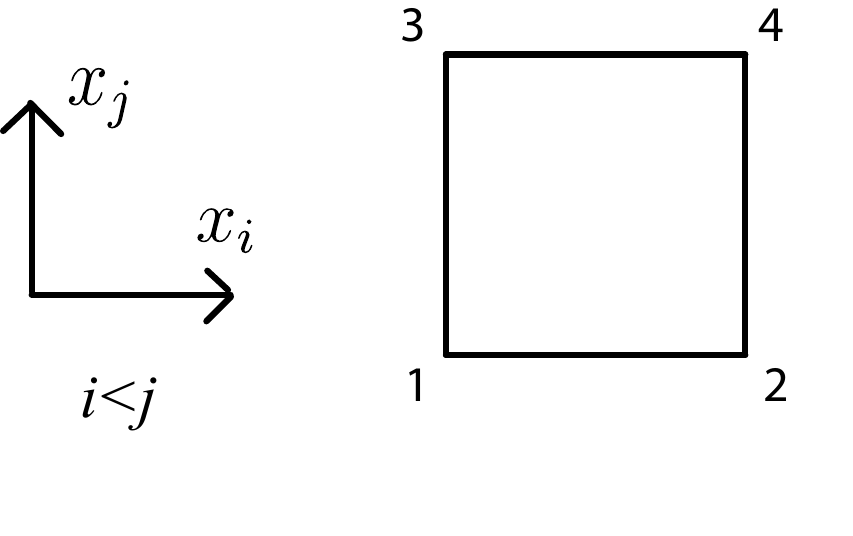}\label{fig: 2d square 1234}}
    \hspace{10ex}
    \subfigure[3d cube]{\raisebox{0ex}{\includegraphics[scale=0.30]{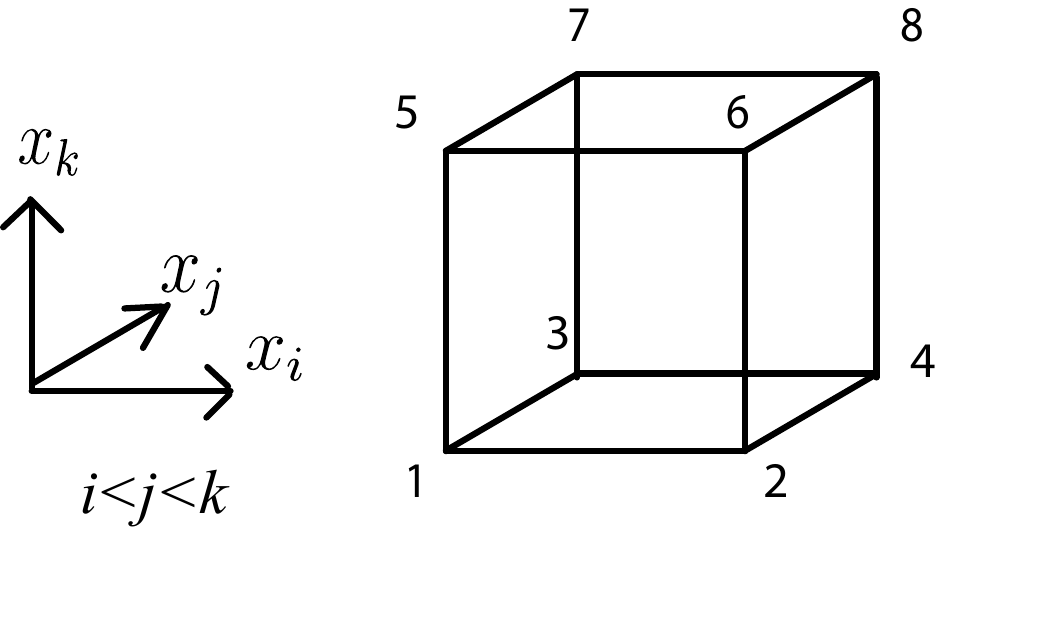}\label{fig: 3d cube 12345678}}}
    \caption{(a) the conventional orientation of a 2d square
    (b) the conventional orientation of a 3d cube
    }
    \label{fig: 2d square and 3d cube}
\end{figure}

The notations and definitions in this section are mostly adapted from Ref.~\cite{Chen2023HigherCup}, but note that the signs of boundary and coboundary operators with integer coefficients presented previously have typos. Here, we provide a corrected sign convention and explicitly list all definitions of the higher cup products used for hypercubes in this paper. 

First, let us establish the notation for cochains on a hypercube. We denote the $d$-dimensional hypercube by $\msquare_d$. As a subset of $\R^d$, we assign coordinates to $\msquare_d$ via
\begin{equation}
\msquare_d = \{(x_1,\cdots,x_d) \in \R^d \mid x_i \in [0,1]\},
\end{equation}
which will be used throughout our discussion of cochain operations. The ordering of the coordinates is shown in Fig.~\ref{fig: 2d square and 3d cube}.

\subsection{The Cells}

We denote the set of $p$-cells of $\msquare_d$ by $F_p(\msquare_d)$. For instance, the vertices of the hypercube are the $0$-cells:
\begin{equation}
F_0(\msquare_d) = \{(\sigma_1, \cdots, \sigma_d) \mid \sigma_i = 0, 1\}.
\end{equation}
Equivalently, each label \((\sigma_1 \cdots \sigma_d)\) can be interpreted as the vertex’s coordinate in $\R^d$.

Next, consider the $p$-cells in $F_p(\msquare_d)$. A $p$-cell is spanned by $p$ directions \(\{i_1 \cdots i_p\}\). However, identifying just the directions is not sufficient to specify a particular cell because there are $2^{d-p}$ possible choices for the remaining coordinates, each of which can be \(0\) or \(1\). Concretely, if we let \(\{\ihat_1 \cdots \ihat_{d-p}\}\) denote the complementary set of directions \(\{1,2,\dots,d\} \backslash \{i_1,i_2,\dots,i_p\}\), then specifying a $p$-cell also requires assigning a value of \(0\) or \(1\) to each of those \(\ihat\)-coordinates.

To organize these labels, we represent each cell in terms of a $d$-tuple \((z_1 \cdots z_d)\). We assign \(z_j=\bullet\) if \(j \in \{i_1 \cdots i_p\}\), indicating that the coordinate in direction \(j\) can vary in \([0,1]\). Otherwise, \(z_j\) is fixed to be either 0 or 1, indicating that the coordinate in that direction does not vary. A cell described by \((z_1 \cdots z_d)\) is a $p$-cell precisely when exactly $p$ of the entries are \(\bullet\). We denote such a cell by \(P_{(z_1 \cdots z_d)}\). As a subset of \(\R^d\), we may write
\begin{equation}
P_{(z_1 \cdots z_d)} := \{(x_1,\cdots,x_d) \in \R^d \mid 
x_i \in [0,1] \text{ if } z_i = \bullet, \text{ or }  x_{\ihat} = z_{\ihat} \text{ if } z_{\ihat} \in \{0,1\}\}.
\end{equation}
For example, the cells of \(\msquare_2\) can be listed as
\begin{equation}
    \begin{split}
        F_0(\msquare_2) &= \{(0,0),(0,1),(1,0),(1,1)\}, \\
        F_1(\msquare_2) &= \{(\bullet,0),(\bullet,1),(0,\bullet),(1,\bullet)\}, \\
        F_2(\msquare_2) &= \{(\bullet,\bullet)\}.
    \end{split}
\end{equation}
Similarly, for a 3-dimensional cube \(\msquare_3\), the cells are:
\begin{equation}
\begin{split}
    F_0(\msquare_3) =&
    \{(0,0,0),\, (0,0,1),\, (0,1,0),\, (0,1,1),\, (1,0,0),\, (1,0,1),\, (1,1,0),\, (1,1,1)\},\\
    F_1(\msquare_3) =&
    \{(\bullet,0,0),\, (\bullet,0,1),\, (\bullet,1,0),\, (\bullet,1,1),\,
      (0,\bullet,0),\, (0,\bullet,1),\, (1,\bullet,0),\, (1,\bullet,1),\\
    &~~
      (0,0,\bullet),\, (0,1,\bullet),\, (1,0,\bullet),\, (1,1,\bullet)\},\\
    F_2(\msquare_3) =&
    \{(\bullet,\bullet,0),\, (\bullet,\bullet,1),\, (\bullet,0,\bullet),\, (\bullet,1,\bullet),\, (0,\bullet,\bullet),\, (1,\bullet,\bullet)\},\\
    F_3(\msquare_3) =& 
    \{(\bullet,\bullet,\bullet)\},
\end{split}
\end{equation}
which are shown in Fig.~\ref{fig: cells of a cube}.

\begin{figure}[htbp]
    \centering
    \raisebox{0.1em}{\subfigure[0d cell(vertices)]{\includegraphics[scale=0.4]{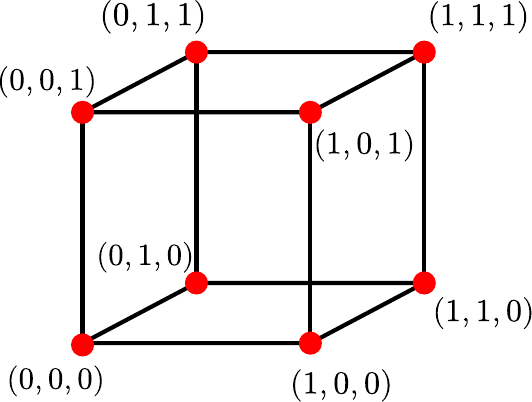}\label{fig: 0d cell(vertices)}}}
    \quad
    \subfigure[1d cell(edges)]{\includegraphics[scale=0.43]{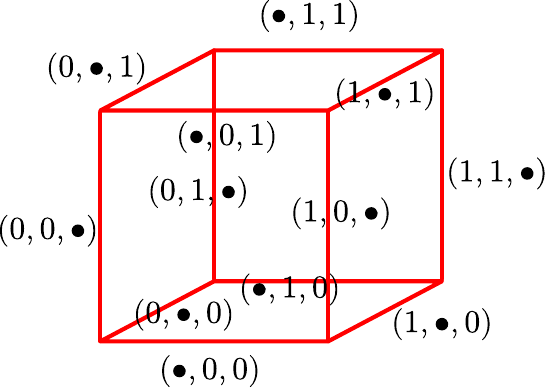}\label{fig: 1d cell(edges)}}
    \quad
\subfigure[2d cell(faces)]{\includegraphics[scale=0.43]{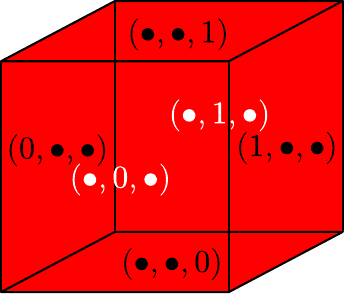}\label{fig: 2d cell(faces)}}
    \caption{(a), (b), and (c) corresponds to diffenrent dimensional cells $F_0$, $F_1$, and $F_2$}
    \label{fig: cells of a cube}
\end{figure}

\subsection{Boundary of a Cell}

We now describe the boundary of a $p$-cell \(P_{(z_1 \cdots z_d)}\). Suppose \(z_j = \bullet\) exactly for \(j \in \{i_1 \cdots i_p\}\). The boundary \(\partial P_{(z_1 \cdots z_d)}\) is formed by replacing precisely one of the \(\bullet\)’s with either 0 or 1. Concretely,
\begin{equation}
    \partial P_{(z_1 \cdots z_d)} 
    = \Bigl\{ P_{(z_1 \cdots z_{j-1}, \tilde{z}_j, z_{j+1} \cdots z_d)} 
    \,\Big|\,
    j \in \{i_1 \cdots i_p\},  
    \tilde{z}_j \in \{0,1\} 
    \Bigr\}.
\end{equation}
For example, using the notation \((z_1 \cdots z_d)\) \(\leftrightarrow\) \(P_{(z_1 \cdots z_d)}\),
\begin{equation}
    \partial (\bullet,\bullet,0) 
    = \Bigl\{ (\bullet,0,0), (\bullet,1,0), (0,\bullet,0), (1,\bullet,0) \Bigr\}.
\end{equation}
When working with integer coefficients $\ZZ$ instead of $\ZZ_2$, the boundary acquires signs. Specifically,\footnote{Note that our sign convention differs from Ref.~\cite{Chen2023HigherCup}, which has a typo in the overall factor.}
\begin{equation}
    \partial (\bullet_1, \cdots, \bullet_d) = \sum_{\substack{\ell=1, \\ a_\ell \in \{0,1\}}}^d (-1)^{\ell+a_\ell} (\bullet, \cdots \underbrace{,a_\ell,}_{\substack{\ell^\text{th} \\ \text{coord.}}} \cdots, \bullet).
\end{equation}
For instance, the boundary of the 1-dimensional edge \(\msquare_1 = (\bullet)\) is
\begin{equation}
    \partial(\bullet)=-(0)+(1),
\end{equation}
the boundary of the 2-dimensional face \(\msquare_2 = (\bullet,\bullet)\) is
\begin{equation}
    \partial(\bullet,\bullet)
    = -(0,\bullet)+(1,\bullet)+(\bullet,0)-(\bullet,1),
\end{equation}
and the boundary of the 3-dimensional cube \(\msquare_3 = (\bullet,\bullet,\bullet)\) is
\begin{equation}
    \partial(\bullet,\bullet,\bullet)
    = -(0,\bullet,\bullet)+(1,\bullet,\bullet)
    +(\bullet,0,\bullet)-(\bullet,1,\bullet)
    -(\bullet,\bullet,0)+(\bullet,\bullet,1).
\end{equation}
In the following parts, we will build upon these definitions to review the higher cup product structures on the cubic lattice.

\subsection{Higher cup products}

Convention for higher cup products on a square is:
\begin{align}
    A_1 \cup B_1(\bullet,\bullet)
    &= A_1(\bullet,0) \, B_1(1,\bullet)
       \;-\;
       A_1(0,\bullet)\, B_1(\bullet,1),
    \\
    A_1 \cup_1 B_2(\bullet,\bullet)
    &= -\,A_1(0,\bullet) \, B_2(\bullet,\bullet)
       \;-\;
       A_1(\bullet,1)\, B_2(\bullet,\bullet),
    \\
    A_2 \cup_1 B_1(\bullet,\bullet)
    &= A_2(\bullet,\bullet)\,B_1(\bullet,0)
       \;+\;
       A_2(\bullet,\bullet)\,B_1(1,\bullet),
    \\
    A_2 \cup_2 B_2(\bullet,\bullet)
    &= A_2(\bullet,\bullet)\,B_2(\bullet,\bullet),
\end{align}
where $(\bullet,\bullet)$ represents the square spanned on the $x_i$ and $x_j$ coordinates, with all other indices that remain constants omitted implicitly. For example, $(\bullet,\bullet)$ could be $(\bullet,\bullet, 0)$, $(1, \bullet,\bullet)$, $(\bullet, 1, \bullet)$, or any square. In the computation, we are free to ignore all other coordinates except $x_i$ and $x_j$.

\begin{align}
    A_1\cup B_1(\msquare_{1234})&=A_1(12)B_1(24)-A_1(13)B_1(34),\\
    A_1\cup B_2(\msquare_{1234})&=-A_1(13)B_2(\msquare_{1234})-A_1(34)B_2(\msquare_{1234}),\\
    A_2\cup B_1(\msquare_{1234})&=A_2(\msquare_{1234})B_1(12)+A_2(\msquare_{1234})B_1(24).
\end{align}

\begin{align}
    A_2\cup B_1(\bullet,\bullet,\bullet)
    &= A_2(0,\bullet,\bullet)\,B_1(\bullet,1,1)
       \;-\;
       A_2(\bullet,0,\bullet)\,B_1(1,\bullet,1)
       \;+\;
       A_2(\bullet,\bullet,0)\,B_1(1,1,\bullet),
    \\
    A_1\cup B_2(\bullet,\bullet,\bullet)
    &= A_1(0,0,\bullet)\,B_2(\bullet,\bullet,1)
       \;-\;
       A_1(0,\bullet,0)\,B_2(\bullet,1,\bullet)
       \;+\;
       A_1(\bullet,0,0)\,B_2(1,\bullet,\bullet),
    \\
    \begin{split}
    A_2\cup_1 B_2(\bullet,\bullet,\bullet)
    &= A_2(\bullet,1,\bullet)\,B_2(\bullet,\bullet,0)
       \;-\;
       A_2(\bullet,\bullet,1)\,B_2(\bullet,0,\bullet)
       \;+\;
       A_2(0,\bullet,\bullet)\,B_2(\bullet,\bullet,0)
    \\
    &\quad +\;
       A_2(0,\bullet,\bullet)\,B_2(\bullet,1,\bullet)
       \;-\;
       A_2(\bullet,\bullet,1)\,B_2(1,\bullet,\bullet)
       \;-\;
       A_2(\bullet,0,\bullet)\,B_2(1,\bullet,\bullet),
    \end{split}
    \\
    A_1\cup_1 B_3(\bullet,\bullet,\bullet)
    &= A_1(\bullet,1,1)\,B_3(\bullet,\bullet,\bullet)
       \;+\;
       A_1(0,\bullet,1)\,B_3(\bullet,\bullet,\bullet)
       \;+\;
       A_1(0,0,\bullet)\,B_3(\bullet,\bullet,\bullet),
    \\
    A_3\cup_1 B_1(\bullet,\bullet,\bullet)
    &= A_3(\bullet,\bullet,\bullet)\,B_1(\bullet,0,0)
       \;+\;
       A_3(\bullet,\bullet,\bullet)\,B_1(1,\bullet,0)
       \;+\;
       A_3(\bullet,\bullet,\bullet)\,B_1(1,1,\bullet),
    \\
    A_2\cup_2 B_3(\bullet,\bullet,\bullet)
    &= A_2(\bullet,\bullet,0)\,B_3(\bullet,\bullet,\bullet)
       \;+\;
       A_2(\bullet,1,\bullet)\,B_3(\bullet,\bullet,\bullet)
       \;+\;
       A_2(0,\bullet,\bullet)\,B_3(\bullet,\bullet,\bullet),
    \\
    A_3\cup_2 B_2(\bullet,\bullet,\bullet)
    &= A_3(\bullet,\bullet,\bullet)\,B_2(1,\bullet,\bullet)
       \;+\;
       A_3(\bullet,\bullet,\bullet)\,B_2(\bullet,0,\bullet)
       \;+\;
       A_3(\bullet,\bullet,\bullet)\,B_2(\bullet,\bullet,1),
    \\
    A_3\cup_3 B_3(\bullet,\bullet,\bullet)
    &= A_3(\bullet,\bullet,\bullet)\,B_3(\bullet,\bullet,\bullet).
\end{align}

In the graphic representation, that is (for simplicity, we will denote the entire cube as C, instead of [12345678])

\begin{align}
    A_2\cup B_1(C)&=A_2(\msquare_{1357})B_1(78)-A_2(\msquare_{1256})B_1(68)+A_2(\msquare_{1234})B_1(48),\\
    A_1\cup B_2(C)&=A_1(15)B_2(\msquare_{5678})-A_1(13)B_2(\msquare_{3478})+A_1(12)B_2(\msquare_{2468}),\\
    \begin{split}
        A_2\cup_1B_2(C)&=A_2(\msquare_{3478})B_2(\msquare_{1234})
        -A_2(\msquare_{5678})B_2(\msquare_{1256})
        +A_2(\msquare_{1357})B_2(\msquare_{1234})\\
        &\quad+A_2(\msquare_{1357})B_2(\msquare_{3478})
        -A_2(\msquare_{5678})B_2(\msquare_{2468})
        -A_2(\msquare_{1256})B_2(\msquare_{2468}),
    \end{split}\\
    A_1\cup_1B_3(C)&=\big[ A_1(79)+A_1(57)+A_1(15) \big] B_3(C),\\
    A_3\cup_1B_1(C)&=A_3(C) \big[B_1(12)+B_1(24)+B_1(48) \big],\\
    A_2\cup_2B_3(C)&=\big[A_2(\msquare_{1234})+A_2(\msquare_{3478})+A_2(\msquare_{1357})\big]B_3(C),\\
    A_3\cup_2B_2(C)&=A_3(C)\big[B_2(\msquare_{2468})+B_2(\msquare_{1256})+B_2(\msquare_{5678})\big].
\end{align}

We can also write down some useful cup products in higher dimensions.

For example, in five dimensions, we have
\begin{eqs}
    &A_2\cup B_3(\bullet,\bullet,\bullet,\bullet,\bullet)\\
    =&+A_2(\bullet,\bullet,0,0,0)B_3(1,1,\bullet,\bullet,\bullet)-A_2(\bullet,0,\bullet,0,0)B_3(1,\bullet,1,\bullet,\bullet)+A_2(\bullet,0,0,\bullet,0)B_3(1,\bullet,\bullet,1,\bullet)\\
    &-A_2(\bullet,0,0,0,\bullet)B_3(1,\bullet,\bullet,\bullet,1)+A_2(0,\bullet,\bullet,0,0)B_3(\bullet,1,1,\bullet,\bullet)-A_2(0,\bullet,0,\bullet,0)B_3(\bullet,1,\bullet,1,\bullet)\\
    &+A_2(0,\bullet,0,0,\bullet)B_3(\bullet,1,\bullet,\bullet,1)+A_2(0,0,\bullet,\bullet,0)B_3(\bullet,\bullet,1,1,\bullet)-A_2(0,0,\bullet,0,\bullet)B_3(\bullet,\bullet,1,\bullet,1)\\
    &+A_2(0,0,0,\bullet,\bullet)B_3(\bullet,\bullet,\bullet,1,1)~,
\end{eqs}
\begin{eqs}
    &A_3\cup B_2(\bullet,\bullet,\bullet,\bullet,\bullet)\\
    =&+A_3(\bullet,\bullet,\bullet,0,0)B_2(1,1,1,\bullet,\bullet)-A_3(\bullet,\bullet,0,\bullet,0)B_2(1,1,\bullet,1,\bullet)+A_3(\bullet,0,\bullet,\bullet,0)B_2(1,\bullet,1,1,\bullet)\\
    &-A_3(0,\bullet,\bullet,\bullet,0)B_2(\bullet,1,1,1,\bullet)+A_3(\bullet,\bullet,0,0,\bullet)B_2(1,1,\bullet,\bullet,1)-A_3(\bullet,0,\bullet,0,\bullet)B_2(1,\bullet,1,\bullet,1)\\
    &+A_3(0,\bullet,\bullet,0,\bullet)B_2(\bullet,1,1\,,\bullet,1)+A_3(\bullet,0,0,\bullet,\bullet)B_2(1,\bullet,\bullet,1,1)-A_3(0,\bullet,0,\bullet,\bullet)B_2(\bullet,1,\bullet,1,1)\\
    &+A_3(0,0,\bullet,\bullet,\bullet)B_2(\bullet,\bullet,1,1,1)~,
\end{eqs}
and
\begin{eqs}
&A_3\cup_1B_3(\bullet,\bullet,\bullet,\bullet,\bullet)\\
=&+A_3(\bullet,\bullet,\bullet,1,1)B_3(\bullet,0,0,\bullet,\bullet)+A_3(\bullet,\bullet,\bullet,1,1)B_3(1,\bullet,0,\bullet,\bullet)+A_3(\bullet,\bullet,\bullet,1,1)B_3(1,1,\bullet,\bullet,\bullet)\\&-A_3(\bullet,\bullet,1,\bullet,1)B_3(\bullet,0,\bullet,0,\bullet)-A_3(\bullet,\bullet,1,\bullet,1)B_3(1,\bullet,\bullet,0,\bullet)+A_3(\bullet,\bullet,0,\bullet,1)B_3(1,1,\bullet,\bullet,\bullet)\\&+A_3(\bullet,\bullet,1,1,\bullet)B_3(\bullet,0,\bullet,\bullet,0)+A_3(\bullet,\bullet,1,1,\bullet)B_3(1,\bullet,\bullet,\bullet,0)+A_3(\bullet,\bullet,0,0,\bullet)B_3(1,1,\bullet,\bullet,\bullet)\\&+A_3(\bullet,1,\bullet,\bullet,1)B_3(\bullet,\bullet,0,0,\bullet)-A_3(\bullet,0,\bullet,\bullet,1)B_3(1,\bullet,\bullet,0,\bullet)-A_3(\bullet,0,\bullet,\bullet,1)B_3(1,\bullet,1,\bullet,\bullet)\\&-A_3(\bullet,1,\bullet,1,\bullet)B_3(\bullet,\bullet,0,\bullet,0)+A_3(\bullet,0,\bullet,1,\bullet)B_3(1,\bullet,\bullet,\bullet,0)-A_3(\bullet,0,\bullet,0,\bullet)B_3(1,\bullet,1,\bullet,\bullet)\\&+A_3(\bullet,1,1,\bullet,\bullet)B_3(\bullet,\bullet,\bullet,0,0)+A_3(\bullet,0,0,\bullet,\bullet)B_3(1,\bullet,\bullet,\bullet,0)+A_3(\bullet,0,0,\bullet,\bullet)B_3(1,\bullet,\bullet,1,\bullet)\\&+A_3(0,\bullet,\bullet,\bullet,1)B_3(\bullet,\bullet,0,0,\bullet)+A_3(0,\bullet,\bullet,\bullet,1)B_3(\bullet,1,\bullet,0,\bullet)+A_3(0,\bullet,\bullet,\bullet,1)B_3(\bullet,1,1,\bullet,\bullet)\\&-A_3(0,\bullet,\bullet,1,\bullet)B_3(\bullet,\bullet,0,\bullet,0)-A_3(0,\bullet,\bullet,1,\bullet)B_3(\bullet,1,\bullet,\bullet,0)+A_3(0,\bullet,\bullet,0,\bullet)B_3(\bullet,1,1,\bullet,\bullet)\\&+A_3(0,\bullet,1,\bullet,\bullet)B_3(\bullet,\bullet,\bullet,0,0)-A_3(0,\bullet,0,\bullet,\bullet)B_3(\bullet,1,\bullet,\bullet,0)    -A_3(0,\bullet,0,\bullet,\bullet)B_3(\bullet,1,\bullet,1,\bullet)\\&+A_3(0,0,\bullet,\bullet,\bullet)B_3(\bullet,\bullet,\bullet,0,0)+A_3(0,0,\bullet,\bullet,\bullet)B_3(\bullet,\bullet,1,\bullet,0)+A_3(0,0,\bullet,\bullet,\bullet)B_3(\bullet,\bullet,1,1,\bullet)
\end{eqs}
The general formula for higher cup products on a hypercubic lattice can be found in Ref.~\cite{Chen2023HigherCup}.

\subsection{Polynomial expression for higher cup products}
We may explicitly write down the expressions for the matrices $M_\mathcal I$ using the definitions of the cup products on the cubic lattice. We tabulate some of them used in the text here

\begin{eqs}
     \bM_{e \cup f} &= \begin{pmatrix}
        \bar y \bar z & 0 & 0\\
       0 & -\bar x \bar z & 0\\
      0 & 0 & \bar x \bar y
    \end{pmatrix}~, \quad
    \bM_{f \cup e} = \begin{pmatrix}
       \bar x & 0 & 0\\
       0 & -\bar y & 0\\
      0 & 0 & \bar z
    \end{pmatrix}, \quad 
        \bM_{f' \cup_1 f} = \begin{pmatrix}
    0 & \bar x & \bar x z\\
    -\bar y & 0 & z\\
    -1 & -y &0
    \end{pmatrix}~,\\
    \bM_{e \cup_1 c} &=
    \begin{pmatrix}
        1\\ \bar x \\ \bar x\bar y
    \end{pmatrix}~, \qquad \qquad ~~
         \bM_{f \cup_2 c} =
    \begin{pmatrix}
        \bar x\\ 1 \\ \bar z
    \end{pmatrix}~, \qquad\qquad
    \bM_{c \cup_2 f} =
    \begin{pmatrix}
        1&y&1
    \end{pmatrix}~.
\end{eqs}

{\change
\section{Explicit matrices for the 3-fermion-type QCA}
\label{app: Explicit matrices for the 3-fermion-type QCA}


\subsection{The $3{+}1$D 3-fermion QCA}
\label{app:3FQCA_TQFT_3D}

\begingroup
The matrix elements of the QCA obtained from the TQFT approach, Eq.~\eqref{eq:3FQCA_TQFT}, are
\begin{align}
    \overline{\bX}^A &=  \left(
        \begin{array}{ccc}
         1 & 0 & 0 \\
         0 & 1 & 0 \\
         0 & 0 & 1 \\
         y z+1 & y^2 z+y z & y z^2+y z \\
         x z+z & x y z+y z+z+1 & z^2+z \\
         x+1 & y+1 & x y z+z \\
         0 & 0 & 0 \\
         0 & 0 & 0 \\
         0 & 0 & 0 \\
         x y z+x+y z+1 & x y^2 z+y^2 z+y z+1 & x y z+y z^2+y z+z \\
         x y z+\frac{x}{y}+x+\frac{1}{y} & x y^2 z+x y z+\frac{1}{y}+1 & x y z^2+x y z+x z+\frac{z}{y} \\
         x y+\frac{x}{z}+x+1 & x y^2+y+\frac{1}{z}+1 & x y z+x y+z+1 \\
        \end{array}
        \right) \\
\overline{\bX}^B &= \left(
        \begin{array}{ccc}
         \frac{1}{x y z}+\frac{1}{x} & \frac{y}{x}+\frac{1}{x} & \frac{z}{x}+\frac{1}{x} \\
         \frac{1}{x y}+\frac{1}{y} & \frac{1}{x y z}+\frac{1}{x y}+\frac{1}{x}+1 & \frac{z}{x y}+\frac{1}{x y} \\
         \frac{1}{x y z}+\frac{1}{y z} & \frac{1}{x y z}+\frac{1}{x z} & \frac{1}{x y}+1 \\
         1 & 0 & 0 \\
         0 & 1 & 0 \\
         0 & 0 & 1 \\
         \frac{1}{x y z}+\frac{1}{x}+\frac{1}{y z}+1 & \frac{1}{x y z}+\frac{y}{x}+\frac{1}{x}+y & \frac{1}{x y}+\frac{z}{x}+\frac{1}{x}+1 \\
         \frac{1}{x y^2 z}+\frac{1}{y^2 z}+\frac{1}{y z}+1 & \frac{1}{x y^2 z}+\frac{1}{x y z}+y+1 & \frac{1}{x y^2}+\frac{1}{y}+z+1 \\
         \frac{1}{x y z}+\frac{1}{y z^2}+\frac{1}{y z}+\frac{1}{z} & \frac{1}{x y z^2}+\frac{1}{x y z}+\frac{1}{x z}+\frac{y}{z} & \frac{1}{x y z}+\frac{1}{x y}+\frac{1}{z}+1 \\
         0 & 0 & 0 \\
         0 & 0 & 0 \\
         0 & 0 & 0 \\
        \end{array}
        \right) 
\end{align}
\begin{align}
    \overline{\bZ}^A &= \left(
        \begin{array}{ccc}
         0 & 0 & 0 \\
         0 & 0 & 0 \\
         0 & 0 & 0 \\
         0 & y z+y & y z+z \\
         x z+x & 0 & x z+z \\
         x y+x & x y+y & 0 \\
         1 & 0 & 0 \\
         0 & 1 & 0 \\
         0 & 0 & 1 \\
         x y z+x & x y+y & x y z+y z \\
         \frac{x}{y}+x & x y z+x y+x+1 & x y z+x z \\
         \frac{x}{z}+x & \frac{x y}{z}+x y & x y+1 \\
        \end{array}
        \right) \\
    \overline{\bZ}^B &= \left(
        \begin{array}{ccc}
         0 & \frac{1}{x z}+\frac{1}{x} & \frac{1}{x y}+\frac{1}{x} \\
         \frac{1}{y z}+\frac{1}{y} & 0 & \frac{1}{x y}+\frac{1}{y} \\
         \frac{1}{y z}+\frac{1}{z} & \frac{1}{x z}+\frac{1}{z} & 0 \\
         0 & 0 & 0 \\
         0 & 0 & 0 \\
         0 & 0 & 0 \\
         \frac{1}{y z}+1 & \frac{1}{x z}+\frac{1}{z} & \frac{1}{x}+1 \\
         \frac{1}{y^2 z}+\frac{1}{y z} & \frac{1}{x y z}+\frac{1}{y z}+\frac{1}{z}+1 & \frac{1}{y}+1 \\
         \frac{1}{y z^2}+\frac{1}{y z} & \frac{1}{z^2}+\frac{1}{z} & \frac{1}{x y z}+\frac{1}{z} \\
         1 & 0 & 0 \\
         0 & 1 & 0 \\
         0 & 0 & 1 \\
        \end{array}
        \right)
\end{align}
\endgroup
And, the $\bH$ matrix that is used to construct the QCA from 3F invertible subalgebra is
\begin{equation}
    \bH = \left(
\begin{array}{cccccccc}
 1 & 0 & 0 & 0 & 1 & 0 & 0 & 0 \\
 0 & 1 & 0 & 0 & 0 & 1 & 0 & 0 \\
 0 & 0 & 1 & 0 & 0 & 0 & 1 & 0 \\
 0 & 0 & 0 & 1 & 0 & 0 & 0 & 1 \\
 y+1 & \frac{y}{x}+\frac{1}{x}+y & 0 & y & \frac{1}{y}+1 & 1 & 0 & 0 \\
 1 & \frac{1}{x}+1 & x & 0 & \frac{x}{y}+x+\frac{1}{y} & x+1 & 0 & 0 \\
 0 & 0 & \frac{1}{y}+1 & 1 & 0 & \frac{1}{x} & y+1 & \frac{y}{x}+\frac{1}{x}+y \\
 0 & 0 & \frac{x}{y}+x+\frac{1}{y} & x+1 & \frac{1}{y} & 0 & 1 & \frac{1}{x}+1 \\
\end{array}
\right)~.
\end{equation}
}

\subsection{The $5{+}1$D ``3-fermion'' QCA}\label{app: Polynomial formalism for Clifford QCA}

In five spatial dimensions, we begin by writing down the matrices for the coboundary operator and the relevant (higher) cup products. For convenience, we label the coordinates by $\{a,b,c,d,e\}$.
\begin{align}
\bd_{c,\delta f}&=\left(
\begin{array}{cccccccccc}
 c-1 & 1-b & 0 & 0 & a-1 & 0 & 0 & 0 & 0 & 0 \\
 d-1 & 0 & 1-b & 0 & 0 & a-1 & 0 & 0 & 0 & 0 \\
 e-1 & 0 & 0 & 1-b & 0 & 0 & a-1 & 0 & 0 & 0 \\
 0 & d-1 & 1-c & 0 & 0 & 0 & 0 & a-1 & 0 & 0 \\
 0 & e-1 & 0 & 1-c & 0 & 0 & 0 & 0 & a-1 & 0 \\
 0 & 0 & e-1 & 1-d & 0 & 0 & 0 & 0 & 0 & a-1 \\
 0 & 0 & 0 & 0 & d-1 & 1-c & 0 & b-1 & 0 & 0 \\
 0 & 0 & 0 & 0 & e-1 & 0 & 1-c & 0 & b-1 & 0 \\
 0 & 0 & 0 & 0 & 0 & e-1 & 1-d & 0 & 0 & b-1 \\
 0 & 0 & 0 & 0 & 0 & 0 & 0 & e-1 & 1-d & c-1 \\
\end{array}
\right)\\
\bM_{f\cup c}&=\left(
\begin{array}{cccccccccc}
 0 & 0 & 0 & 0 & 0 & 0 & 0 & 0 & 0 & a b c \\
 0 & 0 & 0 & 0 & 0 & 0 & 0 & 0 & -a b d & 0 \\
 0 & 0 & 0 & 0 & 0 & 0 & 0 & a b e & 0 & 0 \\
 0 & 0 & 0 & 0 & 0 & 0 & a c d & 0 & 0 & 0 \\
 0 & 0 & 0 & 0 & 0 & -a c e & 0 & 0 & 0 & 0 \\
 0 & 0 & 0 & 0 & a d e & 0 & 0 & 0 & 0 & 0 \\
 0 & 0 & 0 & -b c d & 0 & 0 & 0 & 0 & 0 & 0 \\
 0 & 0 & b c e & 0 & 0 & 0 & 0 & 0 & 0 & 0 \\
 0 & -b d e & 0 & 0 & 0 & 0 & 0 & 0 & 0 & 0 \\
 c d e & 0 & 0 & 0 & 0 & 0 & 0 & 0 & 0 & 0 \\
\end{array}
\right)\\
\bM_{c\cup f}&=\left(
\begin{array}{cccccccccc}
 0 & 0 & 0 & 0 & 0 & 0 & 0 & 0 & 0 & a b \\
 0 & 0 & 0 & 0 & 0 & 0 & 0 & 0 & -a c & 0 \\
 0 & 0 & 0 & 0 & 0 & 0 & 0 & a d & 0 & 0 \\
 0 & 0 & 0 & 0 & 0 & 0 & -a e & 0 & 0 & 0 \\
 0 & 0 & 0 & 0 & 0 & b c & 0 & 0 & 0 & 0 \\
 0 & 0 & 0 & 0 & -b d & 0 & 0 & 0 & 0 & 0 \\
 0 & 0 & 0 & b e & 0 & 0 & 0 & 0 & 0 & 0 \\
 0 & 0 & c d & 0 & 0 & 0 & 0 & 0 & 0 & 0 \\
 0 & -c e & 0 & 0 & 0 & 0 & 0 & 0 & 0 & 0 \\
 d e & 0 & 0 & 0 & 0 & 0 & 0 & 0 & 0 & 0 \\
\end{array}
\right)\\
\bM_{c \cup_1 c}&=\left(
\begin{array}{cccccccccc}
 0 & 0 & 0 & 0 & 0 & \frac{1}{d e} & 0 & 0 & \frac{a}{d e} & \frac{a b}{d e} \\
 0 & 0 & 0 & 0 & -\frac{1}{c e} & 0 & 0 & -\frac{a}{c e} & 0 & \frac{a b}{e} \\
 0 & 0 & 0 & \frac{1}{c d} & 0 & 0 & \frac{a}{c d} & 0 & 0 & a b \\
 0 & 0 & \frac{1}{b e} & 0 & 0 & 0 & 0 & -\frac{a}{e} & -\frac{a c}{e} & 0 \\
 0 & -\frac{1}{b d} & 0 & 0 & 0 & 0 & \frac{a}{d} & 0 & -a c & 0 \\
 \frac{1}{b c} & 0 & 0 & 0 & 0 & 0 & a & a d & 0 & 0 \\
 0 & 0 & \frac{1}{e} & 0 & \frac{b}{e} & \frac{b c}{e} & 0 & 0 & 0 & 0 \\
 0 & -\frac{1}{d} & 0 & -\frac{b}{d} & 0 & b c & 0 & 0 & 0 & 0 \\
 \frac{1}{c} & 0 & 0 & -b & -b d & 0 & 0 & 0 & 0 & 0 \\
 1 & c & c d & 0 & 0 & 0 & 0 & 0 & 0 & 0 \\
\end{array}
\right)
\end{align}
The original qubits $A$ and $B$ are
\begin{eqs}
\bZ^A_c=\left(\begin{array}{c}0\\0\\ \bbone_{10\times 10}\\0\end{array}\right),
\bZ^B_c=\left(\begin{array}{c}0\\0\\0\\ \bbone_{10\times 10}\end{array}\right),
\bX^A_c=\left(\begin{array}{c}\bbone_{10\times 10}\\0\\0\\0\end{array}\right),
\bX^B_c=\left(\begin{array}{c}0\\ \bbone_{10\times 10}\\0\\0\end{array}\right)
\end{eqs}
The $5{+}1$D 3F QCA can then be set up as (we do not perform the tensor contractions for clarity)
\begin{eqs}
\overline{\bZ}^A_c&=\bZ^A_c+(\bX^B_{c'}\bd_{c',\delta f}+\bZ^B_{c'}\bM^\dagger_{c''\cup_1 c'}\bd_{c'',\delta f})\bM^\dagger_{c\cup f}\\
\overline{\bZ}^B_c&=\bZ^B_c+(\bX_{c'}^A\bd_{c',\delta f}+\bZ_{c'}^A\bM^\dagger_{c''\cup_1 c'}\bd_{c'',\delta f})\bM_{f\cup c}\\
\overline{\bX}^A_c&=\bX^A_c+(\bX_{c''}^B\bd_{c'',\delta f}+\bZ_{c'''}^B\bM^\dagger_{c''\cup_1  c'''}\bd_{c'',\delta f})\bM^\dagger_{c'\cup f}\bM_{c'\cup_1 c}\\
\overline{\bX}^B_c&=\bX^B_{c}+(\bX_{c''}^A\bd_{c'',\delta f}+\bZ_{c'''}^A\bM^\dagger_{c''\cup_1 c'''}\bd_{c'',\delta f})\bM_{f\cup c'}\bM_{c'\cup_1 c}
\end{eqs}
The proof that the QCA squares to the identity can be performed similarly to the 3+1D version.

\twocolumngrid

\bibliographystyle{utphys}
\bibliography{bibliography.bib}

\end{document}